\documentclass[letterpaper,12pt,oneside]{book}
\usepackage{fullpage}
\usepackage{amsmath,amssymb,amsfonts,amsthm}
\usepackage[all,arc,knot,poly]{xy}
\usepackage{enumerate}
\usepackage{mathtools}
\usepackage{amsxtra}
\usepackage{needspace}
\usepackage{setspace}

\DeclareMathOperator{\Hom}{\operatorname{Hom}}
\DeclareMathOperator{\Ad}{\operatorname{Ad}}
\DeclareMathOperator{\sgn}{\operatorname{sgn}}
\DeclareMathOperator{\Spec}{\operatorname{Spec}}
\DeclareMathOperator{\vsspan}{\operatorname{span}}
\DeclareMathOperator{\Loc}{\operatorname{Loc}}
\DeclareMathOperator{\Conf}{\operatorname{Conf}}
\DeclareMathOperator{\Aut}{\operatorname{Aut}}
\DeclareMathOperator{\im}{\operatorname{im}}
\DeclareMathOperator{\cross}{\operatorname{cross}}
\DeclareMathOperator{\supp}{\operatorname{supp}}

\makeatletter
\def\Ddots{\mathinner{\mkern1mu\raise\p@
\vbox{\kern7\p@\hbox{.}}\mkern2mu
\raise4\p@\hbox{.}\mkern2mu\raise7\p@\hbox{.}\mkern1mu}}
\makeatother

\newenvironment{step}[1][]{
\medskip
\emph{Step #1.}
}
{\medskip}

\theoremstyle{definition}
\newtheorem{theorem}{Theorem}[section]
\newtheorem{definition}[theorem]{Definition}
\newtheorem{lemma}[theorem]{Lemma}
\newtheorem{proposition}[theorem]{Proposition}

\newtheorem{corollary}[theorem]{Corollary}

\newtheorem{example}[theorem]{Example}

\newtheorem{appendixlemma}{Lemma}[chapter]
\newtheorem{appendixproposition}[appendixlemma]{Proposition}
\newtheorem{appendixtheorem}[appendixlemma]{Theorem}

\begin{document}

\doublespacing

\frontmatter
\begin{titlepage}
\begin{center}
  \vspace*{1cm}

  \Huge
  The Geometry of Cluster Varieties \\
  from Surfaces

  \vspace{2cm}

  \normalsize
  A Dissertation \\
  Presented to the Faculty of the Graduate School \\
  of \\
  Yale University \\
  in Candidacy for the Degree of \\
  Doctor of Philosophy
        
  \vfill
        
  by \\
  Dylan Gregory Lucasi Allegretti \\

  \vspace{1cm}

  Dissertation Director: Alexander Goncharov

  \vspace{1cm}

\end{center}
\end{titlepage}


\newpage

\begin{center}
  \textbf{Abstract}
  \addcontentsline{toc}{chapter}{Abstract}

  \vspace{1cm}

  \textbf{The Geometry of Cluster Varieties from Surfaces}
  
  \vspace{1cm}

  Dylan Gregory Lucasi Allegretti \\
  
  \vspace{1cm}
\end{center}

Cluster varieties are geometric objects introduced by Fock and Goncharov. They have recently found applications in several areas of mathematics and mathematical physics. The goal of this thesis is to study the geometry of a large class of cluster varieties associated to compact oriented surfaces with boundary.

The main original contribution of this thesis is to develop the properties of a particular kind of cluster variety called the \emph{symplectic double}. Let $S$ be a compact oriented surface with boundary together with finitely many marked points on its boundary, and let~$S^\circ$ denote the same surface equipped with the opposite orientation. We consider the surface $S_{\mathcal{D}}$ obtained by gluing $S$ and $S^\circ$ along corresponding boundary components. We show that the symplectic double is birational to a certain moduli space of local systems associated to this surface $S_{\mathcal{D}}$. We define a version of the notion of measured lamination on~$S_{\mathcal{D}}$ and prove that the space of all such laminations is a tropicalization of the symplectic double. We describe a canonical map from this space of laminations into the algebra of rational functions on the symplectic double. There is an explicit formula expressing this map in terms of special polynomials called $F$-polynomials from the theory of cluster algebras.

The second main contribution of this thesis is a proof of Fock and Goncharov's duality conjectures for quantum cluster varieties associated to a disk with finitely many marked points on its boundary. These duality conjectures identify a canonical set of elements in the quantized algebra of functions on a cluster variety satisfying a number of special properties. The results presented here grew out of joint work with Hyun~Kyu~Kim on quantum cluster varieties associated to punctured surfaces.

\newpage

\begin{center}
  \textbf{Acknowledgements}
  \addcontentsline{toc}{chapter}{Acknowledgements}
\end{center}

Most of all, I would like to express my gratitude to my advisor, Alexander~Goncharov, for years of guidance and instruction. His mentorship exposed me to some of the most exciting modern developments in mathematics and gave me the skills I need to contribute to mathematical knowledge. The time and attention he gave me have had a huge impact on my development as a mathematician.

I would like to thank the rest of the mathematics faculty at Yale~University for providing a stimulating environment for research. In particular, I would like to thank Igor~Frenkel, whose vision of mathematics has been an inspiration to me throughout graduate school.

I have benefitted greatly from interactions with my fellow graduate students. I thank Linhui~Shen for many helpful discussions and for contributing several important technical ideas to my research. I thank Hyun~Kyu~Kim for a wonderful collaboration that led to some of the results presented in this thesis. I have also learned a great deal from conversations with Efim~Abrikosov, Ivan~Ip, and Daping~Weng.

Finally, I would like to thank the many mathematicians I have met outside Yale~University at conferences and seminars. In particular, I thank Sergey~Fomin and Andrew~Neitzke for their support and interest in my work.

\tableofcontents

\mainmatter
\chapter{Introduction}
\label{ch:Introduction}

This thesis studies objects called cluster varieties, which are geometric counterparts of the cluster algebras of Fomin and Zelevinsky~\cite{FZI}. Cluster varieties were introduced by Fock and Goncharov in a series of recent papers~\cite{IHES,ensembles,dual,dilog}, and they have already found applications in many different parts of mathematics and mathematical physics.

The central example studied in this thesis is the \emph{cluster symplectic variety} or \emph{symplectic double}. This is a particular kind of cluster variety which carries a natural symplectic form and plays a key role in the quantization of cluster varieties~\cite{dilog}. In this introductory chapter, we will motivate the definition of the symplectic double and summarize the main results of this thesis. As we will see, the formulas used to define the symplectic double come from a remarkable function called the quantum dilogarithm.

\section{Origins of cluster varieties}

\subsection{Seeds and mutations}

We begin with some elementary ideas related to quivers. Recall that a quiver is simply a directed graph. It consists of a set $Q_0$ (the set of \emph{vertices}), a set $Q_1$ (the set of \emph{arrows}), and maps $s:Q_1\rightarrow Q_0$ and $t:Q_1\rightarrow Q_0$ taking an edge to its \emph{source} and \emph{target}, respectively. We typically display an arrow diagrammatically as 
\[
\xymatrix{
s(\alpha) \ar[rr]^{\alpha} & & t(\alpha).
}
\]

A \emph{loop} in a quiver $Q$ is an arrow $\alpha$ whose source and target coincide. A \emph{2-cycle} is a pair of distinct arrows $\alpha$ and $\beta$ such that the target of $\alpha$ is the source of $\beta$ and vice versa. A quiver is said to be \emph{finite} if the sets $Q_0$ and $Q_1$ are both finite.

From now on, all of the quivers that we discuss will be finite quivers with no loops or 2-cycles. Such quivers are determined up to isomorphism by skew-symmetric matrices.

\begin{definition}
If $Q$ is a quiver and $i$ and $j$ are vertices of $Q$, then we define
\[
\varepsilon_{ij}=|\{\text{arrows from $j$ to $i$}\}| - |\{\text{arrows from $i$ to $j$}\}|.
\]
\end{definition}

Abstracting from this situation, we arrive at the notion of a seed.

\begin{definition}
A \emph{seed} $\mathbf{i}=(I,\varepsilon_{ij})$ consists of 
\begin{enumerate}
\item A finite set $I$.
\item A skew-symmetric integer matrix $\varepsilon_{ij}$ ($i$,~$j\in I$).
\end{enumerate}
\end{definition}

The notion of a seed is an important ingredient in the definition of a cluster variety or cluster algebra. If $\mathbf{i}=(I,\varepsilon_{ij})$ is a seed, then we can consider an associated lattice 
\[
\Lambda=\mathbb{Z}[I].
\]
This lattice has a basis $\{e_i\}$ given by $e_i=\{i\}$ for $i\in I$ and a $\mathbb{Z}$-valued skew-symmetric bilinear form $(\cdot,\cdot)$ given on basis elements by 
\[
(e_i,e_j)=\varepsilon_{ij}.
\]
Thus we have the following equivalent definition of a seed.

\begin{definition}
\label{def:altseed}
A \emph{seed} $\mathbf{i}=(\Lambda,\{e_i\},(\cdot,\cdot))$ consists of 
\begin{enumerate}
\item A lattice $\Lambda$.
\item A basis $\{e_i\}$ for $\Lambda$.
\item A $\mathbb{Z}$-valued skew-symmetric bilinear form $(\cdot,\cdot)$ on $\Lambda$.
\end{enumerate}
\end{definition}

Let us now return to the quivers that motivated the definition of a seed. If $k$ is any vertex of a quiver~$Q$, then there is a natural operation that we can perform on~$Q$ to get a new quiver. This operation is easiest to understand in the special case where $k$ is a source (a vertex with no incoming arrows) or a sink (a vertex with no outgoing arrows). In either of these special cases, we can simply reverse the direction of all arrows incident to~$k$. More generally, one has the following operation on quivers.

\begin{definition}
\label{def:quivermutation}
Let $k$ be a vertex of a quiver $Q$. Then we define a new quiver $\mu_k(Q)$, called the quiver obtained by \emph{mutation} in the direction~$k$, as follows.
\begin{enumerate}
\item For each pair of arrows $i\rightarrow k\rightarrow j$, we add a new arrow $i\rightarrow j$.
\item Reverse all arrows incident to~$k$.
\item Remove the arrows from a maximal set of pairwise disjoint 2-cycles.
\end{enumerate}
\end{definition}

The last item in this definition means, for example, that we replace the diagram $\xymatrix{i \ar@< 4pt>[r] \ar@< -4pt>[r] & j \ar[l]}$ by $\xymatrix{i \ar[r] & j}$.

\begin{example}
The diagrams below illustrate the steps of~Definition~\ref{def:quivermutation}.
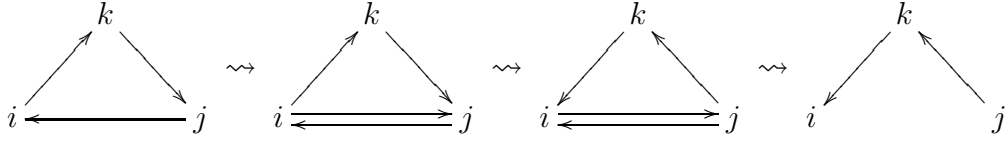
\[
\vcenter{
\xymatrix{
& k \ar[rd] \\
i \ar[ru] & & j \ar[ll]
}
}
\leadsto
\vcenter{
\xymatrix{
& k \ar[rd] \\
i \ar[ru] \ar@< 2pt>[rr] & & j \ar@<2pt>[ll]
}
}
\leadsto
\vcenter{
\xymatrix{
& k \ar[ld] \\
i \ar@< 2pt>[rr] & & j \ar@<2pt>[ll] \ar[lu]
}
}
\leadsto
\vcenter{
\xymatrix{
& k \ar[ld] \\
i & & j \ar[lu]
}
}
\]
\end{example}

It is natural to ask how the associated seed changes when we perform a mutation at some vertex of a quiver. To answer this question, first note that the set of vertices of a quiver coincides with the set of vertices of any quiver obtained by mutation. The change in $\varepsilon_{ij}$ is described by the following proposition.

\begin{proposition}
\label{prop:matrixmutation}
A mutation in the direction~$k$ changes the matrix~$\varepsilon_{ij}$ to the matrix
\[
\varepsilon_{ij}'=
\begin{cases}
-\varepsilon_{ij} & \mbox{if } k\in\{i,j\} \\
\varepsilon_{ij}+\frac{|\varepsilon_{ik}|\varepsilon_{kj}+\varepsilon_{ik}|\varepsilon_{kj}|}{2} & \mbox{if } k\not\in\{i,j\}.
\end{cases}
\]
\end{proposition}

The proof of this proposition is straightforward and will be omitted. Using this result, we can define mutations at the level of seeds.

\begin{definition}
Let $\mathbf{i}=(I,\varepsilon_{ij})$ be a seed and $k\in I$. Then we define a new seed $\mathbf{i}'=(I',\varepsilon_{ij}')$ called the seed obtained by \emph{mutation} in the direction $k$ by setting $I'=I$, and defining $\varepsilon_{ij}'$ by the above formula.
\end{definition}

Naturally, there is also a notion of mutation for seeds in the sense of Definition~\ref{def:altseed}. For any integer $n$, let us write $[n]_+=\max(0,n)$.

\begin{definition}
\label{def:altmutation}
Let $\mathbf{i}=(\Lambda,\{e_i\},(\cdot,\cdot))$ be a seed and $e_k$ a basis vector. Then we define a new seed $\mathbf{i}'=(\Lambda',\{e_i'\},(\cdot,\cdot)')$ called the seed obtained by \emph{mutation} in the direction of $e_k$. It is given by $\Lambda'=\Lambda$, $(\cdot,\cdot)'=(\cdot,\cdot)$, and 
\[
e_i'=
\begin{cases}
-e_k & \mbox{if } i=k \\
e_i+[\varepsilon_{ik}]_+e_k & \mbox{if } i\neq k.
\end{cases}
\]
\end{definition}

\subsection{The quantum dilogarithm}

In the next step of our construction, we employ the following special function.

\begin{definition}
The \emph{quantum dilogarithm} is the formal infinite product
\[
\Psi^q(x)=\prod_{k=1}^\infty(1+q^{2k-1}x)^{-1}.
\]
\end{definition}

The quantum dilogarithm power series first appeared in the 19th century under the names ``$q$-exponential'' and ``infinite Pochhammer symbol''. It was rediscovered in the 1990s by Faddeev and Kashaev, who showed that it satisfies a quantum version of a certain functional equation for the classical dilogarithm function~\cite{FK}. In subsequent works, the quantum dilogarithm was used to quantize the Teichm\"uller space of a surface~\cite{CF,K}.

In this section, we will study the quantum dilogarithm as a function on the following algebra of $q$-commuting variables.

\begin{definition}
\label{def:quantumtorus}
Let $\Lambda$ be a lattice equipped with a $\mathbb{Z}$-valued skew-symmetric bilinear form~$(\cdot,\cdot)$. Then the \emph{quantum torus algebra} is the noncommutative algebra over $\mathbb{Q}[q,q^{-1}]$ generated by variables $Y_v$ ($v\in\Lambda$) subject to the relations 
\[
q^{-(v_1,v_2)}Y_{v_1}Y_{v_2}=Y_{v_1+v_2}.
\]
\end{definition}

This definition allows us to associate to any seed $\mathbf{i}=(\Lambda,\{e_i\}_{i\in I},(\cdot,\cdot))$, a quantum torus algebra $\mathcal{X}_{\mathbf{i}}^q$. The basis $\{e_i\}$ provides a set of generators $X_i^{\pm1}$ given by $X_i=Y_{e_i}$ for this algebra $\mathcal{X}_{\mathbf{i}}^q$. They obey the commutation relations 
\[
X_iX_j=q^{2\varepsilon_{ij}}X_jX_i.
\]
The algebra $\mathcal{X}_{\mathbf{i}}^q$ satisfies the Ore condition from ring theory, so we can form its noncommutative fraction field $\widehat{\mathcal{X}}_{\mathbf{i}}^q$. In addition to associating a quantum torus algebra to every seed, we use the quantum dilogarithm to construct a natural map $\widehat{\mathcal{X}}_{\mathbf{i}'}^q\rightarrow\widehat{\mathcal{X}}_{\mathbf{i}}^q$ whenever two seeds~$\mathbf{i}=(\Lambda,\{e_i\},(\cdot,\cdot))$ and~$\mathbf{i}'=(\Lambda',\{e_i'\},(\cdot,\cdot)')$ are related by a mutation.

\Needspace*{3\baselineskip}
\begin{definition}\mbox{}
\label{def:quantummutationdef}
\begin{enumerate}
\item The automorphism $\mu_k^\sharp:\widehat{\mathcal{X}}_{\mathbf{i}}^q\rightarrow\widehat{\mathcal{X}}_{\mathbf{i}}^q$ is given by conjugation with $\Psi^q(X_k)$:
\[
\mu_k^\sharp=\Ad_{\Psi^q(X_k)}.
\]

\item The isomorphism $\mu_k':\widehat{\mathcal{X}}_{\mathbf{i}'}^q\rightarrow\widehat{\mathcal{X}}_{\mathbf{i}}^q$ is induced by the natural lattice map $\Lambda'\rightarrow\Lambda$.

\item The mutation map $\mu_k^q:\widehat{\mathcal{X}}_{\mathbf{i}'}^q\rightarrow\widehat{\mathcal{X}}_{\mathbf{i}}^q$ is the composition $\mu_k^q=\mu_k^\sharp\circ\mu_k'$.
\end{enumerate}
\end{definition}

Although conjugation by $\Psi^q(X_k)$ produces a priori a formal power series, this construction in fact provides a map $\widehat{\mathcal{X}}_{\mathbf{i}'}^q\rightarrow\widehat{\mathcal{X}}_{\mathbf{i}}^q$ of skew-fields. The proof of this fact is rather tedious, and the interested reader is referred to Appendix~\ref{ch:DerivationOfTheClassicalMutationFormulas}.

The quantum torus algebra $\mathcal{X}_{\mathbf{i}}^q$ used in this construction can be represented as an algebra of operators on an infinite-dimensional  Hilbert space~$\mathcal{H}_{\mathbf{i}}$. In fact, Fock and Goncharov show that there exist intertwining operators $\mathcal{H}_{\mathbf{i'}}\rightarrow\mathcal{H}_{\mathbf{i}}$ whenever the seeds $\mathbf{i}$ and~$\mathbf{i}'$ are related by a mutation~\cite{dilog}.

\subsection{The quantum symplectic double}

To fully understand the ideas of the previous section, particularly the representations on Hilbert spaces, it is useful embed the construction in a larger one. To do this, note that if  $\mathbf{i}=(\Lambda,\{e_i\},(\cdot,\cdot))$ is any seed, then we can form the ``double'' $\Lambda_{\mathcal{D}}=\Lambda_{\mathcal{D},\mathbf{i}}$ of the lattice $\Lambda$ given by the formula 
\[
\Lambda_{\mathcal{D}}=\Lambda\oplus\Lambda^\vee
\]
where $\Lambda^\vee=\Hom(\Lambda,\mathbb{Z})$. The basis $\{e_i\}$ for $\Lambda$ provides a dual basis $\{f_i\}$ for $\Lambda^\vee$, and hence we have a basis $\{e_i,f_i\}$ for $\Lambda_{\mathcal{D}}$. Moreover, there is a natural skew-symmetric bilinear form $(\cdot,\cdot)_{\mathcal{D}}$ on~$\Lambda_{\mathcal{D}}$ given by the formula 
\[
\left((v_1,\varphi_1),(v_2,\varphi_2)\right)_{\mathcal{D}}=(v_1,v_2)+\varphi_2(v_1)-\varphi_1(v_2).
\]
We can apply the construction of Definition~\ref{def:quantumtorus} to these data to get a quantum torus algebra which we denote $\mathcal{D}_{\mathbf{i}}^q$. If we let $X_i$ and $B_i$ denote the generators associated to the basis elements $e_i$ and $f_i$, respectively, then we have the commutation relations 
\[
X_iX_j=q^{2\varepsilon_{ij}}X_jX_i, \quad B_iB_j=B_jB_i, \quad X_iB_j=q^{2\delta_{ij}}B_jX_i.
\]
We will write $\widehat{\mathcal{D}}_{\mathbf{i}}^q$ for the (noncommutative) fraction field of~$\mathcal{D}_{\mathbf{i}}^q$. The following notations will be important in the sequel:
\[
\mathbb{B}_k^+=\prod_{i|(e_k,e_i)>0}B_i^{(e_k,e_i)}, \quad \mathbb{B}_k^-=\prod_{i|(e_k,e_i)<0}B_i^{-(e_k,e_i)}, \quad \widehat{X}_i=X_i\prod_jB_j^{(e_i,e_j)}.
\]
One can check using the above relations that the elements $X_k$ and $\widehat{X}_k$ commute. Just as before, we have a natural map $\widehat{\mathcal{D}}_{\mathbf{i}'}^q\rightarrow\widehat{\mathcal{D}}_{\mathbf{i}}^q$ whenever $\mathbf{i}$ and $\mathbf{i}'$ are two seeds related by a mutation.

\Needspace*{4\baselineskip}
\begin{definition}\mbox{}
\label{def:doublemutation}
\begin{enumerate}
\item The automorphism $\mu_k^\sharp:\widehat{\mathcal{D}}_{\mathbf{i}}^q\rightarrow\widehat{\mathcal{D}}_{\mathbf{i}}^q$ is given by 
\[
\mu_k^\sharp=\Ad_{\Psi^q(X_k)/\Psi^q(\widehat{X}_k)}.
\]

\item The isomorphism $\mu_k':\widehat{\mathcal{D}}_{\mathbf{i}'}^q\rightarrow\widehat{\mathcal{D}}_{\mathbf{i}}^q$ is induced by the natural lattice map $\Lambda_{\mathcal{D},\mathbf{i}'}\rightarrow\Lambda_{\mathcal{D},\mathbf{i}}$.

\item The mutation map $\mu_k^q:\widehat{\mathcal{D}}_{\mathbf{i}'}^q\rightarrow\widehat{\mathcal{D}}_{\mathbf{i}}^q$ is the composition $\mu_k^q=\mu_k^\sharp\circ\mu_k'$.
\end{enumerate}
\end{definition}

Notice that the generators $X_i^{\pm1}$ span a subalgebra of $\mathcal{D}_{\mathbf{i}}^q$ which is isomorphic to the quantum torus algebra $\mathcal{X}_{\mathbf{i}}^q$ defined previously. Moreover, since $X_k$ and $\widehat{X}_k$ commute, the restriction of~$\mu_k^q$ to this subalgebra coincides with the mutation map from Definition~\ref{def:quantummutationdef}.

Thus we have constructed a collection of quantum torus algebras and maps between their fraction fields which extend the construction of the previous section. The importance of this new construction stems from the fact that the algebra $\mathcal{D}_{\mathbf{i}}^q$ has a \emph{canonical} representation as an algebra of difference operators on a Hilbert space $\mathcal{H}_{\mathbf{i}}$. Whenever two seeds $\mathbf{i}$ and~$\mathbf{i}'$ are related by a mutation, there exists a unitary operator $\mathcal{H}_{\mathbf{i'}}\rightarrow\mathcal{H}_{\mathbf{i}}$ intertwining the representations. Details of this construction can be found in~\cite{dilog}.

\subsection{Classical limit}

The mutation map $\mu_k^q$ described in Definition~\ref{def:doublemutation} involves conjugation by a quantum dilogarithm, so the result of applying this map is a priori a formal power series. One can show (see Appendix~\ref{ch:DerivationOfTheClassicalMutationFormulas}) that in fact this definition provides a map $\widehat{\mathcal{D}}_{\mathbf{i}'}^q\rightarrow\widehat{\mathcal{D}}_{\mathbf{i}}^q$ of skew fields. The proof of this fact yields an explicit formula for the action of $\mu_k^q$ on generators.

\begin{theorem}
\label{thm:introclassicallimit}
In the classical limit $q=1$, the map $\mu_k^q$ is given on the generators $B_i'$ and~$X_i'$ of $\mathcal{D}_{\mathbf{i}'}^q$ by the formulas 
\[
\mu_k^1(B_i') =
\begin{cases}
\frac{X_k\mathbb{B}_k^++\mathbb{B}_k^-}{(1+X_k)B_k} & \mbox{if } i=k \\
B_i & \mbox{if } i\neq k
\end{cases}
\]
and 
\[
\mu_k^1(X_i')=
\begin{cases}
X_k^{-1} & \mbox{if } i=k \\
X_i{(1+X_k^{-\sgn(\varepsilon_{ik})})}^{-\varepsilon_{ik}} & \mbox{if } i\neq k.
\end{cases}
\]
\end{theorem}

The two formulas appearing in this theorem were discovered by Fock and Goncharov~\cite{dilog}. Their motivation was to quantize the Teichm\"uller space of a surface and its higher analogs. Remarkably, these same formulas have appeared in several other seemingly unrelated contexts. Indeed, they arise as a special case of the formulas that Fomin and Zelevinsky used to define cluster algebras with coefficients~\cite{FZIV}, and they have appeared in the work of Kontsevich and Soibelman on wall-crossing~\cite{wallcross} (see also~\cite{Neitzke}).

\section{Summary of main results}

In~\cite{dilog}, Fock and Goncharov used the formulas of Theorem~\ref{thm:introclassicallimit} to glue together split algebraic tori and construct a scheme called the cluster symplectic variety or symplectic double. This scheme is one of the main objects studied in this thesis. Here we will summarize our main results. All of the notation and terminology appearing in this section will be explained more precisely later.

The first main result of this thesis deals with a certain moduli space of local systems associated to a surface. More precisely, let $S$ be a compact oriented surface with finitely many marked points on its boundary, and let $S^\circ$ be the same surface equipped with the opposite orientation. We consider the \emph{double} $S_{\mathcal{D}}$ obtained by gluing $S$ and $S^\circ$ along corresponding boundary components. In~\cite{double}, Fock and Goncharov defined a moduli space $\mathcal{D}_{PGL_m,S}$. Roughly speaking, this space parametrizes $PGL_m$-local systems on the double $S_{\mathcal{D}}$, together with some additional data. In addition, Fock and Goncharov describe a rational map 
\[
\mathcal{D}_{PGL_m,S}\dashrightarrow\mathcal{D}
\]
from the moduli space to the symplectic double associated with the surface~$S$. The first main result of this thesis is the following extension of this result of Fock and Goncharov:

\begin{theorem}
The rational map above is a birational equivalence.
\end{theorem}

This theorem gives a concrete way of thinking about the symplectic double as a moduli space of local systems. It is analogous to the results of~\cite{IHES} which relate other moduli spaces of local systems to cluster varieties.

After proving this result, we discuss how the symplectic double is related to various concepts from geometry. We define a space $\mathcal{D}^+(S)$ which we call the \emph{doubled Teichm\"uller space}. Roughly speaking, it is defined as the Teichm\"uller space of $S_{\mathcal{D}}$ together with a choice of orientation for each component of $\partial S$. In addition, we consider a space $\mathcal{D}(\mathbb{R}_{>0})$ associated to the surface~$S$. It is defined using the same formulas from Theorem~\ref{thm:introclassicallimit} where the variables $B_i$ and $X_i$ are taken to be positive real numbers. We prove the following result:

\begin{theorem}
There is a natural homeomorphism $\mathcal{D}^+(S)\cong\mathcal{D}(\mathbb{R}_{>0})$.
\end{theorem}

Another concept from geometry is the notion of a measured lamination. We consider a space $\mathcal{D}_L(S)$ which parametrizes certain laminations on the surface~$S_{\mathcal{D}}$. We also consider a space $\mathcal{D}(\mathbb{Q}^t)$ which is a tropicalization of the symplectic double. This space is defined using the formulas of Theorem~\ref{thm:introclassicallimit} with the operation of addition replaced by maximum and the operation of multiplication replaced by ordinary addition:
\begin{align*}
x_i' &=
\begin{cases}
-x_k & \mbox{if } i=k \\ 
x_i+\varepsilon_{ki}\max\left(0,\sgn(\varepsilon_{ki})x_k\right) & \mbox{if } i\neq k
\end{cases} \\
b_i' &=
\begin{cases}
\max\biggr(x_k+\sum_{j|\varepsilon_{kj}>0}\varepsilon_{kj}b_j,-\sum_{j|\varepsilon_{kj}<0}\varepsilon_{kj}b_j\biggr)-\max(0,x_k)-b_k & \mbox{if } i=k \\
b_i & \mbox{if } i\neq k.
\end{cases} 
\end{align*}
We prove the following:

\begin{theorem}
There is a natural homeomorphism $\mathcal{D}_L(S)\cong\mathcal{D}(\mathbb{Q}^t)$.
\end{theorem}

We use the above result to describe a certain self-duality of the symplectic double. More precisely, we construct a natural subspace $\mathcal{D}(\mathbb{Z}^t)\subseteq\mathcal{D}(\mathbb{Q}^t)$ and a map 
\[
\mathbb{I}_{\mathcal{D}}:\mathcal{D}(\mathbb{Z}^t)\rightarrow\mathbb{Q}(\mathcal{D})
\]
from this subspace to the algebra of rational functions on the symplectic double associated to the surface $S$. This map is similar to maps studied by Fock and Goncharov for other kinds of cluster varieties.

By construction, the symplectic double has an atlas of coordinate charts such that for each seed one has coordinates $B_i$,~$X_i$~($i\in I$), and these coordinates transform by the formulas from Theorem~\ref{thm:introclassicallimit}. Label the elements of $I$ by the numbers $\{1,\dots,n\}$ so that a general point of $\mathcal{D}$ has coordinates $B_1,\dots,B_n,X_1,\dots,X_n$. These $B_i$ and $X_i$ are rational functions on the symplectic double, and we show that the function $\mathbb{I}_\mathcal{D}(l)$ can be expressed in a particularly nice way as a rational function in the variables $B_i$ and $X_i$ for any $l\in\mathcal{D}(\mathbb{Z}^t)$.

To make these statements more precise, let us introduce some notation. Any point $l\in\mathcal{D}(\mathbb{Z}^t)$ corresponds to a collection of disjoint simple closed curves on~$S_{\mathcal{D}}$ where each homotopy class of loops that lie entirely in $S$ or $S^\circ$ appears at most once. If we cut the surface along the image of $\partial S$, we obtain a collection $\mathcal{C}$ of curves on $S$ and a collection $\mathcal{C}^\circ$ of curves on $S^\circ$.

The illustration below shows an example of a closed curve on the double $S_{\mathcal{D}}$ of a genus-2 surface $S$ with two holes. In this case, the set $\mathcal{C}$ consists of the curves labeled $c_1$ and~$c_3$, while the set $\mathcal{C}^\circ$ consists of the curves labeled $c_2$ and~$c_4$. In Chapter~\ref{ch:DualityOfClusterVarieties}, we will see how to associate, to any element $c$ of $\mathcal{C}$ or $\mathcal{C}^\circ$, a polynomial $F_c$. For curves such as the ones below, which intersect the image of~$\partial S$ in~$S_{\mathcal{D}}$, these $F_c$ are in fact the $F$-polynomials of Fomin and Zelevinsky~\cite{FZIV}. We will prove the following theorem, which expresses the function $\mathbb{I}_\mathcal{D}(l)$ in terms of these polynomials.
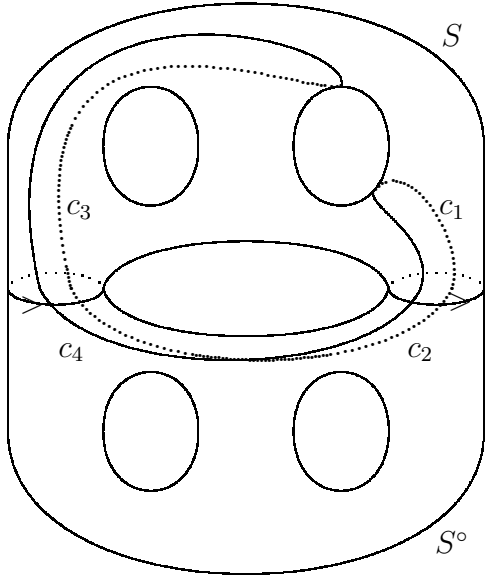
\[
\xy 0;/r.50pc/: 
(-6,-4.5)*\ellipse(3,1){.}; 
(-6,-4.5)*\ellipse(3,1)__,=:a(-180){-}; 
(6,-4.5)*\ellipse(3,1){.}; 
(6,-4.5)*\ellipse(3,1)__,=:a(-180){-}; 
(-13.5,-10)*{>};
(13.5,-9.8)*{>};
(-9,0)*{}="1";
(-3,0)*{}="2";
(3,0)*{}="3";
(9,0)*{}="4";
(-15,0)*{}="A2";
(15,0)*{}="B2";
"1";"2" **\crv{(-9,-5) & (-3,-5)};
"1";"2" **\crv{(-9,5) & (-3,5)};
"3";"4" **\crv{(3,-5) & (9,-5)};
"3";"4" **\crv{(3,5) & (9,5)};
"A2";"B2" **\crv{(-15,12) & (15,12)};
(-9,-9)*{}="A";
(9,-9)*{}="B";
"A";"B" **\crv{(-8,-5) & (8,-5)}; 
(-15,-9)*{}="A1";
(15,-9)*{}="B1";
"B2";"B1" **\dir{-}; 
"A2";"A1" **\dir{-};
"A";"B" **\crv{(-8,-13) & (8,-13)}; 
(-15,-18)*{}="A0";
(15,-18)*{}="B0";
"B1";"B0" **\dir{-}; 
"A1";"A0" **\dir{-};
"A0";"B0" **\crv{(-15,-30) & (15,-30)};
(-9,-18)*{}="5";
(-3,-18)*{}="6";
"5";"6" **\crv{(-9,-23) & (-3,-23)};
"5";"6" **\crv{(-9,-13) & (-3,-13)};
(3,-18)*{}="7";
(9,-18)*{}="8";
"7";"8" **\crv{(3,-23) & (9,-23)};
"7";"8" **\crv{(3,-13) & (9,-13)};
(-11,-9)*{}="V1";
(13,-9)*{}="V2";
(-13,-9)*{}="V3";
(11,-9)*{}="V4";
(6,3.8)*{}="V5";
(8,-3)*{}="V6";
"V1";"V2" **\crv{~*=<2pt>{.} (-8,-15) & (10,-15)};
"V3";"V4" **\crv{(-10,-15) & (8,-15)};
"V3";"V5" **\crv{(-18,13) & (7,7)};
"V1";"V5" **\crv{~*=<2pt>{.} (-16,12) & (4,3)};
"V2";"V6" **\crv{~*=<2pt>{.} (14,-6) & (10,0)};
"V4";"V6" **\crv{(12,-6) & (8,-4)};
(13,7)*{S}; 
(13,-25)*{S^\circ}; 
(13,-4)*{c_1}; 
(11,-13)*{c_2}; 
(-10.5,-4)*{c_3}; 
(-11,-13)*{c_4}; 
\endxy
\]

\begin{theorem}
\label{thm:intromain}
Let $l$ be a point of $\mathcal{D}(\mathbb{Z}^t)$ as above. Then for any choice of seed, we have 
\[
\mathbb{I}_\mathcal{D}(l)=\frac{\prod_{c\in\mathcal{C}^\circ}F_c(\widehat{X}_1,\dots,\widehat{X}_n)}{\prod_{c\in\mathcal{C}}F_c(X_1,\dots,X_n)}B_1^{g_{l,1}}\dots B_n^{g_{l,n}}X_1^{h_{l,1}}\dots X_n^{h_{l,n}}
\]
where the $F_c$ are polynomials, the $g_{l,i}$ and $h_{l,i}$ are integers, and the $\widehat{X}_i$ are given by 
\[
\widehat{X}_i=X_i\prod_j B_j^{\varepsilon_{ij}}.
\]
\end{theorem}

All of the notation appearing in this theorem will be defined precisely below. As we will see, the above formula is in a sense a generalization of Corollary~6.3 in~\cite{FZIV}. If $c$ is a closed loop belonging to~$\mathcal{C}$ or~$\mathcal{C}^\circ$, then $F_c$ does not arise as one of Fomin and Zelevinsky's $F$-polynomials, but Theorem~\ref{thm:intromain} suggests that the $F_c$ should be viewed as ``generalized'' $F$-polynomials. Such generalized $F$-polynomials have appeared previously in the work of Musiker, Schiffler, and Williams~\cite{MW,MSW2} on cluster algebras associated to surfaces.

In addition to the duality map described above, we construct a map 
\[
\mathcal{I}_{\mathcal{D}}:\mathcal{D}(\mathbb{Z}^t)\times\mathcal{D}(\mathbb{Z}^t)\rightarrow\mathbb{Z}
\]
as a kind of intersection pairing for laminations and prove the following.

\begin{theorem}
The map $\mathcal{I}_{\mathcal{D}}$ is the tropicalization of $\mathbb{I}_{\mathcal{D}}$ (in an appropriate sense).
\end{theorem}

In the final part of this thesis, we study cluster varieties associated to a disk with finitely many marked points on its boundary. We first consider an object called the cluster Poisson variety. This is a particular kind of cluster variety defined using the formula for $\mu_k^1(X_i')$ in~Theorem~\ref{thm:introclassicallimit}. We denote this cluster variety by the symbol~$\mathcal{X}$. There is a dual space denoted $\mathcal{A}_0(\mathbb{Z}^t)$ which can be understood as a tropicalization of a certain moduli space associated to the disk. One of the main results of the work of Fock and Goncharov states that there exists a map $\mathbb{I}_{\mathcal{A}}:\mathcal{A}_0(\mathbb{Z}^t)\rightarrow\mathcal{O}(\mathcal{X})$ into the algebra of regular functions on $\mathcal{X}$ satisfying a number of special properties.

By construction, the cluster Poisson variety can be canonically quantized. We will write the expression $\mathcal{O}_q(\mathcal{X})$ for the $q$-deformed algebra of regular functions on the cluster Poisson variety. We prove the following result, which was originally conjectured in~\cite{ensembles}.

\begin{theorem}
\label{thm:introquantumduality}
When $S$ is a disk with finitely many marked points on its boundary, there exists a natural map 
\[
\mathbb{I}_{\mathcal{A}}^q:\mathcal{A}_0(\mathbb{Z}^t)\rightarrow\mathcal{O}_q(\mathcal{X})
\]
which reduces to $\mathbb{I}_{\mathcal{A}}$ in the classical limit $q=1$.
\end{theorem}

A similar map was constructed in~\cite{AK} for cluster varieties associated to punctured surfaces. Interestingly, the proof of Theorem~\ref{thm:introquantumduality} uses several tools that were not present in~\cite{AK}. These include the theory of quantum cluster algebras~\cite{BZq}, the notion of quantum $F$-polynomial from~\cite{Tran}, and the work of Muller on skein algebras~\cite{Muller}. The original construction in~\cite{AK} was instead based on the ``quantum trace map'' introduced by Bonahon and Wong~\cite{BonahonWong}. Recent results of L\^{e} suggest that these two approaches are in fact equivalent~\cite{Le}.

Like the cluster Poisson variety, the symplectic double can be canonically quantized. We will write $\mathcal{D}^q$ for the $q$-deformed algebra of rational functions on the symplectic double.

\begin{theorem}
When $S$ is a disk with finitely many marked points on its boundary, there exists a natural map 
\[
\mathbb{I}_{\mathcal{D}}^q:\mathcal{D}(\mathbb{Z}^t)\rightarrow\mathcal{D}^q
\]
which reduces to $\mathbb{I}_{\mathcal{D}}$ in the classical limit $q=1$.
\end{theorem}

The existence of the map $\mathbb{I}_{\mathcal{D}}^q$ was essentially conjectured by Fock and Goncharov in Conjecture~3.6 of~\cite{dilog}. We will see that the construction of $\mathbb{I}_{\mathcal{D}}^q$ closely parallels the construction of the map $\mathbb{I}_{\mathcal{A}}^q$.

\section{Organization}

In Chapter~\ref{ch:AbstractClusterVarieties}, we use the formulas from this chapter to define three types of cluster varieties: cluster $K_2$-varieties, cluster Poisson varieties, and cluster symplectic varieties. We describe the various additional structures associated with these cluster varieties and the relationships between these structures. We then introduce the notion of a semifield and the notion of points of a cluster variety defined over a semifield.

In Chapter~\ref{ch:ModuliSpacesOfLocalSystems}, we discuss moduli spaces of local systems, one of the main sources of cluster structure. After briefly recalling some terminology related to surfaces and flag varieties, we review the classical results of Fock and Goncharov~\cite{IHES} on the relationship between cluster varieties and moduli spaces of local systems. We define decorated twisted local systems and framed local systems, and we show that the moduli spaces parametrizing these objects are birationally equivalent to the cluster $K_2$- and Poisson varieties, respectively. We then define the symplectic double moduli space from~\cite{double}. We conclude with the author's theorem~\cite{Dmoduli} that this moduli space is birationally equivalent to the symplectic double.

In Chapter~\ref{ch:TeichmullerSpaces}, we discuss three versions of the Teichm\"uller space of a surface. We first define Penner's decorated Teichm\"uller space and prove, following~\cite{Penner} and~\cite{dual}, that this object is identified with the positive real points of the cluster $K_2$-variety. We then define the enhanced Teichm\"uller space and prove, again following~\cite{Penner} and~\cite{dual}, that it is identified with the positive real points of the cluster Poisson variety. Finally, we define a version of the Teichm\"uller space of a doubled surface appearing in the work of Fock and Goncharov~\cite{double} and the author~\cite{Dlam} and prove that it is identified with the positive real points of the symplectic double.

In Chapter~\ref{ch:MeasuredLaminations}, we define three versions of the space of measured laminations on a surface and prove that these spaces are identified with the tropical rational points of the three types of cluster varieties. The metric space completions of these spaces of laminations provide compactifications of the corresponding Teichm\"uller spaces. We identify those laminations that have integral coordinates. The material in this chapter is based on the works of Fock and Goncharov~\cite{IHES,dual} and the author~\cite{Dlam}.

In Chapter~\ref{ch:DualityOfClusterVarieties}, we discuss dualities between the different cluster varieties. We begin with the classical results of Fock and Goncharov which relate the cluster $K_2$- and Poisson varieties~\cite{IHES}. We then review the author's results from~\cite{Dlam} on a self-duality of the symplectic double. We show that the tropical integral points of the symplectic double parametrize an interesting class of rational functions on the symplectic double itself. These are given by an explicit formula involving the $F$-polynomials of Fomin and Zelevinsky~\cite{FZIV}.

In Chapter~\ref{ch:DualityOfQuantumClusterVarieties}, we generalize these results to the quantum setting. We consider quantum cluster varieties associated to a disk with finitely many marked points on its boundary. Using various results from the theory of quantum cluster algebras~\cite{Muller,Tran}, we construct a canonical map from the tropical integral points of the cluster $K_2$-variety into the quantized algebra of regular functions on the cluster Poisson variety. This map satisfies a number of properties conjectured by Fock and Goncharov in~\cite{ensembles}. In addition to this map, we construct a canonical map from the tropical integral points of the symplectic double into its quantized algebra of rational functions. The results of this chapter are based on the paper~\cite{Dquantum}. They extend the author's joint work with Kim~\cite{AK}.

\chapter{Abstract cluster varieties}
\label{ch:AbstractClusterVarieties}

Cluster varieties are a class of schemes obtained by gluing split algebraic tori using certain birational maps called cluster transformations. There are three types of cluster varieties: the cluster $K_2$-varieties, cluster Poisson varieties, and cluster symplectic varieties. In this chapter, we define these objects and the extra structures they carry. In addition, we describe the notion of points of a cluster variety defined over a semifield.

\section{Basic definitions}

In this chapter, we will use a definition of seed which is slightly more general than the one from Chapter~\ref{ch:Introduction}.

\begin{definition}
\label{def:seed}
A \emph{seed} $\mathbf{i}=(I,J,\varepsilon_{ij})$ consists of a finite set $I$, a subset $J\subseteq I$, and a skew-symmetric matrix $\varepsilon_{ij}$ ($i,j\in I$) with integer entries. The matrix $\varepsilon_{ij}$ is called the \emph{exchange matrix}, and the set $I-J$ is called the set of \emph{frozen elements} of $I$.
\end{definition}

In this thesis, all schemes are defined over the field of rational numbers. Thus the \emph{multiplicative group} is the affine scheme $\mathbb{G}_m=\Spec\mathbb{Q}[x,x^{-1}]$. For any field $K$, one has $\mathbb{G}_m(K)=K^*$. A product of multiplicative groups is known as a \emph{split algebraic torus}. Given a seed $\mathbf{i}=(I,J,\varepsilon_{ij})$, we get the following three split algebraic tori:
\[
\mathcal{A}_\mathbf{i} = (\mathbb{G}_m)^{|I|}, \quad
\mathcal{X}_\mathbf{i} = (\mathbb{G}_m)^{|J|}, \quad
\mathcal{D}_\mathbf{i} = (\mathbb{G}_m)^{2|J|}.
\]
We write $\{A_i\}_{i\in I}$ for the natural coordinates on~$\mathcal{A}_\mathbf{i}$ and $\{X_j\}_{j\in J}$ for the natural coordinates on~$\mathcal{X}_\mathbf{i}$. We write $\{B_j, X_j\}_{j\in J}$ for the natural coordinates on $\mathcal{D}_\mathbf{i}$.

\begin{definition}
\label{def:mutation}
Let $\mathbf{i}=(I,J,\varepsilon_{ij})$ be a seed and $k\in J$ a non-frozen element. Then we define a new seed $\mu_k(\mathbf{i})=\mathbf{i}'=(I',J',\varepsilon_{ij}')$, called the seed obtained by \emph{mutation} in the direction~$k$ by setting $I'=I$, $J'=J$, and 
\[
\varepsilon_{ij}'=
\begin{cases}
-\varepsilon_{ij} & \mbox{if } k\in\{i,j\} \\
\varepsilon_{ij}+\frac{|\varepsilon_{ik}|\varepsilon_{kj}+\varepsilon_{ik}|\varepsilon_{kj}|}{2} & \mbox{if } k\not\in\{i,j\}.
\end{cases}
\]
A seed mutation induces birational maps on tori defined by the formulas 
\[
\mu_k^*A_i' =
\begin{cases}
A_k^{-1}\biggr(\prod_{j|\varepsilon_{kj>0}}A_j^{\varepsilon_{kj}} + \prod_{j|\varepsilon_{kj<0}}A_j^{-\varepsilon_{kj}}\biggr) & \mbox{if } i=k \\
A_i & \mbox{if } i\neq k
\end{cases}
\]
and
\[
\mu_k^*X_i'=
\begin{cases}
X_k^{-1} & \mbox{if } i=k \\
X_i{(1+X_k^{-\sgn(\varepsilon_{ik})})}^{-\varepsilon_{ik}} & \mbox{if } i\neq k
\end{cases}
\]
and
\[
\mu_k^*B_i' =
\begin{cases}
\frac{X_k\prod_{j|\varepsilon_{kj>0}}B_j^{\varepsilon_{kj}} + \prod_{j|\varepsilon_{kj<0}}B_j^{-\varepsilon_{kj}}}{(1+X_k)B_k} & \mbox{if } i=k \\
B_i & \mbox{if } i\neq k
\end{cases}
\]
where $\{A_i'\}_{i\in I}$, $\{X_j'\}_{j\in J}$, and $\{B_j',X_j'\}_{j\in J}$ are the coordinates on $\mathcal{A}_{\mathbf{i}'}$, $\mathcal{X}_{\mathbf{i}'}$, and $\mathcal{D}_{\mathbf{i}'}$, respectively.
\end{definition}

\begin{definition}
Two seeds will be called \emph{mutation equivalent} if they are related by a sequence of mutations. We will denote the mutation equivalence class of a seed $\mathbf{i}$ by $|\mathbf{i}|$. A  transformation of the $\mathcal{A}$-, $\mathcal{X}$-, or $\mathcal{D}$-tori obtained by composing the above birational maps is called a \emph{cluster transformation}.
\end{definition}

We can now define cluster varieties by by gluing the tori using cluster transformations.

\Needspace*{2\baselineskip}
\begin{definition} \mbox{}
\begin{enumerate}
\item The \emph{cluster $K_2$-variety} $\mathcal{A}=\mathcal{A}_{|\mathbf{i}|}$ is a scheme obtained by gluing the $\mathcal{A}$-tori for all seeds mutation equivalent to the seed $\mathbf{i}$ using the above birational maps. We will denote by $\mathcal{A}_0$ the subspace given in any cluster coordinate system by the equations $A_i=1$ ($i\in I-J$).

\item The \emph{cluster Poisson variety} $\mathcal{X}=\mathcal{X}_{|\mathbf{i}|}$ is a scheme obtained by gluing the $\mathcal{X}$-tori for all seeds mutation equivalent to the seed $\mathbf{i}$ using the above birational maps.

\item The \emph{cluster symplectic variety} or \emph{symplectic double} $\mathcal{D}=\mathcal{D}_{|\mathbf{i}|}$ is a scheme obtained by gluing the tori $\mathcal{D}_\mathbf{i}$ for all seeds mutation equivalent to the seed $\mathbf{i}$ using the above maps.
\end{enumerate}
\end{definition}

The topic of this thesis is essentially the geometry of these schemes, particularly the symplectic double, which was not so well studied until recently~\cite{double,Dlam,Dmoduli,Dquantum}.

\section{Properties of cluster varieties}

Let us now discuss some of the extra structures on the schemes defined in the previous section. This will help to justify the names of these cluster varieties.

\begin{definition}
For any field $F$, the abelian group $K_2(F)$ is defined as the quotient of~$\Lambda^2 F^*$, the wedge square of the multiplicative group $F^*$ of~$F$, by the subgroup generated by the \emph{Steinberg relations}
\[
(1-x)\wedge x
\]
for $x\in F^*-\{1\}$.
\end{definition}

For example, one can take $F$ to be the field $\mathbb{Q}(X)$ of rational functions on a variety~$X$, and there is a canonical map 
\[
d\log:\Lambda^2\mathbb{Q}(X)^*\rightarrow\Omega_{\log}^2(X)
\]
given by 
\[
f\wedge g\mapsto d\log(f)\wedge d\log(g).
\]
Here $\Omega_{\log}^2(X)$ is the space of 2-forms with logarithmic singularities on~$X$. The map sends the Steinberg relations to zero and therefore induces a well defined map $K_2(X)\coloneqq K_2(\mathbb{Q}(X))\rightarrow \Omega_{\log}^2(X)$.

\begin{definition}
For any seed $\mathbf{i}$, we define an element 
\[
W_{\mathcal{A},\mathbf{i}} = \frac{1}{2}\sum_{i,j\in I}\varepsilon_{ij}A_i\wedge A_j \in \Lambda^2\mathbb{Q}(\mathcal{A}_{\mathbf{i}})^*.
\]
\end{definition}

\begin{proposition}[\cite{dilog}, Corollary~2.18]
This definition provides a canonical element $\mathbf{W}_\mathcal{A}\in K_2(\mathcal{A})$. Its image under the map $d\log$ is a canonical presymplectic structure $\Omega_{\mathcal{A}}=d\log(\mathbf{W}_{\mathcal{A}})$ on the space $\mathcal{A}$.
\end{proposition}

It is this canonical class $\mathbf{W}_{\mathcal{A}}$ that gives $\mathcal{A}$ the name cluster $K_2$-variety. In the context of Teichm\"uller theory, the presymplectic structure $\Omega_{\mathcal{A}}$ is related to the Weil-Petterson symplectic form on Teichm\"uller space.

Turning now to the cluster Poisson variety $\mathcal{X}$, we make the following definition.

\begin{definition}
For any seed $\mathbf{i}$, define a Poisson structure on the torus $\mathcal{X}_{\mathbf{i}}$ by 
\[
\{X_i,X_j\}=\varepsilon_{ij}X_iX_j
\]
where $X_j$~($j\in J$) are the coordinates on $\mathcal{X}_{\mathbf{i}}$.
\end{definition}

\begin{proposition}
Cluster transformations preserve this Poisson structure, and hence $\mathcal{X}$ has a canonical Poisson structure.
\end{proposition}

Finally, let us define an extra structure on the cluster variety $\mathcal{D}$.

\begin{definition}
For any seed $\mathbf{i}$, the torus $\mathcal{D}_{\mathbf{i}}$ has a natural symplectic structure 
\[
\Omega_{\mathbf{i}}=-\frac{1}{2}\sum_{i,j\in J}\varepsilon_{ij}d\log B_i\wedge d\log B_j - \sum_{i\in J}d\log B_i\wedge d\log X_i
\]
and a compatible Poisson structure given by 
\[
\{B_i,B_j\}=0,\quad \{X_i,B_j\}=\delta_{ij}X_iB_j,\quad \{X_i,X_j\}=\varepsilon_{ij}X_iX_j.
\]
where $B_j$,~$X_j$~($j\in J$) are the coordinates on $\mathcal{D}_{\mathbf{i}}$.
\end{definition}

As usual, these formulas define canonical structures on the full cluster variety~$\mathcal{D}$. In fact, one can make a stronger statement.

\begin{proposition}[\cite{dilog}, Proposition~2.14]
There is a canonical class $\mathbf{W}\in K_2(\mathcal{D})$ such that the symplectic structure defined above arises as $\Omega_{\mathcal{D}}=d\log(\mathbf{W})$.
\end{proposition}

In addition to these extra structures, there are various relationships between the different cluster varieties. In the following result, we will assume the seeds used to define $\mathcal{A}$, $\mathcal{X}$, and~$\mathcal{D}$ have no frozen elements.

\begin{theorem}[\cite{dilog}, Theorem~2.3]
\label{thm:doubleproperties}
Assume $I=J$. Then the cluster varieties $\mathcal{A}$, $\mathcal{X}$, and~$\mathcal{D}$ satisfy the following properties.
\begin{enumerate}
\item There is a map $\varphi:\mathcal{A}\times\mathcal{A}\rightarrow\mathcal{D}$ given in any cluster coordinate system by the formulas 
\begin{align*}
\varphi^*(B_i) &= \frac{A_i^\circ}{A_i}, \\
\varphi^*(X_i) &= \prod_j A_j^{\varepsilon_{ij}}
\end{align*}
where $A_i^\circ$ are the coordinates on the second factor of $\mathcal{A}$. It respects the canonical 2-forms in the sense that $\varphi^*\Omega_{\mathcal{D}}=p_-^*\Omega_{\mathcal{A}}-p_+^*\Omega_{\mathcal{A}}$ where $p_-$ and $p_+$ are the projections of $\mathcal{A}\times\mathcal{A}$ onto its two factors.

\item There is a Poisson map $\pi:\mathcal{D}\rightarrow\mathcal{X}\times\mathcal{X}$ given in any cluster coordinate system by the formulas 
\begin{align*}
\pi^*(X_i\otimes1) &= X_i, \\
\pi^*(1\otimes X_i) &= X_i\prod_jB_j^{\varepsilon_{ij}}.
\end{align*}

\item There are commutative diagrams 
\[
\vcenter{
\xymatrix{ 
\mathcal{A}\times\mathcal{A} \ar[rd]^-{\varphi} \ar[dd]_{p\times p} \\
& \mathcal{D}  \ar[ld]^-{\pi} \\
\mathcal{X}\times\mathcal{X}
}
}
\quad
\vcenter{
\xymatrix{ 
\mathcal{X} \ar[r]^{j} \ar[d] & \mathcal{D} \ar[d]^{\pi} \\
\Delta_\mathcal{X} \ar@{^(->}[r] & \mathcal{X}\times\mathcal{X}
}
}
\]
where $\Delta_\mathcal{X}$ denotes the diagonal in $\mathcal{X}\times\mathcal{X}$. Here $p:\mathcal{A}\rightarrow\mathcal{X}$ is a canonical map given in any cluster coordinate system by $p^*X_i=\prod_jA_j^{\varepsilon_{ij}}$, and $j:\mathcal{X}\rightarrow\mathcal{D}$ is an embedding whose image is a Lagrangian subspace given in any cluster coordinate system by the equations $B_j=1$~($j\in J$).

\item There is an involutive isomorphism $\iota:\mathcal{D}\rightarrow\mathcal{D}$ which interchanges the two components of the map $\pi$ and is given in any cluster coordinate system by the formulas 
\begin{align*}
\iota^*(B_i) &= B_i^{-1}, \\
\iota^*(X_i) &= X_i\prod_jB_j^{\varepsilon_{ij}}.
\end{align*}
\end{enumerate}
\end{theorem}

\section{Positivity and semifields}

Like any variety or scheme, one can consider the points of a cluster variety defined over various fields, such as the field of real or complex numbers. However, cluster varieties have an additional property that allows us to consider more exotic constructions. The mutation formulas of Definition~\ref{def:mutation} involve the operations of addition, multiplication, and division, but not subtraction. This positivity property leads us to consider the points of a cluster variety defined over the following kind of algebraic structure.

\begin{definition}
A \emph{semifield} $\mathbb{P}$ is a set equipped with binary operations $+$ and $\cdot$ such that $+$ is commutative and associative, $\mathbb{P}$ is an abelian group under $\cdot$, and the usual distributive law holds: $(a+b)\cdot c=a\cdot c+b\cdot c$ for all~$a$,~$b$,~$c\in\mathbb{P}$.
\end{definition}

\begin{example}
The following examples of semifields will play an important role in our discussion.
\begin{enumerate}
\item Let $\mathbb{P}=\mathbb{R}_{>0}$, the set of positive real numbers. Then $\mathbb{P}$ is a semifield under the usual operations of addition and multiplication.
\item Let $\mathbb{P}=\mathbb{Z}$, $\mathbb{Q}$, or~$\mathbb{R}$. Then $\mathbb{P}$ is a semifield with the operations $\oplus$ and $\otimes$ given by 
\[
a\oplus b=\max(a,b), \quad a\otimes b=a+b
\]
for all $a$, $b\in\mathbb{P}$. The semifields defined in this way are called \emph{tropical semifields} in the works of Fock and~Goncharov. They are denoted $\mathbb{Z}^t$, $\mathbb{Q}^t$, and~$\mathbb{R}^t$.
\item Let $\mathbb{P}=\mathbb{Q}_{\mathrm{sf}}(u_1,\dots,u_n)$ be the set of subtraction-free rational functions in the variables $u_1,\dots,u_n$. This set consists of all rational functions in $u_1,\dots,u_n$ that are expressible as a ratio of two polynomials with positive integral coefficients. It is a semifield whose operations are ordinary addition and multiplication of rational functions.
\item Let $\mathbb{P}=\mathrm{Trop}(y_1,\dots,y_n)$ be the free multiplicative abelian group generated by $y_1,\dots,y_n$ with the auxiliary addition defined by 
\[
\prod_{i=1}^n y_i^{a_i}\oplus \prod_{i=1}^n y_i^{b_i}=\prod_{i=1}^n y_i^{\min(a_i,b_i)}.
\]
This operation makes $\mathbb{P}$ into a semifield which Fomin and Zelevinsky call the \emph{tropical semifield} generated by $y_1,\dots,y_n$.
\end{enumerate}
\end{example}

Given a semifield $\mathbb{P}$ and a split algebraic torus $H$, we can form the set $H(\mathbb{P})=X_*(H)\otimes_\mathbb{Z}\mathbb{P}$. Here $X_*(H)$ is the group of cocharacters of $H$ and we are using the abelian group structure of $\mathbb{P}$. For example, we have the sets $\mathcal{A}_\mathbf{i}(\mathbb{P})$, $\mathcal{X}_\mathbf{i}(\mathbb{P})$, and~$\mathcal{D}_\mathbf{i}(\mathbb{P})$ of $\mathbb{P}$-points of the tori that we used above to construct cluster varieties. Note that the maps 
\[
\psi_{\mathbf{i},\mathbf{i}'}:\mathcal{A}_\mathbf{i}\rightarrow\mathcal{A}_{\mathbf{i}'}
\]
that we used to define the cluster $K_2$-variety induce maps
\[
\psi_{\mathbf{i},\mathbf{i}'*}:\mathcal{A}_\mathbf{i}(\mathbb{P})\rightarrow\mathcal{A}_{\mathbf{i}'}(\mathbb{P})
\]
and similarly for the cluster Poisson and symplectic varieties.

\begin{definition}
The set $\mathcal{A}(\mathbb{P})$ of \emph{$\mathbb{P}$-points} of the cluster $K_2$-variety is defined as the quotient 
\[
\mathcal{A}(\mathbb{P})=\coprod \mathcal{A}_\mathbf{i}(\mathbb{P})/(\text{identifications }\psi_{\mathbf{i},\mathbf{i}'*}).
\]
The sets $\mathcal{X}(\mathbb{P})$ and $\mathcal{D}(\mathbb{P})$ of $\mathbb{P}$-points of the cluster Poisson and symplectic varieties, respectively, are defined analogously.
\end{definition}

A large part of this thesis will be devoted to understanding the geometric objects that arise when one considers the points of a cluster variety defined over a semifield.

\chapter{Moduli spaces of local systems}
\label{ch:ModuliSpacesOfLocalSystems}

Cluster varieties arise naturally as moduli spaces of geometric structures on surfaces. In this chapter, we consider three different moduli spaces of local systems associated to a surface. We will see that these spaces are birationally equivalent to the three types of cluster varieties.

\section{Preliminaries}

\subsection{Decorated surfaces}

We begin by introducing some terminology related to surfaces.

\begin{definition}
A \emph{decorated surface} is a compact oriented surface with boundary together with a finite (possibly empty) collection of marked points on the boundary.
\end{definition}

Given a decorated surface $S$, we can shrink those boundary components without marked points to get a surface $S'$ with punctures and boundary where every boundary component contains at least one marked point.

\begin{definition}
Let $S$ be a decorated surface. An \emph{ideal triangulation} of $S$ is a triangulation of the surface $S'$ described in the preceding paragraph whose vertices are the marked points and the punctures.
\end{definition}

From now on, we will consider only decorated surfaces $S$ that admit an ideal triangulation. Note that in general the sides of a triangle in an ideal triangulation may not be distinct. In this case, the triangle is said to be \emph{self-folded}.

\begin{example}
The illustrations below show an ideal triangulation of a punctured torus and an ideal triangulation of disk with five marked points on its boundary.
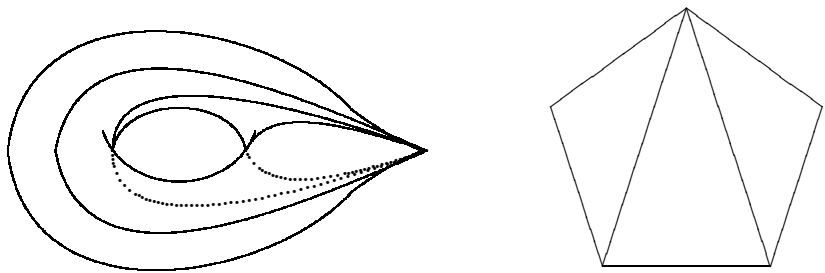
\[
\xy 0;/r.60pc/: 
(-6.5,0)*{}="0";
(-3.5,0)*{}="1";
(3.5,0)*{}="2";
(-4,1)*{}="3";
(4,1)*{}="4";
(-9,0)*{}="A";
(9,2.25)*{}="B1";
(9,-2.25)*{}="B2";
(13,0)*{}="C";
"1";"2" **\crv{(-3,3) & (3,3)};
"3";"4" **\crv{(-2.75,-2.5) & (2.75,-2.5)};
"A";"B1" **\crv{(-8,9) & (4.5,7)};
"A";"B2" **\crv{(-8,-9) & (4.5,-7)};
"B1";"C" **\crv{(8.5,2.5) & (12,0)};
"B2";"C" **\crv{(8.5,-2.5) & (12,0)};
"2";"C" **\crv{(4.5,2) & (7,2)};
"2";"C" **\crv{~*=<2pt>{.} (4.5,-2) & (7,-2)};
"0";"C" **\crv{(-5,8) & (6,3)};
"0";"C" **\crv{(-5,-8) & (6,-3)};
"1";"C" **\crv{(-3.5,5) & (6,2.5)};
"1";"C" **\crv{~*=<2pt>{.} (-3.5,-5) & (6,-2.5)};
\endxy
\qquad\qquad
\xy /l1.5pc/:
{\xypolygon5"A"{~:{(-3,0):}}},
{"A5"\PATH~={**@{-}}'"A2"'"A4"},
\endxy
\]
\end{example}

Let $m\geq2$ be an integer and consider the triangle in $\mathbb{R}^3$ defined by the equation 
\[
x+y+z=m
\]
where $x,y,z\geq0$. We can subdivide this into smaller triangles by drawing the lines $x=p$, $y=p$, and~$z=p$ where $p$  is an integer, $0\leq p\leq m$. An \emph{$m$-triangulation} of the triangle is defined to be a triangulation isotopic to this one. Given an ideal triangulation $T$ of a surface, we can draw a homeomorphic image of an $m$-triangulation within each of its triangles, matching up the vertices on the edges of~$T$. This produces a new triangulation called the \emph{$m$-triangulation} of $T$. The illustration below shows a pair of triangles and the corresponding 3-triangulation.
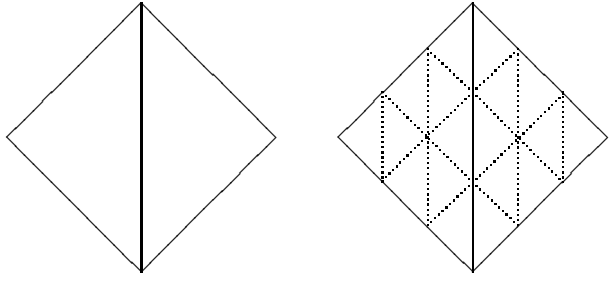
\[
\xy /l1.5pc/:
{\xypolygon4"A"{~:{(2,2):}}},
"A1";"A3" **\dir{-}; 
\endxy
\quad
\quad
\xy /l1.5pc/:
{\xypolygon4"A"{~:{(2,2):}}},
"A1";"A3" **\dir{-}; 
(1,-0.95)*{}="W1"; 
(1,0.95)*{}="W2"; 
(0.05,-1.85)*{}="X1"; 
(-0.9,-0.95)*{}="X2"; 
(0.05,1.85)*{}="Y1"; 
(-0.9,0.95)*{}="Y2"; 
(1.95,-1.85)*{}="U1"; 
(2.9,-0.95)*{}="U2"; 
(1.95,1.85)*{}="V1"; 
(2.9,0.95)*{}="V2"; 
"X1";"Y1" **\dir{.}; 
"X2";"Y2" **\dir{.}; 
"X1";"W1" **\dir{.}; 
"X2";"W2" **\dir{.}; 
"Y2";"W1" **\dir{.}; 
"Y1";"W2" **\dir{.}; 
"U1";"V1" **\dir{.}; 
"U2";"V2" **\dir{.}; 
"U1";"W1" **\dir{.}; 
"U2";"W2" **\dir{.}; 
"V2";"W1" **\dir{.}; 
"V1";"W2" **\dir{.}; 
\endxy
\]

The orientation of the surface $S$ provides an orientation of each triangle in the ideal triangulation~$T$. If $e$ is any edge of the $m$-triangulation that does not lie along an edge of the original ideal triangulation~$T$, then $e$ is parallel to one of the edges of~$T$, and therefore it acquires an orientation.

Given an ideal triangulation $T$ of a decorated surface $S$ and an integer $m\geq2$, we define 
\[
I_m^T=\{\text{vertices of the $m$-triangulation of $T$}\} - \{\text{vertices of $T$}\}
\]
and 
\[
J_m^T=I_m^T-\{\text{vertices on the boundary of $S$}\}.
\]

\begin{definition}
Let $T$ be an ideal triangulation of~$S$ with no self-folded triangles. The \emph{exchange matrix} $\varepsilon_{ij}^T$ ($i,j\in I_m^T\times I_m^T$) is given by the formula 
\[
\varepsilon_{ij}^T=|\{\text{oriented edges from $j$ to $i$}\}| - |\{\text{oriented edges from $i$ to $j$}\}|.
\]
\end{definition}

Thus for a decorated surface with an ideal triangulation~$T$, we can define a seed with $I=I_m^T$, $J=J_m^T$, and $\varepsilon_{ij}=\varepsilon_{ij}^T$.

An important special case of this construction is the case $m=2$. In this case, the $m$-triangulation of $T$ has exactly one vertex in the interior of each edge of $T$, so the set $I$ can be identified with the set of all edges of $T$. An edge of the ideal triangulation~$T$ will be called \emph{external} if it lies along the boundary of $S$, connecting two marked points, and \emph{internal} otherwise. Thus for $m=2$, the set $J$ can be identified with the set of internal edges.

\begin{definition}
If $k$ is an internal edge of the ideal triangulation $T$, then a \emph{flip} at~$k$ is the transformation of $T$ that removes the edge~$k$ and replaces it by the unique different edge that, together with the remaining edges, forms a new ideal triangulation:
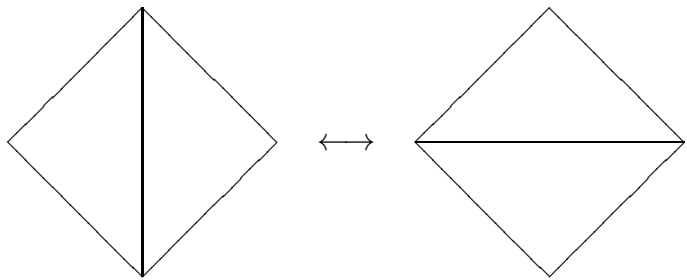
\[
\xy /l1.5pc/:
{\xypolygon4"A"{~:{(2,2):}}},
{"A1"\PATH~={**@{-}}'"A3"},
\endxy
\quad
\longleftrightarrow
\quad
\xy /l1.5pc/:
{\xypolygon4"A"{~:{(2,2):}}},
{"A2"\PATH~={**@{-}}'"A4"}
\endxy
\]
A flip will be called \emph{regular} if none of the triangles above is self-folded.
\end{definition}

\begin{proposition}
Any two isotopy classes of ideal triangulations on a surface are related by a sequence of flips.
\end{proposition}

\subsection{Flag varieties}

The local systems that we are interested in will be equipped with a decoration by flags.

\begin{definition}
A (complete) \emph{flag} in $\mathbb{C}^m$ is a sequence of subspaces 
\[
\{0\}=V_0\subseteq V_1\subseteq V_2 \subseteq\dots\subseteq V_m=\mathbb{C}^m
\]
where $\dim V_i=i$. We denote the set of all such flags by~$\mathcal{B}(\mathbb{C})$.
\end{definition}

If we are given a matrix $X\in GL_m(\mathbb{C})$, then we can write 
\[
X=(\mathbf{x}_1\quad\mathbf{x}_2\quad\dots\quad\mathbf{x}_m)
\]
where $\mathbf{x}_i$ is the $i$th column of~$X$. If we then define $V_i=\vsspan_{\mathbb{C}}\{\mathbf{x}_1,\dots,\mathbf{x}_i\}$, we obtain a flag $\{0\}=V_0\subseteq V_1\subseteq V_2 \subseteq\dots\subseteq V_m=\mathbb{C}^m$. Moreover, if we multiply on the right of~$X$ by a nondegenerate upper triangular matrix, we get 
\[
X
\left( \begin{array}{cccc}
b_{11} & b_{12} & \dots & b_{1m} \\
0 & b_{22} &\dots & b_{2m} \\
\vdots &  & \ddots & \vdots \\
0 & 0 & \dots & b_{mm} \end{array} \right)
=(b_{11}\mathbf{x}_1\quad b_{12}\mathbf{x}_1+b_{22}\mathbf{x}_2 \quad\dots\quad b_{1m}\mathbf{x}_1+\dots+b_{mm}\mathbf{x}_m),
\]
which determines the same flag. Hence there is a well defined map $GL_m(\mathbb{C})/B\rightarrow\mathcal{B}(\mathbb{C})$ where $B$ denotes the group of upper triangular matrices in~$GL_m(\mathbb{C})$. One can check that this map is a bijection, and hence we have an identification $\mathcal{B}(\mathbb{C})=GL_m(\mathbb{C})/B$. Since $B$ contains the center of $GL_m(\mathbb{C})$, we also have an identification $\mathcal{B}(\mathbb{C})=PGL_m(\mathbb{C})/B$ where, by abuse of notation, we write $B$ for the image of the group of upper triangular matrices in~$PGL_m(\mathbb{C})$. This discussion leads to the following definition.

\begin{definition}
The \emph{flag variety} for the group $PGL_m$ is defined as the quotient space $\mathcal{B}=PGL_m/B$ where $B$ is a Borel subgroup of $PGL_m$.
\end{definition}

\begin{example}
For $m=2$, the flag variety $\mathcal{B}=PGL_m/B$ is identified with the projective line $\mathbb{P}^1$.
\end{example}

In addition to the ordinary flag variety, we consider the following object.

\begin{definition}
The \emph{decorated flag variety} of $SL_m$ is isomorphic to the quotient $\mathcal{A}=SL_m/U$ where $U$ is a maximal unipotent subgroup of~$SL_m$.
\end{definition}

\section{Decorated twisted local systems}

We will now use the notion of a flag to define a certain space of $SL_m$-local systems. This construction is easiest to understand when $m$ is odd, so we will consider that case first. Let $S$ be a decorated surface, and let $\mathcal{L}$ be a $SL_m$-local system on~$S$, that is, a principal $SL_m$-bundle with flat connection. There is a natural left action of~$SL_m$ on $\mathcal{A}$, so we can form the associated bundle 
\[
\mathcal{L}_{\mathcal{A}}=\mathcal{L}\times_{SL_m}\mathcal{A}.
\]
Choose points $x_1,\dots,x_r$ on the boundary of~$S$, one on each segment bounded by adjacent marked points. Then the \emph{punctured boundary} of $S$ is defined as the subset of~$\partial S$ given by $\partial^\circ S=\partial S-\{x_1,\dots,x_r\}$. 

\begin{definition}
\label{def:Aspacesimplified}
Let $S$ be a decorated surface and $\mathcal{L}$ an $SL_m$-local system on $S$ for~$m$ odd. A \emph{decoration} for $\mathcal{L}$ is a flat section of the restriction $\mathcal{L}_{\mathcal{A}}|_{\partial^\circ S}$ to the punctured boundary. A \emph{decorated $SL_m$-local system} is a local system $\mathcal{L}$ together together with a decoration. The space of all decorated $SL_m$-local systems on $S$ is denoted~$\mathcal{A}_{SL_m,S}$.
\end{definition}

For general $m$, the definition of $\mathcal{A}_{SL_m,S}$ is more subtle. In the general case, we consider local systems, not on the surface~$S$, but on the punctured tangent bundle~$T'S$, that is, the tangent bundle of~$S$ with the zero section removed. For any point $y\in S$, we have $T_yS\cong\mathbb{R}^2$. Thus $T_y'S=T_yS-0$ is homotopy equivalent to a circle, and we have 
\[
\pi_1(T_y'S,x)\cong\mathbb{Z}
\]
for any choice of basepoint $x\in T_y'S$. Let $\sigma_S$ denote a generator of this fundamental group. It is well defined up to a sign. By abuse of notation, we will also write $\sigma_S$ for the image of this generator under the inclusion $\pi_1(T_y'S,x)\hookrightarrow \pi_1(T'S,x)$. The group $\pi_1(T'S,x)$ fits into a short exact sequence 
\[
\xymatrix{ 
1 \ar[r] & \mathbb{Z} \ar[r] & \pi_1(T'S,x) \ar[r] & \pi_1(S,y) \ar[r] & 1
}
\]
where the group $\mathbb{Z}$ is identified with the central subgroup of $\pi_1(T'S,x)$ generated by~$\sigma_S$.

\begin{definition}
A \emph{twisted $SL_m$-local system} $\mathcal{L}$ on $S$ is an $SL_m$-local system on the punctured tangent bundle $T'S$ with monodromy $(-1)^{m-1}e$ around $\sigma_S$ where $e$ is the identity in~$SL_m$. A \emph{decoration} for $\mathcal{L}$ is a locally constant section of the restriction of $\mathcal{L}_{\mathcal{A}}$ to the lifted punctured boundary. A \emph{decorated twisted $SL_m$-local system} is a twisted local system $\mathcal{L}$ together together with a decoration. The space of all decorated twisted $SL_m$-local systems on $S$ is denoted~$\mathcal{A}_{SL_m,S}$.
\end{definition}

Observe that the element $(-1)^{m-1}e$ has order two, and therefore this definition does not depend on the choice of generator $\sigma_S$. It reduces to Definition~\ref{def:Aspacesimplified} when $m$ is odd.

A twisted $SL_m$-local system on $S$ is thus the same thing as a homomorphism $\rho:\pi_1(T'S,x)\rightarrow SL_m$, considered up to conjugation, such that $\rho(\sigma_S)=(-1)^{m-1}e$. If we define $\bar{\pi}_1(T'S,x)$ to be the quotient of $\pi_1(T'S,x)$ by the central subgroup $2\mathbb{Z}$ generated by the element~$\sigma_S^2$, then we obtain a new central extension 
\[
\xymatrix{ 
1 \ar[r] & \mathbb{Z}/2\mathbb{Z} \ar[r] & \bar{\pi}_1(T'S,x) \ar[r] & \pi_1(S,y) \ar[r] & 1.
}
\]
Let $\bar{\sigma}_S$ be the order two element in the quotient group $\mathbb{Z}/2\mathbb{Z}$. Then we can think of a twisted local system as a representation $\rho:\bar{\pi}_1(T'S,x)\rightarrow SL_m$, considered modulo conjugation, with the property that $\rho(\bar{\sigma}_S)=(-1)^{m-1}e$.

In~\cite{IHES}, Fock and Goncharov show how to associate, to a general point of $\mathcal{A}_{SL_m,S}$ and ideal triangulation $T$ of $S$, a collection of coordinates $A_i$ indexed by the set $I=I_m^T$. These coordinates determine a point of the torus $\mathcal{A}_{\mathbf{i}}$ where $\mathbf{i}$ is the seed corresponding to the triangulation~$T$. In fact, Fock and Goncharov prove the following result:

\begin{proposition}[\cite{IHES}, Theorem~1.17]
\label{prop:Aisom}
There is a birational map 
\[
\mathcal{A}_{SL_m,S}\dashrightarrow\mathcal{A}_{\mathbf{i}}
\]
where $\mathbf{i}$ is the seed corresponding to an ideal triangulation $T$ of~$S$. Let $T$ and $T'$ be ideal triangulations of~$S$ related by a sequence of regular flips, and let $\mathbf{i}$ and $\mathbf{i}'$ be the corresponding seeds. Then the transition map $\mathcal{A}_{\mathbf{i}}\dashrightarrow\mathcal{A}_{\mathbf{i}'}$ is the cluster transformation used to glue these tori in the cluster $K_2$-variety.
\end{proposition}

Since the torus $\mathcal{A}_{\mathbf{i}}$ is open in the cluster $K_2$-variety, we have the following immediate consequence.

\begin{theorem}
There is a natural birational equivalence of the moduli space $\mathcal{A}_{SL_m,S}$ and the cluster $K_2$-variety associated to the surface $S$.
\end{theorem}

\section{Framed local systems}

Let $\mathcal{L}$ be a $PGL_m$-local system. There is a natural left action of the group $PGL_m$ on the flag variety $\mathcal{B}$, and so we can form the associated bundle 
\[
\mathcal{L}_{\mathcal{B}}=\mathcal{L}\times_{PGL_m}\mathcal{B}.
\]
We will consider $PGL_m$-local systems on $S$ with some additional data.

\begin{definition}
Let $S$ be a decorated surface and $\mathcal{L}$ a $PGL_m$-local system on $S$. A \emph{framing} for $\mathcal{L}$ is a flat section of the restriction $\mathcal{L}_{\mathcal{B}}|_{\partial^\circ S}$ to the punctured boundary. A \emph{framed $PGL_m$-local system} is a local system $\mathcal{L}$ together together with a framing. The space of all framed $PGL_m$-local systems on $S$ is denoted~$\mathcal{X}_{PGL_m,S}$.
\end{definition}

There is an alternative way of thinking about framed local systems that will be useful in what follows. Let $S$ be a decorated surface as before, and let $S'$ be the punctured surface obtained by deleting all marked points and shrinking those boundary components without marked points. Fix a complete, finite-area hyperbolic metric on this surface~$S'$ so that $\partial S'$ is totally geodesic. Then its universal cover can be identified with a subset of the hyperbolic plane~$\mathbb{H}$ with totally geodesic boundary.

\begin{definition}
The punctures and deleted marked points on~$S'$ give rise to a set of points on the boundary~$\partial\mathbb{H}$. We call this the \emph{Farey set} and denote it by~$\mathcal{F}_{\infty}(S)$.
\end{definition}

The action of $\pi_1(S)$ by deck transformations on the universal cover gives rise to an action of $\pi_1(S)$ on this set~$\mathcal{F}_{\infty}(S)$. Recall that a $PGL_m(\mathbb{C})$-local system can be viewed as a homomorphism $\rho:\pi_1(S)\rightarrow PGL_m(\mathbb{C})$, modulo the action of $PGL_m(\mathbb{C})$ by conjugation. The following result gives a characterization of framed local systems emphasizing this monodromy representation.

\begin{proposition}[\cite{IHES}, Lemma~1.1]
\label{prop:Xconfig}
Consider a pair $(\rho,\psi)$ where $\rho:\pi_1(S)\rightarrow PGL_m(\mathbb{C})$ is a group homomorphism, and $\psi:\mathcal{F}_{\infty}(S)\rightarrow\mathcal{B}(\mathbb{C})$ is a $(\pi_1(S),\rho)$-equivariant map from the set described above into the flag variety. That is, for any $\gamma\in\pi_1(S)$, we have 
\[
\psi(\gamma c)=\rho(\gamma)\psi(c).
\]
The moduli space $\mathcal{X}_{PGL_m,S}(\mathbb{C})$ parametrizes such pairs modulo the action of~$PGL_m(\mathbb{C})$.
\end{proposition}

In~\cite{IHES}, Fock and Goncharov show how to associate, to a general point of $\mathcal{X}_{PGL_m,S}$ and ideal triangulation $T$ of $S$, a collection of coordinates $X_j$ indexed by the set $J=J_m^T$ described above. These coordinates determine a point of the torus $\mathcal{X}_{\mathbf{i}}$ where $\mathbf{i}$ is the seed corresponding to the triangulation~$T$. In fact, Fock and Goncharov prove the following result:

\begin{proposition}[\cite{IHES}, Theorem~1.17]
\label{prop:Xisom}
There is a birational map 
\[
\mathcal{X}_{PGL_m,S}\dashrightarrow\mathcal{X}_{\mathbf{i}}
\]
where $\mathbf{i}$ is the seed corresponding to an ideal triangulation $T$ of~$S$. Let $T$ and $T'$ be ideal triangulations of~$S$ related by a sequence of regular flips, and let $\mathbf{i}$ and $\mathbf{i}'$ be the corresponding seeds. Then the transition map $\mathcal{X}_{\mathbf{i}}\dashrightarrow\mathcal{X}_{\mathbf{i}'}$ is the cluster transformation used to glue these tori in the cluster Poisson variety.
\end{proposition}

As an immediate consequence, we have the following.

\begin{theorem}
There is a natural birational equivalence of the moduli space $\mathcal{X}_{PGL_m,S}$ and the cluster Poisson variety associated to the surface $S$.
\end{theorem}

\section{The symplectic double moduli space}

There is a similar moduli space corresponding to the symplectic double. In this case, the moduli space essentially parametrizes geometric structures on a doubled surface.

\begin{definition}
Suppose that $S$ is a decorated surface and $S^\circ$ is the same surface equipped with the opposite orientation. The \emph{double} $S_{\mathcal{D}}$ is obtained by gluing $S$ and~$S^\circ$ along corresponding boundary components and deleting the image of each marked point in the resulting surface.
\end{definition}

Denote by $\widetilde{\Loc}_{SL_m,S}$ the space of twisted $SL_m$-local systems on the surface $S$. Suppose we are given a representation $\rho:\bar{\pi}_1(T'S)\rightarrow SL_m$ corresponding to a point in this space and a homomorphism $\sigma:\pi_1(S)\rightarrow Z(SL_m)$ into the center of $SL_m$. Composing the latter map with the projection $\bar{\pi}_1(T'S)\rightarrow\pi_1(S)$, we get a homomorphism $\tilde{\sigma}:\bar{\pi}_1(T'S)\rightarrow SL_m$, which we can multiply pointwise with $\rho$ to get a new point $\sigma\cdot\rho\in\widetilde{\Loc}_{SL_m,S}$. Thus we have defined an action of the group 
\[
\Delta_{SL_m}=\Hom(\pi_1(S),Z(SL_m))
\]
on the space $\widetilde{\Loc}_{SL_m,S}$. The natural homeomorphism $S\rightarrow S^\circ$ induces an isomorphism $\pi_1(S)\cong\pi_1(S^\circ)$, and therefore we have a diagonal action of $\Delta_{SL_m}$ on the product space 
\[
\widetilde{\Loc}_{SL_m,S}\times\widetilde{\Loc}_{SL_m,S^\circ}.
\]

Let $\mathcal{L}$ be a twisted $SL_m$-local system on a decorated surface $S$. If $\beta$ is a framing for $\mathcal{L}$, then we define an $H$-local subsystem $\mathcal{F}_{\beta}$ over the punctured boundary of~$S$, where $H$ is the Cartan group for~$SL_m$. It is the local subsystem of $\mathcal{L}_{\mathcal{A}}$ obtained by taking the preimage of the section $\beta$ under the natural map $\mathcal{L}_{\mathcal{A}}\rightarrow\mathcal{L}_{\mathcal{B}}$.

We now give the definition of the symplectic double moduli space. As in~\cite{double}, we will assume that the surface~$S$ has no marked points on its boundary. Later we will explain how this construction can be extended to general decorated surfaces.

\begin{definition}[\cite{double}, Definition~2.3]
\label{def:Dspace}
Let $S$ be a decorated surface with no marked points. The moduli space $\mathcal{D}_{PGL_m,S}$ parametrizes the data $(\mathcal{L},\mathcal{L}^\circ,\beta,\beta^\circ,\alpha)$ where 
\begin{enumerate}
\item The pair $(\mathcal{L},\mathcal{L}^\circ)$ is an element of 
\[
\left(\widetilde{\Loc}_{SL_m,S}\times\widetilde{\Loc}_{SL_m,S^\circ}\right)/\Delta_{SL_m}.
\]

\item $\beta$ and $\beta^\circ$ are framings for the $PGL_m$-local systems on $S$ and $S^\circ$ corresponding to $\mathcal{L}$ and $\mathcal{L}^\circ$, respectively.

\item $\alpha$ is an $H$-equivariant map $\mathcal{F}_{\beta}\rightarrow\mathcal{F}_{\beta^\circ}$ of local subsystems. (So in particular these local subsystems are isomorphic.)
\end{enumerate}
\end{definition}

Definition~\ref{def:Dspace} describes the moduli space $\mathcal{D}_{PGL_m,S}$ when there are no marked points on the boundary of~$S$. To define this space for a general decorated surface, we must modify the definition slightly. As before, we consider tuples $(\mathcal{L},\mathcal{L}^\circ,\beta,\beta^\circ,\alpha)$. If we choose a decorated flag in the fiber of $\mathcal{F}_\beta$ over each marked point on $\partial S$, then the map $\alpha$ gives a corresponding choice of decorated flags in the fibers of $\mathcal{F}_{\beta^\circ}$. For a generic choice of flags, Fock and Goncharov's construction provides coordinates $A_i$ and~$A_i^\circ$~($i\in I_m^T-J_m^T$) corresponding to these decorations of~$\mathcal{L}$ and~$\mathcal{L}^\circ$, respectively. They are independent of the ideal triangulation~$T$.

\begin{definition}
\label{def:Dspacemarked}
For any decorated surface $S$, the space $\mathcal{D}_{PGL_m,S}$ parametrizes the data $(\mathcal{L},\mathcal{L}^\circ,\beta,\beta^\circ,\alpha)$ as above where we require $A_i=A_i^\circ$ for $i\in I_m^T-J_m^T$ and any choice of decorated flags.
\end{definition}

To construct coordinates on the moduli space $\mathcal{D}_{PGL_m,S}$, fix an ideal triangulation~$T$ of the surface~$S$ and a general point $\mu\in\mathcal{D}_{PGL_m,S}$. This point $\mu$ determines a framed local system on~$S$. By Proposition~\ref{prop:Xisom}, there is a collection of $X_j$, indexed by the set $J=J_m^T$, which are coordinates of this framed local system.

In addition to these coordinates $X_j$ ($j\in J$), we will define a collection of coordinates $B_j$~($j\in J$) as follows. Let $t$ be any triangle in the ideal triangulation $T$. Then the data defining the point $\mu$ allow us to assign, to each vertex $p$ of this triangle, an invariant flag~$b_p$. For each vertex $p$, we can then choose a decorated flag $a_p$ in the fiber of the projection $\mathcal{A}\rightarrow\mathcal{B}$ over the flag $b_p$. By parallel transporting these decorated flags to a common point in the interior of $t$, we get a well defined point in the configuration space 
\[
\Conf_3(\mathcal{A})=SL_m\backslash\mathcal{A}^3
\]
of triples of affine flags. On the other hand, consider the ideal triangulation $T^\circ$ of~$S^\circ$ corresponding to $T$, and let $t^\circ$ be the triangle in this triangulation corresponding to~$t$. We have associated a decorated flag $a_p$ to each vertex $p$, and the choice of $m$ allows us to associate a decorated flag $a_p^\circ$ to the corresponding vertex of $t^\circ$. We can once again transport these decorated flags to a common point in the interior of $t^\circ$. In this way, we obtain a second point of $\Conf_3(\mathcal{A})$. This construction produces a well defined point in 
\[
\left(\Conf_3(\mathcal{A})\times\Conf_3(\mathcal{A})\right)/H^3
\]
where $H$ denotes the Cartan group of $SL_m$.

By the results of~\cite{IHES}, the first factor $\Conf_3(\mathcal{A})$ is identified with the space of decorated twisted local systems on $t$, and by Proposition~\ref{prop:Aisom}, there is a set of coordinates $A_i$ on this space, parametrized by a set of vertices of the $m$-triangulation of~$t$. Similarly, there are coordinates $A_i^\circ$ on the second factor. For each index $i$, we define 
\[
B_i=\frac{A_i^\circ}{A_i}.
\]
One can show that this ratio is independent of all choices in the construction. Thus we have associated a well defined $B_j$ to each index $j$ in the set $J=J_m^T$.

\begin{proposition}[\cite{double}]
This construction defines a rational map 
\[
\mathcal{D}_{PGL_m,S}\dashrightarrow\mathcal{D}_{\mathbf{i}}
\]
where $\mathbf{i}$ is the seed corresponding to an ideal triangulation $T$ of~$S$. Let $T$ and $T'$ be ideal triangulations of~$S$ related by a sequence of regular flips, and let $\mathbf{i}$ and $\mathbf{i}'$ be the corresponding seeds. Then the transition map $\mathcal{D}_{\mathbf{i}}\dashrightarrow\mathcal{D}_{\mathbf{i}'}$ is the cluster transformation used to glue these tori in the symplectic double.
\end{proposition}

We claim that this map is in fact a birational equivalence. To prove this, we will need the following elementary facts.

\begin{lemma}
\label{lem:splitting}
Let 
\[
\xymatrix{ 
1 \ar[r] & A \ar[r]^{u} & B \ar[r]^{v} & C \ar[r] & 1
}
\]
be a central extension of groups with $\Aut(A)=1$, and let $s:C\rightarrow B$ be a homomorphism such that $v\circ s=1_C$. Then there is a natural isomorphism $A\times C\cong B$, $(a,c)\mapsto u(a)s(c)$. If $s':C\rightarrow B$ is any other homomorphism such that $v\circ s'=1_C$, then there exists $\varphi\in\Hom(C,A)$ such that 
\[
s'(c)=u(\varphi(c))s(c)
\]
for all $c\in C$.
\end{lemma}

\begin{proof}
It is well known that if $s:C\rightarrow B$ is a homomorphism satisfying $v\circ s=1_C$, then the rule $(a,c)\mapsto u(a)s(c)$ defines an isomorphism of a semidirect product $A\rtimes C$ with~$B$. Since $\Aut(A)=1$, the semidirect product is isomorphic to the direct product $A\times C$.

If $s$ and $s'$ are homomorphisms $C\rightarrow B$ such that $v\circ s=v\circ s'=1_C$, then we have 
\[
v(s'(c)s(c)^{-1})=1
\]
for all $c\in C$. By exactness, we have $s'(c)s(c)^{-1}\in\ker v=\im u$, and therefore there exists $\varphi(c)\in A$ such that $u(\varphi(c))=s'(c)s(c)^{-1}$, that is,
\[
s'(c)=u(\varphi(c))s(c)
\]
for all $c\in C$. We claim that the map $\varphi:C\rightarrow A$ defined in this way is a group homomorphism. Indeed, suppose $c_1$,~$c_2\in C$. Since the image of $u$ is central in $B$, we have 
\begin{align*}
u(\varphi(c_1c_2)) &= s'(c_1c_2)s(c_1c_2)^{-1} \\
&= s'(c_1)s'(c_2)s(c_2)^{-1}s(c_1)^{-1} \\
&= s'(c_1)u(\varphi(c_2))s(c_1)^{-1} \\
&= s'(c_1)s(c_1)^{-1}u(\varphi(c_2)) \\
&= u(\varphi(c_1))u(\varphi(c_2)).
\end{align*}
Since $u$ is injective, we can cancel $u$ on both sides of $u(\varphi(c_1c_2))=u(\varphi(c_1)\varphi(c_2))$ to see that $\varphi$ is a homomorphism as claimed. This completes the proof.
\end{proof}

\begin{lemma}
\label{lem:lifting}
Let $S$ be a decorated surface with at least one boundary component. Any homomorphism $\pi_1(S)\rightarrow PGL_m(\mathbb{C})$ can be lifted to a homomorphism $\pi_1(S)\rightarrow SL_m(\mathbb{C})$. Any two lifts are related by the action of an element of the group 
\[
\Delta_{SL_m}=\Hom(\pi_1(S),Z(SL_m(\mathbb{C}))).
\]
\end{lemma}

\begin{proof}
Let $\rho:\pi_1(S)\rightarrow PGL_m(\mathbb{C})$ be a homomorphism. Since $\pi_1(S)$ is free and $\mathbb{C}$ is algebraically closed, we can lift this to a homomorphism $\pi_1(S)\rightarrow SL_m(\mathbb{C})$ by choosing a lift for each generator of $\pi_1(S)$.

Suppose $\rho_1$,~$\rho_2:\pi_1(S)\rightarrow SL_m(\mathbb{C})$ are lifts of $\rho$. Then for every $\gamma\in\pi_1(S)$, the elements $\rho_1(\gamma)$ and $\rho_2(\gamma)$ represent the same class in $PGL_m(\mathbb{C})$, so there exists a nonzero scalar $\lambda_{\gamma}$ such that 
\[
\rho_2(\gamma)=\lambda_{\gamma}\rho_1(\gamma).
\]
Taking determinants of both sides, we see that 
\[
1=\lambda_{\gamma}^m.
\]
The group $Z(SL_m(\mathbb{C}))$ is identified with the group of $m$th roots of unity, so we can define a map $\sigma:\pi_1(S)\rightarrow Z(SL_m(\mathbb{C}))$ by putting $\sigma(\gamma)=\lambda_{\gamma}$. One easily checks that this map is a homomorphism. By construction, we have $\rho_2=\sigma\cdot\rho_1$.
\end{proof}

Let $A=\mathbb{Q}[x_1,\dots,x_k]/(f_1,\dots,f_s)$ and $B=\mathbb{Q}[x_1,\dots,x_l]/(g_1,\dots,g_t)$ be reduced algebras over~$\mathbb{Q}$, and let $X=\Spec A$ and $Y=\Spec B$ be the corresponding affine schemes over~$\mathbb{Q}$.

\begin{lemma}
\label{lem:pointsfield}
Let $\varphi^*:B\rightarrow A$ and $\psi^*:A\rightarrow B$ be ring homomorphisms such that the induced maps $\varphi:X(\mathbb{C})\rightarrow Y(\mathbb{C})$ and $\psi:Y(\mathbb{C})\rightarrow X(\mathbb{C})$ satisfy $\varphi\circ\psi=1_{Y(\mathbb{C})}$ and $\psi\circ\varphi=1_{X(\mathbb{C})}$. Then $X$ and $Y$ are isomorphic as schemes over~$\mathbb{Q}$.
\end{lemma}

\begin{proof}
We can think of the $X(\mathbb{C})$ as the closed subset of $\mathbb{C}^k$ defined by the polynomials $f_1,\dots,f_s\in\mathbb{C}[x_1,\dots,x_k]$ and think of $Y(\mathbb{C})$ as the closed subset of $\mathbb{C}^l$ defined by $g_1,\dots,g_t\in\mathbb{C}[x_1,\dots,x_l]$. The maps $\psi$ and $\varphi$ provide inverse isomorphisms of these closed subsets. Since we assume that $A$ and $B$ are reduced, the Nullstellensatz implies that there is an isomorphism of coordinate rings  
\[
\mathbb{C}[x_1,\dots,x_k]/(f_1,\dots,f_s)\stackrel{\cong}{\rightarrow}\mathbb{C}[x_1,\dots,x_l]/(g_1,\dots,g_t).
\]
This isomorphism restricts to the map $\psi^*:A\rightarrow B$, and its inverse restricts to $\varphi^*:B\rightarrow A$. It follows that $A\cong B$ as rings, and therefore $X$ and $Y$ are isomorphic as schemes over~$\mathbb{Q}$.
\end{proof}

\begin{theorem}
\label{thm:Dbirationaltorus}
Fix an ideal triangulation $T$ of the surface $S$, and let $\mathbf{i}$ be the corresponding seed. There is a rational map 
\[
\mathcal{D}_{\mathbf{i}}\dashrightarrow\mathcal{D}_{PGL_m,S}
\]
from the split algebraic torus $\mathcal{D}_{\mathbf{i}}=\mathbb{G}_m^{2|J|}$ into the symplectic double moduli space. Composing this map with the coordinate functions gives the identity whenever the composition is defined.
\end{theorem}

\begin{proof}
The proof will consist of two steps. In the first step, we will construct an inverse to this map at the level of complex points. The definition of this map will require some choices. In the second part of the proof, we will show that the construction is independent of these choices.

\Needspace*{2\baselineskip}
\begin{step}[1]
Construction of the map
\end{step}

To construct this map, suppose we are given $B_j$,~$X_j\in \mathbb{C}^*$ for~$j\in J$. By Proposition~\ref{prop:Xisom}, we can use the $X_j$ ($j\in J$) to construct a point in $\mathcal{X}_{PGL_m,S}(\mathbb{C})$, that is, a $PGL_m(\mathbb{C})$-local system on~$S$ together with a framing~$\beta$. This local system is represented by some homomorphism $\pi_1(S)\rightarrow PGL_m(\mathbb{C})$.

By Lemma~\ref{lem:lifting}, this homomorphism can be lifted to a homomorphism $\rho_0:\pi_1(S)\rightarrow SL_m(\mathbb{C})$. Recall the central extension 
\[
\xymatrix{ 
1 \ar[r] & \mathbb{Z}/2\mathbb{Z} \ar[r]^-{u} & \bar{\pi}_1(T'S) \ar[r]^-{v} & \pi_1(S) \ar[r] & 1.
}
\]
Since $\pi_1(S)$ is free, we can choose a homomorphism $s:\pi_1(S)\rightarrow\bar{\pi}_1(T'S)$ such that $v\circ s=1$. Once we have chosen this homomorphism, Lemma~\ref{lem:splitting} gives an isomorphism 
\[
\eta:\mathbb{Z}/2\mathbb{Z}\times\pi_1(S)\stackrel{\cong}{\rightarrow}\bar{\pi}_1(T'S).
\]
Consider the homomorphism $\mathbb{Z}/2\mathbb{Z}\times\pi_1(S)\rightarrow SL_m(\mathbb{C})$ defined by the rules $(1,\gamma)\mapsto\rho_0(\gamma)$ and $(\bar{\sigma}_S,1)\mapsto(-1)^{m-1}e$. Abusing notation, we will denote this map also by~$\rho_0$. Then the composition $\rho=\rho_0\circ\eta^{-1}:\bar{\pi}_1(T'S)\rightarrow SL_m(\mathbb{C})$ sends $\bar{\sigma}_S$ to the element $(-1)^{m-1}e$ and hence defines a twisted local system $\mathcal{L}$ on the surface~$S$.

Choose a complete, finite area hyperbolic metric with totally geodesic boundary on~$S'$ so that its universal cover is identified with a subset of the hyperbolic plane. By Proposition~\ref{prop:Xconfig}, the framing $\beta$ is the same thing as an equivariant map $\psi:\mathcal{F}_{\infty}(S)\rightarrow\mathcal{B}(\mathbb{C})$ that associates a flag to each point of $\mathcal{F}_{\infty}(S)$. Here $\mathcal{F}_{\infty}(S)$ is defined as the set of vertices of a lift $\tilde{T}$ of the ideal triangulation~$T$ to the hyperbolic plane.

Let $t$ be any triangle of this lifted triangulation. Then there is a flag associated to each vertex of $t$, and we can choose a decorated flag that projects to this flag under the natural map $\mathcal{A}(\mathbb{C})\rightarrow\mathcal{B}(\mathbb{C})$. This gives a triple of decorated flags and hence an element of $\mathcal{A}_{SL_m,t}(\mathbb{C})$. By Proposition~\ref{prop:Aisom}, there are coordinates $A_i$ on this space, corresponding to vertices $i$ in the $m$-triangulation of $t$. Define a new collection of coordinates by putting 
\[
A_i^\circ=B_iA_i
\]
for each vertex $i$.

Consider a second copy of the hyperbolic plane with a triangulation $\tilde{T}^\circ$ obtained by reversing the orientation of $\tilde{T}$. There is a triangle $t^\circ$ in this triangulation corresponding to the triangle~$t$, and we can use the numbers $A_i^\circ$ to construct a configuration of three decorated flags associated to the vertices of this triangle. In this way, we get an $H$-equivariant correspondence between decorated flags associated to the vertices of~$\tilde{T}$ and decorated flags associated to the vertices of~$\tilde{T}^\circ$.

Let $\gamma^\circ\in\pi_1(S^\circ)$ and let $\gamma$ be the corresponding element of $\pi_1(S)$. We can view this element $\gamma$ as a deck transformation of the universal cover of $S$, and it maps any triangle~$t$ of~$\tilde{T}$ to a new triangle $t'$. If we choose a decorated flag at each vertex of~$t$ which projects to the ordinary flag at this vertex, then the $SL_m(\mathbb{C})$ transformation $\rho_0(\gamma)$ gives a new triple of decorated flags at the vertices of~$t'$. Consider the corresponding triangles~$t^\circ$ and~$(t')^\circ$ in the triangulation~$\tilde{T^\circ}$. By the construction described above, there are corresponding decorated flags at the vertices of~$t^\circ$ and~$(t')^\circ$. We define $\rho_0^\circ(\gamma^\circ)$ to be the unique element of~$SL_m(\mathbb{C})$ that takes the triple of decorated flags at the vertices of~$t^\circ$ to the triple at the vertices of~$(t')^\circ$.

This defines a homomorphism $\rho_0^\circ:\pi_1(S^\circ)\rightarrow SL_m(\mathbb{C})$. From the isomorphism $\eta$ considered above, we get an isomorphism 
\[
\bar{\pi}_1(T'S^\circ)\cong\mathbb{Z}/2\mathbb{Z}\times\pi_1(S^\circ).
\]
Thus we can construct a representation $\rho^\circ:\bar{\pi}_1(T'S^\circ)\rightarrow SL_m(\mathbb{C})$ as before, and this defines a twisted $SL_m(\mathbb{C})$-local system $\mathcal{L}^\circ$ on the surface~$S^\circ$. The construction also gives an equivariant assignment of a flag to each vertex of $\tilde{T}^\circ$, so by Proposition~\ref{prop:Xconfig}, we have a framing of the associated $PGL_m(\mathbb{C})$-local system. Finally, the correspondence between decorated flags at the vertices of~$\tilde{T}$ and~$\tilde{T}^\circ$ gives the gluing datum~$\alpha$ appearing in~Definition~\ref{def:Dspace}.

Thus we have associated a point in $\mathcal{D}_{PGL_m,S}(\mathbb{C})$ to a collection of elements $B_j$,~$X_j\in \mathbb{C}^*$ for~$j\in J$. By construction, this map composes with the coordinate functions to give the identity whenever this composition is defined. This completes Step~1 of the proof.

\Needspace*{2\baselineskip}
\begin{step}[2]
Independence of choices
\end{step}

In the above construction, we had to choose a lift $\rho_0:\pi_1(S)\rightarrow SL_m(\mathbb{C})$ of a certain map $\pi_1(S)\rightarrow PGL_m(\mathbb{C})$. It follows from Lemma~\ref{lem:lifting} that the resulting pair $(\rho_0,\rho_0^\circ)$ of representations $\pi_1(S)\rightarrow SL_m(\mathbb{C})$ is well defined modulo the diagonal action of $\Delta_{SL_m}$.

In addition to this choice of lift, we had to choose a homomorphism $s:\pi_1(S)\rightarrow\bar{\pi}_1(T'S)$ with the property that $v\circ s=1$. By~Lemma~\ref{lem:splitting}, this choice provides an isomorphism $\eta:\mathbb{Z}/2\mathbb{Z}\times\pi_1(S)\rightarrow\bar{\pi}_1(T'S)$ given by the formula 
\[
\eta(x,\gamma)=u(x)s(\gamma).
\]
If we choose a different homomorphism $s':\pi_1(S)\rightarrow\bar{\pi}_1(T'S)$ satisfying $v\circ s'=1$, then we get a different isomorphism $\eta':\mathbb{Z}/2\mathbb{Z}\times\pi_1(S)\rightarrow\bar{\pi}_1(T'S)$ given by 
\[
\eta'(x,\gamma) = u(x)s'(\gamma) = u(\varphi(\gamma))u(x)s(\gamma)
\]
for some $\varphi\in\Hom(\pi_1(S),\mathbb{Z}/2\mathbb{Z})$. We constructed the twisted local system $\mathcal{L}$ above by describing a homomorphism $\rho:\bar{\pi}_1(T'S)\rightarrow SL_m(\mathbb{C})$. This was defined as a composition $\rho=\rho_0\circ\eta^{-1}$. If we repeat the same construction with the different map~$s'$, we get a different homomorphism $\rho'=\rho_0\circ\eta'^{-1}$. We can write 
\[
\rho=\rho_0\circ\eta^{-1}\circ\eta'\circ\eta'^{-1}.
\]
By the above formulas, we have 
\[
\eta^{-1}\circ\eta'(x,\gamma)=(\varphi(\gamma)x,\gamma).
\]
Let $\tilde{\gamma}\in\bar{\pi}_1(T'S)$ and write $\eta'^{-1}(\tilde{\gamma})=(x,\gamma)$. Then 
\begin{align*}
\rho(\tilde{\gamma}) &= (\rho_0\circ\eta^{-1}\circ\eta'\circ\eta'^{-1})(\tilde{\gamma}) \\
&= (\rho_0\circ\eta^{-1}\circ\eta')(x,\gamma) \\
&= \rho_0(\varphi(\gamma)x,\gamma) \\
&= \rho_0(\varphi(\gamma),1)\rho_0(x,\gamma).
\end{align*}
Finally, we have 
\begin{align*}
\rho_0(x,\gamma) &= (\rho_0\circ\eta'^{-1})(\tilde{\gamma}) \\
&= \rho'(\tilde{\gamma})
\end{align*}
and therefore 
\[
\rho(\tilde{\gamma})=\sigma(\gamma)\rho'(\tilde{\gamma})
\]
where we have defined $\sigma(\gamma)=\rho_0(\varphi(\gamma),1)$. The representation $\rho^\circ$ changes in exactly the same way, so the pair $(\mathcal{L},\mathcal{L}^\circ)$ of twisted local systems determined by these representations is well defined modulo the diagonal action of the group $\Delta_{SL_m}=\Hom(\pi_1(S),Z(SL_m(\mathbb{C})))$. This completes the Step~2 of the proof.

Now the generic part of the moduli space $\mathcal{D}_{PGL_m,S}$ is an affine scheme over $\mathbb{Q}$, and one can check that the map constructed above is defined by polynomials with coefficients in~$\mathbb{Q}$. Hence, by Lemma~\ref{lem:pointsfield}, we have a birational map $\mathcal{D}_{\mathbf{i}}\dashrightarrow\mathcal{D}_{PGL_m,S}$.
\end{proof}

As an immediate consequence, we have the following.

\begin{theorem}
There is a natural birational equivalence of the moduli space $\mathcal{D}_{PGL_m,S}$ and the symplectic double associated to the surface $S$.
\end{theorem}

\chapter{Teichm\"uller spaces}
\label{ch:TeichmullerSpaces}

In low-dimensional topology and geometry, the Teichm\"uller space of a surface is a space that parametrizes hyperbolic structures on the surface, modulo a certain equivalence. Here we will study three variants of the classical Teichm\"uller space of a surface. We will see that these Teichm\"uller spaces correspond to the positive real points of the three cluster varieties.

\section{Decorated Teichm\"uller space}

Let us first recall the classical Teichm\"uller space of a punctured surface.

\begin{definition}
The \emph{Teichm\"uller space} $\mathcal{T}(S)$ of a punctured surface $S$ is the quotient 
\[
\mathcal{T}(S)=\Hom'(\pi_1(S),PSL_2(\mathbb{R}))/PSL_2(\mathbb{R})
\]
where $\Hom'(\pi_1(S),PSL_2(\mathbb{R}))$ is the set of all discrete and faithful representations of~$\pi_1(S)$ into $PSL_2(\mathbb{R})$ such that the image of a loop surrounding a puncture is parabolic. The group~$PSL_2(\mathbb{R})$ acts on this set by conjugation.
\end{definition}

Let $\rho:\pi_1(S)\rightarrow PSL_2(\mathbb{R})$ be an element of the set described above. Then we can represent~$S$ as a quotient 
\[
S=\mathbb{H}/\Delta
\]
where $\mathbb{H}$ is the upper half plane and $\Delta=\rho(\pi_1(S))$ is a discrete subgroup of $PSL_2(\mathbb{R})$. By definition, the map $\rho$ takes any loop surrounding a puncture to a parabolic transformation. If $x\in\partial\mathbb{H}$ is the fixed point of the parabolic transformation corresponding to a puncture~$p$ in~$S$, then a horocycle in $\mathbb{H}$ centered at $x$ projects to a curve on $S$ which we also call a \emph{horocycle} at $p$.

\begin{definition}
If $S$ is a punctured surface, then we define the \emph{decorated Teichm\"uller space} $\mathcal{A}^+(S)$ to be the space that parametrizes pairs $(\rho,\mathcal{S})$ where $\rho$ is a point of $\mathcal{T}(S)$ and $\mathcal{S}$ is a set of horocycles, one at each puncture.
\end{definition}

More generally, suppose that $S$ is any decorated surface. Delete the marked points on the boundary of $S$ and double the resulting surface along its boundary arcs. This produces a punctured surface $S'$ where each marked point in the original surface gives rise to a puncture in~$S'$. The doubled surface $S'$ comes equipped with a natural involution $\iota:S'\rightarrow S'$.

\begin{definition}
The \emph{decorated Teichm\"uller space} $\mathcal{A}^+(S)$ is the $\iota$-invariant subspace of $\mathcal{A}^+(S')$.
\end{definition}

Note that this space can be identified with the one defined previously in the special case where there are no marked points on $S$. We write $\mathcal{A}_0^+(S)$ for the set of points in $\mathcal{A}^+(S)$ such that if $e$ is the segment of $\partial S$ between two marked points, then the horocycles at the ends of~$e$ are tangent. When there is no possibility of confusion, we will simply write $\mathcal{A}^+$ and~$\mathcal{A}_0^+$.

To construct coordinates on the decorated Teichm\"uller space, fix a point $m\in\mathcal{A}^+$ and let $i$ be an edge of an ideal triangulation. The point $m$ allows us to write $S'$ as a quotient $S'=\mathbb{H}/\Delta$ where $\Delta$ is a discrete subgroup of $PSL_2(\mathbb{R})$. We can then deform $i$ into a geodesic and lift this geodesic to the upper half plane $\mathbb{H}$. By definition of the decorated Teichm\"uller space, we have horocycles at the ends of the resulting geodesic in $\mathbb{H}$. We define~$A_i$ as the exponentiated half length (respectively, negative half length) of the segment of the lifted curve between the intersection points with the horocycles if these horocycles do not intersect (respectively, if they do intersect).
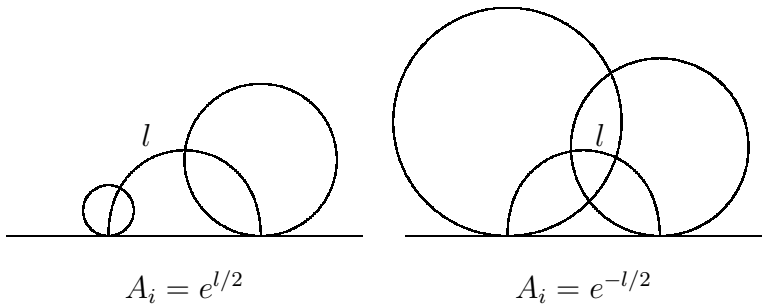
\[
\xy 0;/r.40pc/: 
(0,-6)*{}="1"; 
(12,-6)*{}="2"; 
"1";"2" **\crv{(0,3) & (12,3)}; 
(-8,-6)*{}="X"; 
(20,-6)*{}="Y"; 
"X";"Y" **\dir{-}; 
(0,-4)*\xycircle(2,2){-};
(12,0)*\xycircle(6,6){-};
(6,-10)*{A_i=e^{l/2}}; 
(3,2)*{l}; 
\endxy
\quad
\xy 0;/r.40pc/: 
(0,-6)*{}="1"; 
(12,-6)*{}="2"; 
"1";"2" **\crv{(0,3) & (12,3)}; 
(-8,-6)*{}="X"; 
(20,-6)*{}="Y"; 
"X";"Y" **\dir{-}; 
(0,3)*\xycircle(9,9){-};
(12,1)*\xycircle(7,7){-};
(7.3,2)*{l}; 
(6,-10)*{A_i=e^{-l/2}}; 
\endxy
\]
Doing this for every edge of an ideal triangulation of $S$, we get a collection of numbers $A_i$ ($i\in I$) corresponding to the point $m$.

\begin{proposition}
The numbers $A_i$ ($i\in I$) provide a bijection 
\[
\mathcal{A}^+(S)\rightarrow\mathbb{R}_{>0}^{|I|}.
\]
\end{proposition}

\begin{proof}
We will construct an inverse to this map. Let us suppose that we are given a positive number $A_i$ for each edge $i\in I$. Let $\tilde{S}$ denote the topological universal cover of $S$. Then we can lift the ideal triangulation of $S$ to a triangulation of $\tilde{S}$, and we can associate to each edge of this triangulation the number associated to its projection.

Let $t_0$ be any triangle in the triangulation of~$\tilde{S}$. By~\cite{Penner}, Chapter~1, Corollary~4.8, there exists an ideal triangle $u_0$ in $\mathbb{H}$ and horocycles around the endpoints of $u_0$ realizing the $A$-coordinates associated to the edges of $t_0$. Next, consider a triangle $t$ adjacent to $t_0$ in the triangulation of~$\tilde{S}$. The common edge $t\cap t_0$ corresponds to an edge in~$\mathbb{H}$ with horocycles around its endpoints. By~\cite{Penner}, Chapter~1, Lemma~4.14, there is a unique ideal triangle $u$ in~$\mathbb{H}$ adjacent to $u_0$ with horocycles around its endpoints so that these horocycles agree with the ones already constructed and realize the $A$-coordinates associated to the edges of $t$.

Continuing in this way, we obtain a collection of ideal triangles in~$\mathbb{H}$ where each triangle corresponds to a triangle in~$\tilde{S}$. Now any element $\gamma\in\pi_1(S)$ corresponds to a deck transformation of $\tilde{S}$, and there is a unique element of $PSL_2(\mathbb{R})$ that realizes this deck transformation as an isometry of $\mathbb{H}$ preserving the triangulation. In this way, we obtain a representation $\rho:\pi_1(S)\rightarrow PSL_2(\mathbb{R})$. One can check that this construction provides a two-sided inverse of the map $\mathcal{A}^+(S)\rightarrow\mathbb{R}_{>0}^{|I|}$.
\end{proof}

\begin{proposition}
A regular flip at an edge $k$ of the ideal triangulation changes the coordinates~$A_i$ to new coordinates~$A_i'$ given by the formula 
\[
A_i' =
\begin{cases}
A_k^{-1}\biggr(\prod_{j|\varepsilon_{kj>0}}A_j^{\varepsilon_{kj}} + \prod_{j|\varepsilon_{kj<0}}A_j^{-\varepsilon_{kj}}\biggr) & \mbox{if } i=k \\
A_i & \mbox{if } i\neq k.
\end{cases}
\]
\end{proposition}

\begin{proof}
This result is a restatement of the Ptolemy relation proved in~\cite{Penner}.
\end{proof}

Note that the formula appearing in this proposition is exactly the formula that we used to define the cluster $K_2$-variety. Thus we immediately obtain the following.

\begin{theorem}
The decorated Teichm\"uller space is naturally identified with the space of positive real points of the cluster $K_2$-variety associated to the surface~$S$:
\[
\mathcal{A}^+(S)=\mathcal{A}(\mathbb{R}_{>0}).
\]
\end{theorem}

\section{Enhanced Teichm\"uller space}

The next version of Teichm\"uller space that we will consider parametrizes more general surface group representations. More precisely, we consider the set 
\[
\Hom''(\pi_1(S),PSL_2(\mathbb{R}))/PSL_2(\mathbb{R})
\]
where $\Hom''(\pi_1(S),PSL_2(\mathbb{R}))$ is the set of all discrete and faithful representations of~$\pi_1(S)$ into $PSL_2(\mathbb{R})$ such that the image of a loop surrounding a puncture is either parabolic or hyperbolic. The group~$PSL_2(\mathbb{R})$ again acts by conjugation.

Suppose we are given a representation $\rho:\pi_1(S)\rightarrow PSL_2(\mathbb{R})$ in the set described above. A puncture $p$ in the surface $S$ will be called a \emph{hole} if $\rho$ maps the homotopy class of a loop surrounding $p$ to a hyperbolic transformation.

\begin{definition}
If $S$ is a punctured surface, then we define the \emph{enhanced Teichm\"uller space} $\mathcal{X}^+(S)$ to be the space that parametrizes pairs $(\rho,\mathcal{S})$ where $\rho$ is an element of the above quotient and $\mathcal{S}$ is a set of orientations, one for each hole.
\end{definition}

More generally, if $S$ is any decorated surface, then by doubling $S$, we can construct the punctured surface $S'$ with the natural involution $\iota:S'\rightarrow S'$ as before. One then has the following definition.

\begin{definition}
The \emph{enhanced Teichm\"uller space} $\mathcal{X}^+(S)$ of $S$ is the $\iota$-invariant subspace of $\mathcal{X}^+(S')$.
\end{definition}

This space can be identified with the one defined previously in the special case where there are no marked points on $S$. When there is no possibility of confusion, we will simply write~$\mathcal{X}^+$.

To construct coordinates on the enhanced Teichm\"uller space, fix a point $m\in\mathcal{X}^+$ and an ideal triangulation $T$. Modify a small neighborhood of each hole in $S$ to get a new surface~$S'$ with geodesic boundary. The edges of $T$ correspond to arcs on this surface $S'$, and there is a canonical way to deform these arcs. If an arc ends on a hole, we wind its endpoint around the hole infinitely many times in the direction prescribed by the orientation so that the arcs spiral into the holes in~$S'$. More precisely, each geodesic boundary component of~$S'$ lifts to a geodesic in~$\mathbb{H}$ and we deform the preimage of an edge by dragging its endpoints along these geodesics until they coincide with points of~$\partial\mathbb{H}$.

Once we have deformed the edges of our triangulation in this way, we can lift the triangulation to the upper half plane to get a collection of ideal triangles. Let $k$ be any internal edge and consider the two ideal triangles that share this edge. Together they form an ideal quadrilateral, and we number the vertices of this ideal quadrilateral in counterclockwise order as shown below so that $k$ joins vertices 1 and~3.
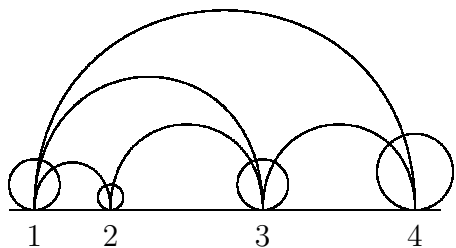
\[
\xy 0;/r.40pc/: 
(-6,-8)*{1}; 
(0,-8)*{2}; 
(12,-8)*{3}; 
(24,-8)*{4}; 
(-6,-6)*{}="1"; 
(0,-6)*{}="2"; 
(12,-6)*{}="3"; 
(24,-6)*{}="4"; 
"1";"2" **\crv{(-6,-1) & (0,-1)}; 
"2";"3" **\crv{(0,3) & (12,3)}; 
"3";"4" **\crv{(12,3) & (24,3)}; 
"1";"4" **\crv{(-6,15) & (24,15)}; 
"1";"3" **\crv{(-6,8) & (12,8)}; 
(-6,-4)*\xycircle(2,2){-};
(0,-5)*\xycircle(1,1){-};
(12,-4)*\xycircle(2,2){-};
(24,-3)*\xycircle(3,3){-};
(-8,-6)*{}="X"; 
(26,-6)*{}="Y"; 
"X";"Y" **\dir{-}; 
\endxy
\]

Choose a horocycle at each vertex, and put $A_{ij}=e^{l_{ij}/2}$ where $l_{ij}$ is the signed length of the segment between the horocycles at $i$ and $j$. We then define the \emph{cross ratio} 
\[
X_k=\frac{A_{12}A_{34}}{A_{14}A_{23}}.
\]
It is easy to see that there are two ways of numbering the vertices, and both give the same value for the cross ratio. One can also show that the numbers $X_k$ ($k\in J$) are independent of the chosen horocycles. They are the numbers that we associate to the point~$m$.

\begin{proposition}
The numbers $X_j$ ($j\in J$) provide a bijection 
\[
\mathcal{X}^+(S)\rightarrow\mathbb{R}_{>0}^{|J|}.
\]
\end{proposition}

\begin{proof}
Suppose we are given a positive number $X_j$ for each edge $j\in J$. To recover the orientation of a hole~$h$, we compute the number $\sum\log X_j$, where the sum is over all edges~$j$ incident to~$h$. This sum is negative (respectively, positive) when the orientation is induced from the orientation of $S$ (respectively, the opposite of this orientation).

Let $\tilde{S}$ denote the topological universal cover of the surface $S$. We can lift the ideal triangulation of $S$ to an ideal triangulation of $\tilde{S}$, and we can associate to each edge of this triangulation the number associated to its projection.

Let $t_0$ be any triangle in the triangulation of~$\tilde{S}$, and choose a corresponding ideal triangle $u_0$ in $\mathbb{H}$. Next, consider a triangle $t$ adjacent to $t_0$ in the triangulation of~$\tilde{S}$. There is an $X$-coordinate associated to the common edge $t\cap t_0$. Since the cross ratio is a complete invariant of four points on~$\partial\mathbb{H}$, there is a unique ideal triangle $u$ in~$\mathbb{H}$ adjacent to $u_0$ so that these triangles realize the $X$-coordinate associated to the common edge.

Continuing in this way, we obtain a triangulation of a region in~$\mathbb{H}$ by geodesic triangles where each triangle corresponds to a triangle in~$\tilde{S}$. Now any element $\gamma\in\pi_1(S)$ determines a deck transformation of $\tilde{S}$, and there is a unique element of $PSL_2(\mathbb{R})$ that realizes this deck transformation as an isometry of $\mathbb{H}$ preserving the triangulation. In this way, we obtain a representation $\rho:\pi_1(S)\rightarrow PSL_2(\mathbb{R})$. One can check that this construction provides a two-sided inverse of the map $\mathcal{X}^+(S)\rightarrow\mathbb{R}_{>0}^{|J|}$.
\end{proof}

\begin{proposition}
A regular flip at an edge $k$ of the ideal triangulation changes the coordinates~$X_i$ to new coordinates~$X_i'$ given by the formula 
\[
X_i'=
\begin{cases}
X_k^{-1} & \mbox{if } i=k \\
X_i{(1+X_k^{-\sgn(\varepsilon_{ik})})}^{-\varepsilon_{ik}} & \mbox{if } i\neq k.
\end{cases}
\]
\end{proposition}

\begin{proof}
See~\cite{Penner}.
\end{proof}

This proposition immediately implies the following.

\begin{theorem}
The enhanced Teichm\"uller space is naturally identified with the space of positive real points of the cluster Poisson variety associated to the surface~$S$:
\[
\mathcal{X}^+(S)=\mathcal{X}(\mathbb{R}_{>0}).
\]
\end{theorem}

\section{Doubled Teichm\"uller space}

We now describe a version of the Teichm\"uller space of a doubled surface first introduced in~\cite{double}. Let $\Sigma$ be any oriented punctured surface.

\begin{definition}
A \emph{simple lamination} on $\Sigma$ is a finite collection $\gamma=\{\gamma_i\}$ of simple nontrivial disjoint nonisotopic loops on $\Sigma$ which do not retract to the punctures, considered modulo isotopy.
\end{definition}

Let $\Sigma$ be an oriented surface equipped with a simple lamination $\gamma=\{\gamma_i\}$. For any subset $\{\gamma_1,\dots,\gamma_k\}$ of the set of loops $\gamma_i$, we let $\Sigma_{p_1,\dots,p_k}$ denote the singular surface obtained by pinching these loops to get the nodes $p_1,\dots,p_k$. This singular surface is equipped with a simple lamination $\gamma_{p_1,\dots,p_k}$ given by the image of $\gamma-\{\gamma_1,\dots,\gamma_k\}$. Note that if we cut the surface $\Sigma_{p_1,\dots,p_k}$ at the nodes $p_1,\dots,p_k$, we get a punctured surface $\Sigma_{p_1,\dots,p_k}'$. When we talk about a point in the Teichm\"uller space of $\Sigma_{p_1,\dots,p_k}$, we really mean a point in the Teichm\"uller space of the surface $\Sigma_{p_1,\dots,p_k}'$. Similarly, when we talk about a horocycle at a node of $\Sigma_{p_1,\dots,p_k}$, we really mean a horocycle at one of the corresponding punctures in $\Sigma_{p_1,\dots,p_k}'$.

\begin{definition}
The set $\mathcal{X}_{\Sigma;\gamma;p_1,\dots,p_k}^+$ parametrizes points $\rho\in\mathcal{T}(\Sigma_{p_1,\dots,p_k})$ of the Teichm\"uller space of $\Sigma_{p_1,\dots,p_k}$, together with the following data:
\begin{enumerate}
\item An orientation for each loop of the simple lamination $\gamma_{p_1,\dots,p_k}$.

\item For every node $p_i$, a pair of horocycles $(c_{-,i},c_{+,i})$ centered at $p_i$, located on opposite sides of the node, and defined up to simultaneous shift by any real number.
\end{enumerate}
We write $\mathcal{X}_{\Sigma;\gamma}^+$ for the union of these sets $\mathcal{X}_{\Sigma;\gamma;p_1,\dots,p_k}^+$.
\end{definition}

Now let $S$ be a decorated surface, and let $S^\circ$ be the same surface with the opposite orientation. The \emph{double} $S_\mathcal{D}$ of $S$ is the punctured surface obtained by gluing $S$ and~$S^\circ$ along corresponding boundary components and deleting the image of each marked point in the resulting surface. The surface $\Sigma=S_\mathcal{D}$ carries a natural simple lamination $\gamma$ given by the image of the boundary loops of $S$.

\begin{definition}
The \emph{doubled Teichm\"uller space} $\mathcal{D}^+(S)$ is the space $\mathcal{X}_{\Sigma;\gamma}^+$ with $\Sigma=S_\mathcal{D}$.
\end{definition}

When there is no possibility of confusion, we will denote the doubled Teichm\"uller space by $\mathcal{D}^+$. The next two results will be used to construct coordinates on this space.

\begin{lemma}
If $\rho\in\mathcal{T}(S_\mathcal{D})$ is a point in the Teichm\"uller space of $S_\mathcal{D}$, then there exists a hyperbolic structure representing $\rho$ such that $\partial S\subseteq S_\mathcal{D}$ is geodesic.
\end{lemma}

\begin{proof}
A point $\rho$ of the Teichm\"uller space can be viewed as a marked hyperbolic surface, that is, a hyperbolic surface $X$ together with a diffeomorphism $\phi:S_\mathcal{D}\rightarrow X$. The image of $\partial S\subseteq S_\mathcal{D}$ under $\phi$ may not be geodesic in $X$, but we can deform $\phi$ to a homotopic map $\phi'$ such that $\phi'(\partial S)$ is geodesic. This is the same as saying that the diagram 
\[
\xymatrix{ 
& S_\mathcal{D} \ar[rd]^-{\phi'} \ar[ld]_{\phi} \\
X \ar[rr]_-{1_X} & & X 
}
\]
commutes up to homotopy. Thus $(X,\phi)$ and $(X,\phi')$ represent the same point of Teichm\"uller space. Pulling back the hyperbolic structure of $X$ along $\phi'$, we obtain a hyperbolic structure on $S_\mathcal{D}$ representing the point $\rho$.
\end{proof}

Let $\rho\in\mathcal{T}(S_\mathcal{D})$. By the above lemma, we can view $\rho$ as a hyperbolic structure on $S_\mathcal{D}$ such that $\partial S\subseteq S_\mathcal{D}$ is geodesic. By cutting along $\partial S$, we recover the surfaces $S$ and $S^\circ$ equipped with hyperbolic structures with geodesic boundary. Then the universal cover $\tilde{S}$ of $S$ can be identified with a subset of $\mathbb{H}$ obtained by removing countably many geodesic half disks. Similarly, the universal cover $\tilde{S^\circ}$ of $S^\circ$ can be identified with a subset of $\mathbb{H}$.

\begin{lemma}
\label{lem:correspondence}
Let $\iota$ be the natural map $S\rightarrow S^\circ$. Then there exists a map $\tilde{\iota}:\tilde{S}\rightarrow\tilde{S^\circ}$, unique up to the action of $\pi_1(S)$ by deck transformations, such that the diagram 
\[
\xymatrix{ 
\tilde{S} \ar[r]^{\tilde{\iota}} \ar[d] & \tilde{S^\circ} \ar[d] \\
S \ar_{\iota}[r] & S^\circ
}
\]
commutes. This map $\tilde{\iota}$ restricts to an isometry on the geodesic boundary of $\tilde{S}$.
\end{lemma}

\begin{proof}
Since $\tilde{S}$ is simply connected, the composition $\tilde{S}\rightarrow S\stackrel{\iota}{\rightarrow} S^\circ$ induces the zero map on fundamental groups. Hence, by the lifting criterion, there exists a map $\tilde{\iota}:\tilde{S}\rightarrow\tilde{S^\circ}$ making the above diagram commutative. If $\tilde{\iota}':\tilde{S}\rightarrow\tilde{S^\circ}$ is any other map with this property, then $\tilde{\iota}$ and $\tilde{\iota}'$ are both lifts of the composition $\tilde{S}\rightarrow S\stackrel{\iota}{\rightarrow} S^\circ$. It follows that they coincide up to a deck transformation.

We know that the map $\tilde{\iota}$ restricts to an isometry on the geodesic boundary of~$\tilde{S}$ because the other three maps of the diagram restrict to isometries on boundary geodesics.
\end{proof}

To construct coordinates on $\mathcal{D}^+(S)$, fix a point $m\in\mathcal{D}^+(S)$ and let $i\in J$ be an edge of the ideal triangulation $T$. The point $m$ determines a point of the Teichm\"uller space of~$\Sigma_{p_1,\dots,p_k}$ where $\Sigma=S_\mathcal{D}$, and we can represent this by a hyperbolic structure such that the image of the boundary $\partial S$ is geodesic. By cutting along the image of $\partial S$, we recover the surfaces $S$ and $S^\circ$ equipped with hyperbolic structures with geodesic boundary.

Suppose the edge $i$ corresponds to an arc connecting two holes in $S$. Choose a pair of geodesics $g_1$ and~$g_2$ in the upper half plane $\mathbb{H}$ which project to these boundary components. Deform the arc connecting the holes by winding its endpoints around the holes infinitely many times in the direction prescribed by the orientation. This corresponds to deforming the preimage $\tilde{i}$ in $\mathbb{H}$ so that its endpoints coincide with endpoints of~$g_1$ and~$g_2$. If we let $i^\circ$ be the image of the arc $i$ under the tautological map $S\rightarrow S^\circ$, then we can apply the same procedure to $i^\circ$ to get an arc $\tilde{i}^\circ$ in $\mathbb{H}$.

Choose horocycles $c_1$ and~$c_2$ around the endpoints of $\tilde{i}$. The horocycle $c_k$ intersects~$g_k$ in a unique point. By Lemma~\ref{lem:correspondence}, there is a corresponding point on $g_k^\circ$ and thus a corresponding horocycle $c_k^\circ$. The map constructed in Lemma~\ref{lem:correspondence} restricts to an isometry on $g_k$, so if we shift $c_k$ by some amount, then the corresponding horocycle $c_k^\circ$ will shift by the same amount.
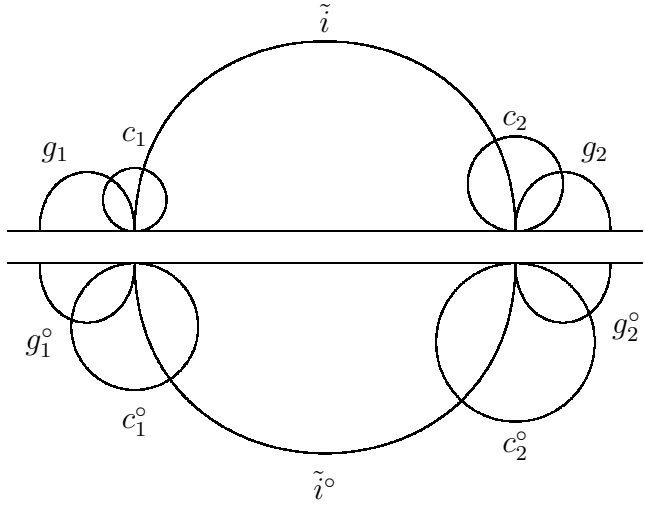
\[
\xy 0;/r.50pc/: 
(-18,-6)*{}="0"; 
(-12,-6)*{}="1"; 
(12,-6)*{}="2"; 
(18,-6)*{}="3"; 
"0";"1" **\crv{(-18,-1) & (-12,-1)}; 
"1";"2" **\crv{(-12,10) & (12,10)}; 
"2";"3" **\crv{(12,-1) & (18,-1)}; 
(-20,-6)*{}="X"; 
(20,-6)*{}="Y"; 
"X";"Y" **\dir{-}; 
(-12,-4)*\xycircle(2,2){-};
(12,-3)*\xycircle(3,3){-};
(0,7.5)*{\tilde{i}}; 
(-12,0)*{c_1}; 
(12,1)*{c_2}; 
(-17,-1)*{g_1}; 
(17,-1)*{g_2}; 
(-20,-8)*{}="Z"; 
(20,-8)*{}="W"; 
"Z";"W" **\dir{-}; 
(-18,-8)*{}="0"; 
(-12,-8)*{}="1"; 
(12,-8)*{}="2"; 
(18,-8)*{}="3"; 
"0";"1" **\crv{(-18,-13) & (-12,-13)}; 
"1";"2" **\crv{(-12,-24) & (12,-24)}; 
"2";"3" **\crv{(12,-13) & (18,-13)}; 
(12,-13)*\xycircle(5,5){-};
(-12,-12)*\xycircle(4,4){-};
(0,-22)*{\tilde{i}^\circ};
(-12,-18)*{c_1^\circ}; 
(12,-19.5)*{c_2^\circ}; 
(-18,-13)*{g_1^\circ}; 
(19,-12)*{g_2^\circ}; 
\endxy
\]

Denote by $A_i$ the exponentiated signed half distance between~$c_1$ and~$c_2$, and denote by $A_i^\circ$ the exponentiated signed half distance between~$c_1^\circ$ and~$c_2^\circ$. We can then define a number associated to the edge $i$ by 
\[
B_i=\frac{A_i^\circ}{A_i}.
\]
This defines $B_i$ when $i$ corresponds to an arc connecting two holes. If one or both of the endpoints of $i$ are punctures, then the data of the point $m\in\mathcal{D}^+(S)$ provides a horocycle~$c_k$ at any endpoint of $i$ which is a puncture, as well as a horocycle $c_k^\circ$ at the corresponding endpoint of $i^\circ$. Thus we can associate numbers $A_i$ and $A_i^\circ$ to these arc as before, and we can define $B_i$ by the above formula. One can show that the $|J|$ numbers obtained in this way are independent of all choices made in the construction.

In addition to the $B_i$, there are $|J|$ numbers $X_i$ associated to a point in the space $\mathcal{D}^+(S)$. Given a point of $\mathcal{D}^+(S)$, these are simply defined as the $X$-coordinates of the point of the enhanced Teichm\"uller space of the surface $S$ obtained by cutting $S_\mathcal{D}$ along the image of $\partial S$.

\begin{proposition}
The numbers $B_i$ and $X_i$ provide a bijection 
\[
\mathcal{D}^+(S)\rightarrow\mathbb{R}_{>0}^{2|J|}.
\]
\end{proposition}

\begin{proof}
Call this map $\phi$. We will construct a map $\psi:\mathbb{R}_{>0}^{2|J|}\rightarrow\mathcal{D}^+(S)$ and prove that $\phi$ and $\psi$ are inverses.

\Needspace*{2\baselineskip}
\begin{step}[1]
Definition of $\psi:\mathbb{R}_{>0}^{2|J|}\rightarrow\mathcal{D}^+(S)$.
\end{step}

Suppose we are given positive numbers $B_j$ and $X_j$ for every edge $j\in J$. By the reconstruction procedure for the enhanced Teichm\"uller space, we can glue together ideal triangles in $\mathbb{H}$, using the $X_i$ as gluing parameters, to get a tiling of a region $\tilde{S}\subseteq\mathbb{H}$ with geodesic boundary.

Choose a horocycle around each vertex of this triangulation. (We do not require these horocycles to be covariant under the action of the deck transformation group.) Then for every edge~$i$ of this triangulation, there is a number $A_i$ defined as the exponentiated signed half length between the horocycles at the ends of~$i$. If $B_i$ is the $B$-coordinate corresponding to the lifted edge~$i$, define 
\[
A_i^\circ=B_iA_i.
\]
We will use these numbers $A_i^\circ$ to construct another region $\tilde{S^\circ}$ with geodesic boundary in a copy of~$\mathbb{H}$.

To construct $\tilde{S^\circ}$, first choose an edge $i$ in the triangulation of $\tilde{S}$. Let $i^\circ$ be any geodesic in the second copy of $\mathbb{H}$. There are horocycles around the endpoints of $i$ with corresponding exponentiated half length $A_i$, and we can choose horocycles around the endpoints of $i^\circ$ so that the resulting exponentiated half length equals $A_i^\circ$. Consider one of the triangles adjacent to~$i$. There are numbers $A_j$ and $A_k$ corresponding to its other sides. By \cite{Penner}, Chapter~1, Lemma~4.14, there is a unique point in the second copy of $\partial\mathbb{H}$ and a horocycle about this point so that the exponentiated half lengths of the decorated arcs incident to this point are $A_j^\circ$ and $A_k^\circ$, and the cyclic order of the vertices of the resulting triangle agrees with the original cyclic order of the vertices.

Repeating this process for every edge in the triangulation of $\tilde{S}$, we obtain a region $\tilde{S^\circ}\subseteq\mathbb{H}$ with geodesic boundary. To construct this region, we had to choose horocycles around the endpoints of the triangulation of $\tilde{S}$, but the construction is independent of these choices. We also had to choose edges $i$ and $i^\circ$ and horocycles around the endpoints of $i^\circ$. If we had made a different choice we would get a different region obtained from $\tilde{S^\circ}$ by a transformation in $PSL_2(\mathbb{R})$. Thus $\tilde{S^\circ}$ is well defined up to isometry.

Now the region $\tilde{S}$ is obtained from $\mathbb{H}$ by removing infinitely many geodesic half disks. We can extend $\tilde{S}$ to a larger region by gluing a copy of $\tilde{S^\circ}$ in each of these half disks in such a way that the horocycles around identified vertices coincide. The resulting region is again obtained from $\mathbb{H}$ by removing infinitely many geodesic half disks, and we can enlarge it by gluing a copy of $\tilde{S}$ in each of these half disks. Continuing this process ad infinitum, we partition the entire hyperbolic plane into ideal triangles. The illustration below shows an example of how the spaces $\tilde{S}$ and $\tilde{S}^\circ$ may be glued together in the disk model of the hyperbolic plane.
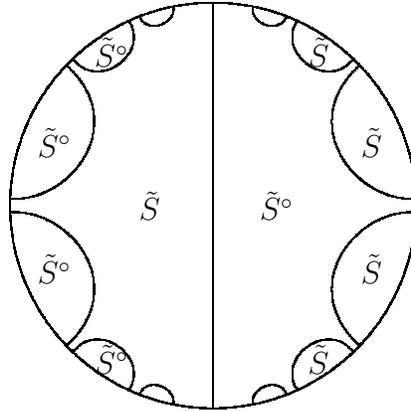
\[
\xy 0;/r.40pc/: 
(0,0)*\xycircle(16,16){-};
(0,16)*{}="X"; 
(0,-16)*{}="Y"; 
(-3.12,15.69)*{}="A1"; 
(-5.8,14.9)*{}="A2"; 
(-6.6,14.5)*{}="A3"; 
(-11,11.5)*{}="A4"; 
(-11.5,11)*{}="A5"; 
(-15.9,0.5)*{}="A6"; 
(-15.9,-0.5)*{}="A7"; 
(-11.5,-11)*{}="A8"; 
(-11,-11.5)*{}="A9"; 
(-6.6,-14.5)*{}="A10"; 
(-5.8,-14.9)*{}="A11"; 
(-3.12,-15.69)*{}="A12"; 
(3.12,15.69)*{}="B1"; 
(5.8,14.9)*{}="B2"; 
(6.6,14.5)*{}="B3"; 
(11,11.5)*{}="B4"; 
(11.5,11)*{}="B5"; 
(15.9,0.5)*{}="B6"; 
(15.9,-0.5)*{}="B7"; 
(11.5,-11)*{}="B8"; 
(11,-11.5)*{}="B9"; 
(6.6,-14.5)*{}="B10"; 
(5.8,-14.9)*{}="B11"; 
(3.12,-15.69)*{}="B12"; 
"X";"Y" **\dir{-}; 
"A1";"A2" **\crv{(-2.9,14.4) & (-5.1,13.4)}; 
"A3";"A4" **\crv{(-5,12) & (-9,9)}; 
"A5";"A6" **\crv{(-6.6,6.5) & (-11,0.4)}; 
"A7";"A8" **\crv{(-11,-0.4) & (-6.6,-6.5)}; 
"A9";"A10" **\crv{(-9,-9) & (-5,-12)}; 
"A11";"A12" **\crv{(-5.1,-13.4) & (-2.9,-14.4)}; 
"B1";"B2" **\crv{(2.9,14.4) & (5.1,13.4)}; 
"B3";"B4" **\crv{(5,12) & (9,9)}; 
"B5";"B6" **\crv{(6.6,6.5) & (11,0.4)}; 
"B7";"B8" **\crv{(11,-0.4) & (6.6,-6.5)}; 
"B9";"B10" **\crv{(9,-9) & (5,-12)}; 
"B11";"B12" **\crv{(5.1,-13.4) & (2.9,-14.4)}; 
(-5,0)*{\tilde{S}}; 
(5,0)*{\tilde{S}^\circ}; 
(-12.5,5)*{\tilde{S}^\circ}; 
(12.5,5)*{\tilde{S}}; 
(-12.5,-5)*{\tilde{S}^\circ}; 
(12.5,-5)*{\tilde{S}}; 
(-8,12.1)*{\tilde{S}^\circ}; 
(8.3,12.1)*{\tilde{S}}; 
(-8,-12.1)*{\tilde{S}^\circ}; 
(8.3,-12.1)*{\tilde{S}}; 
\endxy
\]

Any element of $\pi_1(S_\mathcal{D})$ corresponds to a deck transformation and is thus represented as a unique element of $PSL_2(\mathbb{R})$. In this way we recover a discrete and faithful representation $\rho:\pi_1(S_\mathcal{D})\rightarrow PSL_2(\mathbb{R})$ representing a point in $\mathcal{T}(S_\mathcal{D})$.

\Needspace*{2\baselineskip}
\begin{step}[2]
Proof of $\phi\circ\psi=1_{\mathbb{R}_{>0}^{2|J|}}$.
\end{step}

Let $B_j$ and $X_j$ ($j\in J$) be given. By the reconstruction procedure described above, we construct triangulated regions $\tilde{S}\subseteq\mathbb{H}$ and $\tilde{S}^\circ\subseteq\mathbb{H}$ which we then glue together to get a triangulation of the hyperbolic plane. From this triangulation, we get a point $m\in\mathcal{D}^+(S)$. We want to show that the coordinates of $m$ are the numbers $B_j$ and $X_j$.

This point $m$ determines a hyperbolic metric on $S_{\mathcal{D}}$, and the universal cover of the resulting hyperbolic surface is isometric to the triangulated surface $\tilde{S}_{\mathcal{D}}$ that we got by gluing together copies of $\tilde{S}$ and $\tilde{S}^\circ$. Consider a copy of $\tilde{S}$ in $\tilde{S}_{\mathcal{D}}$. To find the coordinates of~$m$, we choose horocycles around the vertices of the triangulation of~$\tilde{S}$, and we can take these to be exactly the horocycles used in the construction of~$m$ from the~$B_j$ and~$X_j$. Now consider a copy of $\tilde{S}^\circ$ adjacent to $\tilde{S}$ in $\tilde{S}_{\mathcal{D}}$. Choose a basepoint~$x$ on the geodesic separating these regions. If $g$ is any boundary geodesic of $\tilde{S}$, then there is a horocycle $h$ around one of the endpoints of $g$, and this horocycle intersects~$g$ in a unique point $p$. Draw a curve $\tilde{\alpha}$ in~$\tilde{S}$ from the point~$p$ to the point~$x$.
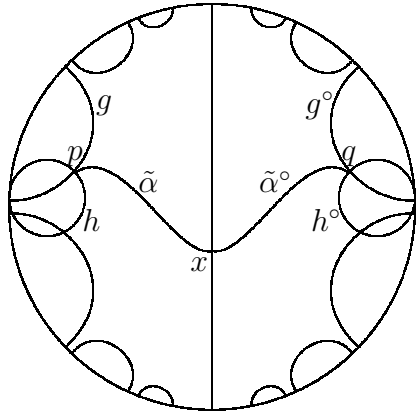
\[
\xy 0;/r.40pc/: 
(0,0)*\xycircle(16,16){-};
(-13,0.7)*\xycircle(3,3){-};
(13,0.7)*\xycircle(3,3){-};
(-10.7,2.75)*{}="p"; 
(0,-3.5)*{}="x"; 
(10.7,2.75)*{}="q"; 
(0,16)*{}="X"; 
(0,-16)*{}="Y"; 
(-3.12,15.69)*{}="A1"; 
(-5.8,14.9)*{}="A2"; 
(-6.6,14.5)*{}="A3"; 
(-11,11.5)*{}="A4"; 
(-11.5,11)*{}="A5"; 
(-15.9,0.5)*{}="A6"; 
(-15.9,-0.5)*{}="A7"; 
(-11.5,-11)*{}="A8"; 
(-11,-11.5)*{}="A9"; 
(-6.6,-14.5)*{}="A10"; 
(-5.8,-14.9)*{}="A11"; 
(-3.12,-15.69)*{}="A12"; 
(3.12,15.69)*{}="B1"; 
(5.8,14.9)*{}="B2"; 
(6.6,14.5)*{}="B3"; 
(11,11.5)*{}="B4"; 
(11.5,11)*{}="B5"; 
(15.9,0.5)*{}="B6"; 
(15.9,-0.5)*{}="B7"; 
(11.5,-11)*{}="B8"; 
(11,-11.5)*{}="B9"; 
(6.6,-14.5)*{}="B10"; 
(5.8,-14.9)*{}="B11"; 
(3.12,-15.69)*{}="B12"; 
"p";"x" **\crv{(-7,5) & (-3,-4)}; 
"x";"q" **\crv{(3,-4) & (7,5)}; 
"X";"Y" **\dir{-}; 
"A1";"A2" **\crv{(-2.9,14.4) & (-5.1,13.4)}; 
"A3";"A4" **\crv{(-5,12) & (-9,9)}; 
"A5";"A6" **\crv{(-6.6,6.5) & (-11,0.4)}; 
"A7";"A8" **\crv{(-11,-0.4) & (-6.6,-6.5)}; 
"A9";"A10" **\crv{(-9,-9) & (-5,-12)}; 
"A11";"A12" **\crv{(-5.1,-13.4) & (-2.9,-14.4)}; 
"B1";"B2" **\crv{(2.9,14.4) & (5.1,13.4)}; 
"B3";"B4" **\crv{(5,12) & (9,9)}; 
"B5";"B6" **\crv{(6.6,6.5) & (11,0.4)}; 
"B7";"B8" **\crv{(11,-0.4) & (6.6,-6.5)}; 
"B9";"B10" **\crv{(9,-9) & (5,-12)}; 
"B11";"B12" **\crv{(5.1,-13.4) & (2.9,-14.4)}; 
(-1,-4.5)*{x}; 
(-5,2)*{\tilde{\alpha}}; 
(5,2)*{\tilde{\alpha}^\circ}; 
(-10.8,4)*{p}; 
(10.8,4)*{q}; 
(-9.5,-1)*{h}; 
(9,-1)*{h^\circ}; 
(-8.5,8)*{g}; 
(8.5,8)*{g^\circ}; 
\endxy
\]
This curve $\tilde{\alpha}$ projects to a curve~$\alpha$ on the surface $S$, and we can apply the map~$\iota$ to get a curve $\alpha^\circ$ on $S^\circ$. Finally, we lift $\alpha^\circ$  to a curve $\tilde{\alpha}^\circ$ in~$S^\circ$ that starts at $x$. This curve $\tilde{\alpha}^\circ$ ends at some point $q=\tilde{\iota}(p)$, and there is a horocycle $h^\circ$ that intersects $g^\circ$ at $q$. In this way, we draw horocycles around all of the vertices of $\tilde{S}^\circ$ and can define the coordinates of $m$.

On the other hand, consider the monodromy along the curve in $S_{\mathcal{D}}$ obtained by concatenating $\alpha$ and $\alpha^\circ$. It maps the region~$\tilde{S}$ isometrically into the region bounded by~$g^\circ$. The horocycle used in the construction of~$m$ from the~$B_j$ and~$X_j$ is obtained by applying this transformation to $h$. Thus the horocycles used to construct $m$ agree with the ones obtained using the map $\tilde{\iota}$.

It follows that the $B$-coordinates of $m$ are simply the numbers $B_j$ that we started with. It is easy to see that the $X$-coordinates of $m$ are the numbers $X_j$. This completes Step~2 of the proof.

\Needspace*{2\baselineskip}
\begin{step}[3]
Proof of $\psi\circ\phi=1_{\mathcal{D}^+}$.
\end{step}

Let $m\in\mathcal{D}^+(S)$ be given. To calculate the coordinates of~$m$, let us pass to the universal cover~$\tilde{S}_{\mathcal{D}}$ of~$S_{\mathcal{D}}$ and choose a copy of $\tilde{S}$ in~$\tilde{S}_{\mathcal{D}}$. Choose horocycles around the vertices of the triangulation of this copy of $\tilde{S}$. Using the construction described above involving lifts of the map $\iota:S\rightarrow S^\circ$, we can get horocycles around all vertices of the triangulation of $\tilde{S}_{\mathcal{D}}$. We can use these horocycles to calculate the numbers $A_j$ and~$A_j^\circ$, and thus the coordinates $B_j$ and~$X_j$.

We can now use the $X_i$ to reconstruct the region~$\tilde{S}$. We must choose horocycles around the vertices of the triangulation of~$\tilde{S}$, and we can assume that these are exactly the ones used above to define the~$B_j$ and~$X_j$. Then we can use the $B_j$ to reconstruct an adjacent region $\tilde{S}^\circ$ with exactly the horocycles used above. By gluing together copies of $\tilde{S}$ and $\tilde{S}^\circ$, we recover $\tilde{S}_{\mathcal{D}}$ together with all of the horocycles used to compute the coordinates. Quotienting this universal cover by the action of $\pi_1(S_{\mathcal{D}})$, we recover the point $m\in\mathcal{D}^+(S)$. This completes Step~3 of the proof.
\end{proof}

We have defined the $B_i$ and $X_i$ coordinates in terms of a fixed ideal triangulation of~$S$. The next result says how these coordinates vary when we change the triangulation.

\begin{proposition}
A regular flip at an edge $k$ of the triangulation changes the coordinates~$X_i$ and~$B_i$ to new coordinates~$X_i'$ and~$B_i'$ given by the formulas 
\begin{align*}
X_i' &=
\begin{cases}
X_k^{-1} & \mbox{if } i=k \\
X_i{(1+X_k^{-\sgn(\varepsilon_{ik})})}^{-\varepsilon_{ik}} & \mbox{if } i\neq k
\end{cases} \\
B_i' &=
\begin{cases}
\frac{X_k\prod_{j|\varepsilon_{kj>0}}B_j^{\varepsilon_{kj}} + \prod_{j|\varepsilon_{kj<0}}B_j^{-\varepsilon_{kj}}}{(1+X_k)B_k} & \mbox{if } i=k \\
B_i & \mbox{if } i\neq k.
\end{cases}
\end{align*}
\end{proposition}

\begin{proof}
By definition, the coordinates $X_i$ on $\mathcal{D}^+(S)$ coincide with the coordinates of the corresponding point of the enhanced Teichm\"uller space of $S$. Therefore the transformation rule for the $X_i$ is just the usual transformation rule for the coordinates on the enhanced Teichm\"uller space.

To prove the second rule, deform the arcs corresponding to the edges of the ideal triangulation by winding their endpoints around the curves of the simple lamination infinitely many times. Lift these deformed curves to the universal cover of $S$, and consider the ideal quadrilateral formed by the two triangles adjacent to the geodesic arc $\tilde{k}$ corresponding to~$k$.

Number the vertices of this quadrilateral in counterclockwise order so that the edge $\tilde{k}$ joins vertices~1 and~3. Choose a horocycle around each vertex of this ideal quadrilateral and let $A_{ij}$ denote the exponentiated signed half length between the horocycles at $i$ and $j$. There is a corresponding ideal quadrilateral in the universal cover of $S^\circ$ and a number $A_{ij}^\circ$ corresponding to the edge connecting vertices $i$ and $j$ of this ideal quadrilateral.

If we flip at the edge $k$, then the ideal quadrilateral in the universal cover of $S$ is replaced by the same ideal quadrilateral triangulated by the arc from~2 to~4, and the number associated to this new arc is 
\[
A_{24}=\frac{A_{12}A_{34}+A_{14}A_{23}}{A_{13}}.
\]
Similarly, the number $A_{13}^\circ$ transforms to 
\[
A_{24}^\circ=\frac{A_{12}^\circ A_{34}^\circ+A_{14}^\circ A_{23}^\circ}{A_{13}^\circ}.
\]
Applying these rules to the quotient $B_{13}=A_{13}^\circ/A_{13}$, we obtain 
\begin{align*}
B_{24} &= \frac{A_{12}^\circ A_{34}^\circ+A_{14}^\circ A_{23}^\circ}{A_{12}A_{34}+A_{14}A_{23}}\frac{A_{13}}{A_{13}^\circ} \\
&= \frac{\frac{A_{12}^\circ A_{34}^\circ}{A_{14}A_{23}}+\frac{A_{14}^\circ A_{23}^\circ}{A_{14}A_{23}}}{1+\frac{A_{12}A_{34}}{A_{14}A_{23}}}\frac{A_{13}}{A_{13}^\circ} \\
&= \frac{X_{13}B_{12}B_{34}+B_{14}B_{23}}{(1+X_{13})B_{13}}
\end{align*}
where we have used the relation $X_{13}=\frac{A_{12}A_{34}}{A_{14}A_{23}}$. This proves the transformation rule for the~$B_i$.
\end{proof}

As in the case of the decorated or enhanced Teichm\"uller space, we have the following result.

\begin{theorem}
The doubled Teichm\"uller space is naturally identified with the space of positive real points of the symplectic double associated to the surface~$S$:
\[
\mathcal{D}^+(S)=\mathcal{D}(\mathbb{R}_{>0}).
\]
\end{theorem}

\chapter{Measured laminations}
\label{ch:MeasuredLaminations}

The notion of a measured lamination is a fundamental concept from low-dimensional topology and geometry. This notion was famously used by Thurston to define a natural compactification of Teichm\"uller space and to classify elements of the mapping class group of a surface. In this chapter, we define three versions of the space of measured laminations on a surface and show that these spaces are ``tropicalizations'' of the three types of cluster varieties.

\section{Tropical rational points}

\subsection{Bounded laminations}

Let $S$ be a decorated surface. In this chapter, the term ``curve'' will always refer to the following technical notion.

\begin{definition}
By a \emph{curve} in $S$, we mean an embedding $C\rightarrow S$ of a compact, connected, one-dimensional manifold $C$ with (possibly empty) boundary into $S$. We require that any endpoints of $C$ map to points on the boundary of $S$ disjoint from the marked points. When we talk about homotopies, we mean homotopies within the class of such curves. A curve is called \emph{special} if it is retractable to a puncture or to an interval on~$\partial S$ containing exactly one marked point. A curve is \emph{contractible} if it can be retracted to a point within this class of curves.
\end{definition}

\begin{definition}
A \emph{rational bounded lamination} on $S$ is the homotopy class of a collection of finitely many simple nonintersecting noncontractible curves on $S$, either closed or ending on a segment of the boundary bounded by adjacent marked points, with rational weights and subject to the following conditions and equivalence relations:
\begin{enumerate}
\item The weight of a curve is nonnegative unless the curve is special.
\item A lamination containing a curve of weight zero is equivalent to the lamination with this curve removed.
\item A lamination containing homotopic curves of weights $a$ and $b$ is equivalent to the lamination with one curve removed and the weight $a+b$ on the other.
\end{enumerate}
\end{definition}

The set of all rational bounded laminations on $S$ is denoted $\mathcal{A}_L(S,\mathbb{Q})$. We will write $\mathcal{A}_L(S,\mathbb{Z})$ for the set of all bounded laminations on $S$ that can be represented by a collection of curves with integral weights. We write $\mathcal{A}_L^0(S,\mathbb{Z})$ for the subset of all laminations such that if $e$ is the segment of $\partial S$ between two marked points, then the total weight of the curves ending on $e$ vanishes. When there is no possibility of confusion, we will simply write $\mathcal{A}_L$ and~$\mathcal{A}_L^0$.

To construct coordinates on $\mathcal{A}_L(S,\mathbb{Q})$, fix a bounded lamination $l$ and an ideal triangulation of~$S$. Deform the curves of $l$ so that each curve intersects each edge of the triangulation in the minimal number of points. Then we define $a_i$ to be half the total weight of curves that intersect the edge $i$.

\begin{proposition}
The numbers $a_i$ ($i\in I$) provide a bijection 
\[
\mathcal{A}_L(S,\mathbb{Q})\rightarrow\mathbb{Q}^{|I|}.
\]
\end{proposition}

\begin{proof}
Suppose we are given a rational number $a_i$ for each edge $i\in I$. It is enough to construct a bounded lamination with coordinates 
\[
\tilde{a_i}=pa_i+q
\]
for some rational numbers $p$ and~$q$. Indeed, if we construct a lamination with these coordinates, then we can add a special curve of weight $-q$ around each puncture and each marked point and then divide the weight of each curve by $p$ to get a lamination with coordinates~$a_i$.

Let $\tilde{S}$ denote the universal cover of $S$. Then we can lift the triangulation of $S$ to a triangulation of $\tilde{S}$, and we can associate to each edge of this triangulation the number $\tilde{a_i}$ associated to its projection.

Let $t$ be any triangle in the triangulation of $\tilde{S}$. By an argument in~\cite{dual}, we can fix $p$ and $q$ so that there exists a corresponding topological triangle $u$ with finitely many nonintersecting curves joining each pair of adjacent edges where the edge corresponding to $i$ intersects the curves exactly $2\tilde{a_i}$ times. By gluing the triangles $u$ and matching the curves at each edge, we recover the universal cover of $S$ together with a collection of curves. Quotienting the resulting space by the group of deck transformations, we obtain the desired bounded lamination on the surface~$S$.
\end{proof}

\begin{proposition}
\label{prop:flipAlam}
A regular flip at an edge $k$ of the triangulation changes the coordinates~$a_i$ to new coordinates~$a_i'$ given by the formula 
\[
a_i' =
\begin{cases}
\max\biggr(\sum_{j|\varepsilon_{kj}>0}\varepsilon_{kj}a_j,-\sum_{j|\varepsilon_{kj}<0}\varepsilon_{kj}a_j\biggr)-a_k & \mbox{if } i=k \\
a_i & \mbox{if } i\neq k.
\end{cases}
\]
\end{proposition}

This proposition immediately implies the following.

\begin{theorem}
The space of rational bounded laminations is naturally identified with the space of tropical rational points of the cluster $K_2$-variety associated to the surface~$S$:
\[
\mathcal{A}_L(S,\mathbb{Q})=\mathcal{A}(\mathbb{Q}^t).
\]
\end{theorem}

Since the transformation rules in Proposition~\ref{prop:flipAlam} are continuous with respect to the standard topology on $\mathbb{Q}^{|I|}$, the coordinates define a natural topology on $\mathcal{A}_L(S,\mathbb{Q})$. 

\begin{definition}
The space of \emph{real bounded laminations} is defined as the metric space completion of $\mathcal{A}_L(S,\mathbb{Q})$.
\end{definition}

This space of real bounded laminations is identified with the space $\mathcal{A}(\mathbb{R}^t)$ of $\mathbb{R}^t$-points of the cluster $K_2$-variety associated to the surface~$S$. There is a natural action of the group $\mathbb{R}_{>0}$ on this space where an element $\lambda\in\mathbb{R}_{>0}$ acts by multiplying the coordinates in any coordinate system by~$\lambda$. The \emph{spherical tropical space} $\mathcal{SA}(\mathbb{R}^t)$ is the quotient 
\[
\mathcal{SA}(\mathbb{R}^t)=\left(\mathcal{A}(\mathbb{R}^t)-0\right)/\mathbb{R}_{>0}.
\]
By Proposition~2.2 of~\cite{infinity}, we know that $\mathcal{SA}(\mathbb{R}^t)$ can be viewed as a boundary of the decorated Teichm\"uller space~$\mathcal{A}^+(S)$. In~\cite{Penner}, Chapter~5, Section~5.4 (see also~\cite{PP}), the points of $\mathcal{A}(\mathbb{R}^t)$ are identified with compactly supported measured geodesic laminations in the sense of Thurston~\cite{Th}, with extra data associated to the punctures.

\subsection{Unbounded laminations}

There is an analogous notion of unbounded lamination.

\begin{definition}
A \emph{rational unbounded lamination} on $S$ is the homotopy class of a collection of finitely many simple nonintersecting noncontractible and non-special curves on $S$ with positive rational weights and a choice of orientation for each puncture in $S$ that meets a curve. A lamination containing homotopic curves of weights $a$ and~$b$ is equivalent to the lamination with one curve removed and the weight $a+b$ on the other.
\end{definition}

The set of all rational unbounded laminations on $S$ is denoted $\mathcal{X}_L(S,\mathbb{Q})$. We will write $\mathcal{X}_L(S,\mathbb{Z})$ for the set of all unbounded laminations on $S$ that can be represented by a collection of curves with integral weights. When there is no possibility of confusion, we will simply write $\mathcal{X}_L$.

Now consider an unbounded lamination $l$ and an ideal triangulation $T$. Remove a small neighborhood of each puncture in $S$ that meets a curve to get a new surface $S'$ with boundary. The edges of $T$ correspond to arcs on this surface $S'$. If an arc ends on a hole, we wind its endpoint around the hole infinitely many times in the direction prescribed by the orientation so that the arcs spiral into the holes in~$S'$.

Let $k$ be an internal edge of the resulting triangulation of $S'$, and consider the quadrilateral on the surface with diagonal $k$. There will be finitely many curves that connect opposite sides of the quadrilateral and possibly infinitely many curves that join adjacent sides. Number the vertices of this quadrilateral in counterclockwise order as shown below so that the edge $k$ joins vertices~1 and~3.
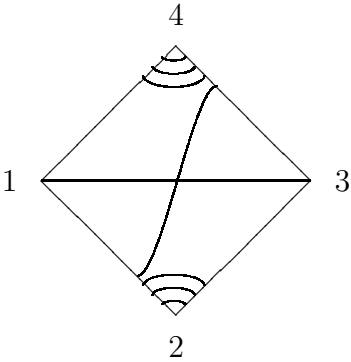
\[
\xy /l1.5pc/:
{\xypolygon4"A"{~:{(2,2):}}};
{\xypolygon4"B"{~:{(2.5,0):}~>{}}};
{\xypolygon4"C"{~:{(0.8,0.8):}~>{}}};
{"A2"\PATH~={**@{-}}'"A4"};
(4.5,0)*{1};
(1,3.5)*{2};
(-2.5,0)*{3}; 
(1,-3.5)*{4};
(1.8,2)*{}="A0";
(0.4,-2.2)*{}="A1";
(1.7,-2.2)*{}="A2";
(0.6,-2.4)*{}="A3";
(1.5,-2.4)*{}="A4";
(0.8,-2.6)*{}="A5";
(1.3,-2.6)*{}="A6";
(0.15,-2)*{}="B0";
(0.4,2.2)*{}="B1";
(1.7,2.2)*{}="B2";
(0.6,2.4)*{}="B3";
(1.5,2.4)*{}="B4";
(0.8,2.6)*{}="B5";
(1.3,2.6)*{}="B6";
"A0";"B0" **\crv{(1.3,1.9) & (0.65,-1.9)};
"A1";"A2" **\crv{(0.5,-1.9) & (1.6,-1.9)};
"A3";"A4" **\crv{(0.7,-2.2) & (1.4,-2.2)};
"A5";"A6" **\crv{(0.9,-2.5) & (1.2,-2.5)};
"B1";"B2" **\crv{(0.5,1.9) & (1.6,1.9)};
"B3";"B4" **\crv{(0.7,2.2) & (1.4,2.2)};
"B5";"B6" **\crv{(0.9,2.5) & (1.2,2.5)};
\endxy
\]

Let $p$ be a vertex of this quadrilateral. If there are infinitely many curves connecting the edges that meet at $p$, then we can choose one such curve $\alpha_p$ and delete all of the curves between $\alpha_p$ and the point $p$. By doing this for each vertex, we remove all but finitely many curves from the quadrilateral and get a bounded lamination on a disk with four marked points. Let $a_{ij}$ be the coordinate of this bounded lamination corresponding to the edge connecting vertices $i$ and $j$. Define 
\[
x_k=a_{12}+a_{34}-a_{14}-a_{23}.
\]
It is easy to see that this number is independent of all choices made in the construction. It is the number associated to the edge $k$.

\begin{proposition}
The numbers $x_j$ ($j\in J$) provide a bijection 
\[
\mathcal{X}_L(S,\mathbb{Q})\rightarrow\mathbb{Q}^{|J|}.
\]
\end{proposition}

\begin{proof}
Suppose we are given a rational number $x_j$ for each edge $j\in J$. To recover the orientation of a puncture $h$, we compute the number $\sum x_i$, where the sum is over all edges $j$ incident to~$h$. This sum is negative (respectively, positive) when the orientation is induced from the orientation of~$S$ (respectively, the opposite of this orientation).

Let $\tilde{S}$ denote the universal cover of the surface $S$. We can lift the ideal triangulation of~$S$ to an ideal triangulation of~$\tilde{S}$, and we can associate to each edge of this triangulation the number associated to its projection.

Let $t_0$ be any triangle in the triangulation of $\tilde{S}$, and choose a corresponding ideal triangle $u_0$ in the hyperbolic plane $\mathbb{H}$. There is a unique triple of horocycles about its endpoints that are pairwise tangent. The points of tangency provide three canonical points on the edges of~$u_0$. Parametrize the edges of this triangle by $\mathbb{R}$, respecting the orientation induced by the orientation of the triangle, so that the point with parameter $s\in\mathbb{R}$ lies at distance $|s|$ from the distinguished point. Connect points with parameter $s\in\frac{1}{2}+\mathbb{Z}_{\geq0}$ on one side to points with parameter~$-s$ on the next side in the clockwise direction by a horocyclic arc. In this way, we obtain the triangle $u_0$ with infinitely many arcs connecting adjacent sides.
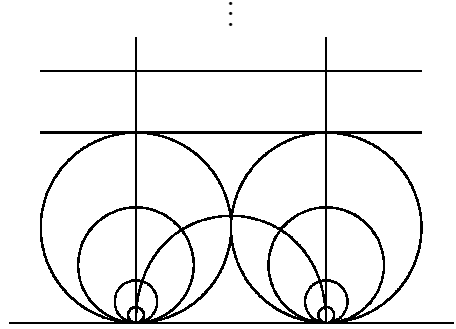
\[
\xy 0;/r.50pc/: 
(6,14)*{\vdots}; 
(0,-6)*{}="2"; 
(12,-6)*{}="3"; 
"2";"3" **\crv{(0,3) & (12,3)}; 
(0,12)*{}="B"; 
(12,12)*{}="C";
(0,-6)*{}="B1"; 
(12,-6)*{}="C1";
(-8,-6)*{}="X"; 
(20,-6)*{}="Y"; 
"B";"B1" **\dir{-}; 
"C";"C1" **\dir{-}; 
"X";"Y" **\dir{-}; 
(-6,6)*{}="X1"; 
(18,6)*{}="Y1"; 
"X1";"Y1" **\dir{-}; 
(-6,9.892)*{}="X2"; 
(18,9.892)*{}="Y2"; 
"X2";"Y2" **\dir{-}; 
(0,0)*\xycircle(6,6){-};
(0,-2.361)*\xycircle(3.639,3.639){-};
(0,-4.661)*\xycircle(1.339,1.339){-};
(0,-5.507)*\xycircle(0.493,0.493){-};
(12,0)*\xycircle(6,6){-};
(12,-2.361)*\xycircle(3.639,3.639){-};
(12,-4.661)*\xycircle(1.339,1.339){-};
(12,-5.507)*\xycircle(0.493,0.493){-};
\endxy
\]

Next, consider a triangle $t$ adjacent to $t_0$ in the ideal triangulation of $\tilde{S}$. There is a number $x_j$ associated to the common edge $j$, and we can find an ideal triangle $u$ adjacent to~$u_0$ so that the cross ratio of the resulting quadrilateral is $e^{x_j}$. We can draw infinitely many horocyclic arcs on $u$ as we did for $u_0$. By~\cite{Penner}, Chapter~1, Corollary~4.16, these arcs connect to the ones already drawn on $u_0$.

Continuing in this way, we obtain a triangulation of a region in $\mathbb{H}$ by ideal triangles where each triangle corresponds to a triangle in $\tilde{S}$. Quotienting this region by the action of the fundamental group, we obtain a surface homeomorphic to $S$ with curves drawn on it. One can check that this construction provides a two-sided inverse of the map $\mathcal{X}_L(S)\rightarrow\mathbb{Q}^{|J|}$.
\end{proof}

\begin{proposition}
\label{prop:flipXlam}
A regular flip at an edge $k$ of the triangulation changes the coordinates~$x_i$ to new coordinates $x_i'$ given by the formula 
\begin{align*}
x_i' &=
\begin{cases}
-x_k & \mbox{if } i=k \\ 
x_i+\varepsilon_{ki}\max\left(0,\sgn(\varepsilon_{ki})x_k\right) & \mbox{if } i\neq k.
\end{cases}
\end{align*}
\end{proposition}

This proposition immediately implies the following.

\begin{theorem}
The space of rational unbounded laminations is naturally identified with the space of tropical rational points of the cluster Poisson variety associated to the surface~$S$:
\[
\mathcal{X}_L(S,\mathbb{Q})=\mathcal{X}(\mathbb{Q}^t).
\]
\end{theorem}

Since the transformation rules in Proposition~\ref{prop:flipXlam} are continuous with respect to the standard topology on $\mathbb{Q}^{|J|}$, the coordinates define a natural topology on $\mathcal{X}_L(S,\mathbb{Q})$. 

\begin{definition}
The space of \emph{real unbounded laminations} is defined as the metric space completion of $\mathcal{X}_L(S,\mathbb{Q})$.
\end{definition}

This space of real unbounded laminations is identified with the space $\mathcal{X}(\mathbb{R}^t)$ of $\mathbb{R}^t$-points of the cluster Poisson variety associated to the surface~$S$. There is a natural action of the group $\mathbb{R}_{>0}$ on this space where an element $\lambda\in\mathbb{R}_{>0}$ acts by multiplying the coordinates in any coordinate system by~$\lambda$. The \emph{spherical tropical space} $\mathcal{SX}(\mathbb{R}^t)$ is the quotient 
\[
\mathcal{SX}(\mathbb{R}^t)=\left(\mathcal{X}(\mathbb{R}^t)-0\right)/\mathbb{R}_{>0}.
\]
By Proposition~2.2 of~\cite{infinity}, we know that $\mathcal{SX}(\mathbb{R}^t)$ can be viewed as a boundary of the enhanced Teichm\"uller space~$\mathcal{X}^+(S)$.

\subsection{Doubled laminations}

Finally, we state our results for doubled laminations. As before, a \emph{simple lamination} on a surface is a finite collection $\gamma=\{\gamma_i\}$ of simple noncontractible nonspecial disjoint nonisotopic loops considered up to isotopy. The surface~$S_\mathcal{D}$ comes equipped with a simple lamination given by the image of the boundary loops of~$S$ in~$S_\mathcal{D}$.

\begin{definition}
A \emph{rational doubled lamination} on $S_\mathcal{D}$ is the homotopy class of a collection of finitely many simple nonintersecting noncontractible and non-special closed curves with positive rational weights and a choice of orientation for each component of $\gamma$ which meets or is homotopic to a curve. A lamination containing homotopic curves of weights $a$ and $b$ is equivalent to the lamination with one curve removed and the weight $a+b$ on the other.
\end{definition}

The set of all rational doubled laminations on $S_\mathcal{D}$ will be denoted $\mathcal{D}_L(S,\mathbb{Q})$. We will write $\mathcal{D}_L(S,\mathbb{Z})$ for the set of all doubled laminations on $S_\mathcal{D}$ that can be represented by a collection of curves with integral weights. When there is no possibility of confusion, we will simply write~$\mathcal{D}_L$.

Let $l\in\mathcal{D}_L(S,\mathbb{Q})$. Then $l$ can be represented by a collection of nonintersecting simple closed curves on $S_\mathcal{D}$ such that all curves have the same rational weight. By cutting along~$\partial S$, we recover the surfaces $S$ and $S^\circ$ with curves drawn on them.

Let $\tilde{S}$ and $\tilde{S^\circ}$ denote the universal covers of $S$ and $S^\circ$, respectively. It is useful when drawing pictures to choose hyperbolic structures on $S$ and $S^\circ$. We can choose these hyperbolic structures so that if $\gamma$ is a component of $\partial S$ or $\partial S^\circ$ that does not meet a curve of the lamination, then the monodromy around $\gamma$ is parabolic. Then the universal covers can be obtained from $\mathbb{H}$ by removing countably many geodesic half disks. Below we will assume that these hyperbolic structures have been specified. None of our constructions will depend on the choice of hyperbolic structures.

\begin{lemma}
\label{lem:correspondence2}
There exists a map $\tilde{\iota}:\tilde{S}\rightarrow\tilde{S^\circ}$, unique up to the action of $\pi_1(S)$ by deck transformations, such that the diagram 
\[
\xymatrix{ 
\tilde{S} \ar[r]^{\tilde{\iota}} \ar[d] & \tilde{S^\circ} \ar[d] \\
S \ar_{\iota}[r] & S^\circ
}
\]
commutes. Points on $\partial\tilde{S}$ that project to points on the curves of the lamination are mapped bijectively to points on $\partial\tilde{S^\circ}$ that project to points on the curves, and this bijection preserves the natural order of these points.
\end{lemma}

\begin{proof}
The lifting criterion ensures that there is a map $\tilde{\iota}:\tilde{S}\rightarrow\tilde{S^\circ}$ lifting the composition $\tilde{S}\rightarrow S\rightarrow S^\circ$. It is easy to check that this map has the required properties.
\end{proof}

To construct coordinates on $\mathcal{D}_L(S,\mathbb{Q})$, fix a point $l\in\mathcal{D}_L(S,\mathbb{Q})$ and let $i\in J$ be an edge of the ideal triangulation $T$. Realize the lamination $l$ by a collection of nonintersecting simple closed curves on $S_\mathcal{D}$ such that all curves have the same rational weight. By cutting along the image of $\partial S$, we recover the surfaces $S$ and $S^\circ$ with curves drawn on them.

Suppose the edge $i$ corresponds to an arc connecting two holes $\gamma_1$ and~$\gamma_2$ in~$S$ which meet the curves of $l$. Choose a pair of geodesics $g_1$ and $g_2$ in $\partial\tilde{S}$ which project to these boundary components. Deform the curve connecting the holes by winding its endpoint around the holes infinitely many times in the direction prescribed by the orientation. This corresponds to deforming the preimage $\tilde{i}$ in~$\tilde{S}$ so that its endpoints coincide with endpoints of $g_1$ and $g_2$. If we let $i^\circ$ denote the image of $i$ under the tautological map $S\rightarrow S^\circ$, then we can apply the same procedure to $i^\circ$ to get an arc $\tilde{i^\circ}$ in $\tilde{S^\circ}$.
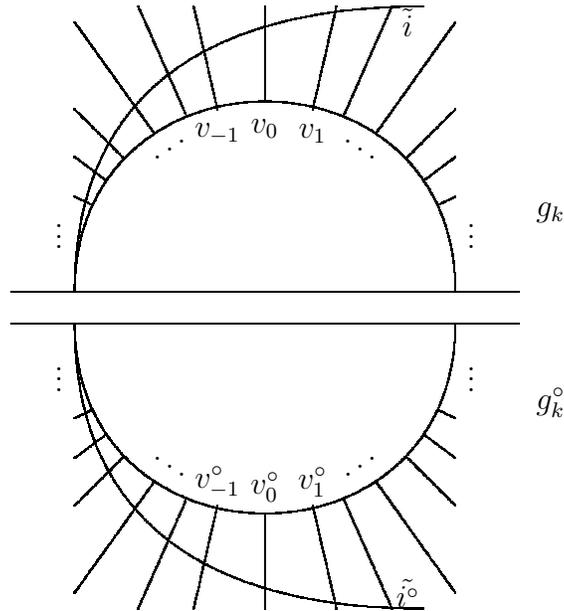
\[
\xy 0;/r.50pc/: 
(-12,-6)*{}="1"; 
(12,-6)*{}="2"; 
"1";"2" **\crv{(-12,10) & (12,10)}; 
(-16,-6)*{}="X"; 
(16,-6)*{}="Y"; 
"X";"Y" **\dir{-}; 
(18,-1)*{g_k}; 
(-12,0)*{}="P1";
(-11,-0.5)*{}="Q1"; 
(-12,2.5)*{}="P2";
(-10,1)*{}="Q2"; 
(-12,5.5)*{}="P3";
(-9,2.5)*{}="Q3"; 
(-12,11)*{}="P4";
(-7,4)*{}="Q4";
(-8,12)*{}="P5";
(-5,5.2)*{}="Q5"; 
(-4.5,12)*{}="P6";
(-3,5.5)*{}="Q6"; 
(0,12)*{}="P7";
(0,6)*{}="Q7"; 
(4.5,12)*{}="P8";
(3,5.5)*{}="Q8"; 
(8,12)*{}="P9";
(5,5.2)*{}="Q9"; 
(12,11)*{}="P10";
(7,4)*{}="Q10"; 
(12,5.5)*{}="P11";
(9,2.5)*{}="Q11"; 
(12,2.5)*{}="P12";
(10,1)*{}="Q12"; 
(12,0)*{}="P13";
(11,-0.5)*{}="Q13"; 
"P1";"Q1" **\crv{(-12,0) & (-11,-0.5)}; 
"P2";"Q2" **\crv{(-12,2.5) & (-10,1)}; 
"P3";"Q3" **\crv{(-12,5.5) & (-9,2.5)}; 
"P4";"Q4" **\crv{(-12,11) & (-7,4)}; 
"P5";"Q5" **\crv{(-8,12) & (-5,5.2)}; 
"P6";"Q6" **\crv{(-4.5,12) & (-3,5.5)}; 
"P7";"Q7" **\crv{(0,12) & (0,6)}; 
"P8";"Q8" **\crv{(4.5,12) & (3,5.5)}; 
"P9";"Q9" **\crv{(8,12) & (5,5.2)}; 
"P10";"Q10" **\crv{(12,11) & (7,4)}; 
"P11";"Q11" **\crv{(12,5.5) & (9,2.5)}; 
"P12";"Q12" **\crv{(12,2.5) & (10,1)}; 
"P13";"Q13" **\crv{(12,0) & (11,-0.5)}; 
(-3,4)*{v_{-1}};
(0,4.3)*{v_0};
(3,4)*{v_1};
(-6,3)*{\Ddots};
(6,3.5)*{\ddots};
(13,-2)*{\vdots};
(-13,-2)*{\vdots};
(10,12)*{}="B";
"1";"B" **\crv{(-12,9) & (0,12)}; 
(9,11)*{\tilde{i}};
(-16,-8)*{}="Z"; 
(16,-8)*{}="W"; 
"Z";"W" **\dir{-}; 
(-12,-8)*{}="1"; 
(12,-8)*{}="2"; 
"1";"2" **\crv{(-12,-24) & (12,-24)}; 
(18,-13)*{g_k^\circ};
(-12,-14)*{}="P1";
(-11,-13.5)*{}="Q1"; 
(-12,-16.5)*{}="P2";
(-10,-15)*{}="Q2"; 
(-12,-19.5)*{}="P3";
(-9,-16.5)*{}="Q3"; 
(-12,-25)*{}="P4";
(-7,-18)*{}="Q4"; 
(-8,-26)*{}="P5";
(-5,-19.2)*{}="Q5"; 
(-4.5,-26)*{}="P6";
(-3,-19.5)*{}="Q6"; 
(0,-26)*{}="P7";
(0,-20)*{}="Q7"; 
(4.5,-26)*{}="P8";
(3,-19.5)*{}="Q8"; 
(8,-26)*{}="P9";
(5,-19.2)*{}="Q9"; 
(12,-25)*{}="P10";
(7,-18)*{}="Q10"; 
(12,-19.5)*{}="P11";
(9,-16.5)*{}="Q11"; 
(12,-16.5)*{}="P12";
(10,-15)*{}="Q12"; 
(12,-14)*{}="P13";
(11,-13.5)*{}="Q13"; 
"P1";"Q1" **\crv{(-12,-14) & (-11,-13.5)}; 
"P2";"Q2" **\crv{(-12,-16.5) & (-10,-15)}; 
"P3";"Q3" **\crv{(-12,-19.5) & (-9,-16.5)}; 
"P4";"Q4" **\crv{(-12,-25) & (-12,-25)}; 
"P5";"Q5" **\crv{(-8,-26) & (-5,-19.2)}; 
"P6";"Q6" **\crv{(-4.5,-26) & (-3,-19.5)}; 
"P7";"Q7" **\crv{(0,-26) & (0,-20)}; 
"P8";"Q8" **\crv{(4.5,-26) & (3,-19.5)}; 
"P9";"Q9" **\crv{(8,-26) & (5,-19.2)}; 
"P10";"Q10" **\crv{(12,-25) & (7,-18)}; 
"P11";"Q11" **\crv{(12,-19.5) & (9,-16.5)}; 
"P12";"Q12" **\crv{(12,-16.5) & (10,-15)}; 
"P13";"Q13" **\crv{(12,-14) & (11,-13.5)}; 
(6,-17)*{\Ddots};
(-6,-16.5)*{\ddots};
(13,-11)*{\vdots};
(-13,-11)*{\vdots};
(-3,-18)*{v_{-1}^\circ};
(0,-18.3)*{v_0^\circ};
(3,-18)*{v_1^\circ};
(10,-26)*{}="B";
"1";"B" **\crv{(-12,-23) & (0,-26)}; 
(9,-25)*{\tilde{i^\circ}};
\endxy
\]

Observe that the curves of the lamination that end on~$\gamma_k$ can be lifted to infinitely many curves in~$\tilde{S}$ that end on~$g_k$. Similarly, if $g_k^\circ$ is the geodesic in~$\tilde{S^\circ}$ that projects to $\gamma_k$, then the curves in~$S^\circ$ that end on~$\gamma_k$ can be lifted to infinitely many curves in~$\tilde{S^\circ}$ that end on~$g_k^\circ$. Label the endpoints of curves on $g_k$ by the symbols~$v_\alpha$ ($\alpha\in\mathbb{Z}$), and label the endpoints of curves on $g_k^\circ$ by~$v_\alpha^\circ$ ($\alpha\in\mathbb{Z}$). By Lemma~\ref{lem:correspondence2}, there is a map $\tilde{\iota}:\tilde{S}\rightarrow\tilde{S^\circ}$ that projects to the natural map $\iota:S\rightarrow S^\circ$ and is unique up to the action of $\pi_1(S)$ by deck transformations. This provides a bijection 
\[
f:\{\text{vertices }v_\alpha\}\rightarrow\{\text{vertices }v_\alpha^\circ\}
\]
which preserves the order of the vertices. Choose a vertex $v_{\alpha(k)}$ on $g_k$ and let $v_{\beta(k)}^\circ$ be the corresponding vertex given by $v_{\beta(k)}^\circ=f(v_{\alpha(k)})$. We can choose $v_{\alpha(k)}$ in such a way that the curve ending at $v_{\alpha(k)}$ intersects $\tilde{i}$ and the curve ending at $v_{\beta(k)}^\circ$ intersects $\tilde{i^\circ}$. Notice that if we choose a different vertex $v_{\alpha(k)}$, then the vertex $v_{\beta(k)}^\circ$ will change by a corresponding amount.

By construction, the curve that ends at $v_{\alpha(k)}$ intersects $\tilde{i}$ at some point $p(k)$. Denote by~$a_i$ half the number of intersections between the lifted curves of $l$ and the lifted edge $\tilde{i}$ between the points~$p(1)$ and~$p(2)$. Similarly, the curve that ends at $v_{\beta(k)}^\circ$ intersects $\tilde{i^\circ}$ at some point $p^\circ(k)$. Let $a_i^\circ$ be half the number of intersections between the lifted curves of $l$ and the lifted edge $\tilde{i^\circ}$ between the points~$p^\circ(1)$ and~$p^\circ(2)$. We can then define a coordinate associated to the edge $i$ by 
\[
b_i=a_i^\circ-a_i.
\]
This defines $b_i$ when $i$ corresponds to an arc connecting two holes. If one or both of the endpoints of $i$ are punctures, then $i$ intersects only finitely many curves near these punctures, and $i^\circ$ intersects only finitely many curves near the corresponding punctures in~$S^\circ$. Thus we can associate the half intersection numbers $a_i$ and $a_i^\circ$ to these arcs as before, and we can define $b_i$ by the above formula. One can show that the $|J|$ numbers obtained in this way are independent of all choices made in the construction.

In addition to the $b_i$, there are $|J|$ numbers $x_i$ associated to a point in the space $\mathcal{D}_L(S,\mathbb{Q})$. Given a point of $\mathcal{D}_L(S,\mathbb{Q})$, these are simply defined as the $X$-coordinates of the unbounded lamination on the surface $S$ obtained by cutting $S_\mathcal{D}$ along the image of $\partial S$.

\begin{proposition}
The numbers $b_i$ and $x_i$ provide a bijection 
\[
\mathcal{D}_L(S,\mathbb{Q})\rightarrow\mathbb{Q}^{2|J|}.
\]
\end{proposition}

\begin{proof}
Call this map $\phi$. We will construct a map $\psi:\mathbb{Q}^{2|J|}\rightarrow\mathcal{D}_L(S,\mathbb{Q})$ and prove that $\phi$ and $\psi$ are inverses.

\Needspace*{2\baselineskip}
\begin{step}[1]
Definition of $\psi:\mathbb{Q}^{2|J|}\rightarrow\mathcal{D}_L(S,\mathbb{Q})$.
\end{step}

Suppose we are given rational numbers $b_j$ and $x_j$ for every $j\in J$. By the reconstruction procedure for unbounded laminations, we can use the $x_j$ to glue together ideal triangles to get a region $\tilde{S}\subseteq\mathbb{H}$ with geodesic boundary together with infinitely many curves.

Let $t_0$ be a triangle in $\tilde{S}$. There are infinitely many curves connecting each pair of adjacent sides of this triangle. For each vertex $p$ of $t_0$, choose a curve connecting the two sides of this triangle that meet at $p$. Then for every edge $i$, there is a number $a_i$ defined as half the total weight of the curves that intersect $i$ between the distinguished curves. Define 
\[
a_i^\circ=b_i+a_i.
\]
We will use these numbers $a_i^\circ$ to construct another region $\tilde{S^\circ}$ with geodesic boundary in a copy of $\mathbb{H}$.

To construct $\tilde{S^\circ}$, let $u_0$ be an ideal triangle with infinitely many curves connecting each pair of adjacent sides as before. If we choose the distinguished curves on $t_0$ sufficiently close to the vertices, then we can choose a triple of distinguished curves near the vertices of $u_0$ so that $a_i^\circ$ is half the total weight of the curves that intersect an edge of $u_0$ between distinguished curves.

Now suppose $t$ is a triangle adjacent to $t_0$ in the triangulation of $\tilde{S}$. As before we can choose a distinguished curve near each vertex of $t$ to get a triple of numbers~$a_i$. Then we can draw an ideal triangle $u$ with infinitely many curves connecting each pair of adjacent sides, and there are distinguished curves near each vertex of $u$ realizing the numbers $a_i^\circ$. We can choose the distinguished curves on $t_0$ and $t$ so that they coincide at the common edge $t_0\cap t$, and then we can glue $u_0$ and $u$ so that their distinguished curves coincide. Continuing this process inductively, we obtain the desired space $\tilde{S^\circ}$.

Now the region $\tilde{S}$ is obtained from $\mathbb{H}$ by removing infinitely many geodesic half disks. We can extend $\tilde{S}$ to a larger region by gluing a copy of $\tilde{S^\circ}$ in each of these half disks in such a way that corresponding curves are identified. The resulting region is again obtained from $\mathbb{H}$ by removing infinitely many geodesic half disks, and we can enlarge it by gluing a copy of $\tilde{S}$ in each of these half disks. Continuing this process ad infinitum, we partition the hyperbolic plane into ideal triangles together with infinitely many curves. Quotienting this space by the group of deck transformations, we recover the surface $S_{\mathcal{D}}$ with a collection of curves. There may be infinitely many curves homotopic to some component $\gamma_i$ of $\gamma$. In this case we delete a maximal collection of such curves between the distinguished curves. If there are any remaining curves homotopic to $\gamma_i$, then they all lie on one of the surfaces $S$ or $S^\circ$, and we choose the orientation of $\gamma_i$ to agree with the orientation of this surface.

\Needspace*{2\baselineskip}
\begin{step}[2]
Proof of $\phi\circ\psi=1_{\mathbb{Q}^{2|J|}}$.
\end{step}

Let $b_j$ and $x_j$ ($j\in J$) be given. By the reconstruction procedure described above, we construct triangulated regions $\tilde{S}\subseteq\mathbb{H}$ and $\tilde{S}^\circ\subseteq{H}$ which we then glue together to get a triangulation of the hyperbolic plane together with a collection of curves. From this data, we get a doubled lamination $l$ on $S_{\mathcal{D}}$. We want to show that the coordinates of $l$ are the numbers~$b_j$ and~$x_j$.

We can lift $l$ to its universal cover, which is exactly the triangulated surface $\tilde{S}_{\mathcal{D}}$ that we got by gluing together copies of $\tilde{S}$ and $\tilde{S}^\circ$. Consider a copy of $\tilde{S}$ in $\tilde{S}_{\mathcal{D}}$, and let $\tilde{i}$ be an edge of the triangulation of $\tilde{S}$. To find the coordinate of $l$ corresponding to the edge $\tilde{i}$, we must choose near each endpoint of $\tilde{i}$ a distinguished curve in $\tilde{S}$ which intersects $\tilde{i}$. Now consider a copy of $\tilde{S}^\circ$ adjacent to $\tilde{S}$ in $\tilde{S}_{\mathcal{D}}$. Choose a basepoint~$x$ on the geodesic separating these regions. Suppose the edge $\tilde{i}$ is asymptotic to an endpoint of the boundary geodesic $g$ of $\tilde{S}$. Then there is a distinguished curve $c$ that intersects $\tilde{i}$ and $g$. Let $p$ be the point of intersection with $g$. Draw a curve $\tilde{\alpha}$ in~$\tilde{S}$ from the point~$p$ to the point~$x$.
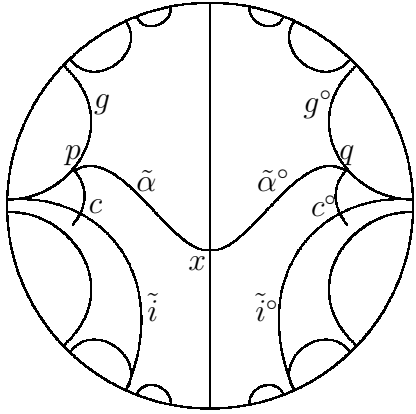
\[
\xy 0;/r.40pc/: 
(0,0)*\xycircle(16,16){-};
(-10.7,2.75)*{}="p"; 
(-10.8,-1.5)*{}="p1"; 
(0,-3.5)*{}="x"; 
(10.7,2.75)*{}="q"; 
(10.8,-1.5)*{}="q1"; 
(0,16)*{}="X"; 
(0,-16)*{}="Y"; 
(-3.12,15.69)*{}="A1"; 
(-5.8,14.9)*{}="A2"; 
(-6.6,14.5)*{}="A3"; 
(-11,11.5)*{}="A4"; 
(-11.5,11)*{}="A5"; 
(-15.9,0.5)*{}="A6"; 
(-15.9,-0.5)*{}="A7"; 
(-11.5,-11)*{}="A8"; 
(-11,-11.5)*{}="A9"; 
(-6.6,-14.5)*{}="A10"; 
(-5.8,-14.9)*{}="A11"; 
(-3.12,-15.69)*{}="A12"; 
(3.12,15.69)*{}="B1"; 
(5.8,14.9)*{}="B2"; 
(6.6,14.5)*{}="B3"; 
(11,11.5)*{}="B4"; 
(11.5,11)*{}="B5"; 
(15.9,0.5)*{}="B6"; 
(15.9,-0.5)*{}="B7"; 
(11.5,-11)*{}="B8"; 
(11,-11.5)*{}="B9"; 
(6.6,-14.5)*{}="B10"; 
(5.8,-14.9)*{}="B11"; 
(3.12,-15.69)*{}="B12"; 
"p";"x" **\crv{(-7,5) & (-3,-4)}; 
"x";"q" **\crv{(3,-4) & (7,5)}; 
"X";"Y" **\dir{-}; 
"A1";"A2" **\crv{(-2.9,14.4) & (-5.1,13.4)}; 
"A3";"A4" **\crv{(-5,12) & (-9,9)}; 
"A5";"A6" **\crv{(-6.6,6.5) & (-11,0.4)}; 
"A7";"A8" **\crv{(-11,-0.4) & (-6.6,-6.5)}; 
"A9";"A10" **\crv{(-9,-9) & (-5,-12)}; 
"A11";"A12" **\crv{(-5.1,-13.4) & (-2.9,-14.4)}; 
"B1";"B2" **\crv{(2.9,14.4) & (5.1,13.4)}; 
"B3";"B4" **\crv{(5,12) & (9,9)}; 
"B5";"B6" **\crv{(6.6,6.5) & (11,0.4)}; 
"B7";"B8" **\crv{(11,-0.4) & (6.6,-6.5)}; 
"B9";"B10" **\crv{(9,-9) & (5,-12)}; 
"B11";"B12" **\crv{(5.1,-13.4) & (2.9,-14.4)}; 
"A6";"A10" **\crv{(-8,1) & (-3,-6)}; 
"B6";"B10" **\crv{(8,1) & (3,-6)}; 
"p";"p1" **\crv{(-9.6,1.5) & (-9.6,0)}; 
"q";"q1" **\crv{(9.6,1.5) & (9.6,0)}; 
(-1,-4.5)*{x}; 
(-5,2)*{\tilde{\alpha}}; 
(5,2)*{\tilde{\alpha}^\circ}; 
(-10.8,4)*{p}; 
(10.8,4)*{q}; 
(-9,0)*{c}; 
(9,0)*{c^\circ}; 
(-4.5,-8)*{\tilde{i}}; 
(4.5,-8)*{\tilde{i}^\circ}; 
(-8.5,8)*{g}; 
(8.5,8)*{g^\circ}; 
\endxy
\]
This curve $\tilde{\alpha}$ projects to a curve~$\alpha$ on the surface $S$, and we can apply the map~$\iota$ to get a curve $\alpha^\circ$ on $S^\circ$. Finally, we lift $\alpha^\circ$  to a curve $\tilde{\alpha}^\circ$ in~$S^\circ$ that starts at $x$. This curve $\tilde{\alpha}^\circ$ ends at some point $q=\tilde{\iota}(p)$, and there is a unique lifted curve $c^\circ$ of the lamination that passes through $q$. In this way, we get distinguished curves near the endpoints of the edge $\tilde{i}^\circ$. These can be used to define the coordinates of $l$.

On the other hand, consider the monodromy along the curve in $S_{\mathcal{D}}$ obtained by concatenating $\alpha$ and $\alpha^\circ$. It maps the region~$\tilde{S}$ isometrically into the region bounded by~$g^\circ$. In the construction of $l$ from the coordinates~$b_j$ and~$x_j$, there is a correspondence between curves that end on $g$ and curves that end on $g^\circ$, and this correspondence is obtained by applying this transformation. Thus the distinguished curves used to construct $l$ agree with the ones obtained using the map $\tilde{\iota}$.

It follows that the $b$-coordinates of $l$ are simply the numbers $b_j$ that we started with. It is easy to see that the $x$-coordinates of $l$ are the numbers $x_j$. This completes Step~2 of the proof.

\Needspace*{2\baselineskip}
\begin{step}[3]
Proof of $\psi\circ\phi=1_{\mathcal{D}_L}$.
\end{step}

Let $l\in\mathcal{D}_L(S,\mathbb{Q})$ be given. To calculate the coordinates of~$l$, let us pass to the universal cover~$\tilde{S}_{\mathcal{D}}$ of~$S_{\mathcal{D}}$ and choose a copy of $\tilde{S}$ in~$\tilde{S}_{\mathcal{D}}$. We associate numbers $x_i$ to the edges of the triangulation in the usual way. If $\tilde{i}$ is an edge of the triangulation of~$\tilde{S}$, then we can choose a distinguished curve near each endpoint of~$\tilde{i}$. Using the construction described above involving lifts of the map $\iota:S\rightarrow S^\circ$, we can get a pair of distinguished curves near the endpoints of a corresponding edge $\tilde{i}^\circ$ in~$\tilde{S}^\circ$. We can use these curves to calculate the numbers~$a_i$ and~$a_i^\circ$, and thus the coordinate $b_i$.

We can now use the $x_i$ to reconstruct the region~$\tilde{S}$. We must choose a distinguished curve near each endpoint of the edge~$\tilde{i}$, and we can assume that these are exactly the ones used above to define~$b_i$. We can use the curves to begin constructing the space~$\tilde{S}^\circ$ with a collection of curves. By gluing together copies of $\tilde{S}$ and $\tilde{S}^\circ$, we recover $\tilde{S}_{\mathcal{D}}$ with a collection of curves. Quotienting this space by the action of $\pi_1(S_{\mathcal{D}})$, we recover the point $l\in\mathcal{D}_L(S,\mathbb{Q})$. This completes Step~3 of the proof.
\end{proof}

\begin{proposition}
\label{prop:flipDlam}
A regular flip at an edge $k$ of the triangulation changes the coordinates $x_i$ and~$b_i$ to new coordinates $x_i'$ and~$b_i'$ given by the formulas 
\begin{align*}
x_i' &=
\begin{cases}
-x_k & \mbox{if } i=k \\ 
x_i+\varepsilon_{ki}\max\left(0,\sgn(\varepsilon_{ki})x_k\right) & \mbox{if } i\neq k
\end{cases} \\
b_i' &=
\begin{cases}
\max\biggr(x_k+\sum_{j|\varepsilon_{kj}>0}\varepsilon_{kj}b_j,-\sum_{j|\varepsilon_{kj}<0}\varepsilon_{kj}b_j\biggr)-\max(0,x_k)-b_k & \mbox{if } i=k \\
b_i & \mbox{if } i\neq k.
\end{cases} 
\end{align*}
\end{proposition}

\begin{proof}
By definition, the coordinates $x_i$ on $\mathcal{D}_L(S,\mathbb{Q})$ coincide with the coordinates of the unbounded lamination on the surface $S$. Therefore the transformation rule for the $x_i$ is just the usual transformation rule for the coordinates of an unbounded lamination.

To prove the second rule, deform the curves corresponding to edges of the ideal triangulation by winding their endpoints around the holes infinitely many times. Lift these deformed curves to the universal cover of $S$, and consider the ideal quadrilateral formed by the two triangles adjacent to the geodesic arc $\tilde{k}$ corresponding to $k$.
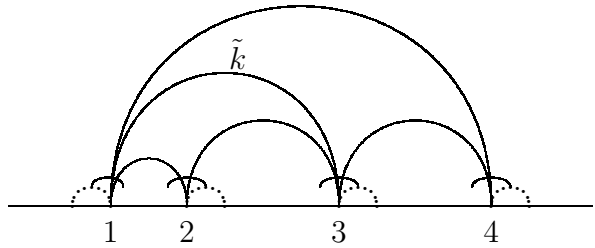
\[
\xy 0;/r.40pc/: 
(4,6)*{\tilde{k}}; 
(-6,-8)*{1}; 
(0,-8)*{2}; 
(12,-8)*{3}; 
(24,-8)*{4}; 
(-9,-6)*{}="A"; 
(3,-6)*{}="B"; 
(15,-6)*{}="C"; 
(27,-6)*{}="D"; 
(-6,-6)*{}="1"; 
(0,-6)*{}="2"; 
(12,-6)*{}="3"; 
(24,-6)*{}="4"; 
"A";"1" **\crv{~*=<2pt>{.}(-9,-4) & (-6,-4)}; 
"2";"B" **\crv{~*=<2pt>{.}(0,-4) & (3,-4)}; 
"3";"C" **\crv{~*=<2pt>{.}(12,-4) & (15,-4)}; 
"4";"D" **\crv{~*=<2pt>{.}(24,-4) & (27,-4)}; 
"1";"2" **\crv{(-6,-1) & (0,-1)}; 
"2";"3" **\crv{(0,3) & (12,3)}; 
"3";"4" **\crv{(12,3) & (24,3)}; 
"1";"4" **\crv{(-6,15) & (24,15)}; 
"1";"3" **\crv{(-6,8) & (12,8)}; 
(-7.5,-4.5)*{}="P"; 
(-5,-4.5)*{}="Q"; 
(-1.5,-4.5)*{}="R"; 
(1.5,-4.5)*{}="S"; 
(10.5,-4.5)*{}="T"; 
(13.5,-4.5)*{}="U"; 
(22.5,-4.5)*{}="V"; 
(25.5,-4.5)*{}="W"; 
"P";"Q" **\crv{(-7.5,-3.5) & (-5,-3.5)}; 
"R";"S" **\crv{(-1.5,-3.5) & (1.5,-3.5)}; 
"T";"U" **\crv{(10.5,-3.5) & (13.5,-3.5)}; 
"V";"W" **\crv{(22.5,-3.5) & (25.5,-3.5)}; 
(-14,-6)*{}="X"; 
(32,-6)*{}="Y"; 
"X";"Y" **\dir{-}; 
\endxy
\]

Number the vertices of this quadrilateral in counterclockwise order so that the edge $\tilde{k}$ joins vertices~1 and~3. In the construction of the $b_i$, we considered near each endpoint $e$ of a lifted edge $\tilde{i}$ a curve that intersects $\tilde{i}$ and projects down to a curve of the lamination. We can choose these curves to intersect both edges of the quadrilateral that meet at $e$.

Our construction will then associate a number $a_{ij}$ to the geodesic connecting~$i$ and~$j$. There is a corresponding ideal quadrilateral in the universal cover of $S^\circ$ and a number $a_{ij}^\circ$ corresponding to the edge connecting vertices~$i$ and~$j$ of this ideal quadrilateral.

If we flip at the edge~$k$, then the ideal quadrilateral in the universal cover of $S$ is replaced by the same ideal quadrilateral triangulated by the arc from~2 to~4, and the number associated to this new arc is 
\[
a_{24}=\max(a_{12}+a_{34},a_{14}+a_{23})-a_{13}.
\]
Similarly, the number $a_{13}^\circ$ transforms to 
\[
a_{24}^\circ=\max(a_{12}^\circ+a_{34}^\circ,a_{14}^\circ+a_{23}^\circ)-a_{13}^\circ.
\]
Applying these rules to the difference $b_{13}=a_{13}^\circ-a_{13}$, we obtain 
\begin{align*}
b_{24} &=\max(a_{12}^\circ+a_{34}^\circ,a_{14}^\circ+a_{23}^\circ) - \max(a_{12}+a_{34},a_{14}+a_{23})-(a_{13}^\circ-a_{13}) \\
&= \max(a_{12}^\circ+a_{34}^\circ-a_{14}-a_{23},a_{14}^\circ+a_{23}^\circ-a_{14}-a_{23}) \\
&\qquad - \max(0,a_{12}+a_{34}-a_{14}-a_{23})-(a_{13}^\circ-a_{13}) \\
&= \max(x_{13}+b_{12}+b_{34},b_{14}+b_{23}) - \max(0,x_{13}) - b_{13}
\end{align*}
where we have used the relation $x_{13}=a_{12}+a_{34}-a_{14}-a_{23}$. This proves the transformation rule for the $b_i$.
\end{proof}

This proposition immediately implies the following.

\begin{theorem}
The space of rational doubled laminations is naturally identified with the space of tropical rational points of the symplectic double associated to the surface~$S$:
\[
\mathcal{D}_L(S,\mathbb{Q})=\mathcal{D}(\mathbb{Q}^t).
\]
\end{theorem}

Since the transformation rules in Proposition \ref{prop:flipDlam} are continuous with respect to the standard topology on $\mathbb{Q}^{2|J|}$, the coordinates define a natural topology on $\mathcal{D}_L(S,\mathbb{Q})$. 

\begin{definition}
The space of \emph{real doubled laminations} is defined as the metric space completion of $\mathcal{D}_L(S,\mathbb{Q})$.
\end{definition}

This space of real $\mathcal{D}$-laminations is identified with the space $\mathcal{D}(\mathbb{R}^t)$ of $\mathbb{R}^t$-points of the symplectic double associated to the surface~$S$. There is a natural action of the group $\mathbb{R}_{>0}$ on this space where an element $\lambda\in\mathbb{R}_{>0}$ acts by multiplying the coordinates in any coordinate system by~$\lambda$. The \emph{spherical tropical space} $\mathcal{SD}(\mathbb{R}^t)$ is the quotient 
\[
\mathcal{SD}(\mathbb{R}^t)=\left(\mathcal{D}(\mathbb{R}^t)-0\right)/\mathbb{R}_{>0}.
\]
By Proposition~2.2 of~\cite{infinity}, we know that $\mathcal{SD}(\mathbb{R}^t)$ can be viewed as a boundary of the doubled Teichm\"uller space~$\mathcal{D}^+(S)$.

\section{Tropical integral points}

We will now describe those laminations that have integral coordinates. In other words, we are going to characterize the tropical integral points of the three cluster varieties.

In the case of bounded laminations, we will describe a certain topological condition on the curves. Consider a lamination $l\in\mathcal{A}_L^0(S,\mathbb{Z})$ represented by a collection of finitely many curves of weight~$\pm1$ on~$S$. If $e$ is any edge of $\partial S$ bounded by adjacent marked points, then we can drag the endpoints of the curves that end on $e$ so that they all end at the same point in the interior of $e$. Any of the resulting arcs that connects two points on $\partial S$ can be viewed as a singular 1-simplex with $\mathbb{Z}/2\mathbb{Z}$-coefficients. Similarly, if $c$ is a closed loop of the lamination, then we can view this loop as a 1-cycle with $\mathbb{Z}/2\mathbb{Z}$-coefficients. This cycle consists of a single 0-simplex in the interior of each triangle that $l$ intersects in a fixed ideal triangulation~$T$ together with a 1-simplex connecting each pair of consecutive 0-simplices. Let $\sigma_l$ be the sum of all the 1-simplices obtained in this way from arcs and closed loops. Then $\sigma_l$ is a cycle representing an element $[\sigma_l]$ of the singular homology $H_1(S,\mathbb{Z}/2\mathbb{Z})$ of $S$ with coefficients in~$\mathbb{Z}/2\mathbb{Z}$.

\begin{theorem}
\label{thm:integralAlam}
Let $l\in\mathcal{A}_L^0(S,\mathbb{Z})$ be a bounded lamination. Then $l$ has integral coordinates if and only if 
\[
[\sigma_l]=0\in H_1(S,\mathbb{Z}/2\mathbb{Z}).
\]
\end{theorem}

Theorem~\ref{thm:integralAlam} is a consequence of the analogous result for doubled laminations. We will therefore postpone the proof of Theorem~\ref{thm:integralAlam} to the end of this section.

\begin{theorem}
\label{thm:integralXlam}
A rational unbounded lamination has integral coordinates if and only if it can be represented by a collection of curves with integral weights. That is, 
\[
\mathcal{X}(\mathbb{Z}^t)=\mathcal{X}_L(S,\mathbb{Z}).
\]
\end{theorem}

\begin{proof}
Let $l$ be an unbounded lamination in $\mathcal{X}_L(S,\mathbb{Z})$. Then the definition implies that $l$ has integral coordinates, so $l\in\mathcal{X}(\mathbb{Z}^t)$, and hence $\mathcal{X}_L(S,\mathbb{Z})\subseteq\mathcal{X}(\mathbb{Z}^t)$. On the other hand, if $l\in\mathcal{X}(\mathbb{Z}^t)$, then $l$ can be reconstructed from a set of integral coordinates, and it is clear from the reconstruction procedure that the curves of the resulting lamination will have integral weights. Hence $l\in\mathcal{X}_L(S,\mathbb{Z})$ and $\mathcal{X}(\mathbb{Z}^t)\subseteq\mathcal{X}_L(S,\mathbb{Z})$.
\end{proof}

For doubled laminations, we must again impose a topological condition on the curves. To state this condition, we need some terminology and notation.

\begin{definition}
A closed curve $c$ on~$S_\mathcal{D}$ is called an \emph{intersecting curve} if every curve in its homotopy class intersects the image of $\partial S$ in~$S_\mathcal{D}$.
\end{definition}

Let us now consider an element of $\mathcal{D}_L(S,\mathbb{Z})$ represented by a collection of finitely many closed curves of weight~1 on~$S_{\mathcal{D}}$. If we cut the surface~$S_{\mathcal{D}}$ along $\partial S$, we obtain a collection~$\mathcal{C}$ of curves on~$S$ and a collection $\mathcal{C}^\circ$ of curves on~$S^\circ$.

To state the relevant topological condition, we first use the natural map $S^\circ\rightarrow S$ to draw all of the curves in $\mathcal{C}$ or $\mathcal{C}^\circ$ on the surface $S$. For example, the intersecting curve illustrated in the introduction gives the following picture.
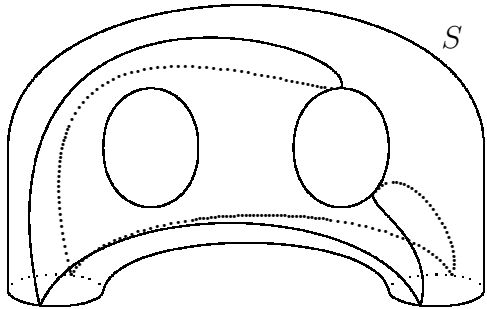
\[
\xy 0;/r.5pc/: 
(-6,-4.5)*\ellipse(3,1){.}; 
(-6,-4.5)*\ellipse(3,1)__,=:a(-180){-}; 
(6,-4.5)*\ellipse(3,1){.}; 
(6,-4.5)*\ellipse(3,1)__,=:a(-180){-}; 
(-9,0)*{}="1";
(-3,0)*{}="2";
(3,0)*{}="3";
(9,0)*{}="4";
(-15,0)*{}="A2";
(15,0)*{}="B2";
"1";"2" **\crv{(-9,-5) & (-3,-5)};
"1";"2" **\crv{(-9,5) & (-3,5)};
"3";"4" **\crv{(3,-5) & (9,-5)};
"3";"4" **\crv{(3,5) & (9,5)};
"A2";"B2" **\crv{(-15,12) & (15,12)};
(-9,-9)*{}="A";
(9,-9)*{}="B";
"A";"B" **\crv{(-8,-5) & (8,-5)}; 
(-15,-9)*{}="A1";
(15,-9)*{}="B1";
"B2";"B1" **\dir{-}; 
"A2";"A1" **\dir{-};
(-11,-8)*{}="V1";
(13,-8)*{}="V2";
(-13,-10)*{}="V3";
(11,-10)*{}="V4";
(6,3.8)*{}="V5";
(8,-3)*{}="V6";
"V1";"V2" **\crv{~*=<2pt>{.} (-8,-3) & (10,-3)};
"V3";"V4" **\crv{(-10,-3) & (8,-3)};
"V3";"V5" **\crv{(-18,13) & (7,7)};
"V1";"V5" **\crv{~*=<2pt>{.} (-16,12) & (4,3)};
"V2";"V6" **\crv{~*=<2pt>{.} (14,-6) & (10,0)};
"V4";"V6" **\crv{(12,-6) & (8,-4)};
(13,7)*{S}; 
\endxy
\]

In this way, we get a collection of curves on $S$, and any arc in this collection that connects two holes can be viewed as a singular 1-simplex with $\mathbb{Z}/2\mathbb{Z}$-coefficients. Similarly, if $c$ is a closed loop in this collection, then we can view this loop as a 1-cycle with $\mathbb{Z}/2\mathbb{Z}$-coefficients. This cycle consists of a single 0-simplex in the interior of each triangle that $c$ intersects in a fixed ideal triangulation~$T$ together with a 1-simplex connecting each pair of consecutive 0-simplices. Let $\sigma_l$ be the sum of all the 1-simplices obtained in this way from arcs and closed loops. Then $\sigma_l$ is a cycle representing an element $[\sigma_l]$ of the singular homology $H_1(S,\mathbb{Z}/2\mathbb{Z})$ of $S$ with coefficients in~$\mathbb{Z}/2\mathbb{Z}$. With this notation, one has the following result.

\begin{theorem}
\label{thm:integralDlam}
Let $l\in\mathcal{D}_L(S,\mathbb{Z})$ be a doubled lamination. Then $l$ has integral coordinates if and only if 
\[
[\sigma_l]=0\in H_1(S,\mathbb{Z}/2\mathbb{Z}).
\]
\end{theorem}

\begin{proof}
To simplify notation, we will denote $\sigma_l$ simply by $\sigma$. We assume that the ideal triangulation $T$ has been drawn on~$S$ in such a way that the edges spiral into the holes of~$S$ in the direction specified by~$l$.

\Needspace*{2\baselineskip}
\begin{step}[1]
Integral coordinates $\implies[\sigma_l]=0$ 
\end{step}

Suppose $l\in\mathcal{D}_L(S,\mathbb{Z})$ is a lamination with integral coordinates. The 1-simplices that make up~$\sigma$ may intersect the edges of the triangulation in infinitely many points. We will begin by constructing a homologous cycle~$\sigma'$ that intersects each edge of the triangulation in a finite, and in fact an \emph{even}, number of points. We will then show that $[\sigma']=0\in H_1(S,\mathbb{Z}/2\mathbb{Z})$.

To define this cycle $\sigma'$, consider any hole in~$S$ that meets one of the simplices of~$\sigma$. Modify the cycle $\sigma$ near this hole as illustrated below so that the simplices meet at an edge of the triangulation spiraling into the hole rather than at a point of $\partial S$. Doing this near every hole gives the desired cycle $\sigma'$.
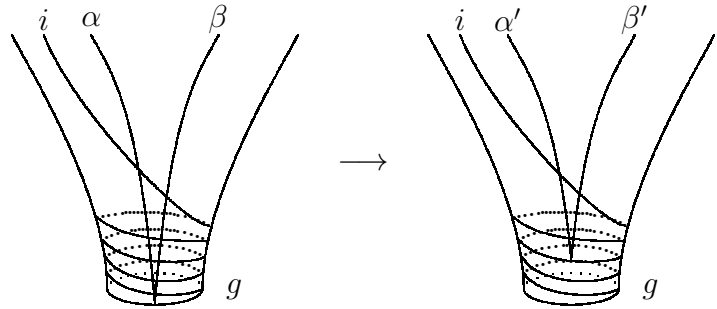
\[
\xy 0;/r.50pc/: 
(0,-4)*\ellipse(3,1){.}; 
(0,-4)*\ellipse(3,1)__,=:a(180){-}; 
(-9,8)*{}="1"; 
(-3,-8)*{}="2"; 
"1";"2" **\crv{(-8,6) & (-3,-2)}; 
(9,8)*{}="3"; 
(3,-8)*{}="4"; 
"3";"4" **\crv{(8,6) & (3,-2)}; 
(-4,8)*{}="L"; 
(0,-9)*{}="M";
(4,8)*{}="N"; 
"L";"M" **\crv{(-3,6) & (-1,4)}; 
"N";"M" **\crv{(3,6) & (1,4)}; 
(-3,-7)*{}="A1"; 
(3,-8)*{}="B1";
(-3.1,-6)*{}="A2"; 
(3,-7)*{}="B2";
(-3.3,-4.9)*{}="A3"; 
(3.1,-6)*{}="B3";
(-3.6,-3.5)*{}="A4"; 
(3.3,-5)*{}="B4";
(-7,8)*{}="A5"; 
(3.4,-4)*{}="B5";
"A1";"B1" **\crv{(-2,-8.6) & (2,-8.6)}; 
"A1";"B2" **\crv{~*=<2pt>{.} (-2,-6) & (2,-6)};
"A2";"B2" **\crv{(-2,-7.8) & (2,-7.8)}; 
"A2";"B3" **\crv{~*=<2pt>{.} (-2,-5) & (2,-5)};
"A3";"B3" **\crv{(-2,-6.5) & (2,-6.5)}; 
"A3";"B4" **\crv{~*=<2pt>{.} (-2,-4) & (2,-4)};
"A4";"B4" **\crv{(-2,-5) & (2,-5)}; 
"A4";"B5" **\crv{~*=<2pt>{.} (-2,-3) & (2,-3)};
"A5";"B5" **\crv{(-5,3) & (2,-4)}; 
(-7,9)*{i}; 
(-4,9)*{\alpha}; 
(4,9)*{\beta}; 
(5,-8)*{g}; 
\endxy
\quad
\longrightarrow
\quad
\xy 0;/r.50pc/: 
(0,-4)*\ellipse(3,1){.}; 
(0,-4)*\ellipse(3,1)__,=:a(180){-}; 
(-9,8)*{}="1"; 
(-3,-8)*{}="2"; 
"1";"2" **\crv{(-8,6) & (-3,-2)}; 
(9,8)*{}="3"; 
(3,-8)*{}="4"; 
"3";"4" **\crv{(8,6) & (3,-2)}; 
(-4,8)*{}="L"; 
(0,-6.1)*{}="M";
(4,8)*{}="N"; 
"L";"M" **\crv{(-3,6) & (-1,4)}; 
"N";"M" **\crv{(3,6) & (1,4)}; 
(-3,-7)*{}="A1"; 
(3,-8)*{}="B1";
(-3.1,-6)*{}="A2"; 
(3,-7)*{}="B2";
(-3.3,-4.9)*{}="A3"; 
(3.1,-6)*{}="B3";
(-3.6,-3.5)*{}="A4"; 
(3.3,-5)*{}="B4";
(-7,8)*{}="A5"; 
(3.4,-4)*{}="B5";
"A1";"B1" **\crv{(-2,-8.6) & (2,-8.6)}; 
"A1";"B2" **\crv{~*=<2pt>{.} (-2,-6) & (2,-6)};
"A2";"B2" **\crv{(-2,-7.8) & (2,-7.8)}; 
"A2";"B3" **\crv{~*=<2pt>{.} (-2,-5) & (2,-5)};
"A3";"B3" **\crv{(-2,-6.5) & (2,-6.5)}; 
"A3";"B4" **\crv{~*=<2pt>{.} (-2,-4) & (2,-4)};
"A4";"B4" **\crv{(-2,-5) & (2,-5)}; 
"A4";"B5" **\crv{~*=<2pt>{.} (-2,-3) & (2,-3)};
"A5";"B5" **\crv{(-5,3) & (2,-4)}; 
(-7,9)*{i}; 
(-4,9)*{\alpha'}; 
(4,9)*{\beta'}; 
(5,-8)*{g}; 
\endxy
\]
It is easy to see that $\sigma'$ is homologous to~$\sigma$. We claim that it intersects each edge of the triangulation of~$S$ in an even number of points. (Here the number of intersections is counted with multiplicity so that if two 1-simplices intersect an edge $i$ at the same point, then we say there are two intersections.)

Indeed, consider the preimages of $\sigma$ and $\sigma'$ in the universal cover $\tilde{S}$ of~$S$.
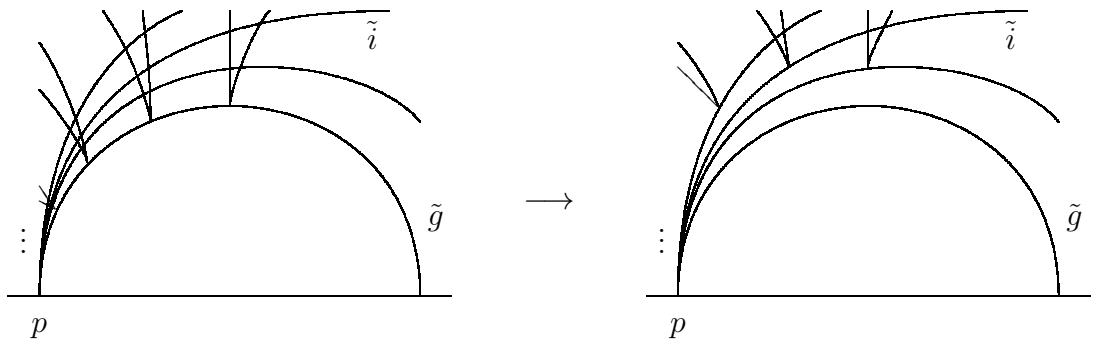
\[
\xy 0;/r.50pc/: 
(-12,-6)*{}="1"; 
(12,-6)*{}="2"; 
"1";"2" **\crv{(-12,10) & (12,10)}; 
(-14,-6)*{}="X"; 
(14,-6)*{}="Y"; 
"X";"Y" **\dir{-}; 
(-12,-8)*{p}; 
(13,-1)*{\tilde{g}}; 
(-12,0)*{}="P1";
(-12,1)*{}="Q1";
(-11,-0.5)*{}="R1"; 
(-12,7)*{}="P2";
(-12,10)*{}="Q2";
(-9,2.5)*{}="R2"; 
(-8,12)*{}="P3";
(-5.5,12)*{}="Q3";
(-5,5.2)*{}="R3"; 
(0,12)*{}="P4";
(2.5,12)*{}="Q4";
(0,6)*{}="R4"; 
"P1";"R1" **\dir{-}; 
"Q1";"R1" **\dir{-}; 
"P2";"R2" **\crv{(-11,6) & (-9,3)}; 
"Q2";"R2" **\crv{(-10,7) & (-9,3)}; 
"P3";"R3" **\crv{(-7,10.5) & (-6,9)}; 
"Q3";"R3" **\crv{(-5,8) & (-5,6)}; 
"P4";"R4" **\crv{(0,12) & (0,6)}; 
"Q4";"R4" **\crv{(1.5,11) & (0,7)}; 
(-13,-2)*{\vdots};
(12,5)*{}="A";
(10,12)*{}="B";
(-3,12)*{}="C";
"1";"A" **\crv{(-12,12) & (8,10)}; 
"1";"B" **\crv{(-12,9) & (0,12)}; 
"1";"C" **\crv{(-12,9) & (-3,12)}; 
(9,10.5)*{\tilde{i}};
(-16,-8)*{}="Z"; 
(16,-8)*{}="W"; 
\endxy
\quad
\longrightarrow
\quad
\xy 0;/r.50pc/: 
(-12,-6)*{}="1"; 
(12,-6)*{}="2"; 
"1";"2" **\crv{(-12,10) & (12,10)}; 
(-14,-6)*{}="X"; 
(14,-6)*{}="Y"; 
"X";"Y" **\dir{-}; 
(-12,-8)*{p}; 
(13,-1)*{\tilde{g}}; 
(-12,8.5)*{}="P2";
(-12,10)*{}="Q2";
(-9.5,6)*{}="R2"; 
(-7,12)*{}="P3";
(-5.5,12)*{}="Q3";
(-5,8.5)*{}="R3"; 
(0,12)*{}="P4";
(1.5,12)*{}="Q4";
(0,8.5)*{}="R4"; 
"P2";"R2" **\dir{-}; 
"Q2";"R2" **\crv{(-11,9) & (-9.5,6.5)}; 
"P3";"R3" **\crv{(-5.5,10) & (-5,8.5)}; 
"Q3";"R3" **\crv{(-5.5,12) & (-5,8.5)}; 
"P4";"R4" **\crv{(0,12) & (0,8.5)}; 
"Q4";"R4" **\crv{(0.5,10) & (0,9)}; 
(-13,-2)*{\vdots};
(12,5)*{}="A";
(10,12)*{}="B";
(-3,12)*{}="C";
"1";"A" **\crv{(-12,12) & (8,10)}; 
"1";"B" **\crv{(-12,9) & (0,12)}; 
"1";"C" **\crv{(-12,9) & (-3,12)}; 
(9,10.5)*{\tilde{i}};
(-16,-8)*{}="Z"; 
(16,-8)*{}="W"; 
\endxy
\]
Let $g$ be the boundary of a hole in~$S$, and let $\tilde{g}$ be an arc in the universal cover that projects to $g$. Then the preimage of $\sigma$ near $g$ consists of infinitely many curves meeting~$\tilde{g}$ as illustrated above on the left. These curves are alternately preimages of~$\alpha$ or~$\beta$. By definition, the coordinate $b_i$ is obtained by choosing, near each endpoint $p$ of~$\tilde{i}$, a distinguished pair of preimages $\tilde{\alpha}_p$ and $\tilde{\beta}_p$ of $\alpha$ and $\beta$, respectively. We then count the number $n_i$ of preimages of curves from $S$ that intersect $\tilde{i}$ between the distinguished curves and the number $n_i^\circ$ of preimages of curves from $S^\circ$ that intersect~$\tilde{i}$ between the distinguished curves. Then $b_i$ is given by 
\[
b_i=\frac{n_i^\circ}{2}-\frac{n_i}{2}.
\]
Since we are assuming $b_i\in\mathbb{Z}$ for all $i$, we know that $n_i^\circ-n_i\in2\mathbb{Z}$, or equivalently that $n_i^\circ+n_i\in2\mathbb{Z}$. When we modify $\sigma$ to get the cycle $\sigma'$, the preimage of $\sigma$ in $\tilde{S}$ is modified as in the illustration above. The distinguished curves $\tilde{\alpha}_p$ and $\tilde{\beta}_p$ that we chose above determine preimages $\tilde{\alpha}_p'$ and $\tilde{\beta}_p'$ of $\alpha'$ and $\beta'$, respectively. As before, we can count the number of intersections between $\tilde{i}$ and the preimages of 1-simplices of~$\sigma'$. The total number of such intersections differs from $n_i^\circ+n_i$ by an even number. Hence the total number of these intersections is even. But this number equals the number of intersections of $\sigma'$ with the edge~$i$. This proves that $\sigma'$ intersects each edge of the triangulation in an even number of points.

It remains to show that $[\sigma']=0\in H_1(S,\mathbb{Z}/2\mathbb{Z})$. If $\sigma_c$ is the summand of $\sigma'$ obtained from a closed loop $c$ on $S$ or $S^\circ$, we can modify this summand to get a homologous cycle whose vertices all lie on edges of $T$. In general, the simplices of the resulting cycle may intersect, and we will deform these arcs so that each of the intersections occurs on an edge of the ideal triangulation. We will consider the refinement of the resulting cycle whose 1-simplices are the intersections of the 1-simplices with the triangles of the ideal triangulation. Abusing notation, we will denote this refinement also by $\sigma'$. We want to show that this $\sigma'$ represents the zero class in homology.

Consider a triangle on the surface with sides~$i$,~$j$, and~$k$. The simplices of $\sigma'$ contained in this triangle are arcs connecting adjacent sides of the triangle so that the total number of arcs that terminate on a given side is always even. Let $\alpha$ be an arc joining two sides, say~$i$ and~$j$. We may assume that there are no additional arcs between $\alpha$ and the vertex $u$ that it surrounds.
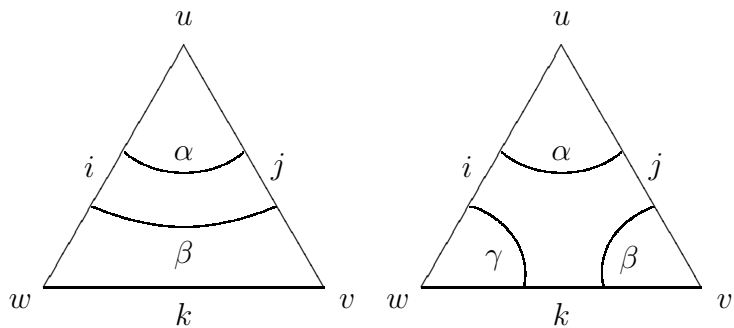
\[
\xy /l1.7pc/:
{\xypolygon3"A"{~:{(-3,0):}}},
{\xypolygon3"B"{~:{(-3.5,0):}~>{}}},
(-0.1,-1)*{}="j1";
(2.1,-1)*{}="i1";
(-0.7,0)*{}="j2";
(2.7,0)*{}="i2";
"B1"*{u};
"B3"*{v};
"B2"*{w};
(2.75,-0.75)*{i};
(-0.75,-0.75)*{j};
(1,2)*{k};
"j1";"i1" **\crv{(0.5,-0.5) & (1.5,-0.5)};
"j2";"i2" **\crv{(0.5,0.5) & (1.5,0.5)};
(1,-1)*{\alpha};
(1,0.9)*{\beta};
\endxy
\quad
\xy /l1.7pc/:
{\xypolygon3"A"{~:{(-3,0):}}},
{\xypolygon3"B"{~:{(-3.5,0):}~>{}}},
(2.1,-1)*{}="i1";
(2.7,0)*{}="i2";
(-0.1,-1)*{}="j1";
(-0.7,0)*{}="j2";
(1.7,1.5)*{}="k1";
(0.2,1.5)*{}="k2";
"B1"*{u};
"B3"*{v};
"B2"*{w};
(2.75,-0.75)*{i};
(-0.75,-0.75)*{j};
(1,2)*{k};
"i1";"j1" **\crv{(1.5,-0.5) & (0.5,-0.5)};
"j2";"k2" **\crv{(0.5,0.5) & (0.25,1.5)};
"i2";"k1" **\crv{(2.5,0) & (1.5,0.5)};
(1,-1)*{\alpha};
(-0.25,1)*{\beta};
(2.25,1)*{\gamma};
\endxy
\]
There is an even number of arcs that meet the edge $j$. It follows that $\alpha$ cannot be the only one and hence there is another arc $\beta$ that terminates on this edge. By our choice of $\alpha$, this arc must not lie between $\alpha$ and $u$, so $\beta$ looks like one of the arcs illustrated above depending on whether it ends on~$i$ or~$k$.

First consider the case where $\beta$ ends on $i$. In this case, there is a cycle formed by the arcs $\alpha$ and~$\beta$ together with the portions of~$i$ and~$j$ between the endpoints of these arcs. This cycle is obviously zero in $H_1(S,\mathbb{Z}/2\mathbb{Z})$.

Now consider the case where $\beta$ joins the edges $j$ and~$k$. We can assume that there are no arcs between $\beta$ and the vertex $v$ that it surrounds. Since the total number of arcs that meet the edge $k$ is even, there is another arc $\gamma$ that meets $k$. By our choice of $\beta$, this arc must not lie between $\beta$ and $v$. Moreover, we can assume that it does not join $k$ and $j$ because this case was already treated above. It follows that $\gamma$ joins~$k$ and~$i$ as illustrated above. There is a cycle formed by the arcs $\alpha$,~$\beta$, and~$\gamma$ together with the portions of $i$,~$j$, and~$k$ between their endpoints. This cycle is again zero in $H_1(S,\mathbb{Z}/2\mathbb{Z})$.

Thus we define for any triangle a cycle with $\mathbb{Z}/2\mathbb{Z}$-coefficients. Consider the sum of this cycle with $\sigma'$ in the space of all such cycles. It is another cycle representing the same homology class as $\sigma'$. It consists of arcs on the edges of the triangulation together with arcs connecting different edges, and the total number of arcs of the latter type that meet at a given edge is always even. Thus we can iterate the above construction and continue adding cycles to get new cycles homologous to $\sigma'$. At each step, the number of arcs connecting different edges decreases, so eventually we are left with zero. This proves that $\sigma'$ is homologous to zero.

\Needspace*{2\baselineskip}
\begin{step}[2]
$[\sigma_l]=0\implies$ integral coordinates
\end{step}

Suppose our cycle satisfies $[\sigma]=0\in H_1(S,\mathbb{Z}/2\mathbb{Z})$. Since $l\in\mathcal{D}_L(S,\mathbb{Z})$, we can realize $l$ as a collection of curves on $S_\mathcal{D}$ with integral weights. Then the portion of $l$ that lies on $S\subseteq S_{\mathcal{D}}$ is a collection of curves with integral weights. Since $\mathcal{X}_L(S,\mathbb{Z})=\mathcal{X}(\mathbb{Z}^t)$, it follows that $l$ has integral $x$-coordinates. Thus we only need to show that $l$ has integral $b$-coordinates.

To prove this fact, we first modify $\sigma$ as in Step~1 to get a new cycle $\sigma'$. In general, the interiors of two 1-simplices of~$\sigma'$ may intersect. If such an intersection point lies on an edge $i$ of~$T$, let us deform this edge slightly to avoid the intersection without changing the number of intersections between $i$ and the interior points of 1-simplices of~$\sigma'$. We then take a refinement of $\sigma'$ in which every point of self intersection is a 0-simplex. Abusing notation, we will denote this refinement also by $\sigma'$.

Since $\sigma'$ is homologous to~$\sigma$, our assumption implies that $[\sigma']=0\in H_1(S,\mathbb{Z}/2\mathbb{Z})$. We can realize $S$ as a two-dimensional simplicial complex $\Delta$ such that every simplex of~$\sigma'$ is included in~$\Delta$. This simplicial complex defines a simplicial homology group $H_1^{\Delta}(S,\mathbb{Z}/2\mathbb{Z})$ with coefficients in $\mathbb{Z}/2\mathbb{Z}$. There is an isomorphism 
\[
H_1^{\Delta}(S,\mathbb{Z}/2\mathbb{Z})\cong H_1(S,\mathbb{Z}/2\mathbb{Z})
\]
between simplicial and singular homology, and so $\sigma'$ also represents the zero class in \emph{simplicial} homology. This means there is a chain $\eta$ of simplices in $\Delta$ with $\sigma'=\partial\eta$. Since $\sigma'$ does not meet a marked point or hole on~$S$, we know that $\eta$ does not meet a marked point or hole~$S$. It follows that the endpoints of an edge $i$ of~$T$ cannot lie on~$\eta$. If $i$ crosses $\sigma'$ into the interior of a 2-simplex of~$\eta$, then it must leave $\eta$ by crossing $\sigma'$ again.

It follows that $i$ intersects $\sigma'$ in an even number of points. As we showed in Step~1, this number of intersections differs by an even number from $n_i^\circ+n_i$, so we have $n_i^\circ+n_i\in2\mathbb{Z}$ and hence $n_i^\circ-n_i\in2\mathbb{Z}$. Hence $b_i=\frac{1}{2}(n_i^\circ-n_i)\in\mathbb{Z}$ as desired.
\end{proof}

\begin{proof}[Proof of Theorem~\ref{thm:integralAlam}]
There is a natural map $\varphi:\mathcal{A}_0\times\mathcal{A}_0\rightarrow\mathcal{D}$ given in coordinates by 
\begin{align*}
\varphi^*(B_i) &= \frac{A_i^\circ}{A_i}, \\
\varphi^*(X_i) &=\prod_{j\in J}A_j^{\varepsilon_{ij}}
\end{align*}
where $A_i^\circ$ are the coordinates on the second factor of $\mathcal{A}_0$ and $A_i$ are the coordinates on the first factor. It induces a map $\mathcal{A}_0(\mathbb{Q}^t)\times\mathcal{A}_0(\mathbb{Q}^t)\rightarrow\mathcal{D}(\mathbb{Q}^t)$ on tropical rational points. Let $\varphi_2:\mathcal{A}_0(\mathbb{Q}^t)\rightarrow\mathcal{D}(\mathbb{Q}^t)$ be the map obtained by precomposing the above map with the inclusion $\mathcal{A}_0(\mathbb{Q}^t)\hookrightarrow \mathcal{A}_0(\mathbb{Q}^t)\times\mathcal{A}_0(\mathbb{Q}^t)$ into the second factor. Now, if $l\in\mathcal{A}_L^0(S,\mathbb{Z})$, then the lamination $\varphi_2(l)$ lies in $\mathcal{D}_L(S,\mathbb{Z})$, and $l$ satisfies the topological condition in the statement of Theorem~\ref{thm:integralAlam} if and only if $\varphi_2(l)$ satisfies the topological condition in the statement of Theorem~\ref{thm:integralDlam}. Each $x_j$ coordinate of $\varphi_2(l)$ equals~0, and each $b_j$ equals the coordinate $a_j$ of the lamination~$l$. Thus Theorem~\ref{thm:integralAlam} is a consequence of Theorem~\ref{thm:integralDlam}.
\end{proof}

\chapter{Duality of cluster varieties}
\label{ch:DualityOfClusterVarieties}

In~\cite{IHES}, Fock and Goncharov showed that the cluster $K_2$- and Poisson varieties are dual in the sense that the tropical integral points of the cluster $K_2$-variety associated to a surface parametrize a canonical basis for the algebra of regular functions on the cluster Poisson variety. In this chapter, we review the construction of Fock and Goncharov. We then show that the tropical integral points of the symplectic double parametrize an interesting class of rational functions on the symplectic double itself. These functions are given by a remarkable formula involving the $F$-polynomials of Fomin and Zelevinsky~\cite{FZIV}.

\section{Duality map for the cluster Poisson variety}

\subsection{Calculation of monodromies}
\label{subsec:CalculationOfMonodromies}

In~\cite{IHES}, Fock and Goncharov describe a method for calculating the monodromy of a framed local system in terms of the coordinates on the space $\mathcal{X}_{PGL_2,S}(\mathbb{C})$. We will review the method here and discuss some of its implications.

To begin, let $m$ be a point of $\mathcal{X}_{PGL_2,S}(\mathbb{C})$ and let $l$ be an oriented loop on~$S$. Fix an ideal triangulation $T$ of $S$. Then for each edge $i\in I$, there is a coordinate $X_i\in\mathbb{C}^*$. For each $i\in I$, let us choose a square root $X_i^{1/2}$ of $X_i$.

Now deform $l$ so that it intersects each edge of the ideal triangulation $T$ in the minimal number of points. Suppose $i_1,\dots,i_s$ are the edges of $T$ that $l$ intersects in order (so an edge may appear more than once on this list). After crossing the edge~$i_k$, the curve $l$ enters a triangle $t$ of~$T$ before leaving through the next edge. If the curve $l$ turns to the right before leaving $t$, then we form the matrix 
\[
M_k=
\left( \begin{array}{cc}
X_{i_k}^{1/2} & X_{i_k}^{1/2} \\
0 & X_{i_k}^{-1/2} \end{array} \right).
\]
On the other hand, if $l$ turns to the left before leaving $t$, then we form the matrix
\[
M_k=
\left( \begin{array}{cc}
X_{i_k}^{1/2} & 0 \\
X_{i_k}^{-1/2} & X_{i_k}^{-1/2} \end{array} \right).
\]
We can then multiply the matrices $M_k$ defined in this way to get a matrix representing the monodromy:
\[
\rho(l)=M_1\dots M_s.
\]

\begin{definition}
By a \emph{peripheral curve} on $S$, we mean a closed curve retractible to a puncture of~$S$.
\end{definition}

\begin{proposition}
\label{prop:monodromyperipheral}
Let $l$ be a peripheral curve on~$S$. Then $\rho(l)$ has eigenvalues $X_1^{a_1}\dots X_n^{a_n}$ and $X_1^{-a_1}\dots X_n^{-a_n}$ where $a_i$ is the coordinate of~$l$ associated to the edge~$i$.
\end{proposition}

\begin{proof}
Deform the loop so that it intersects each edge in the minimal number of points. We may assume $l$ is oriented so that it turns always to the left. If $i_1,\dots,i_n$ are the edges of~$T$ that $l$ crosses in order, then we can use the construction described above to compute the matrix $\rho(l)$. It is given by 
\[
\prod_{k=1}^s\left( \begin{array}{cc}
X_{i_k}^{1/2} & 0 \\
X_{i_k}^{-1/2} & X_{i_k}^{-1/2} \end{array} \right)
=\left( \begin{array}{cc}
X_{i_1}^{1/2}\dots X_{i_s}^{1/2} & 0 \\
C & X_{i_1}^{-1/2}\dots X_{i_s}^{-1/2} \end{array} \right)
\]
where $C$ is a polynomial in the variables $X_{i_k}^{\pm1/2}$. The eigenvalues of this matrix are the diagonal elements.
\end{proof}

\subsection{The multiplicative canonical pairing}

We will now use the construction described above to associate a canonical function $\mathbb{I}_{\mathcal{A}}(l)$ to a lamination $l\in\mathcal{A}(\mathbb{Z}^t)$. We will assume here that the surface $S$ has no marked points and later generalize to arbitrary decorated surfaces.

\Needspace*{2\baselineskip}
\begin{definition}[\cite{IHES}, Definition~12.4]
\label{def:canonicalAX}
Let $S$ be a decorated surface with no marked points.
\begin{enumerate}
\item Let $l\in\mathcal{A}_L(S,\mathbb{Z})$ be a lamination consisting of a single nonperipheral curve of weight~$k$. Then the value of $\mathbb{I}_{\mathcal{A}}(l)$ on~$m\in\mathcal{X}_{PGL_2,S}(\mathbb{C})$ is the trace of the $k$th power of the matrix~$\rho(l)$ constructed above.

\item Let $l\in\mathcal{A}_L(S,\mathbb{Z})$ be a lamination consisting of a single peripheral curve of weight~$k$. The data of $m$ provide a distinguished eigenspace of $\rho(l)$ with eigenvalue $\lambda_l$, and we define the value of $\mathbb{I}_{\mathcal{A}}(l)$ at $m$ to be $\lambda_l^k$.

\item Let $l\in\mathcal{A}_L(S,\mathbb{Z})$ and write $l=\sum_ik_il_i$ where $l_i$ are the curves of $l$ with each homotopy class of curves appearing at most once in the sum and $k_i\in\mathbb{Z}$. Then
\[
\mathbb{I}_{\mathcal{A}}(l)=\prod_i\mathbb{I}_{\mathcal{A}}(k_il_i).
\]
\end{enumerate}
\end{definition}

Thus for any choice of square roots $X_i^{1/2}$, we define $\mathbb{I}_{\mathcal{A}}(l)(m)$ as a certain Laurent polynomial in the $X_i^{1/2}$. Notice that if $l\in\mathcal{A}(\mathbb{Z}^t)$, then the total weight of curves intersecting an edge $i$ of the triangulation is always even. It follows that in this case $\mathbb{I}_{\mathcal{A}}(l)(m)$ is a Laurent polynomial in the variables $X_i$ which is independent of the choice of square roots. Thus we get a canonical map 
\[
\mathbb{I}_{\mathcal{A}}:\mathcal{A}(\mathbb{Z}^t)\rightarrow\mathbb{Q}(\mathcal{X}).
\]
Fock and Goncharov prove the following in~\cite{IHES}.

\begin{theorem}[\cite{IHES}, Theorem~12.2]
\label{thm:classicalproperties}
The canonical functions $\mathbb{I}_{\mathcal{A}}(l)$ for $l\in\mathcal{A}(\mathbb{Z}^t)$ satisfy the following properties:
\begin{enumerate}
\item For any choice of ideal triangulation, $\mathbb{I}_{\mathcal{A}}(l)$ is a Laurent polynomial in the coordinates~$X_i$ with highest term $X_1^{a_1}\dots X_n^{a_n}$ where $a_i$ is the coordinate of $l$ associated to the edge~$i$.

\item The coefficients of the Laurent polynomial $\mathbb{I}_{\mathcal{A}}(l)$ are positive integers.

\item For any laminations $l$,~$l'\in\mathcal{A}(\mathbb{Z}^t)$, we have 
\[
\mathbb{I}_{\mathcal{A}}(l)\mathbb{I}_{\mathcal{A}}(l')=\sum_{l''\in\mathcal{A}(\mathbb{Z}^t)}c(l,l';l'')\mathbb{I}_{\mathcal{A}}(l'')
\]
where the coefficients $c(l,l';l'')$ are nonnegative integers and only finitely many terms are nonzero.
\end{enumerate}
\end{theorem}

Fock and Goncharov argue moreover that the image of this map $\mathbb{I}_{\mathcal{A}}$ is a canonical vector space basis for the algebra of regular functions on the cluster Poisson variety~\cite{IHES}. Recent work of Gross, Hacking, Keel, and Kontsevich describes a more general method for constructing canonical bases~\cite{GHKK}. The construction described above has also found applications in physics, where it was used by Gaiotto, Moore, and Neitzke to calculate BPS degeneracies in certain four-dimensional supersymmetric quantum field theories~\cite{GMN}.

\subsection{The intersection pairing}

In addition to the multiplicative canonical map $\mathbb{I}_{\mathcal{A}}$, we have the following canonical map, which should be viewed as a degeneration of~$\mathbb{I}_{\mathcal{A}}$.

\begin{definition}
Let $S$ be a punctured surface and choose $l\in\mathcal{A}_L(S,\mathbb{Z})$ and $m\in\mathcal{X}_L(S,\mathbb{Z})$. Assume that $m$ provides the negative orientation for each hole. Then we define $\mathcal{I}(l,m)$ to be half the minimal number of intersections between~$l$ and~$m$. Here we take into account the weights of the curves so that if a curve of weight $k_1$ intersects a curve of weight $k_2$, then this intersection contributes the term $k_1k_2/2$ to $\mathcal{I}(l,m)$.
\end{definition}

This defines $\mathcal{I}(l,m)$ in the special case where $m$ provides the negative orientation for each hole. Note that there are natural actions of the group $(\mathbb{Z}/2\mathbb{Z})^r$ on $\mathcal{A}_L(S,\mathbb{Z})$ and $\mathcal{X}_L(S,\mathbb{Z})$ where $r$ is the number of punctures in $S$. The generator of the $i$th factor of $\mathbb{Z}/2\mathbb{Z}$ acts on $\mathcal{A}_L(S,\mathbb{Z})$ by changing the sign of the weight of curves surrounding the $i$th hole. It acts on $\mathcal{X}_L(S,\mathbb{Z})$ by changing the orientation of the $i$th hole. We extend $\mathcal{I}$ to a map
\[
\mathcal{I}:\mathcal{A}_L(S,\mathbb{Z})\times\mathcal{X}_L(S,\mathbb{Z})\rightarrow\mathbb{Q}
\]
called the \emph{intersection pairing} by requiring that it be equivariant with respect to these group actions.

To understand the relationship between the intersection pairing and the multiplicative canonical pairing, we recall the notion of the tropicalization of a subtraction-free rational function~\cite{dual}. Let $F(u_1,\dots,u_n)$ be such a function. We define its \emph{tropicalization} $F^t(u_1,\dots,u_n)$ by the formula 
\[
F^t(u_1,\dots,u_n)=\lim_{C\rightarrow\infty}\frac{\log F(e^{Cu_1},\dots,e^{Cu_n})}{C}.
\]
It is easy to see that 
\[
\lim_{C\rightarrow\infty}\frac{\log(e^{Cv_1}+\dots+e^{Cv_n})}{C}=\max(v_1,\dots,v_n),
\]
so tropicalization takes the operations~$+$ and~$\cdot$ to~$\max$ and~$+$, respectively.

For any laminations $l\in\mathcal{A}_L(S,\mathbb{Z})$ and $m\in\mathcal{X}_L(S,\mathbb{Z})$, one has 
\[
\mathcal{I}(l,m)=(\mathbb{I}_{\mathcal{A}}(l))^t(m)
\]
so that the intersection pairing is the tropicalization of the multiplicative canonical pairing. Later we will give a similar characterization of the tropicalization of the canonical mapping~$\mathbb{I}_{\mathcal{D}}$ on doubled laminations.

\section{Duality map for the symplectic double}

\subsection{The multiplicative canonical pairing}

Let $S$ be an arbitrary decorated surface. Recall that a simple closed curve $l$ on $S_\mathcal{D}$ is called an \emph{intersecting curve} if every curve in its homotopy class intersects the image of $\partial S$ in $S_\mathcal{D}$.

Suppose we are given an intersecting curve $l$ on $S_\mathcal{D}$. Deform this curve so that it intersects the image of~$\partial S$ in the minimal number of points. Then, starting from any component of this image which we may call $\gamma_1$, there is a segment $c_1$ of the curve $l$ which lies entirely in~$S$ and connects the component~$\gamma_1$ to another component which we may call $\gamma_2$. Starting from this component, there is a segment $c_2$ of $l$ which lies entirely in~$S^\circ$ and connects~$\gamma_2$ to a component~$\gamma_3$. Continue labeling in this way until the curve closes.

Now suppose we are given a general point $m\in\mathcal{D}_{PGL_2,S}(\mathbb{C})$. Cut the surface $S_\mathcal{D}$ along~$\partial S$. Then each $c_i$ is a curve that connects two boundary components of~$S$ or~$S^\circ$. Consider a curve~$c_i$ for $i$ odd. Choose a decorated flag at the endpoints of $c_i$, and define~$A_{c_i}$ as the invariant obtained by parallel transporting these flags to a common point in~$S$. For even indices $i$, the flags already chosen determine corresponding flags at the endpoints of $c_i$. We define~$A_{c_i}^\circ$ to be the invariant obtained by parallel transporting these flags to a common point in $S^\circ$.

\Needspace*{3\baselineskip}
\begin{definition} \mbox{}
Fix a general point $m\in\mathcal{D}_{PGL_2,S}(\mathbb{C})$.
\begin{enumerate}
\item Let $l$ be an intersecting curve of weight $k$ on $S_\mathcal{D}$, and assume that the orientation of each component of $\gamma$ agrees with the orientation of~$S$. Then 
\[
\mathbb{I}_\mathcal{D}(l)(m)=\left(\frac{\prod_{\text{$i$ even}}A_{c_i}^\circ}{\prod_{\text{$i$ odd}}A_{c_i}}\right)^k.
\]

\item Let $l$ be a curve of weight $k$ on $S_\mathcal{D}$ which is not an intersecting curve and is not homotopic to a loop in the simple lamination. If $l$ lies in $S^\circ$, then $\mathbb{I}_\mathcal{D}(l)(m)$ is defined as the trace of the $k$th power of the of the monodromy around $l$. If $l$ lies in $S$, then $\mathbb{I}_\mathcal{D}(l)(m)$ is the reciprocal of this quantity.

\item Let $l$ be a curve of weight $k$ on $S_\mathcal{D}$ which is homotopic to a loop $\gamma_i$ in the simple lamination. The point $m$ determines a distinguished eigenvalue $\lambda_i$ of the monodromy around $\gamma_i$. If the orientation that $l$ provides for $\gamma_i$ agrees with the orientation of~$S^\circ$, then $\mathbb{I}_\mathcal{D}(l)(m)$ is defined as $\lambda_i^k$. If this orientation agrees with the orientation of $S$, then $\mathbb{I}_\mathcal{D}(l)(m)$ is the reciprocal of this quantity.

\item Let $l\in\mathcal{D}_L(S,\mathbb{Z})$ and write $l=\sum_ik_il_i$ where $l_i$ are the curves of $l$ with each homotopy class of curves appearing at most once in the sum and $k_i\in\mathbb{Z}$. Then
\[
\mathbb{I}_{\mathcal{D}}(l)=\prod_i\mathbb{I}_{\mathcal{D}}(k_il_i).
\]
\end{enumerate}
\end{definition}

This defines the canonical map in the special case where the orientation of any~$\gamma_i$ that meets a curve agrees with the orientation of the surface~$S$. This definition can be extended in a natural way to an arbitrary lamination. Indeed, suppose $l$ is a lamination for which the function $\mathbb{I}(l)$ has been defined. Suppose the orientation that $l$ provides for $\gamma_i$ agrees with the orientation of $S$, and let $l'$ be the lamination obtained from $l$ by reversing the orientation of~$\gamma_i$. There is a natural birational automorphism $\tau$ of $\mathcal{D}_{PGL_2,S}(\mathbb{C})$ described in~\cite{Dlam} which maps a general point $m\in\mathcal{D}_{PGL_2,S}(\mathbb{C})$ to the point obtained by changing the decoration that $m$ assigns to $\gamma_i$. We define 
\[
\mathbb{I}_{\mathcal{D}}(l')=\mathbb{I}_{\mathcal{D}}(l)\circ\tau.
\]

We can use the above definition to extend the map $\mathbb{I}_{\mathcal{A}}$ to laminations on arbitrary decorated surfaces. Indeed, recall that there is a natural map $\varphi:\mathcal{A}_0\times\mathcal{A}_0\rightarrow\mathcal{D}$. It induces a map $\mathcal{A}_0(\mathbb{Z}^t)\times\mathcal{A}_0(\mathbb{Z}^t)\rightarrow\mathcal{D}(\mathbb{Z}^t)$ on tropical integral points. Let $\varphi_2:\mathcal{A}_0(\mathbb{Z}^t)\rightarrow\mathcal{D}(\mathbb{Z}^t)$ be the map obtained by precomposing the above map with the inclusion $\mathcal{A}_0(\mathbb{Z}^t)\hookrightarrow \mathcal{A}_0(\mathbb{Z}^t)\times\mathcal{A}_0(\mathbb{Z}^t)$ into the second factor. The function $\mathbb{I}_{\mathcal{A}}(l)$ is defined by applying $\mathbb{I}_{\mathcal{D}}$ to $\varphi_2(l)$ and restricting to the Lagrangian subspace $\mathcal{X}\hookrightarrow\mathcal{D}$.

\subsection{Expression in terms of $F$-polynomials}

\subsubsection{Intersecting curves}

We will prove our formula for $\mathbb{I}_{\mathcal{D}}(l)$ in several steps. We begin by examining the special case where $l$ is an intersecting curve. In this case, the proof involves ideas from the theory of cluster algebras. A summary of the results that we need here can be found in Appendix~\ref{ch:ReviewOfClusterAlgebras}.

Fix an ideal triangulation $T$ of $S$, and a point $m\in\mathcal{D}_{PGL_2,S}(\mathbb{C})$. If we choose a hyperbolic structure on $S$, then its universal cover can be identified with a subset of the hyperbolic plane~$\mathbb{H}$, and the triangulation $T$ can be lifted to a triangulation $\tilde{T}$ of the universal cover. Using the natural map $S^\circ\rightarrow S$, we can draw all of the curves $c_i$ on the surface $S$. We can then lift each curve to a geodesic $\tilde{c}_i$ in the universal cover in such a way that $\tilde{c}_i$ and $\tilde{c}_{i+1}$ share a common endpoint. Let $P$ be a triangulated ideal polygon, formed from triangles in $\tilde{T}$, which includes all of the triangles that the curves~$\tilde{c}_i$ pass through. Let $T(P)$ be the triangulation of $P$ provided by the triangulation $\tilde{T}$. Choose a decorated flag at each vertex of $T(P)$. Then we can define~$A_i$ as the invariant associated to an edge $i$ of an ideal triangle between the chosen flags.

In exactly the same way, the universal cover of $S^\circ$ can be identified with a subset of the hyperbolic plane, and the triangulation $T$ provides a triangulation $\tilde{T}^\circ$ of this universal cover. The polygon $P$ gives rise to a polygon $P^\circ$ in this universal cover. The latter polygon has a triangulation $T(P^\circ)$ provided by the triangulation $\tilde{T}^\circ$. The decorated flags that we already chose at the vertices of $T(P)$ provide decorated flags at the vertices of $T(P^\circ)$. We can define~$A_i^\circ$ as the invariant associated to an edge $i$ of an ideal triangle between the chosen flags.

It is convenient at this point to adopt a kind of multi-index notation. If $\mathbf{v}=(v_i)$ is a vector indexed by the edges of the triangulation $T(P)$, then we will write 
\[
A^{\mathbf{v}}=\prod_i A_i^{v_i}.
\]
Similarly, if $\mathbf{v}=(v_i)$ is indexed by the edges of $T(P^\circ)$, then we will write 
\[
(A^\circ)^{\mathbf{v}}=\prod_i (A_i^\circ)^{v_i}.
\]
Since there is a natural bijection between the edges of $T(P^\circ)$ and the edges of $T(P)$, the indexing sets for these vectors can be identified.

\begin{lemma}
\label{lem:Adecomposition}
For each $i$, there is a polynomial $F_{c_i}$ and an integral vector $\mathbf{g}_{c_i}$, indexed by the edges of the triangulation $T(P)$, such that 
\[
A_{c_i} = F_{c_i}(X_1,\dots,X_n)A^{\mathbf{g}_{c_i}}
\]
and 
\[
A_{c_i}^\circ = F_{c_i}(\widehat{X}_1,\dots,\widehat{X}_n)(A^\circ)^{\mathbf{g}_{c_i}}.
\]
\end{lemma}

\begin{proof}
Let $i_1,\dots,i_m$ be the edges of $T(P)$. Associated to the polygon $P$, there is a cluster algebra generated by the variables $A_i$ over the trivial semifield $\mathbb{P}=\{1\}$. Applying Proposition~\ref{prop:FZ63} to this cluster algebra, we see that
\[
A_{c_i} = F_{c_i}\biggr(\prod_{j}A_j^{\varepsilon_{i_1j}},\dots,\prod_{j}A_j^{\varepsilon_{i_mj}}\biggr)A_{i_1}^{g_{i_1}}\dots A_{i_m}^{g_{i_m}}
\]
for some integers $g_{i_1},\dots,g_{i_m}$. Now for any internal edge $k$ of $T(P)$, the product $\prod_{j}A_j^{\varepsilon_{kj}}$ equals the coordinate $X_k$ associated to the corresponding edge $k$ of $T$. Moreover, by the matrix formula of~\cite{MW}, this $F_{c_i}$ is a polynomial only in the variables associated to edges that $\tilde{c}_i$ crosses, which are all internal edges. Therefore we can write it as $F_{c_i}(X_1,\dots,X_n)$, a~polynomial in the $X$-coordinates. This proves the first equation. The proof of the second equation is similar. In this case, one uses the fact that $\widehat{X}_k=\prod_{j}(A_j^\circ)^{\varepsilon_{kj}}$.
\end{proof}

Recall that the variables $X_k$ are defined as cross ratios $X_k=\frac{A_iA_m}{A_jA_l}$ where $i$, $j$, $l$, $m$ are the edges of the quadrilateral with diagonal $k$. Here we will consider additional variables associated to the edges $i$ of $\tilde{T}$. Consider a triangle $\Delta$ in $\tilde{T}$ that includes $i$ as one of its edges, and label the other edges of this triangle as follows:
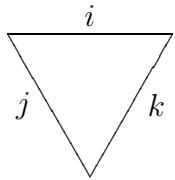
\[
\xy /l1.5pc/:
{\xypolygon3"A"{~:{(2,0):}}};
(2.4,0.5)*{j};
(-0.4,0.5)*{k}; 
(1,-1.4)*{i};
\endxy
\]
Then we define the variable associated to $i$ by $W_{i,\Delta}=\frac{A_iA_j}{A_k}$. Note that this depends on the chosen triangle as well as the edge $i$.

Fix an edge $i$ of the triangulation $T(P)$. For any endpoint $v$ of $i$, there is a collection of edges in $\tilde{T}$ that start at $v$ and lie in the counterclockwise direction from~$i$. Consider a curve~$i'$ that goes diagonally across $i$, intersecting finitely many of these edges transversely before terminating on one of them. An example is illustrated below.
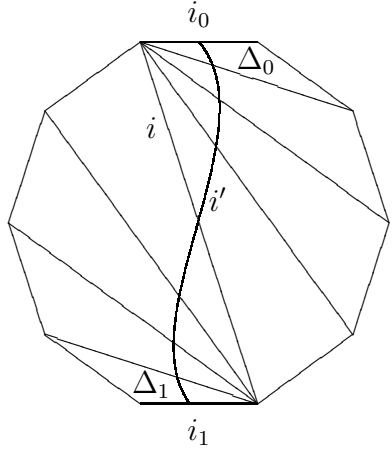
\[
\xy /l3pc/:
{\xypolygon10"A"{~:{(2,0):}}};
{"A9"\PATH~={**@{-}}'"A4"},
{"A9"\PATH~={**@{-}}'"A5"},
{"A9"\PATH~={**@{-}}'"A6"},
{"A9"\PATH~={**@{-}}'"A7"},
{"A10"\PATH~={**@{-}}'"A4"},
{"A1"\PATH~={**@{-}}'"A4"},
{"A2"\PATH~={**@{-}}'"A4"},
(1,-1.9)*{}="1";
(1.1,1.9)*{}="2";
"1";"2" **\crv{(0.25,-1) & (1.75,1)};
(1.5,-1)*{i};
(0.8,-0.25)*{i'};
(1,-2.2)*{i_0};
(1,2.2)*{i_1};
(0.4,-1.7)*{\Delta_0};
(1.5,1.7)*{\Delta_1};
\endxy
\]

Given such a curve $i'$, let $E_i$ be the set of all edges in $T(P)$ that $i'$ crosses. Then we can form the product 
\[
P_i=W_{i_0,\Delta_0}\cdot\prod_{j\in E_i} X_j\cdot W_{i_1,\Delta_1}
\]
where $i_0$ and $i_1$ are the edges on which $i'$ terminates. One can check that this expression equals $A_i^2$. We will use this fact to prove the following result.

\begin{lemma}
\label{lem:factorization}
Let $\mathbf{s}=\sum_{i\text{ even}}\mathbf{g}_{c_i}-\sum_{i\text{ odd}}\mathbf{g}_{c_i}$. Then there exists a half integral vector $\mathbf{h}=(h_i)_{i\in J}$, indexed by the internal edges of the triangulation $T$, such that 
\[
A^{\mathbf{s}}=X^{\mathbf{h}}=\prod_{i\in J}X_i^{h_i}.
\]
\end{lemma}

\begin{proof}
Consider an arc $\tilde{c}_i$ in $T(P)$. If this arc coincides with an edge $i$ of the triangulation, then the associated $\mathbf{g}$-vector equals the standard basis vector $\mathbf{e}_i$. Otherwise $\tilde{c}_i$ intersects one or more edges of $T(P)$, and we can compute the corresponding $\mathbf{g}$-vector~$\mathbf{g}_{c_i}$ using the results of \cite{MSW1} that we reviewed in Appendix~\ref{sec:ClusterAlgebrasAssociatedToSurfaces}. In this case, $\mathbf{g}_{c_i}$ is given by the formula 
\[
\mathbf{g}_{c_i}=\deg\left(\frac{x(P_{-})}{\cross(T,c_i)}\right).
\]
By definition of a perfect matching, we know that any endpoint of a diagonal in the graph~$\bar{G}_{T,c_i}$ meets exactly one edge of the minimal perfect matching $P_{-}$. It follows that the vector $\mathbf{g}_{c_i}$ is an alternating sum of standard basis vectors corresponding to the edges of a path in the graph $\bar{G}_{T,c_i}$.

Consider the path formed by the $\tilde{c}_i$ in $P$. An example of such a path is illustrated below.
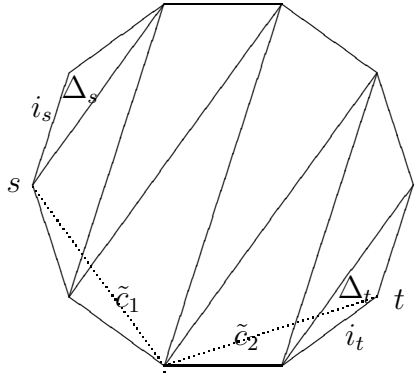
\[
\xy /l3pc/:
{\xypolygon10"A"{~:{(2,0):}}};
{"A7"\PATH~={**@{-}}'"A3"},
{"A7"\PATH~={**@{-}}'"A4"},
{"A6"\PATH~={**@{-}}'"A4"},
{"A8"\PATH~={**@{-}}'"A3"},
{"A8"\PATH~={**@{-}}'"A2"},
{"A9"\PATH~={**@{-}}'"A2"},
{"A9"\PATH~={**@{-}}'"A1"},
{"A1"\PATH~={**@{.}}'"A3"},
{"A3"\PATH~={**@{.}}'"A5"},
(3.2,0)*{s};
(-0.85,1.2)*{t};
(2,1.2)*{\tilde{c}_1};
(0.75,1.6)*{\tilde{c}_2};
(-0.39,1.1)*{\Delta_t};
(2.5,-1)*{\Delta_s};
(-0.39,1.6)*{i_t};
(2.9,-0.8)*{i_s};
\endxy
\]
Let $s$ and $t$ be the endpoints of the path formed by the $\tilde{c}_i$. Consider the closed path on the surface $S$ obtained by drawing all the $c_i$ on $S$ using the natural map $S^\circ\rightarrow S$. Let $\Delta_t$ be the triangle in $P$ that contains $t$ and is that last triangle that the lifted arcs $c_i$ pass through. Let $\Delta_s$ be the preimage of this triangle under the monodromy.

By the above discussion, the vector $\mathbf{s}$ is an alternating sum of standard basis vectors associated to the edges of a path in $T(P)$. To each edge $i$ on this path, we associate a curve~$i'$ as above so that $i'$ and $j'$ terminate on a common edge whenever $i$ and $j$ terminate on a common vertex. We can choose these curves so that the first and last ones terminate at corresponding edges~$i_s$ of~$\Delta_s$ and~$i_t$ of~$\Delta_t$. One can show in this case that $W_{i_s,\Delta_s}=W_{i_t,\Delta_t}$. It follows from the above discussion that 
\[
\prod_iP_i^{s_i}=\biggr(\prod_iA_i^{s_i}\biggr)^2
\]
where $\mathbf{s}=(s_i)$. Since $\mathbf{s}$ is an alternating sum, all $W$-factors cancel on the left hand side of the equation. Therefore the left hand side is a product of the $X_i$, and the lemma follows by taking square roots on both sides.
\end{proof}

We can now prove our formula for $\mathbb{I}_{\mathcal{D}}(l)$ in the case where $l$ is an intersecting curve.

\begin{proposition}
\label{prop:intersecting}
Let $l$ be a doubled lamination represented by a single intersecting curve of weight~1, and suppose that the orientation of each component of $\gamma$ agrees with the orientation of~$S$. Then 
\[
\mathbb{I}_\mathcal{D}(l)=\frac{\prod_{\text{$i$ even}}F_{c_i}(\widehat{X}_1,\dots,\widehat{X}_n)}{\prod_{\text{$i$ odd}}F_{c_i}(X_1,\dots,X_n)}B_1^{g_{l,1}}\dots B_n^{g_{l,n}}X_1^{h_{l,1}}\dots X_n^{h_{l,n}}
\]
where the $g_{l,i}$ are integers and the $h_{l,i}$ are half integers.
\end{proposition}

\begin{proof}
By Lemma~\ref{lem:Adecomposition}, we know that
\[
A_{c_i} = F_{c_i}(X_1,\dots,X_n)A^{\mathbf{g}_{c_i}}
\]
and
\[
A_{c_i}^\circ = F_{c_i}(\widehat{X}_1,\dots,\widehat{X}_n)(A^\circ)^{\mathbf{g}_{c_i}}.
\]
Inserting these expressions into the formula in the definition of $\mathbb{I}_{\mathcal{D}}$, we obtain 
\[
\mathbb{I}_\mathcal{D}(l) = \frac{\prod_{\text{$i$ even}}F_{c_i}(\widehat{X}_1,\dots,\widehat{X}_n)}{\prod_{\text{$i$ odd}}F_{c_i}(X_1,\dots,X_n)}\frac{(A^\circ)^{\mathbf{g}_{\text{even}}}}{A^{\mathbf{g}_{\text{odd}}}}
\]
where we have defined $\mathbf{g}_{\text{even}}=\sum_{i\text{ even}}\mathbf{g}_{c_i}$ and $\mathbf{g}_{\text{odd}}=\sum_{i\text{ odd}}\mathbf{g}_{c_i}$. Substituting $A_i^\circ=B_iA_i$ into this expression, we obtain 
\[
\mathbb{I}_\mathcal{D}(l) =\frac{\prod_{\text{$i$ even}}F_{c_i}(\widehat{X}_1,\dots,\widehat{X}_n)}{\prod_{\text{$i$ odd}}F_{c_i}(X_1,\dots,X_n)}B^{\mathbf{g}_{\text{even}}}A^{\mathbf{g}_{\text{even}}-\mathbf{g}_{\text{odd}}}.
\]
Finally, by Lemma~\ref{lem:factorization}, we can write the last factor as
\[
A^{\mathbf{g}_{\text{even}}-\mathbf{g}_{\text{odd}}}=X^{\mathbf{h}}
\]
for some half integral vector $\mathbf{h}$. This completes the proof.
\end{proof}

Notice that the proof of Lemma~\ref{lem:factorization} actually gives an explicit description of the product $X_1^{h_{l,1}}\dots X_n^{h_{l,n}}$ appearing in this theorem. In proving the lemma, we have essentially described a cycle $\eta_l$ with $\mathbb{Z}/2\mathbb{Z}$-coefficients on~$S$ such that this monomial equals $X_{i_1}^{\pm1/2}\dots X_{i_s}^{\pm1/2}$ where $i_1,\dots,i_s$ are the edges of $T$ that $\eta_l$ intersects. This fact will be important below when we discuss rational functions obtained from laminations.

\subsubsection{Other curves}

We will now derive formulas for $\mathbb{I}_{\mathcal{D}}(c)$ in special cases where $c$ is a closed curve on $S_{\mathcal{D}}$ that is not an intersecting curve.

\begin{proposition}
\label{prop:Scircloop}
Let $c$ be a doubled lamination consisting of a single loop of weight~$k$ that lies entirely in~$S^\circ$ and is not retractible to the simple lamination. Then 
\[
\mathbb{I}_{\mathcal{D}}(c)=F_c(\widehat{X}_1,\dots,\widehat{X}_n)B_1^{g_{c,1}}\dots B_n^{g_{c,n}}X_1^{h_{c,1}}\dots X_n^{h_{c,n}}
\]
where $F_c$ is a polynomial, the $g_{c,i}$ are integers, and the $h_{c,i}$ are half integers.
\end{proposition}

\begin{proof}
Let $c$ be a doubled lamination satisfying the hypotheses of the proposition. In this case, $\mathbb{I}_{\mathcal{D}}(c)$ is defined as the trace of the $k$th power of the monodromy around~$c$. This monodromy is a product of the matrices 
\[
\left( \begin{array}{cc}
\widehat{X}_i^{1/2} & 0 \\
\widehat{X}_i^{-1/2} & \widehat{X}_i^{-1/2} \end{array} \right)
\quad
\text{and}
\quad
\left( \begin{array}{cc}
\widehat{X}_i^{1/2} & \widehat{X}_i^{1/2} \\
0 & \widehat{X}_i^{-1/2} \end{array} \right),
\]
one for each edge $i$ that $c$ intersects. These matrices factor as 
\[
\left( \begin{array}{cc}
\widehat{X}_i & 0 \\
1 & 1 \end{array} \right)
\cdot \widehat{X}_i^{-1/2}
\quad
\text{and}
\quad
\left( \begin{array}{cc}
\widehat{X}_i & \widehat{X}_i \\
0 & 1 \end{array} \right)
\cdot \widehat{X}_i^{-1/2},
\]
so the monodromy factors as $M\cdot\widehat{X}_{i_1}^{-1/2}\dots\widehat{X}_{i_s}^{-1/2}$ where $i_1,\dots,i_s$ are the edges that $c$ intersects and $M$ is a matrix with polynomial entries. Let us write $F_c(\widehat{X}_1,\dots,\widehat{X}_n)$ for the trace of the matrix $M^k$. Then we have 
\[
\mathbb{I}_{\mathcal{D}}(c)=F_c(\widehat{X}_1,\dots,\widehat{X}_n)\widehat{X}_{i_1}^{-k/2}\dots\widehat{X}_{i_s}^{-k/2}.
\]
Consider the product $\prod_l\prod_j B_j^{\varepsilon_{i_lj}}$. Let~$i$ be any edge of the triangulation. Using the definition of the matrix $\varepsilon_{ij}$, it is easy to show that the total degree of $B_i$ in this product is even. Hence
\begin{align*}
\widehat{X}_{i_1}^{-k/2}\dots\widehat{X}_{i_s}^{-k/2} &= \prod_l {\biggr(X_{i_l}\prod_jB_j^{\varepsilon_{i_lj}}\biggr)}^{-k/2} \\
&= {\biggr(\prod_{j,l} B_j^{\varepsilon_{i_lj}}\biggr)}^{-k/2}\biggr(\prod_lX_{i_l}^{-k/2}\biggr) \\
&= B_1^{g_{c,1}}\dots B_n^{g_{c,n}}X_1^{h_{c,1}}\dots X_n^{h_{c,n}}
\end{align*}
for some integers $g_{c,i}$ and half integers $h_{c,i}$. This completes the proof.
\end{proof}

\begin{proposition}
\label{prop:Scirchole}
Let $c$ be a doubled lamination consisting of a single loop of weight~$k$ on~$S_{\mathcal{D}}$. Assume that this loop is homotopic to a loop in the simple lamination and the orientation that $c$ provides for this loop agrees with the orientation of $S^\circ$. Then 
\[
\mathbb{I}_{\mathcal{D}}(c)=B_1^{g_{c,1}}\dots B_n^{g_{c,n}}X_1^{h_{c,1}}\dots X_n^{h_{c,n}}
\]
for integers $g_{c,i}$ and half integers $h_{c,i}$.
\end{proposition}

\begin{proof}
Let $c$ be a doubled lamination satisfying the hypotheses of the proposition. In this case, $\mathbb{I}_{\mathcal{D}}(c)$ is defined as $\lambda^k$ where $\lambda$ is the distinguished eigenvalue of the monodromy around~$c$. Let $i_1,\dots,i_s$ be the edges of the triangulation that meet this loop in the simple lamination. Proposition~\ref{prop:monodromyperipheral} implies 
\[
\mathbb{I}_{\mathcal{D}}(c)=\left(\widehat{X}_{i_1}^{1/2}\dots\widehat{X}_{i_s}^{1/2}\right)^k.
\]
The same argument that we used in the proof of Proposition~\ref{prop:Scircloop} can be used to show that this expression has the form $B_1^{g_{c,1}}\dots B_n^{g_{c,n}}X_1^{h_{c,1}}\dots X_n^{h_{c,n}}$ for some integers~$g_{c,i}$ and half integers~$h_{c,i}$.
\end{proof}

\begin{proposition}
\label{prop:Sloop}
Let $c$ be a doubled lamination consisting of a single loop of weight~$k$ that lies entirely in~$S$ and is not retractible to the simple lamination. Then 
\[
\mathbb{I}_{\mathcal{D}}(c)=\frac{1}{F_c(X_1,\dots,X_n)}X_1^{h_{c,1}}\dots X_n^{h_{c,n}}
\]
where $F_c$ is a polynomial and the $h_{c,i}$ are half integers.
\end{proposition}

\begin{proof}
Let $c$ be a doubled lamination satisfying the hypotheses of the proposition. In this case, $\mathbb{I}_{\mathcal{D}}(c)$ is defined as the reciprocal of the trace of the $k$th power of the monodromy around~$c$. This monodromy is represented by a product of the matrices 
\[
\left( \begin{array}{cc}
X_i^{1/2} & 0 \\
X_i^{-1/2} & X_i^{-1/2} \end{array} \right)
\quad
\text{and}
\quad
\left( \begin{array}{cc}
X_i^{1/2} & X_i^{1/2} \\
0 & X_i^{-1/2} \end{array} \right),
\]
one for each edge $i$ that $c$ intersects. These matrices factor as 
\[
\left( \begin{array}{cc}
X_i & 0 \\
1 & 1 \end{array} \right)
\cdot X_i^{-1/2}
\quad
\text{and}
\quad
\left( \begin{array}{cc}
X_i & X_i \\
0 & 1 \end{array} \right)
\cdot X_i^{-1/2},
\]
so the monodromy factors as $M\cdot X_{i_1}^{-1/2}\dots X_{i_s}^{-1/2}$ where $i_1,\dots,i_s$ are the edges that $c$ intersects and $M$ is a matrix with polynomial entries. Let us write $F_c(X_1,\dots,X_n)$ for the trace of the matrix $M^k$. Then we have 
\begin{align*}
\mathbb{I}_{\mathcal{D}}(c) &= \frac{1}{F_c(X_1,\dots,X_n)}X_{i_1}^{k/2}\dots X_{i_s}^{k/2} \\
&= \frac{1}{F_c(X_1,\dots,X_n)}X_1^{h_{c,1}}\dots X_n^{h_{c,n}}
\end{align*}
for half integers $h_{c,i}$.
\end{proof}

\begin{proposition}
\label{prop:Shole}
Let $c$ be a doubled lamination consisting of a single loop of weight~$k$ on~$S_{\mathcal{D}}$. Assume that this loop is homotopic to a loop in the simple lamination and the orientation that $c$ provides for this loop agrees with the orientation of~$S$. Then 
\[
\mathbb{I}_{\mathcal{D}}(c)=X_1^{h_{c,1}}\dots X_n^{h_{c,n}}
\]
for half integers $h_{c,i}$.
\end{proposition}

\begin{proof}
Let $c$ be a doubled lamination satisfying the hypotheses of the proposition.  In this case, $\mathbb{I}_{\mathcal{D}}(c)$ is defined as $\lambda^{-k}$ where $\lambda$ is the distinguished eigenvalue of the monodromy around~$c$. Let $i_1,\dots,i_s$ be the edges of the triangulation that meet this loop in the simple lamination. Proposition~\ref{prop:monodromyperipheral} implies 
\[
\mathbb{I}_{\mathcal{D}}(c)=\left(X_{i_1}^{1/2}\dots X_{i_s}^{1/2}\right)^{-k}.
\]
Thus $\mathbb{I}_{\mathcal{D}}(c)$ has the desired form.
\end{proof}

Propositions~\ref{prop:Scircloop}, \ref{prop:Scirchole}, \ref{prop:Sloop}, and~\ref{prop:Shole} provide expressions for $\mathbb{I}_{\mathcal{D}}(c)$ in the cases where the curve $c$ is not an intersecting curve. In each case, the expression includes a product of the form $X_1^{h_{c,1}}\dots X_n^{h_{c,n}}$ where each $h_{c,i}$ is a half integer. In fact, the proofs of these propositions provide a more explicit description of this product: It is obtained by multiplying one factor $X_i^{\pm k/2}$ each time the curve $c$ crosses an edge $i$. This fact will be important below when we discuss rational functions obtained from laminations.

\subsubsection{The general case}

We will now combine the above results to get a general expression for the function $\mathbb{I}_{\mathcal{D}}(l)$. To state this result, we will need the following notation. Any doubled lamination $l\in\mathcal{D}_L(S,\mathbb{Z})$ can be represented by a collection of curves of weight~1. If this collection contains homotopic curves of weights $a$ and $b$ which are not intersecting curves, let us replace these by a single curve of weight $a+b$. In this way, we obtain a new collection of curves representing~$l$. If we now cut the surface along the image of $\partial S$, we obtain a collection $\mathcal{C}^\circ$ of curves on $S^\circ$ and a collection $\mathcal{C}$ of curves on $S$.

\begin{theorem}
\label{thm:main}
Let $l\in\mathcal{D}_L(S,\mathbb{Z})$ and suppose that the orientation of the curve $\gamma_i$ agrees with~$S$ whenever $\gamma_i$ meets a curve of~$l$. Then 
\[
\mathbb{I}_\mathcal{D}(l)=\frac{\prod_{c\in\mathcal{C}^\circ}F_c(\widehat{X}_1,\dots,\widehat{X}_n)}{\prod_{c\in\mathcal{C}}F_c(X_1,\dots,X_n)}B_1^{g_{l,1}}\dots B_n^{g_{l,n}}X_1^{h_{l,1}}\dots X_n^{h_{l,n}}
\]
where the $F_c$ are polynomials, the $g_{l,i}$ are integers and the $h_{l,i}$ are half integers.
\end{theorem}

\begin{proof}
If $c$ is a closed curve which is homotopic to a loop in the simple lamination, we set $F_c=1$. Then the above formula for $\mathbb{I}_{\mathcal{D}}(l)$ is an immediate consequence of Propositions~\ref{prop:Scircloop}, \ref{prop:Scirchole}, \ref{prop:Sloop}, \ref{prop:Shole}, and~\ref{prop:intersecting} and the multiplicativity of $\mathbb{I}_{\mathcal{D}}$.
\end{proof}

Since the $h_{l,i}$ in Theorem~\ref{thm:main} are only half integers, we see that $\mathbb{I}_\mathcal{D}(l)$ is in general not a rational function in the variables $B_i$ and $X_i$. We will now show that we get a rational function whenever $l\in\mathcal{D}(\mathbb{Z}^t)$.

\begin{theorem}
\label{thm:restrictintegral}
The function $\mathbb{I}_{\mathcal{D}}(l)$ is rational if and only if $l$ has integral coordinates.
\end{theorem}

\begin{proof}
We have seen in Theorem~\ref{thm:integralDlam} that a lamination $l$ has integral coordinates if and only if the homology class $[\sigma_l]\in H_1(S,\mathbb{Z}/2\mathbb{Z})$ vanishes.

\Needspace*{2\baselineskip}
\begin{step}[1]
$[\sigma_l]=0\implies$ rational function
\end{step}

Suppose the cycle $\sigma=\sigma_l$ satisfies $[\sigma]=0\in H_1(S,\mathbb{Z}/2\mathbb{Z})$. We will begin by replacing $\sigma$ by a homologous cycle.

Let $c$ be a closed curve in $\mathcal{C}$ or $\mathcal{C}^\circ$. Then Propositions~\ref{prop:Scircloop}, \ref{prop:Scirchole}, \ref{prop:Sloop}, and~\ref{prop:Shole} give an expression for $\mathbb{I}_{\mathcal{D}}(c)$ in terms of the $B$- and $X$-coordinates. This expression includes a factor which is a monomial in the $X$-coordinates with half integral exponents. By the remarks following Proposition~\ref{prop:Shole}, we know that this monomial has the form $X_{i_1}^{\pm k/2}\dots X_{i_s}^{\pm k/2}$ where $i_1,\dots,i_s$ are the edges of $T$ that $\sigma_c$ intersects. Similarly, if $c$ is an intersecting curve on~$S_\mathcal{D}$, then Proposition~\ref{prop:intersecting} gives an expression for $\mathbb{I}_{\mathcal{D}}(c)$ in terms of the $B$- and $X$-coordinates. This expression includes a factor which is a monomial in the $X$-coordinates with half integral exponents. In proving Proposition~\ref{prop:intersecting}, we essentially constructed a cycle $\sigma_c'$ with $\mathbb{Z}/2\mathbb{Z}$-coefficients such that this monomial has the form $X_{i_1}^{\pm1/2}\dots X_{i_s}^{\pm1/2}$ where $i_1,\dots,i_s$ are the edges of $T$ that $\sigma_c'$ intersects.

It is easy to see that $\sigma_c'$ is homologous to $\sigma_c$ for any intersecting curve $c$. Thus for an arbitrary doubled lamination $l$, we can replace $\sigma$ by a homologous cycle $\sigma'$ in which every summand of $\sigma$ of the form $\sigma_c$ for $c$ an intersecting curve is replaced by $\sigma_c'$. Our assumption that $[\sigma]=0\in H_1(S,\mathbb{Z}/2\mathbb{Z})$ implies that $[\sigma']=0$ as well. Arguing as in the proof of Theorem~\ref{thm:integralDlam}, we find that each edge of $T$ intersects $\sigma'$ in an even number of points. 

Now Theorem~\ref{thm:main} leads to an expression for $\mathbb{I}_{\mathcal{D}}(l)$ in terms of $B$- and $X$-coordinates for any doubled lamination $l$. This expression includes a factor which is a monomial in the $X$-coordinates with half integral exponents $h_{l,i}$, and we want to show that each $h_{l,i}$ is in fact an integer. Since $\mathbb{I}_{\mathcal{D}}(l)$ is the product of the functions $\mathbb{I}_{\mathcal{D}}(c)$ where $c$ is an intersecting curve in $l$ or a loop in~$\mathcal{C}$ or~$\mathcal{C}^\circ$, this follows immediately from the fact that any edge $i$ of the triangulation $T$ intersects $\sigma'$ in an even number of points.

\Needspace*{2\baselineskip}
\begin{step}[2]
Rational function $\implies[\sigma_l]=0$
\end{step}

Suppose $\mathbb{I}_{\mathcal{D}}(l)$ is a rational function in the variables $B_j$ and~$X_j$. Then each of the exponents $h_{l,i}$ in Theorem~\ref{thm:main} is an integer. This means that each edge of the ideal triangulation~$T$ intersects the cycle $\sigma'$ constructed in~Step~1 in an even number of points. Arguing as in the proof of Theorem~\ref{thm:integralDlam}, we can recursively add cycles to~$\sigma'$ to prove that $[\sigma']=0\in H_1(S,\mathbb{Z}/2\mathbb{Z})$. Since $\sigma'$ is homologous to $\sigma$, this implies $[\sigma]=0\in H_1(S,\mathbb{Z}/2\mathbb{Z})$.
\end{proof}

By Theorem~\ref{thm:restrictintegral}, we have constructed a canonical map 
\[
\mathbb{I}_{\mathcal{D}}:\mathcal{D}(\mathbb{Z}^t)\rightarrow\mathbb{Q}(\mathcal{D})
\]
in complete analogy with the classical construction of $\mathbb{I}_{\mathcal{A}}$.

\subsection{The intersection pairing}

In addition to the multiplicative canonical pairing that we defined above, we have a canonical map $\mathcal{I}_{\mathcal{D}}:\mathcal{D}_L(S,\mathbb{Z})\times\mathcal{D}_L(S,\mathbb{Z})\rightarrow\mathbb{Q}$ defined as follows.

Suppose we are given an intersecting curve $l$ on $S_\mathcal{D}$. Deform this curve so that it intersects the image of~$\partial S$ in the minimal number of points. Then, starting from a component of this image which we may call $\gamma_1$, there is a segment $c_1$ of the curve $l$ which lies entirely in~$S$ and connects the component~$\gamma_1$ to another component which we may call~$\gamma_2$. Starting from this component, there is a segment $c_2$ of $l$ which lies entirely in~$S^\circ$ and connects~$\gamma_2$ to a component~$\gamma_3$. Continue labeling in this way until the curve closes.

Now suppose we are given a point $m\in\mathcal{D}_L(S,\mathbb{Z})$. Draw this lamination on the same copy of $S_{\mathcal{D}}$ as $l$ and deform its curves so that they intersect the image of~$\partial S$ in the minimal number of points. Cut the surface $S_\mathcal{D}$ along $\partial S$. Then each $c_i$ is a curve that connects two boundary components of~$S$ or~$S^\circ$. We can wind the ends of $c_i$ around the holes infinitely many times in the direction prescribed by the orientations from $m$.

Consider a curve $c_i$ for $i$ odd. Lifting this curve to the universal cover of $S$, we obtain a geodesic $\tilde{c}_i$ connecting two points on the boundary of $\mathbb{H}$. Choose a distinguished curve intersecting $\tilde{c}_i$ near each of these boundary points, and define~$a_{c_i}$ as half the number of intersections of the lifted curves between the distinguished curves. Next consider $c_i$ for~$i$~even. Lifting to the universal cover of $S^\circ$, we again get a geodesic connecting two points on the boundary of $\mathbb{H}$, and the distinguished curves already chosen determine a pair of distinguished curves near these points. We define~$a_{c_i}^\circ$ as half the number of intersections of the lifted curves between the distinguished curves.

\Needspace*{5\baselineskip}
\begin{definition}\label{def:intersectionD} \mbox{}
Fix a point $m\in\mathcal{D}_L(S,\mathbb{Z})$.
\begin{enumerate}
\item Let $l$ be an intersecting curve of weight $k$ on $S_\mathcal{D}$, and assume that the orientation of each component of~$\gamma$ agrees with the orientation of~$S$. Then 
\[
\mathcal{I}_\mathcal{D}(l,m)=k\biggr(\sum_{\text{$i$ even}}a_{c_i}^\circ-\sum_{\text{$i$ odd}}a_{c_i}\biggr).
\]

\item Let $l$ be a curve of weight $k$ on $S_\mathcal{D}$ which is not an intersecting curve and is not homotopic to a loop in the simple lamination. If $l$ lies in $S^\circ$, then $\mathcal{I}_\mathcal{D}(l,m)$ is defined as $k/2$ times the minimal number of intersections between~$l$ and~$m$. If $l$ lies in $S$, then $\mathcal{I}_\mathcal{D}(l,m)$ is minus this quantity.

\item Let $l$ be a curve of weight $k$ on $S_\mathcal{D}$ which is homotopic to a loop $\gamma_i$ in the simple lamination. If $l$ and~$m$ provide the same orientation for $\gamma_i$, then $\mathcal{I}_\mathcal{D}(l,m)$ is defined as $k/2$ times the minimal number of intersections between~$l$ and~$m$. If $l$ and~$m$ provide different orientations for $\gamma$, then $\mathcal{I}_\mathcal{D}(l,m)$ is minus this quantity.

\item Let $l_1$ and $l_2$ be laminations on $S_\mathcal{D}$ such that no curve from $l_1$ intersects a curve from~$l_2$. Then $\mathcal{I}_\mathcal{D}(l_1+l_2,m)=\mathcal{I}_\mathcal{D}(l_1,m)+\mathcal{I}_\mathcal{D}(l_2,m)$.
\end{enumerate}
\end{definition}

In the classical setting of bounded and unbounded laminations, the intersection pairing is the tropicalization of the multiplicative canonical pairing. We will now prove the analogous result for doubled laminations.

\begin{proposition}
\label{prop:intersectingintersection}
Let $l$ be an intersecting curve of weight~1 on~$S_{\mathcal{D}}$. If $\gamma_i$ is a component of~$\gamma$ that meets a curve of $l$, assume that the orientation that $l$ provides for $\gamma_i$ agrees with the orientation of~$S$. Then $\mathcal{I}_{\mathcal{D}}(l,m)=(\mathbb{I}_{\mathcal{D}}(l))^t(m)$ for any~$m\in\mathcal{D}_L(S,\mathbb{Z})$.
\end{proposition}

\begin{proof}
Let $m\in\mathcal{D}_L(S,\mathbb{Z})$ be a doubled lamination with coordinates $b_j$ and~$x_j$ for $j\in J$, and let~$e^{Cm}$ denote the point of $\mathcal{D}^+(S)$ with coordinates 
\[
B_j=e^{Cb_j}
\]
and
\[
X_j=e^{Cx_j}
\]
for $j\in J$. This point $e^{Cm}\in\mathcal{D}^+(S)$ determines hyperbolic structures on~$S$ and $S^\circ$, so we can view the universal covers $\tilde{S}$ and~$\tilde{S}^\circ$ as subsets of the hyperbolic plane.

Cut the intersecting curve $l$ along the image of $\partial S$ and then, using the natural map $S^\circ\rightarrow S$, draw all of the resulting curves on $S$. These curves form a cycle on $S$ which we can lift to a path in the universal cover $\tilde{S}$ connecting points on $\partial\mathbb{H}$. Choose an ideal triangulation $T$ of $S$, and lift this to a triangulation $\tilde{T}$ of $\tilde{S}$. Let $P$ be a triangulated ideal polygon, formed from the triangles of $\tilde{T}$, that contains all of the triangles that intersect this path. The curves of the path that project to curves from $S$ determine a lamination $u\in\mathcal{X}_L(P)$.

In exactly the same way, we can lift $S^\circ$ to its universal cover $\tilde{S}^\circ$, and the triangulation $\tilde{T}$ and polygon $P$ determine a triangulation $\tilde{T}^\circ$ and a polygon~$P^\circ$ in~$\tilde{S}^\circ$. The curves of the path in $P^\circ$ that project to curves from $S^\circ$ determine a lamination $u^\circ\in\mathcal{X}_L(P^\circ)$.

Now consider the portion of $m$ that intersects $S$. Lifting these curves to the universal cover, we get a collection of (possibly infinitely many) curves on the polygon~$P$. Let $p$ be any vertex of $P$. If there are infinitely many curves connecting the sides that meet at $p$, then we can choose one such curve $\alpha_p$ and delete all of the curves between~$\alpha_p$ and the point $p$. By doing this for each vertex, we remove all but finitely many curves and get a lamination $v\in\mathcal{A}_L(P)$. 

In the same way, we can consider the portion of $m$ that intersects $S^\circ$. Lifting these curves to the universal cover, we get a collection of (possibly infinitely many) curves on the polygon~$P^\circ$. The distinguished curves $\alpha_p$ that we chose near the vertices of $P$ determine corresponding curves $\alpha_p^\circ$ near the the vertices of $P^\circ$. For each $p$, delete all of the curves on $P^\circ$ between $\alpha_p^\circ$ and $p$. In this way, we remove all but finitely many curves and get a lamination $v^\circ\in\mathcal{A}_L(P^\circ)$.

For any edge $i$ of the triangulation of~$P$, write $a_i$ for the corresponding coordinate of the lamination $v$. For any edge $i$ of the triangulation of~$P^\circ$, write $a_i^\circ$ for the corresponding coordinate of the lamination $v^\circ$. Consider the point $e^{Cv}$ of $\mathcal{A}^+(P)$ with coordinates $A_i=e^{Ca_i}$ and the point $e^{Cv^\circ}$ of $\mathcal{A}^+(P^\circ)$ with coordinates $A_i^\circ=e^{Ca_i^\circ}$. Since the cross ratios of the numbers $A_i$ are simply the coordinates $X_j$ introduced above, we know that the point $e^{Cv}\in\mathcal{A}^+(P)$ determines the same hyperbolic structure on $P$ that we get from the point~$e^{Cm}$. Similarly, the point $e^{Cv^\circ}\in\mathcal{A}^+(P^\circ)$ determines the same hyperbolic structure on~$P^\circ$ that we get from~$e^{Cm}$. It follows that 
\[
\mathbb{I}_{\mathcal{D}}(l)(e^{Cm})=\frac{\mathbb{I}_{\mathcal{A}}(u^\circ)(e^{Cv^\circ})}{\mathbb{I}_{\mathcal{A}}(u)(e^{Cv})}.
\]
Using the definition of tropicalization and the result from the classical theory of the cluster $K_2$- and Poisson varieties, we see that 
\begin{align*}
(\mathbb{I}_{\mathcal{D}}(l))^t(m) &= \lim_{C\rightarrow\infty}\frac{\log\mathbb{I}_{\mathcal{D}}(l)(e^{Cm})}{C} \\
&= \lim_{C\rightarrow\infty}\frac{\log\mathbb{I}_{\mathcal{A}}(u^\circ)(e^{Cv^\circ})}{C} - \frac{\log\mathbb{I}_{\mathcal{A}}(u)(e^{Cv})}{C}\\
&= \mathcal{I}_{\mathcal{D}}(l,m)
\end{align*}
as desired.
\end{proof}

\begin{proposition}
\label{prop:loopintersection}
Let $l$ be a loop of weight $k$ on $S_{\mathcal{D}}$ which is not an intersecting curve. Then $\mathcal{I}_{\mathcal{D}}(l,m)=(\mathbb{I}_{\mathcal{D}}(l))^t(m)$ for any~$m\in\mathcal{D}_L(S,\mathbb{Z})$.
\end{proposition}

\begin{proof}
Let $m\in\mathcal{D}_L(S,\mathbb{Z})$ be a $\mathcal{D}$-lamination with coordinates $b_j$ and~$x_j$ ($j\in J$), and let $e^{Cm}$ denote the point of $\mathcal{D}^+(S)$ with coordinates $B_j=e^{Cb_j}$ and~$X_j=e^{Cx_j}$.

If $l$ lies entirely on $S^\circ$ and is not homotopic to a curve in the simple lamination, then we define $u$ to be the point of $\mathcal{A}_L(S^\circ)$ obtained by cutting $S_{\mathcal{D}}$. On the other hand, suppose $l$ is homotopic to a loop $\gamma_i$ in the simple lamination and $l$ and~$m$ provide the same orientation for this loop. If the orientation of~$\gamma_i$ agrees with the orientation of~$S^\circ$, then we define $u$ to be the point of $\mathcal{A}_L(S)$ obtained by shifting $l$ onto the surface $S$ and cutting $S_{\mathcal{D}}$. If the orientation agrees with that of $S$, then we define $u$ to be the point of $\mathcal{A}_L(S^\circ)$ obtained by shifting $l$ onto $S^\circ$ and cutting.

Now let $v$ be the point of $\mathcal{X}_L(S)$ or $\mathcal{X}_L(S^\circ)$ that we get by drawing the lamination $m$ on $S_{\mathcal{D}}$ and cutting. Let $e^{Cv}$ denote the point of $\mathcal{X}^+(S)$ or $\mathcal{X}^+(S^\circ)$ whose coordinates are obtained by scaling the coordinates of $v$ by $C$ and exponentiating. By construction, we have 
\[
\mathbb{I}_{\mathcal{D}}(l)(e^{Cm})=\mathbb{I}_{\mathcal{A}}(u)(e^{Cv}).
\]
It follows that 
\begin{align*}
(\mathbb{I}_{\mathcal{D}}(l))^t(m) &= \lim_{C\rightarrow\infty}\frac{\log\mathbb{I}_{\mathcal{D}}(l)(e^{Cm})}{C} \\
&= \lim_{C\rightarrow\infty}\frac{\log\mathbb{I}_{\mathcal{A}}(u)(e^{Cv})}{C} \\
&= (\mathbb{I}_{\mathcal{A}}(u))^t(v).
\end{align*}
Applying the result from the classical theory of the $\mathcal{A}$- and $\mathcal{X}$-spaces, we see that the last expression is $k/2$ times the minimal number of intersections between $l$ and~$m$.

If $l$ lies entirely on $S$ and is not homotopic to a curve in the simple lamination, then we define $u$ to be the point of $\mathcal{A}_L(S)$ obtained by cutting $S_{\mathcal{D}}$. On the other hand, suppose $l$ is homotopic to a loop $\gamma_i$ in the simple lamination and $l$ and $m$ provide different orientations for this loop. If the orientation that $l$ provides for~$\gamma_i$ agrees with the orientation of~$S$, then we define $u$ to be the point of $\mathcal{A}_L(S)$ obtained by shifting $l$ onto the surface $S$ and cutting. If the orientation agrees with that of $S^\circ$, then we define $u$ to be the point of $\mathcal{A}_L(S^\circ)$ obtained by shifting~$l$ onto~$S^\circ$ and cutting.

Let $v$ and $e^{Cv}$ be as above. Then by construction, we have 
\[
\mathbb{I}_{\mathcal{D}}(l)(e^{Cm})=\mathbb{I}_{\mathcal{A}}(u)(e^{Cv})^{-1}
\]
so that 
\begin{align*}
(\mathbb{I}_{\mathcal{D}}(l))^t(m) &= \lim_{C\rightarrow\infty}\frac{\log\mathbb{I}_{\mathcal{D}}(l)(e^{Cm})}{C} \\
&= \lim_{C\rightarrow\infty}\frac{\log\mathbb{I}_{\mathcal{A}}(u)(e^{Cv})^{-1}}{C} \\
&= -(\mathbb{I}_{\mathcal{A}}(u))^t(v).
\end{align*}
Applying the result from the classical theory once again, we see that this last expression is minus $k/2$ times the minimal number of intersections between $l$ and~$m$.
\end{proof}

\begin{theorem}
Let $l$,~$m\in\mathcal{D}_L(S,\mathbb{Z})$. If $\gamma_i$ is a component of~$\gamma$ that meets a curve of $l$, assume that the orientation that $l$ provides for $\gamma_i$ agrees with the orientation of~$S$. Then $\mathcal{I}_{\mathcal{D}}(l,m)=(\mathbb{I}_{\mathcal{D}}(l))^t(m)$.
\end{theorem}

\begin{proof}
Let $l\in\mathcal{D}_L(S,\mathbb{\mathbb{Z}})$. Then we can write $l=l_1+\dots+l_d$ where $l_1,\dots,l_d$ are mutually nonintersecting and nonhomotopic closed curves. We then have 
\begin{align*}
\mathcal{I}_{\mathcal{D}}(l,m) &= \mathcal{I}_{\mathcal{D}}(l_1,m)+\dots+\mathcal{I}_{\mathcal{D}}(l_d,m) \\
&= (\mathbb{I}_{\mathcal{D}}(l_1))^t(m)+\dots+(\mathbb{I}_{\mathcal{D}}(l_d))^t(m) \\
&= (\mathbb{I}_{\mathcal{D}}(l_1)\dots\mathbb{I}_{\mathcal{D}}(l_d))^t(m) \\
&= (\mathbb{I}_{\mathcal{D}}(l))^t(m)
\end{align*}
by Propositions~\ref{prop:intersectingintersection} and~\ref{prop:loopintersection} and the multiplicativity and additivity of the maps $\mathbb{I}_{\mathcal{D}}$ and~$\mathcal{I}_{\mathcal{D}}$.
\end{proof}

We can naturally extend $\mathcal{I}_{\mathcal{D}}$ to a pairing $\mathcal{D}_L(S,\mathbb{Z})\times\mathcal{D}_L(S,\mathbb{Z})\rightarrow\mathbb{Q}$ using the tropicalization of the map $\tau$ from~\cite{Dlam}.

\chapter{Duality of quantum cluster varieties}
\label{ch:DualityOfQuantumClusterVarieties}

The cluster Poisson variety and symplectic double can be canonically quantized. In this chapter, we define a canonical map from the tropical integral points of the cluster $K_2$-variety into the quantized algebra of regular functions on the cluster Poisson variety. We define a similar map from the tropical integral points of the symplectic double into its quantized algebra of rational functions. All of the cluster varieties considered here correspond to a disk with finitely many marked points on its boundary. In~\cite{AK}, similar results were obtained for cluster varieties associated to a punctured surface.

\section{Quantization of cluster varieties}

We have already discussed the main elements of the theory of quantum cluster varieties in the introduction. For convenience, we restate the main definitions here. We begin with a notion of seed which is equivalent to the one in~Definition~\ref{def:seed}.

\begin{definition}
A \emph{seed} $\mathbf{i}=(\Lambda,\{e_i\}_{i\in I},\{e_j\}_{j\in J},(\cdot,\cdot))$ is a quadruple where 
\begin{enumerate}
\item $\Lambda$ is a lattice with basis $\{e_i\}_{i\in I}$.
\item $\{e_j\}_{j\in J}$ is a subset of the basis.
\item $(\cdot,\cdot)$ is a $\mathbb{Z}$-valued skew-symmetric bilinear form on $\Lambda$.
\end{enumerate}
\end{definition}

A basis vector $e_i$ with $i\in I-J$ is said to be \emph{frozen}. Note that if we are given a seed, we can form a skew-symmetric integer matrix with entries $\varepsilon_{ij}=(e_i,e_j)$~($i$,~$j\in I$).

\begin{definition}
Let $\mathbf{i}=(\Lambda,\{e_i\}_{i\in I},\{e_j\}_{j\in J},(\cdot,\cdot))$ be a seed and $e_k$ ($k\in J$) a non-frozen basis vector. Then we define a new seed $\mathbf{i}'=(\Lambda',\{e_i'\}_{i\in I},\{e_j'\}_{j\in J},(\cdot,\cdot)')$ called the seed obtained by \emph{mutation} in the direction of $e_k$. It is given by $\Lambda'=\Lambda$, $(\cdot,\cdot)'=(\cdot,\cdot)$, and 
\[
e_i'=
\begin{cases}
-e_k & \mbox{if } i=k \\
e_i+[\varepsilon_{ik}]_+e_k & \mbox{if } i\neq k.
\end{cases}
\]
\end{definition}

\begin{definition}
The \emph{quantum dilogarithm} is the formal infinite product 
\[
\Psi^q(x)=\prod_{k=1}^\infty(1+q^{2k-1}x)^{-1}.
\]
\end{definition}

\begin{definition}
Let $\Lambda$ be a lattice equipped with a $\mathbb{Z}$-valued skew-symmetric bilinear form~$(\cdot,\cdot)$. Then the \emph{quantum torus algebra} is the noncommutative algebra over $\mathbb{Z}[q,q^{-1}]$ generated by variables $Y_v$ ($v\in\Lambda$) subject to the relations 
\[
q^{-(v_1,v_2)}Y_{v_1}Y_{v_2}=Y_{v_1+v_2}.
\]
\end{definition}

This definition allows us to associate to any seed $\mathbf{i}=(\Lambda,\{e_i\}_{i\in I},\{e_j\}_{j\in J},(\cdot,\cdot))$, a quantum torus algebra $\mathcal{X}_{\mathbf{i}}^q$. The set $\{e_j\}_{j\in J}$ provides a set of generators $X_j^{\pm1}$ given by $X_j=Y_{e_j}$ for this algebra. They obey the commutation relations 
\[
X_iX_j=q^{2\varepsilon_{ij}}X_jX_i.
\]
This algebra $\mathcal{X}_{\mathbf{i}}^q$ satisfies the Ore condition from ring theory, so we can form its noncommutative fraction field $\widehat{\mathcal{X}}_{\mathbf{i}}^q$. In addition to associating a quantum torus algebra to every seed, we use the quantum dilogarithm to construct a natural map $\widehat{\mathcal{X}}_{\mathbf{i}'}^q\rightarrow\widehat{\mathcal{X}}_{\mathbf{i}}^q$ whenever two seeds $\mathbf{i}$ and~$\mathbf{i}'$ are related by a mutation.

\Needspace*{3\baselineskip}
\begin{definition}\mbox{}
\begin{enumerate}
\item The automorphism $\mu_k^\sharp:\widehat{\mathcal{X}}_{\mathbf{i}}^q\rightarrow\widehat{\mathcal{X}}_{\mathbf{i}}^q$ is given by conjugation with $\Psi^q(X_k)$:
\[
\mu_k^\sharp=\Ad_{\Psi^q(X_k)}.
\]

\item The isomorphism $\mu_k':\widehat{\mathcal{X}}_{\mathbf{i}'}^q\rightarrow\widehat{\mathcal{X}}_{\mathbf{i}}^q$ is induced by the natural lattice map $\Lambda'\rightarrow\Lambda$.

\item The mutation map $\mu_k^q:\widehat{\mathcal{X}}_{\mathbf{i}'}^q\rightarrow\widehat{\mathcal{X}}_{\mathbf{i}}^q$ is the composition $\mu_k^q=\mu_k^\sharp\circ\mu_k'$.
\end{enumerate}
\end{definition}

If $\mathbf{i}=(\Lambda,\{e_i\}_{i\in I},\{e_j\}_{j\in J},(\cdot,\cdot))$ is any seed, then we can form the ``double'' $\Lambda_{\mathcal{D}}=\Lambda_{\mathcal{D},\mathbf{i}}$ of the lattice $\Lambda$ given by the formula 
\[
\Lambda_{\mathcal{D}}=\Lambda\oplus\Lambda^\vee
\]
where $\Lambda^\vee=\Hom(\Lambda,\mathbb{Z})$. The basis $\{e_i\}$ for $\Lambda$ provides a dual basis $\{f_i\}$ for $\Lambda^\vee$, and hence we have a basis $\{e_i,f_i\}$ for $\Lambda_{\mathcal{D}}$. Moreover, there is a natural skew-symmetric bilinear form $(\cdot,\cdot)_{\mathcal{D}}$ on~$\Lambda_{\mathcal{D}}$ given by the formula 
\[
\left((v_1,\varphi_1),(v_2,\varphi_2)\right)_{\mathcal{D}}=(v_1,v_2)+\varphi_2(v_1)-\varphi_1(v_2).
\]
We can associate to these data a quantum torus algebra~$\mathcal{D}_{\mathbf{i}}^q$. If we let $X_i$ and $B_i$ denote the generators associated to the basis elements $e_i$ and $f_i$, respectively, then we have the commutation relations 
\[
X_iX_j=q^{2\varepsilon_{ij}}X_jX_i, \quad B_iB_j=B_jB_i, \quad X_iB_j=q^{2\delta_{ji}}B_jX_i.
\]
We will write $\widehat{\mathcal{D}}_{\mathbf{i}}^q$ for the (noncommutative) fraction field of~$\widehat{\mathcal{D}}_{\mathbf{i}}^q$.

\Needspace*{3\baselineskip}
\begin{definition}\mbox{}
\begin{enumerate}
\item The automorphism $\mu_k^\sharp:\widehat{\mathcal{D}}_{\mathbf{i}}^q\rightarrow\widehat{\mathcal{D}}_{\mathbf{i}}^q$ is given by 
\[
\mu_k^\sharp=\Ad_{\Psi^q(X_k)/\Psi^q(\widehat{X}_k)}.
\]

\item The isomorphism $\mu_k':\widehat{\mathcal{D}}_{\mathbf{i}'}^q\rightarrow\widehat{\mathcal{D}}_{\mathbf{i}}^q$ is induced by the natural lattice map $\Lambda_{\mathcal{D},\mathbf{i}'}\rightarrow\Lambda_{\mathcal{D},\mathbf{i}}$.

\item The mutation map $\mu_k^q:\widehat{\mathcal{D}}_{\mathbf{i}'}^q\rightarrow\widehat{\mathcal{D}}_{\mathbf{i}}^q$ is the composition $\mu_k^q=\mu_k^\sharp\circ\mu_k'$.
\end{enumerate}
\end{definition}

Notice that the generators $X_i^{\pm1}$ span a subalgebra of $\mathcal{D}_{\mathbf{i}}^q$ which is isomorphic to the quantum torus algebra $\mathcal{X}_{\mathbf{i}}^q$ defined previously. Moreover, since $X_k$ and $\widehat{X}_k$ commute, the restriction of~$\mu_k^q$ to this subalgebra coincides with the previous map $\mu_k^q$.

The proof of the following theorem can be found in Appendix~\ref{ch:DerivationOfTheClassicalMutationFormulas}.

\begin{theorem}
\label{thm:transformation}
The map $\mu_k^q$ is given on generators by the formulas 
\[
\mu_k^q(B_i')=
\begin{cases}
(qX_k\mathbb{B}_k^++\mathbb{B}_k^-)B_k^{-1}(1+q^{-1}X_k)^{-1} & \mbox{if } i=k \\
B_i & \mbox{if } i\neq k
\end{cases}
\]
and
\[
\mu_k^q(X_i')=
\begin{cases}
X_i\prod_{p=0}^{|\varepsilon_{ik}|-1}(1+q^{2p+1}X_k) & \mbox{if } \varepsilon_{ik}\leq0 \mbox{ and } i\neq k \\
X_iX_k^{\varepsilon_{ik}}\prod_{p=0}^{\varepsilon_{ik}-1}(X_k+q^{2p+1})^{-1} & \mbox{if } \varepsilon_{ik}\geq0 \mbox{ and } i\neq k \\
X_k^{-1} & \mbox{if } i=k.
\end{cases}
\]
\end{theorem}

\begin{definition}
We set 
\begin{align*}
\mathcal{X}^q &= \coprod_{\mathbf{i}'\in|\mathbf{i}|}\widehat{\mathcal{X}}_{\mathbf{i}'}^q/\text{identifications}, \\
\mathcal{D}^q &= \coprod_{\mathbf{i}'\in|\mathbf{i}|}\widehat{\mathcal{D}}_{\mathbf{i}'}^q/\text{identifications},
\end{align*}
where we are quotienting by the identifications given by the maps $\mu_k^q$.
\end{definition}

The sets $\mathcal{X}^q$ and $\mathcal{D}^q$ inherit natural algebra structures and in the classical limit $q=1$ are identified with the function fields of the cluster Poisson variety and symplectic double, respectively. For $q\neq1$, we think of $\mathcal{X}^q$ and $\mathcal{D}^q$ as the function fields of corresponding ``quantum cluster varieties''.

\section{The skein algebra}

In this section, we will define a version of the skein algebra, following~\cite{Muller}. Throughout this chapter, the term ``curve'' will refer to the following technical notion.

\begin{definition}
By a \emph{curve} in $S$, we mean an immersion $C\rightarrow S$ of a compact, connected, one-dimensional manifold $C$ with (possibly empty) boundary into $S$. We require that any boundary points of $C$ map to the marked points on $\partial S$ and no point in the interior of $C$ maps to a marked point. By a \emph{homotopy} of two curves $\alpha$ and $\beta$, we mean a homotopy of $\alpha$ and $\beta$ within the class of such curves. Two curves are said to be \emph{homotopic} if they can be related by homotopy and orientation-reversal.
\end{definition}

\begin{definition}
A \emph{multicurve} is an unordered finite set of curves which may contain duplicates. Two multicurves are \emph{homotopic} if there is a bijection between their constituent curves which takes each curve to a homotopic one.
\end{definition}

\begin{definition}
A \emph{framed link} in $S$ is a multicurve such that each intersection of strands is transverse and 
\begin{enumerate}
\item At each crossing, there is an ordering of the strands.
\item At each marked point, there is an equivalence relation on the strands and an ordering on equivalence classes of strands.
\end{enumerate}
By a \emph{homotopy} of framed links, we mean a homotopy through the class of multicurves with transverse intersections where the crossing data are not changed.
\end{definition}

When drawing pictures of framed links, we indicate the ordering of strands at a transverse intersection or marked point by making one strand pass ``over'' the other:
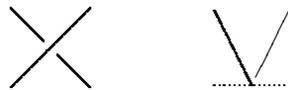
\[
\xygraph{
    !{0;/r2.5pc/:}
    [u(0.5)]!{\xoverv}
}
\qquad
\qquad
\xy 
(0,5)*{}; (5,-5)*{} **\crv{(0,5)&(5,-5)};
\POS?(1)*{\hole}="x"; 
(10,5)*{}; "x" **\crv{}; 
(0,-5)*{}; (10,-5)*{} **\dir{.}
\endxy
\]
(In the second of these pictures, the dotted line indicates a portion of $\partial S$ containing a marked point.)
If two strands are identified by the equivalence relation at a marked point, we indicate this as in the following picture: 
\[
\xy 
(0,5)*{}; (-5,-5)*{} **\crv{(0,5)&(-5,-5)};
(-10,5)*{}; (-5,-5)*{} **\crv{}; 
(0,-5)*{}; (-10,-5)*{} **\dir{.}
\endxy
\]

\begin{definition}
A multicurve with transverse intersections is said to be \emph{simple} if it has no interior intersections and no contractible curves. Note that any simple multicurve can be regarded as a framed link with the simultaneous ordering at each endpoint.
\end{definition}

\begin{definition}
\label{def:skein}
Let us write $\mathcal{K}(S)$ for the free $\mathbb{Z}[\omega,\omega^{-1}]$-module generated by equivalence classes of framed links in~$S$. The \emph{skein module} $\mathrm{Sk}_\omega(S)$ is defined as the quotient of $\mathcal{K}(S)$ by the following local relations. In each of these expressions, we depict the portion of a framed link over a small disk in~$S$. The framed links appearing in a given relation are assumed to be identical to each other outside of the small disk. In the last two relations, the dotted line segment represents a portion of $\partial S$. In these pictures, there may be additional undrawn curves ending at marked points, provided their order with respect to the drawn curves and each other does not change.
\begin{align*}
\xygraph{
    !{0;/r2.5pc/:}
    [u(0.5)]!{\xoverv}
}
\quad
&=
\quad
\omega^{-2}
\quad
\xygraph{
    !{0;/r2.5pc/:}
    [u(0.5)]!{\xunoverv}
}
\quad
+
\quad
\omega^2
\quad
\xygraph{
    !{0;/r2.5pc/:}
    [u(0.5)]!{\xunoverh}
} \\
\xygraph{
    !{0;/r2.5pc/:}
    !{\vcap-}
    !{\vcap}
}
\quad
&=
\quad
-(\omega^4+\omega^{-4}) \\
\omega
\quad
{\xy 
(0,5)*{}; (-5,-5)*{} **\crv{(0,5)&(-5,-5)};
\POS?(1)*{\hole}="x"; 
(-10,5)*{}; "x" **\crv{}; 
(0,-5)*{}; (-10,-5)*{} **\dir{.}
\endxy}
\quad
&=
\quad
{\xy 
(0,5)*{}; (-5,-5)*{} **\crv{(0,5)&(-5,-5)};
(-10,5)*{}; (-5,-5)*{} **\crv{}; 
(0,-5)*{}; (-10,-5)*{} **\dir{.}
\endxy}
\quad
=
\quad
\omega^{-1}
\quad
{\xy 
(0,5)*{}; (5,-5)*{} **\crv{(0,5)&(5,-5)};
\POS?(1)*{\hole}="x"; 
(10,5)*{}; "x" **\crv{}; 
(0,-5)*{}; (10,-5)*{} **\dir{.}
\endxy} \\
{\xy 
(-2,2)*{}; (-5,-5)*{} **\crv{(-2,2)&(-5,-5)};
\POS?(1)*{\hole}="x"; 
(-8,2)*{}; "x" **\crv{}; 
(-2,2)*{}; (-8,2)*{} **\crv{(0,8)&(-10,8)};
(0,-5)*{}; (-10,-5)*{} **\dir{.}
\endxy}
\quad
&=
\quad
{\xy 
(-2,2)*{}; (-5,-5)*{} **\crv{(-2,2)&(-5,-5)};
(-8,2)*{}; (-5,-5)*{} **\crv{}; 
(-2,2)*{}; (-8,2)*{} **\crv{(0,8)&(-10,8)};
(0,-5)*{}; (-10,-5)*{} **\dir{.}
\endxy}
\quad
=
\quad
{\xy 
(2,2)*{}; (5,-5)*{} **\crv{(2,2)&(5,-5)};
\POS?(1)*{\hole}="x"; 
(8,2)*{}; "x" **\crv{}; 
(2,2)*{}; (8,2)*{} **\crv{(0,8)&(10,8)};
(0,-5)*{}; (10,-5)*{} **\dir{.}
\endxy}
\quad
=
\quad 0
\end{align*}
If $K$ is a framed link in $S$, then the class of $K$ in $\mathrm{Sk}_\omega(S)$ will be denoted $[K]$.
\end{definition}

Suppose $K$ and $L$ are two links such that the union of the underlying multicurves has transverse intersections. Then the \emph{superposition} $K\cdot L$ is the framed link whose underlying multicurve is the union of the underlying multicurves of $K$ and $L$, with each strand of $K$ crossing over each strand of $L$ and all other crossings ordered as in~$K$ and~$L$.

\begin{proposition}[\cite{Muller}, Proposition~3.5]
\label{prop:superpositionhomotopy}
$[K\cdot L]$ depends only on the homotopy classes of~$K$ and~$L$.
\end{proposition}

\begin{definition}
For any framed links $K$ and $L$, choose homotopic links $K'$ and $L'$ such that the union of the multicurves underlying $K'$ and $L'$ is transverse. Then the \emph{superposition product} is defined by 
\[
[K][L]\coloneqq[K'\cdot L'].
\]
This extends to a product on $\mathrm{Sk}_\omega(S)$ by bilinearity.
\end{definition}

Note that the superposition product is well defined and independent of the choice of~$K'$ and~$L'$ by Proposition~\ref{prop:superpositionhomotopy}.

We conclude this section by describing the relation, first discovered by Muller in~\cite{Muller}, between skein algebras and quantum cluster algebras. A review of the relevant notions from the theory of quantum cluster algebras can be found in Appendix~\ref{ch:ReviewOfQuantumClusterAlgebras}. For the rest of this chapter, we take $S$ to be a disk with finitely many marked points on its boundary.

\begin{definition}
Let $T$ be an ideal triangulation of $S$. For any $i\in I$ and $j\in J$, we define 
\[
b_{ij}=
\begin{cases}
1 & \mbox{if $i$, $j$ share a vertex and $i$ is immediately clockwise to $j$} \\
-1 & \mbox{if $i$, $j$ share a vertex and $j$ is immediately clockwise to $i$} \\
0 & \mbox{otherwise}.
\end{cases}
\]
These are entries of a skew-symmetric $|I|\times|J|$ matrix which we denote $\mathbf{B}_T$.
\end{definition}

\begin{definition}
Let $T$ be an ideal triangulation of $S$. For $i$,~$j\in I$, we define 
\[
\lambda_{ij}=
\begin{cases}
1 & \mbox{if $i$, $j$ share a vertex and $i$ is clockwise to $j$} \\
-1 & \mbox{if $i$, $j$ share a vertex and $j$ is clockwise to $i$} \\
0 & \mbox{otherwise}.
\end{cases}
\]
These are entries of a skew-symmetric $|I|\times|I|$ matrix which we denote $\Lambda_T$. We will use the same notation $\Lambda_T$ for the associated skew-symmetric bilinear form $\mathbb{Z}^{I}\times\mathbb{Z}^{I}\rightarrow\mathbb{Z}$.
\end{definition}

\begin{proposition}[\cite{Muller}, Proposition~7.8]
The matrices $\Lambda_T$ and $\mathbf{B}_T$ satisfy the compatibility condition 
\[
\sum_kb_{kj}\lambda_{ki}=4\delta_{ij}.
\]
\end{proposition}

\begin{definition}
Let $\mathcal{F}$ denote the skew-field of fractions of the skein algebra $\mathrm{Sk}_\omega(S)$. For any ideal triangulation $T$ and any vector $v=(v_1,\dots,v_m)\in\mathbb{Z}_{\geq0}^{I}$, we will write $T^v$ for the simple multicurve having $v_i$-many curves homotopic to $i\in I$ and no other components. The corresponding class~$[T^v]$ is called a \emph{monomial} in the triangulation~$T$. More generally, we write 
\[
[T^{u'-u}]=\omega^{-\Lambda_T(u,u')}[T^u]^{-1}[T^{u'}].
\]
This is well defined and provides a map $M_T:\mathbb{Z}^{I}\rightarrow\mathcal{F}-\{0\}$ given by $M_T(v)=[T^v]$.
\end{definition}

One can show that the pair $(\mathbf{B}_T,M_T)$ is a quantum seed and $\Lambda_T$ is the compatibility matrix associated to the toric frame~$M_T$.

\begin{proposition}[\cite{Muller}, Theorem~7.9]
\label{prop:quantumflip}
Let $T$ be an ideal triangulation of $S$, and let $T'$ be the ideal triangulation obtained from~$T$ by performing a flip of the edge~$k$. Then the quantum seed $(\mathbf{B}_{T'},M_{T'})$ is obtained from $(\mathbf{B}_T,M_T)$ by a mutation in the direction~$k$.
\end{proposition}

It follows from Proposition~\ref{prop:quantumflip} that there is a quantum cluster algebra $\mathcal{A}$ canonically associated to the disk~$S$. This algebra is generated by the cluster variables inside of the skew-field of fractions of the skein algebra.

\section{Duality map for the quantum Poisson variety}

\subsection{Realization of the cluster Poisson variety}

We have associated a quantum cluster algebra $\mathcal{A}$ to the disk $S$. Let $\mathcal{F}$ be the ambient skew-field of this algebra $\mathcal{A}$.

\begin{definition}
Let $T$ be an ideal triangulation of $S$. For any $j\in J$, we define an element $X_j=X_{j;T}$ of $\mathcal{F}$ by the formula 
\[
X_j = M_T\big(\sum_s\varepsilon_{js}\mathbf{e}_s\big).
\]
where the $\mathbf{e}_s$ are standard basis vectors. In addition, we will use the notation $q=\omega^4$.
\end{definition}

As the notation suggests, these elements are related to the generators of the quantum Poisson variety.

\begin{proposition}
\label{prop:Xcommutation}
The elements $X_j$~($j\in J$) satisfy the commutation relations 
\[
X_iX_j=q^{2\varepsilon_{ij}}X_jX_i.
\]
\end{proposition}

\begin{proof}
By the properties of toric frames, we have 
\[
X_iX_j=\omega^{-2\Lambda_T(\sum_s\varepsilon_{is}\mathbf{e}_s,\sum_t\varepsilon_{jt}\mathbf{e}_t)} X_jX_i.
\]
Now the compatibility condition in the definition of a quantum seed implies 
\begin{align*}
\Lambda_T\big(\sum_s\varepsilon_{is}\mathbf{e}_s,\sum_t\varepsilon_{jt}\mathbf{e}_t\big) &= \sum_{s,t}\varepsilon_{is}\varepsilon_{jt}\Lambda_T(\mathbf{e}_s,\mathbf{e}_t) \\
&= \sum_{s,t}\varepsilon_{is}\varepsilon_{jt}\lambda_{st} \\
&= -\sum_s\varepsilon_{is}\sum_tb_{tj}\lambda_{ts} \\
&= -4\varepsilon_{ij}.
\end{align*}
Therefore $X_iX_j=\omega^{8\varepsilon_{ij}}X_jX_i=q^{2\varepsilon_{ij}}X_jX_i$.
\end{proof}

In the following result, $T'$ denotes the triangulation obtained from $T$ by a flip of the edge~$k$.

\begin{proposition}
\label{prop:Xtrans}
For each $i$, let $X_i'=X_{i;T'}$. Then 
\[
X_i'=
\begin{cases}
X_i\prod_{p=0}^{|\varepsilon_{ik}|-1}(1+q^{2p+1}X_k) & \mbox{if } \varepsilon_{ik}\leq0 \mbox{ and } i\neq k \\
X_iX_k^{\varepsilon_{ik}}\prod_{p=0}^{\varepsilon_{ik}-1}(X_k+q^{2p+1})^{-1} & \mbox{if } \varepsilon_{ik}\geq0 \mbox{ and } i\neq k \\
X_k^{-1} & \mbox{if } i=k.
\end{cases}
\]
\end{proposition}

\begin{proof}
According to Lemma~5.4 of~\cite{Tran}, a mutation in the direction $k$ transforms $X_i$ to the new element 
\[
X_i'=
\begin{cases}
X_i\prod_{p=0}^{|b_{ki}|-1}(1+\omega^{-2(-d_kp-\frac{d_k}{2})}X_k) & \mbox{if } b_{ki}\leq0 \mbox{ and } i\neq k \\
X_iX_k^{b_{ki}}\prod_{p=0}^{b_{ki}-1}(X_k+\omega^{-2(-d_kp-\frac{d_k}{2})})^{-1} & \mbox{if } b_{ki}\geq0 \mbox{ and } i\neq k \\
X_k^{-1} & \mbox{if } i=k
\end{cases}
\]
where $d_k$ is the number appearing in Definition~\ref{def:compatibility}. The statement follows if we substitute $d_k=4$ into this formula.
\end{proof}

\subsection{Construction of the map}

We will now define the map $\mathbb{I}_{\mathcal{A}}^q$. To do this, let $l$ be a point of $\mathcal{A}_0(\mathbb{Z}^t)$, represented by a collection of arcs on~$S$. We deform each arc by dragging its endpoints along the boundary in the counterclockwise direction until they hit the marked points. We can assume that the resulting curves are nonintersecting, so there exists an ideal triangulation $T_l$ of $S$ such that each of these curves coincides with an edge of $T_l$. Let 
\[
\mathbf{w}=(w_1,\dots,w_m)
\]
be the integral vector whose $i$th component $w_i$ is the weight of the curve corresponding to the edge $i$ of $T_l$.

\begin{definition}
We will write $\mathbb{I}_{\mathcal{A}}^q(l)$ for the element of $\mathcal{F}$ given by 
\[
\mathbb{I}_{\mathcal{A}}^q(l)=M_{T_l}(\mathbf{w}).
\]
\end{definition}

It is clear from the definition of the toric frames that this is independent of the choice of~$T_l$. Our goal is to show that this definition provides a map $\mathbb{I}_{\mathcal{A}}^q:\mathcal{A}_0(\mathbb{Z}^t)\rightarrow\mathcal{X}^q$.

\begin{theorem}
\label{thm:Xcanonicalregular}
Let $T$ be an ideal triangulation of $S$, and let $X_j=X_{j;T}$ be defined as above. Then for any $l\in\mathcal{A}_0(\mathbb{Z}^t)$, the element $\mathbb{I}_{\mathcal{A}}^q(l)$ is a Laurent polynomial in the~$X_j$ with coefficients in~$\mathbb{Z}_{\geq0}[q,q^{-1}]$.
\end{theorem}

\begin{proof}
By applying the relations from Definition~\ref{def:skein}, we can write 
\[
\mathbb{I}_{\mathcal{A}}^q(l) = \omega^\alpha\prod_{i=1}^mM_{T_l}(\mathbf{e}_i)^{w_i}
\]
where $\alpha=\sum_{i<j}\Lambda_{T_l}(\mathbf{e}_i,\mathbf{e}_j)w_iw_j$. Then Corollary~\ref{cor:Fpositivity} implies 
\[
\mathbb{I}_{\mathcal{A}}^q(l) = \omega^\alpha\prod_{i=1}^m(F_i\cdot M_T(\mathbf{g}_i))^{w_i}.
\]
Each $F_i$ in this last line denotes a polynomial in the expressions $M_T(\sum_j\varepsilon_{kj}\mathbf{e}_j)$ with coefficients in $\mathbb{Z}_{\geq0}[\omega^4,\omega^{-4}]$. The $\mathbf{g}_i$ are integral vectors. We have 
\begin{align*}
\Lambda_T\big(\sum_i\varepsilon_{ki}\mathbf{e}_i,\mathbf{e}_j\big) = \sum_i\varepsilon_{ki}\lambda_{ij} = \sum_i b_{ik}\lambda_{ij} = 4\delta_{kj}
\end{align*}
so that 
\[
M_T\big(\sum_i\varepsilon_{ki}\mathbf{e}_i\big)M_T(\mathbf{e}_j) = \omega^{-8\delta_{kj}} M_T(\mathbf{e}_j)M_T\big(\sum_i\varepsilon_{ki}\mathbf{e}_i\big).
\]
It follows that we can rearrange the factors in the last expression for $\mathbb{I}_{\mathcal{A}}^q(l)$ to get 
\[
\mathbb{I}_{\mathcal{A}}^q(l) = \omega^\alpha Q\cdot\prod_{i=1}^m M_T(\mathbf{g}_i)^{w_i}
\]
where $Q$ is a polynomial in the expressions $M_T(\sum_j\varepsilon_{kj}\mathbf{e}_j)$ with coefficients in the semiring $\mathbb{Z}_{\geq0}[\omega^4,\omega^{-4}]$. By Lemma~6.2 in~\cite{Tran}, we know that $\Lambda_{T_l}(\mathbf{e}_i,\mathbf{e}_j)=\Lambda_T(\mathbf{g}_i,\mathbf{g}_j)$, and so 
\begin{align*}
M_T\big(\sum_iw_i\mathbf{g}_i\big) &= \omega^{\sum_{i<j}\Lambda_T(\mathbf{g}_i,\mathbf{g}_j)w_iw_j} \prod_{i=1}^m M_T(\mathbf{g}_i)^{w_i} \\
&= \omega^\alpha \prod_{i=1}^m M_T(\mathbf{g}_i)^{w_i}.
\end{align*}
Hence 
\[
\mathbb{I}_{\mathcal{A}}^q(l) = Q\cdot M_T\big(\sum_iw_i\mathbf{g}_i\big).
\]
Arguing as in the proof of Lemma~\ref{lem:factorization}, one shows that the second factor is a monomial in the $X_j$ with coefficient in $\mathbb{Z}_{\geq0}[\omega^4,\omega^{-4}]$. This completes the proof.
\end{proof}

By the properties established in Propositions~\ref{prop:Xcommutation}, and~\ref{prop:Xtrans}, we can regard the Laurent polynomial of Theorem~\ref{thm:Xcanonicalregular} as an element of the algebra $\mathcal{X}^q$. Thus we have constructed a canonical map $\mathbb{I}_{\mathcal{A}}^q:\mathcal{A}_0(\mathbb{Z}^t)\rightarrow\mathcal{X}^q$ as desired.

In recent joint work with Kim~\cite{AK}, the author constructed a similar map for cluster varieties associated to a punctured surface. Unlike the one presented here, the construction of~\cite{AK} relied heavily on the ``quantum trace'' map introduced by Bonahon and Wong~\cite{BonahonWong}. Recent results of L\^{e} suggest that these two approaches are in fact equivalent~\cite{Le}.

\subsection{Properties}

We will now prove several properties of the map $\mathbb{I}_{\mathcal{A}}^q$ conjectured in~\cite{IHES,ensembles}.

\begin{theorem}
The map $\mathbb{I}_{\mathcal{A}}^q:\mathcal{A}_0(\mathbb{Z}^t)\rightarrow\mathcal{X}^q$ satisfies the following properties:
\begin{enumerate}
\item Each $\mathbb{I}_{\mathcal{A}}^q(l)$ is a Laurent polynomial in the variables $X_i$ with coefficients in $\mathbb{Z}_{\geq0}[q,q^{-1}]$.

\item The expression $\mathbb{I}_{\mathcal{A}}^q(l)$ agrees with the Laurent polynomial $\mathbb{I}_{\mathcal{A}}(l)$ when~$q=1$.

\item Let $*$ be the canonical involutive antiautomorphism of $\mathcal{X}^q$ that fixes each $X_i$ and sends~$q$ to~$q^{-1}$. Then $*\mathbb{I}_{\mathcal{A}}^q(l)=\mathbb{I}_{\mathcal{A}}^q(l)$.

\item The highest term of $\mathbb{I}_{\mathcal{A}}^q(l)$ is
\[
q^{-\sum_{i<j}\varepsilon_{ij}a_ia_j}X_1^{a_1}\dots X_n^{a_n}
\]
where $X_1^{a_1}\dots X_n^{a_n}$ is the highest term of the classical expression $\mathbb{I}_{\mathcal{A}}(l)$.

\item For any $l$,~$l'\in\mathcal{A}_0(\mathbb{Z}^t)$, we have
\[
\mathbb{I}_{\mathcal{A}}^q(l)\mathbb{I}_{\mathcal{A}}^q(l') = \sum_{l''\in\mathcal{A}_0(\mathbb{Z}^t)}c^q(l,l';l'')\mathbb{I}_{\mathcal{A}}^q(l'')
\]
where $c^q(l,l';l'')\in\mathbb{Z}[q,q^{-1}]$ and only finitely many terms are nonzero.
\end{enumerate}
\end{theorem}

\begin{proof}
1. This was proved in Theorem~\ref{thm:Xcanonicalregular}.

2. This follows immediately from the definitions.

3. If $K$ is any framed link in~$S$, let $K^\dagger$ be the framed link with the same underlying multicurve and the order of each crossing reversed. By Proposition~3.11 of~\cite{Muller} the map sending $[K]$ to $[K^\dagger]$ and $\omega$ to $\omega^\dagger\coloneqq\omega^{-1}$ extends to an involutive antiautomorphism of the skein algebra $\mathrm{Sk}_\omega(S)$. It extends further to an antiautomorphism of the fraction field $\mathcal{F}$ by the rule $(xy^{-1})^\dagger=(y^\dagger)^{-1}x^\dagger$. This operation is compatible with $*$ in the sense that 
\[
X_j^\dagger=X_j=*X_j
\]
and 
\[
q^\dagger=q^{-1}=*q.
\]
It is easy to check that $\dagger$ preserves $M_{T_l}(\mathbf{w})=\omega^{\sum_{i<j}\Lambda_{T_l}(\mathbf{e}_i,\mathbf{e}_j)}M_{T_l}(\mathbf{e}_1)^{w_1}\dots M_{T_l}(\mathbf{e}_m)^{w_m}$. It follows that $\mathbb{I}_{\mathcal{A}}^q(l)$ is $*$-invariant.

4. By part~1, we can write
\[
\mathbb{I}_{\mathcal{A}}^q(l) = \sum_{(i_1,\dots,i_n)\in\supp(l)}c_{i_1,\dots,i_n}X_1^{i_1}\dots X_n^{i_n}
\]
for some finite subset $\supp(l)\subseteq\mathbb{Z}^n$ where the coefficient $c_{i_1,\dots,i_n}\in\mathbb{Z}_{\geq0}[q,q^{-1}]$ is nonzero for all $(i_1,\dots,i_n)\in\supp(l)$. Consider the coefficient $c=c_{a_1,\dots,a_n}$. We can expand this coefficient as $c=\sum_sc_sq^s$. By setting $q=1$, we see that $\sum_sc_s=1$. It follows that we must have $c_s=1$ for some $s$, and all other terms must vanish. Thus $c=q^s$, and the leading term of $\mathbb{I}_{\mathcal{A}}^q(l)$ has the form
\[
q^sX_1^{a_1}\dots X_n^{a_n}
\]
for some integer $s$. By part~3, we know that the expression $\mathbb{I}_{\mathcal{A}}^q(l)$, and hence its leading term, is $*$-invariant. This means 
\[
q^{-s} X_n^{a_n}\dots X_1^{a_1}=q^s X_1^{a_1}\dots X_n^{a_n}.
\]
We have 
\begin{align*}
X_n^{a_n}\dots X_1^{a_1} &= q^{-2\sum_{j=2}^n\varepsilon_{1j}a_1a_j} X_1^{a_1}(X_n^{a_n}\dots X_2^{a_2}) \\
&= q^{-2\sum_{j=2}^n\varepsilon_{1j}a_1a_j} q^{-2\sum_{j=3}^n\varepsilon_{2j}a_2a_j} X_1^{a_1}X_2^{a_2} (X_n^{a_n}\dots X_3^{a_3}) \\
&= \dots \\
&= q^{-2\sum_{i<j}\varepsilon_{ij}a_ia_j} X_1^{a_1}\dots X_n^{a_n},
\end{align*}
so $*$-invariance of the leading term is equivalent to 
\[
q^{-s-2\sum_{i<j}\varepsilon_{ij}a_ia_j} X_1^{a_1}\dots X_n^{a_n}=q^s X_1^{a_1}\dots X_n^{a_n}.
\]
Equating coefficients, we see that $s=-\sum_{i<j}\varepsilon_{ij}a_ia_j$. Thus we see that the highest term equals $q^{-\sum_{i<j}\varepsilon_{ij}a_ia_j}X_1^{a_1}\dots X_n^{a_n}$.

5. Let $l$,~$l'\in\mathcal{A}_0(\mathbb{Z}^t)$. Then there exist ideal triangulations $T_l$ and $T_{l'}$ of~$S$ and integral vectors $\mathbf{w}$ and $\mathbf{w}'$ such that $\mathbb{I}_{\mathcal{A}}^q(l)=M_{T_l}(\mathbf{w})$ and $\mathbb{I}_{\mathcal{A}}^q(l')=M_{T_{l'}}(\mathbf{w}')$. For any integral vector $\mathbf{v}=(v_1,\dots,v_m)$, we can consider the vector $\mathbf{v}_+$ whose $i$th component equals $v_i$ if $v_i\geq0$ and zero otherwise. We can likewise consider the vector $\mathbf{v}_-$ whose $i$th component equals $v_i$ if $v_i\leq0$ and zero otherwise. Then there exists an integer $N$ such that 
\begin{align*}
\mathbb{I}_{\mathcal{A}}^q(l)\mathbb{I}_{\mathcal{A}}^q(l') &= M_{T_l}(\mathbf{w}) M_{T_{l'}}(\mathbf{w}') \\
&= \omega^N M_{T_l}(\mathbf{w}_-)M_{T_l}(\mathbf{w}_+)M_{T_{l'}}(\mathbf{w}_+')M_{T_{l'}}(\mathbf{w}_-') \\
&= \omega^N M_{T_l}(\mathbf{w}_-)[T_l^{\mathbf{w}_+}][T_{l'}^{\mathbf{w}_+'}]M_{T_{l'}}(\mathbf{w}_-').
\end{align*}
By applying the skein relations, we can write the product $[T_l^{\mathbf{w}_+}][T_{l'}^{\mathbf{w}_+'}]$ as 
\[
[T_l^{\mathbf{w}_+}][T_{l'}^{\mathbf{w}_+'}]=\sum_{i=1}^kd_i^\omega[K_i]
\]
where the $K_i$ are distinct simple multicurves on~$S$ and $d_i^\omega\in\mathbb{Z}[\omega,\omega^{-1}]$ are nonzero. Therefore 
\[
\mathbb{I}_{\mathcal{A}}^q(l)\mathbb{I}_{\mathcal{A}}^q(l') = \sum_{i=1}^k\omega^Nd_i^\omega M_{T_i}(\mathbf{w}_-)M_{T_i}(\mathbf{v}_i)M_{T_i}(\mathbf{w}_-')
\]
where $T_i$ is an ideal triangulation of $S$ such that each each curve of $K_i$ coincides with an edge of $T_i$. Here we have written $\mathbf{v}_i$ for the unique integral vector such that $[K_i]=M_{T_i}(\mathbf{v}_i)$. Each of the products $\omega^Nd_i^\omega M_{T_i}(\mathbf{w}_-)M_{T_i}(\mathbf{v}_i)M_{T_i}(\mathbf{w}_-')$ in this last expression can be written as $c^\omega(l,l';l_i)\mathbb{I}_\mathcal{A}^q(l_i)$ for some $l_i\in\mathcal{A}_0(\mathbb{Z}^t)$ and some $c^\omega(l,l';l_i)\in\mathbb{Z}[\omega,\omega^{-1}]$. Hence 
\[
\mathbb{I}_{\mathcal{A}}^q(l)\mathbb{I}_{\mathcal{A}}^q(l') = \sum_{i=1}^kc^\omega(l,l';l_i)\mathbb{I}_\mathcal{A}^q(l_i).
\]
The left hand side of this last equation is a Laurent polynomial with coefficients in~$\mathbb{Z}[q,q^{-1}]$. Let us write $[\mathbb{I}_{\mathcal{A}}^1(l_i)]^H$ for the highest term of $\mathbb{I}_{\mathcal{A}}^1(l_i)$. We can impose a lexicographic total ordering $\geq$ on the set of all commutative Laurent monomials in $X_1,\dots,X_n$ so that $[\mathbb{I}_{\mathcal{A}}^1(l_i)]^H$ is indeed the highest term of $\mathbb{I}_{\mathcal{A}}^1(l_i)$ with respect to this total ordering. These highest terms are distinct, so we may assume 
\[
[\mathbb{I}_{\mathcal{A}}^1(l_1)]^H > [\mathbb{I}_{\mathcal{A}}^1(l_2)]^H > \dots > [\mathbb{I}_{\mathcal{A}}^1(l_k)]^H.
\]
Consider the expression $c^\omega(l,l';l_1)[\mathbb{I}_{\mathcal{A}}^q(l_1)]^H$. It cannot cancel with any other term in the sum, so we must have $c^\omega(l,l';l_1)\in\mathbb{Z}[q,q^{-1}]$. It follows that the sum 
\[
\sum_{i=2}^kc^\omega(l,l';l_i)\mathbb{I}_\mathcal{A}^q(l_i)
\]
is a Laurent polynomial with coefficients in~$\mathbb{Z}[q,q^{-1}]$. Arguing as before, we see that $c^\omega(l,l';l_2)\in\mathbb{Z}[q,q^{-1}]$. Continuing in this way, we see that all $c^\omega(l,l';l_i)$ lie in $\mathbb{Z}[q,q^{-1}]$. This completes the proof.
\end{proof}

\section{Duality map for the quantum symplectic double}

\subsection{Realization of the symplectic double}

We have associated a quantum cluster algebra $\mathcal{A}$ to the disk $S$. One can likewise associate a quantum cluster algebra $\mathcal{A}^\circ$ to the disk $S^\circ$ obtained from~$S$ by reversing its orientation. If $T$ is any ideal triangulation of $S$, then there is a corresponding triangulation $T^\circ$ of~$S^\circ$. We will write $\Lambda_{T^\circ}$ and $M_{T^\circ}$, respectively, for the compatibility matrix and toric frame corresponding to the triangulation~$T^\circ$. We will write $\mathbf{e}_i$ for the standard basis vector in $\mathbb{Z}^m$ associated to the edge $i$ of $T$ or the corresponding edge of $T^\circ$. Note that we have 
\[
\Lambda_{T^\circ}(\mathbf{e}_i,\mathbf{e}_j)=-\Lambda_T(\mathbf{e}_i,\mathbf{e}_j)
\]
for all $i$ and~$j$. Let $\mathcal{F}$ be the ambient skew-field of the quantum cluster algebra $\mathcal{A}$, and let $\mathcal{F}^\circ$ be the skew-field of the quantum cluster algebra $\mathcal{A}^\circ$. Each of these skew-fields is a two-sided module over the ring $\mathbb{Z}[\omega,\omega^{-1}]$, and we will consider the following elements of the tensor product $\mathcal{F}\otimes_{\mathbb{Z}[\omega,\omega^{-1}]}\mathcal{F}^\circ$.

\begin{definition}
\label{def:tensorgenerators}
Let $T$ be an ideal triangulation of $S$. For any $i\in I$ and any $j\in J$, we define elements $B_i=B_{i;T}$ and $X_j=X_{j;T}$ of $\mathcal{F}\otimes_{\mathbb{Z}[\omega,\omega^{-1}]}\mathcal{F}^\circ$ by the formulas 
\begin{align*}
X_j &= M_T\big(\sum_s\varepsilon_{js}\mathbf{e}_s\big)\otimes1, \\
B_i &= M_T(-\mathbf{e}_i)\otimes M_{T^\circ}(\mathbf{e}_i).
\end{align*}
In addition, we will use the notation $q=\omega^4$.
\end{definition}

Let us denote by $\mathcal{G}$ the subalgebra of $\mathcal{F}\otimes_{\mathbb{Z}[\omega,\omega^{-1}]}\mathcal{F}^\circ$ consisting of Laurent polynomials in the variables $B_i$ for $i\in I$ whose coefficients are rational expressions in the $X_j$ with coefficients in $\mathbb{Z}[\omega,\omega^{-1}]$. It is easy to check that the $B_i$ for $i\in I-J$ are central elements of this algebra, and in what follows it will be important to consider the quotient $\mathcal{F}_{\mathcal{D}}=\mathcal{G}/\mathcal{I}$ where $\mathcal{I}$ is the two-sided ideal of $\mathcal{G}$ generated by $B_i-1$ for $i\in I-J$.

The following result relates the elements $X_j$ and $B_j$ to the generators appearing in the definition of the quantum symplectic double.

\begin{proposition}
\label{prop:commutation}
The elements $B_j$ and $X_j$~($j\in J$) satisfy the following commutation relations:
\[
X_iX_j=q^{2\varepsilon_{ij}}X_jX_i, \quad B_iB_j=B_jB_i, \quad X_iB_j=q^{2\delta_{ij}}B_jX_i.
\]
\end{proposition}

\begin{proof}
The proof of the first relation is identical to the proof of Proposition~\ref{prop:Xcommutation}. To prove the second relation, observe that 
\begin{align*}
B_iB_j &= M_T(-\mathbf{e}_i)M_T(-\mathbf{e}_j)\otimes M_{T^\circ}(\mathbf{e}_i)M_{T^\circ}(\mathbf{e}_j) \\
&= \omega^{-2(\Lambda_T(\mathbf{e}_i,\mathbf{e}_j)+\Lambda_{T^\circ}(\mathbf{e}_i,\mathbf{e}_j))} M_T(-\mathbf{e}_j)M_T(-\mathbf{e}_i)\otimes M_{T^\circ}(\mathbf{e}_j)M_{T^\circ}(\mathbf{e}_i) \\
&= \omega^{-2(\Lambda_T(\mathbf{e}_i,\mathbf{e}_j)+\Lambda_{T^\circ}(\mathbf{e}_i,\mathbf{e}_j))} B_jB_i
\end{align*}
and 
\[
\Lambda_T(\mathbf{e}_i,\mathbf{e}_j)+\Lambda_{T^\circ}(\mathbf{e}_i,\mathbf{e}_j)=0.
\]
Therefore $B_iB_j=B_jB_i$. Finally, we have 
\begin{align*}
X_iB_j &= M_T\big(\sum_s\varepsilon_{is}\mathbf{e}_s\big)M_T(-\mathbf{e}_j)\otimes M_{T^\circ}(\mathbf{e}_j) \\
&= \omega^{-2\Lambda_T(\sum_s\varepsilon_{is}\mathbf{e}_s,-\mathbf{e}_j)} M_T(-\mathbf{e}_j)M_T\big(\sum_s\varepsilon_{is}\mathbf{e}_s\big)\otimes M_{T^\circ}(\mathbf{e}_j) \\
&= \omega^{-2\Lambda_T(\sum_s\varepsilon_{is}\mathbf{e}_s,-\mathbf{e}_j)} B_jX_i.
\end{align*}
By the compatibility condition, we have 
\begin{align*}
\Lambda_T\big(\sum_s\varepsilon_{is}\mathbf{e}_s,-\mathbf{e}_j\big) &= -\sum_s\varepsilon_{is}\Lambda_T(\mathbf{e}_s,\mathbf{e}_j) \\
&= -\sum_sb_{si}\lambda_{sj} \\
&= -4\delta_{ij}.
\end{align*}
Therefore $X_iB_j=\omega^{8\delta_{ij}}B_jX_i=q^{2\delta_{ij}}B_jX_i$.
\end{proof}

In the following result, $T'$ denotes the triangulation obtained from $T$ by a flip of the edge~$k$.

\begin{proposition}
\label{prop:Btrans}
For each $i$, let $B_i'=B_{i;T'}$. Then 
\[
B_i'=
\begin{cases}
(qX_k\mathbb{B}_k^++\mathbb{B}_k^-)B_k^{-1}(1+q^{-1}X_k)^{-1} & \mbox{if } i=k \\
B_i & \mbox{if } i\neq k.
\end{cases}
\]
\end{proposition}

\begin{proof}
By Proposition~\ref{prop:exchangetoricframe}, we have 
\begin{align*}
M_{T'}(\mathbf{e}_k) &= M_T\big(-\mathbf{e}_k+\sum_i[\varepsilon_{ki}]_+\mathbf{e}_i\big) + M_T\big(-\mathbf{e}_k+\sum_i[-\varepsilon_{ki}]_+\mathbf{e}_i\big) \\
&= M_T\big(\sum_i\varepsilon_{ki}\mathbf{e}_i-\mathbf{e}_k+\sum_i[-\varepsilon_{ki}]_+\mathbf{e}_i\big) + M_T\big(-\mathbf{e}_k+\sum_i[-\varepsilon_{ki}]_+\mathbf{e}_i\big) \\
&= \omega^\alpha M_T\big(\sum_i\varepsilon_{ki}\mathbf{e}_i\big) M_T(-\mathbf{e}_k) M_T\big(\sum_i[-\varepsilon_{ki}]_+\mathbf{e}_i\big) \\
&\qquad\qquad + \omega^\beta M_T(-\mathbf{e}_k) M_T\big(\sum_i[-\varepsilon_{ki}]_+\mathbf{e}_i\big)
\end{align*}
where we have written 
\[
\alpha = \Lambda_T\big(\sum_i\varepsilon_{ki}\mathbf{e}_i-\mathbf{e}_k, \sum_i[-\varepsilon_{ki}]_+\mathbf{e}_i\big) + \Lambda_T\big(\sum_i\varepsilon_{ki}\mathbf{e}_i,-\mathbf{e}_k\big)
\]
and 
\[
\beta = \Lambda_T\big(-\mathbf{e}_k, \sum_i[-\varepsilon_{ki}]_+\mathbf{e}_i\big).
\]
Factoring, we obtain 
\[
M_{T'}(\mathbf{e}_k) = \left(1+\omega^{\alpha-\beta}M_T\big(\sum_i\varepsilon_{ki}\mathbf{e}_i\big)\right) \omega^\beta M_T(-\mathbf{e}_k) M_T\big(\sum_i[-\varepsilon_{ki}]_+\mathbf{e}_i\big).
\]
A straightforward calculation using the compatibility condition shows $\alpha-\beta=-4$. Substituting this into the last expression, we see that 
\[
M_{T'}(\mathbf{e}_k)\otimes1 = (1+q^{-1}X_k)\omega^\beta \left(M(-\mathbf{e}_k)\otimes1\right) \left(M_T\big(\sum_i[-\varepsilon_{ki}]_+\mathbf{e}_i\big)\otimes1\right).
\]
On the other hand, we have 
\begin{align*}
1\otimes M_{(T')^\circ}(\mathbf{e}_k) &= 1\otimes M_{T^\circ}\big(-\mathbf{e}_k+\sum_i[\varepsilon_{ki}]_+\mathbf{e}_i\big) + 1\otimes M_{T^\circ}\big(-\mathbf{e}_k+\sum_i[-\varepsilon_{ki}]_+\mathbf{e}_i\big) \\
&= M_T\big(\sum_i\varepsilon_{ki}\mathbf{e}_i - \sum_i[\varepsilon_{ki}]_+\mathbf{e}_i + \sum_i[-\varepsilon_{ki}]_+\mathbf{e}_i\big)\otimes M_{T^\circ}\big(-\mathbf{e}_k+\sum_i[\varepsilon_{ki}]_+\mathbf{e}_i\big) \\
&\qquad+ M_T\big(-\sum_i[-\varepsilon_{ki}]_+\mathbf{e}_i + \sum_i[-\varepsilon_{ki}]_+\mathbf{e}_i\big) \otimes M_{T^\circ}\big(-\mathbf{e}_k+\sum_i[-\varepsilon_{ki}]_+\mathbf{e}_i\big).
\end{align*}
By the properties of toric frames, this equals 
\begin{align*}
&\omega^\gamma M_T\big(\sum_i\varepsilon_{ki}\mathbf{e}_i\big) M_T\big( - \sum_i[\varepsilon_{ki}]_+\mathbf{e}_i\big) M_T\big(\sum_i[-\varepsilon_{ki}]_+\mathbf{e}_i\big) \otimes M_{T^\circ}\big(\sum_i[\varepsilon_{ki}]_+\mathbf{e}_i\big)M_{T^\circ}(-\mathbf{e}_k) \\
&\qquad+\omega^\delta M_T\big(-\sum_i[-\varepsilon_{ki}]_+\mathbf{e}_i\big) M_T\big(\sum_i[-\varepsilon_{ki}]_+\mathbf{e}_i\big) \otimes M_{T^\circ}\big(\sum_i[-\varepsilon_{ki}]_+\mathbf{e}_i\big) M_{T^\circ}(-\mathbf{e}_k)
\end{align*}
for certain exponents $\gamma$ and $\delta$. Factoring, we obtain 
\begin{align*}
&\biggr(\omega^{\gamma-\delta} M_T\big(\sum_i\varepsilon_{ki}\mathbf{e}_i\big) M_T\big( - \sum_i[\varepsilon_{ki}]_+\mathbf{e}_i\big) \otimes M_{T^\circ}\big(\sum_i[\varepsilon_{ki}]_+\mathbf{e}_i\big) \\
&\qquad+ M_T\big(-\sum_i[-\varepsilon_{ki}]_+\mathbf{e}_i\big) \otimes M_{T^\circ}\big(\sum_i[-\varepsilon_{ki}]_+\mathbf{e}_i\big)\biggr) \omega^\delta M_T\big(\sum_i[-\varepsilon_{ki}]_+\mathbf{e}_i\big) \otimes M_{T^\circ}(-\mathbf{e}_k) \\
&=(\omega^{\gamma-\delta}X_k\mathbb{B}_k^++\mathbb{B}_k^-) \omega^\delta \left(1\otimes M_{T^\circ}(-\mathbf{e}_k)\right) \biggr(M_T\big(\sum_i[-\varepsilon_{ki}]_+\mathbf{e}_i\big)\otimes1\biggr).
\end{align*}
A straightforward calculation shows that $\delta=\beta$ and $\gamma-\delta=4$. Substituting these into the last expression, we see that 
\[
1\otimes M_{(T')^\circ}(\mathbf{e}_k) = (qX_k\mathbb{B}_k^++\mathbb{B}_k^-) \omega^\beta \left(1\otimes M_{T^\circ}(-\mathbf{e}_k)\right) \biggr(M_T\big(\sum_i[-\varepsilon_{ki}]_+\mathbf{e}_i\big)\otimes1\biggr).
\]
Finally, by the definition of $B_k'$, we have 
\begin{align*}
B_k' &= M_{T'}(-\mathbf{e}_k)\otimes M_{(T')^\circ}(\mathbf{e}_k) =(1\otimes M_{(T')^\circ}(\mathbf{e}_k)) (M_{T'}(\mathbf{e}_k)\otimes1)^{-1} \\
&= (qX_k\mathbb{B}_k^++\mathbb{B}_k^-) \omega^\beta \left(1\otimes M_{T^\circ}(-\mathbf{e}_k)\right) \biggr(M_T\big(\sum_i[-\varepsilon_{ki}]_+\mathbf{e}_i\big)\otimes1\biggr) \\
&\qquad\cdot \biggr(M_T\big(\sum_i[-\varepsilon_{ki}]_+\mathbf{e}_i\big)\otimes1\biggr)^{-1} \left(M_T(-\mathbf{e}_k)\otimes1\right)^{-1} \omega^{-\beta} (1+q^{-1}X_k)^{-1} \\
&= (qX_k\mathbb{B}_k^++\mathbb{B}_k^-)B_k^{-1}(1+q^{-1}X_k)^{-1}.
\end{align*}
For $i\neq k$, we clearly have $B_i'=B_i$. This completes the proof.
\end{proof}

By Proposition~\ref{prop:Xtrans}, there is a similar statement describing the transformation of the~$X_j$ under a flip of the triangulation.

\subsection{Construction of the map}

We are now ready to define the map $\mathbb{I}_{\mathcal{D}}^q$. To do this, let $l$ be a doubled lamination on~$S_{\mathcal{D}}$, represented by a collection of closed curves of weight~1. Let us deform the curves of $l$ so that they intersect the image of $\partial S$ in the minimal number of points and then cut along the image of $\partial S$. In this way, we obtain two collections of arcs drawn on the surfaces $S$ and~$S^\circ$. We deform each arc on $S$ by dragging its endpoints along the boundary in the counterclockwise direction until they hit the marked points. We deform each arc on $S^\circ$ by dragging its endpoints in the clockwise direction until they hit the marked points. Write $\mathcal{C}$ for the resulting collection of curves on $S$ and $\mathcal{C}^\circ$ for the resulting collection of curves on~$S^\circ$. There is an ideal triangulation $T_{\mathcal{C}}$ of $S$ such that each curve in $\mathcal{C}$ coincides with an edge of~$T_{\mathcal{C}}$, and there is an ideal triangulation $T_{\mathcal{C}^\circ}$ of $S^\circ$ such that each curve in $\mathcal{C}^\circ$ coincides with an edge of $T_{\mathcal{C}^\circ}$. Let 
\begin{align*}
\mathbf{w} &= (w_1,\dots,w_m), \\
\mathbf{w}^\circ &= (w_1^\circ,\dots,w_m^\circ)
\end{align*}
be integral vectors where $w_i$ is the total weight of curves homotopic to the edge $i$ of~$T_{\mathcal{C}}$ and $w_i^\circ$ is the total weight of curves homotopic to the edge $i$ of $T_{\mathcal{C}^\circ}$.

Fix an ideal triangulation $T$ of $S$. Then each $c\in\mathcal{C}$ determines an integral $\mathbf{g}$-vector~$\mathbf{g}_c$. The triangulation $T$ determines a corresponding ideal triangulation $T^\circ$ of~$S^\circ$, and so we can likewise associate to each $c^\circ\in\mathcal{C}^\circ$ an integral vector $\mathbf{g}_{c^\circ}$. One can show that the number 
\[
N_l\coloneqq\Lambda_T(\mathbf{g}_{\mathcal{C}},\mathbf{g}_{\mathcal{C}^\circ}),
\]
where $\mathbf{g}_{\mathcal{C}}=\sum_{c\in\mathcal{C}}\mathbf{g}_c$ and $\mathbf{g}_{\mathcal{C}^\circ}=\sum_{c^\circ\in\mathcal{C}^\circ}\mathbf{g}_{c^\circ}$, is independent of the choice of triangulation~$T$.

\begin{definition}
We will write $\mathbb{I}_{\mathcal{D}}^q(l)$ for the element of $\mathcal{F}\otimes_{\mathbb{Z}[\omega,\omega^{-1}]}\mathcal{F}^\circ$ given by 
\[
\mathbb{I}_{\mathcal{D}}^q(l)=\omega^{-N_l}\cdot[\mathcal{C}]^{-1}\otimes[\mathcal{C}^\circ].
\]
Here we are regarding $\mathcal{C}$ and $\mathcal{C}^\circ$ as simple multicurves on $S$ and $S^\circ$, respectively, and we are writing $[\mathcal{C}]$ and $[\mathcal{C}^\circ]$ for the corresponding classes in the skein algebra.
\end{definition}

Our goal is to show that the above definition provides a map $\mathbb{I}_{\mathcal{D}}^q:\mathcal{D}(\mathbb{Z}^t)\rightarrow\mathcal{D}^q$.

\begin{theorem}
\label{thm:Dcanonicalrational}
Let $T$ be an ideal triangulation of $S$, and let $B_i$,~$X_j$ for $i\in I$ and $j\in J$ be defined as above. Then for any $l\in\mathcal{D}_{PGL_2,S}(\mathbb{Z}^t)$, the element $\mathbb{I}_{\mathcal{D}}^q(l)$ is a Laurent polynomial in the $B_i$ whose coefficients are rational expressions in the $X_j$ with coefficients in~$\mathbb{Z}_{\geq0}[q,q^{-1}]$.
\end{theorem}

\begin{proof}
By applying the relations from Definition~\ref{def:skein}, we can write 
\[
\mathbb{I}_{\mathcal{D}}^q(l) = \omega^{-N_l}\cdot\omega^\alpha\prod_{i=1}^mM_{T_{\mathcal{C}}}(\mathbf{e}_i)^{-w_i} \otimes\omega^{\alpha^\circ}\prod_{i=1}^mM_{T_{\mathcal{C}^\circ}}(\mathbf{e}_i)^{w_i^\circ}
\]
where
\[
\alpha=\sum_{i<j}\Lambda_{T_{\mathcal{C}}}(\mathbf{e}_i,\mathbf{e}_j)w_iw_j
\]
and 
\[
\alpha^\circ=\sum_{i<j}\Lambda_{T_{\mathcal{C}^\circ}}(\mathbf{e}_i,\mathbf{e}_j)w_i^\circ w_j^\circ.
\]
Then Corollary~\ref{cor:Fpositivity} implies 
\[
\mathbb{I}_{\mathcal{D}}^q(l) = \omega^{-N_l+\alpha+\alpha^\circ}\cdot\prod_{i=1}^m\left(F_i\cdot M_T(\mathbf{g}_i)\right)^{-w_i}\otimes \prod_{i=1}^m\left(F_i^\circ\cdot M_{T^\circ}(\mathbf{g}_i^\circ)\right)^{w_i^\circ}.
\]
Each $F_i$ in the last line denotes a polynomial in the expressions $M_T(\sum_j\varepsilon_{kj}\mathbf{e}_j)$ with coefficients in $\mathbb{Z}_{\geq0}[\omega^4,\omega^{-4}]$, and each $F_i^\circ$ denotes a polynomial in the expressions $M_{T^\circ}(\sum_j\varepsilon_{kj}\mathbf{e}_j)$ with coefficients in $\mathbb{Z}_{\geq0}[\omega^4,\omega^{-4}]$. The $\mathbf{g}_i$ and $\mathbf{g}_i^\circ$ are integral vectors. We have 
\begin{align*}
\Lambda_T\big(\sum_i\varepsilon_{ki}\mathbf{e}_i,\mathbf{e}_j\big) = \sum_i\varepsilon_{ki}\lambda_{ij} = \sum_i b_{ik}\lambda_{ij} = 4\delta_{kj}
\end{align*}
so that 
\[
M_T\big(\sum_i\varepsilon_{ki}\mathbf{e}_i\big)M_T(\mathbf{e}_j) = \omega^{-8\delta_{kj}} M_T(\mathbf{e}_j)M_T\big(\sum_i\varepsilon_{ki}\mathbf{e}_i\big).
\]
Similarly, we have 
\[
M_{T^\circ}\big(\sum_i\varepsilon_{ki}\mathbf{e}_i\big)M_{T^\circ}(\mathbf{e}_j) = \omega^{8\delta_{kj}} M_{T^\circ}(\mathbf{e}_j)M_{T^\circ}\big(\sum_i\varepsilon_{ki}\mathbf{e}_i\big).
\]
It follows that we can rearrange the factors in the last expression for $\mathbb{I}_{\mathcal{D}}^q(l)$ to get 
\begin{align*}
\mathbb{I}_{\mathcal{D}}^q(l) &= \omega^{-N_l+\alpha+\alpha^\circ} P\cdot\prod_{i=1}^m M_T(\mathbf{g}_i)^{-w_i}\otimes Q\cdot\prod_{i=1}^m M_{T^\circ}(\mathbf{g}_i^\circ)^{w_i^\circ} \\
&= \omega^{-N_l+\alpha+\alpha^\circ}(P\otimes Q)\cdot\big(\prod_{i=1}^m M_T(\mathbf{g}_i)^{-w_i}\otimes \prod_{i=1}^m M_{T^\circ}(\mathbf{g}_i^\circ)^{w_i^\circ}\big)
\end{align*}
where $P$ is a rational function in the expressions $M_T(\sum_i\varepsilon_{kj}\mathbf{e}_j)$ with coefficients in $\mathbb{Z}_{\geq0}[\omega^4,\omega^{-4}]$ and $Q$ is a polynomial in the expressions $M_{T^\circ}(\sum_i\varepsilon_{kj}\mathbf{e}_j)$ with coefficients in $\mathbb{Z}_{\geq0}[\omega^4,\omega^{-4}]$. We claim that the factor $P\otimes Q$ is a Laurent polynomial in the $B_j$ whose coefficients are rational expressions in the~$X_j$ with coefficients in~$\mathbb{Z}_{\geq0}[q,q^{-1}]$. Indeed, $P\otimes 1$ is a rational function in the $X_j$, while $1\otimes Q$ is a polynomial in the expressions 
\begin{align*}
1\otimes M_{T^\circ}\big(\sum_i\varepsilon_{ki}\mathbf{e}_i\big) &= M_T\big(\sum_i\varepsilon_{ki}\mathbf{e}_i-\sum_i\varepsilon_{ki}\mathbf{e}_i\big) \otimes M_{T^\circ}\big(\sum_i\varepsilon_{ki}\mathbf{e}_i\big) \\
&= M_T\big(\sum_i\varepsilon_{ki}\mathbf{e}_i\big)M_T\big(-\sum_i\varepsilon_{ki}\mathbf{e}_i\big) \otimes M_{T^\circ}\big(\sum_i\varepsilon_{ki}\mathbf{e}_i\big) \\
&= X_k\mathbb{B}_k.
\end{align*}
Hence the product 
\[
P\otimes Q=(P\otimes1)\cdot(1\otimes Q)
\]
is a rational expression as claimed. By Lemma~6.2 in~\cite{Tran}, we have $\Lambda_{T_\mathcal{C}}(\mathbf{e}_i,\mathbf{e}_j)=\Lambda_T(\mathbf{g}_i,\mathbf{g}_j)$, and so we can write 
\[
\mathbb{I}_{\mathcal{D}}(l) = \omega^{-N_l}(P\otimes Q)\cdot\left(M_T(-\mathbf{g}_{\mathcal{C}})\otimes M_{T^\circ}(\mathbf{g}_{\mathcal{C}^\circ})\right).
\]
We have 
\[
M_T(-\mathbf{g}_{\mathcal{C}^\circ})\otimes M_{T^\circ}(\mathbf{g}_{\mathcal{C}^\circ})=\prod_{i=1}^mB_i^{g_i}
\]
for some integral exponents $g_i$, and so 
\begin{align*}
\mathbb{I}_{\mathcal{D}}(l) &= \omega^{-N_l}(P\otimes Q)\cdot\left(M_T(-\mathbf{g}_{\mathcal{C}}) M_T(\mathbf{g}_{\mathcal{C}^\circ})\otimes1\right)\cdot\prod_{i=1}^mB_i^{g_i} \\
&=(P\otimes Q)\cdot\left(M_T(-\mathbf{g}_{\mathcal{C}} + \mathbf{g}_{\mathcal{C}^\circ})\otimes1\right)\cdot\prod_{i=1}^mB_i^{g_i}.
\end{align*}
Arguing as in the proof of Lemma~\ref{lem:factorization}, one shows that the factor $M_T(-\mathbf{g}_{\mathcal{C}} + \mathbf{g}_{\mathcal{C}^\circ})\otimes1$ is a monomial in the $X_j$ with coefficients in $\mathbb{Z}_{\geq0}[\omega^4,\omega^{-4}]$. This completes the proof.
\end{proof}

Thus we see that $\mathbb{I}_{\mathcal{D}}(l)$ is an element of the algebra $\mathcal{G}$ introduced following Definition~\ref{def:tensorgenerators}. By the properties established in Propositions~\ref{prop:commutation}, \ref{prop:Btrans}, and~\ref{prop:Xtrans}, we can regard its image in the quotient $\mathcal{F}_{\mathcal{D}}$ as an element of the algebra $\mathcal{D}_{PGL_2,S}^q$. Thus we have constructed a canonical map $\mathbb{I}_{\mathcal{D}}^q:\mathcal{D}_{PGL_2,S}(\mathbb{Z}^t)\rightarrow\mathcal{D}_{PGL_2,S}^q$ as desired.

\appendix
\chapter{Derivation of the classical mutation formulas}
\label{ch:DerivationOfTheClassicalMutationFormulas}

In this appendix, we calculate the action of the map $\mu_k^q$ of Definition~\ref{def:doublemutation} on the generators~$B_i$ and~$X_i$. As a result of this calculation, we obtain a proof of Theorems~\ref{thm:introclassicallimit} and~\ref{thm:transformation}.

\begin{appendixlemma}
\label{lem:commuteseries}
Let $\varphi(x)$ be any formal power series in $x$. Then 
\[
\varphi(X_k)B_i = B_i\varphi(q^{2\delta_{ik}}X_k), \quad
\varphi(\widehat{X}_k)B_i = B_i\varphi(q^{2\delta_{ik}}\widehat{X}_k),
\]
and
\[
\varphi(X_k)X_i = X_i\varphi(q^{2\varepsilon_{ki}}X_k).
\]
\end{appendixlemma}

\begin{proof}
Write $\varphi(x)=a_0+a_1x+a_2x^2+\dots$. Then the relation $X_kX_i=q^{2\varepsilon_{ki}}X_iX_k$ implies 
\begin{align*}
\varphi(X_k)X_i &= a_0X_i+a_1X_kX_i+a_2X_k^2X_i+\dots \\
&= X_ia_0+X_ia_1q^{2\varepsilon_{ki}}X_k+X_ia_2q^{4\varepsilon_{ki}}X_k^2+\dots \\
&= X_i\left(a_0+a_1(q^{2\varepsilon_{ki}}X_k)+a_2(q^{2\varepsilon_{ki}}X_k)^2+\dots\right) \\
&= X_i\varphi(q^{2\varepsilon_{ki}}X_k).
\end{align*}
This proves the third relation. The first two relations are proved similarly. For these one uses the facts $X_kB_i=q^{2\delta_{ik}}B_iX_k$ and $\widehat{X}_kB_i=q^{2\delta_{ik}}B_i\widehat{X}_k$.
\end{proof}

\begin{appendixproposition}
\label{prop:quantumautomorphism}
The map $\mu_k^\sharp$ is given on generators by the formulas 
\[
\mu_k^\sharp(B_i)=
\begin{cases}
B_k(1+qX_k)(1+q\widehat{X}_k)^{-1} & \mbox{if } i=k \\
B_i & \mbox{if } i\neq k
\end{cases}
\]
and 
\[
\mu_k^\sharp(X_i)=
\begin{cases}
X_i(1+qX_k)(1+q^3X_k)\dots(1+q^{2|\varepsilon_{ik}|-1}X_k) & \mbox{if } \varepsilon_{ik}\leq0 \\
X_i{\left((1+q^{-1}X_k)(1+q^{-3}X_k)\dots(1+q^{1-2|\varepsilon_{ik}|}X_k)\right)}^{-1} & \mbox{if } \varepsilon_{ik}\geq0.
\end{cases}
\]
\end{appendixproposition}

\begin{proof}
By Lemma~\ref{lem:commuteseries}, we have 
\begin{align*}
\mu_k^\sharp(B_i) &= \Psi^q(X_k)\Psi^q(\widehat{X}_k)^{-1}B_i\Psi^q(\widehat{X}_k)\Psi^q(X_k)^{-1} \\
&= B_i\Psi^q(q^{2\delta_{ik}}X_k)\Psi^q(q^{2\delta_{ik}}\widehat{X}_k)^{-1}\Psi^q(\widehat{X}_k)\Psi^q(X_k)^{-1}. 
\end{align*}
If $i\neq k$, then this equals $B_i$ as desired. Suppose on the other hand that $i=k$. Using the identity $\Psi^q(q^2x)=(1+qx)\Psi^q(x)$ and commutativity of~$X_i$ and~$\widehat{X}_k$, we can rewrite this last expression as
\begin{align*}
\mu_k^\sharp(B_i) &= B_k\Psi^q(q^2X_k)\Psi^q(q^2\widehat{X}_k)^{-1}\Psi^q(\widehat{X}_k)\Psi^q(X_k)^{-1} \\
&= B_k(1+qX_k)\Psi^q(X_k) \left((1+q\widehat{X}_k)\Psi^q(\widehat{X}_k)\right)^{-1} \Psi^q(\widehat{X}_k)\Psi^q(X_k)^{-1} \\
&= B_k(1+qX_k)(1+q\widehat{X}_k)^{-1}.
\end{align*}
This completes the proof of the first formula.

To prove the second formula, observe that by Lemma~\ref{lem:commuteseries} and the commutativity of~$X_i$ and~$\widehat{X}_k$ we have 
\begin{align*}
\mu_k^\sharp(X_i) &= \Psi^q(X_k)\Psi^q(\widehat{X}_k)^{-1}X_i\Psi^q(\widehat X_k)\Psi^q(X_k)^{-1} \\
&= \Psi^q(X_k)X_i\Psi^q(X_k)^{-1} \\
&= X_i\Psi^q(q^{-2\varepsilon_{ik}}X_k)\Psi^q(X_k)^{-1}.
\end{align*}
If $\varepsilon_{ik}\leq0$, then the identity $\Psi^q(q^2x)=(1+qx)\Psi^q(x)$ implies 
\begin{align*}
\mu_k^\sharp(X_i) &= X_i\Psi^q(q^2q^{2|\varepsilon_{ik}|-2}X_k)\Psi^q(X_k)^{-1} \\
&= X_i(1+q^{2|\varepsilon_{ik}|-1}X_k)\Psi^q(q^{2|\varepsilon_{ik}|-2}X_k)\Psi^q(X_k)^{-1} \\
&= X_i(1+q^{2|\varepsilon_{ik}|-1}X_k)(1+q^{2|\varepsilon_{ik}|-3}X_k)\Psi^q(q^{2|\varepsilon_{ik}|-4}X_k)\Psi^q(X_k)^{-1} \\
&= \dots \\
&= X_i(1+q^{2|\varepsilon_{ik}|-1}X_k)\dots(1+q^3X_k)(1+qX_k)
\end{align*}
as desired. If $\varepsilon_{ik}\geq0$, there is a similar argument using the identity $\Psi^q(q^{-2}x)=(1+q^{-1}x)^{-1}\Psi^q(x)$.
\end{proof}

\begin{appendixlemma}
\label{lem:dualmutation}
Let $(\Lambda,\{e_i\},(\cdot,\cdot))$ be a seed. If we mutate this seed in the direction of a basis vector~$e_k$, then the basis $\{f_i\}$ for $\Lambda^\vee$ transforms to a new basis $\{f_i'\}$ given by 
\[
f_i'=
\begin{cases}
-f_i+\sum_j[-\varepsilon_{kj}]_+f_j & \mbox{if } i=k \\
f_i & \mbox{if } i\neq k
\end{cases}
\]
\end{appendixlemma}

\begin{proof}
The transformation $e_i\mapsto e_i'$ can be represented by an explicit matrix by Definition~\ref{def:altmutation}, and the transformation rule appearing in the lemma is represented by the transpose of this matrix.
\end{proof}

\begin{appendixproposition}
\label{prop:quantummutation}
The map $\mu_k'$ is given on generators by the formulas 
\[
\mu_k'(B_i')=
\begin{cases}
\mathbb{B}_k^-/B_k & \mbox{if } i=k \\
B_i & \mbox{if } i\neq k
\end{cases}
\]
and 
\[
\mu_k'(X_i')=
\begin{cases}
X_k^{-1} & \mbox{if } i=k \\
q^{-[\varepsilon_{ik}]_+\varepsilon_{ik}}X_iX_k^{[\varepsilon_{ik}]_+} & \mbox{if } i\neq k.
\end{cases}
\]
\end{appendixproposition}

\begin{proof}
Let $Y_v$ be the generator of $\mathcal{D}_{\mathbf{i}}$ associated to $v\in\Lambda_{\mathcal{D}}$ as in Definition~\ref{def:quantumtorus}. By Lemma~\ref{lem:dualmutation} and the fact that $(f_i,f_j)_{\mathcal{D}}=0$ for all $i$,~$j$, we have 
\begin{align*}
\mu_k'(B_k') &= Y_{-f_k+\sum_j[-\varepsilon_{kj}]_+f_j} \\
&= B_k^{-1}\mathbb{B}_k^{-}.
\end{align*}
Similarly we have 
\begin{align*}
\mu_k'(X_i') &= Y_{e_i+[\varepsilon_{ik}]_+e_k} \\
&= q^{-[\varepsilon_{ik}]_+(e_i,e_k)_{\mathcal{D}}}Y_{e_i}Y_{[\varepsilon_{ik}]_+e_k} \\
&= q^{-[\varepsilon_{ik}]_+\varepsilon_{ik}}X_iX_k^{[\varepsilon_{ik}]_+}
\end{align*}
for $i\neq k$ and $\mu_k'(X_i')=Y_{-e_k}=X_k^{-1}$ for $i=k$.
\end{proof}

\begin{appendixtheorem}
The map $\mu_k^q$ is given on generators by the formulas 
\[
\mu_k^q(B_i')=
\begin{cases}
(qX_k\mathbb{B}_k^++\mathbb{B}_k^-)B_k^{-1}(1+q^{-1}X_k)^{-1} & \mbox{if } i=k \\
B_i & \mbox{if } i\neq k
\end{cases}
\]
and
\[
\mu_k^q(X_i')=
\begin{cases}
X_i\prod_{p=0}^{|\varepsilon_{ik}|-1}(1+q^{2p+1}X_k) & \mbox{if } \varepsilon_{ik}\leq0 \mbox{ and } i\neq k \\
X_iX_k^{\varepsilon_{ik}}\prod_{p=0}^{\varepsilon_{ik}-1}(X_k+q^{2p+1})^{-1} & \mbox{if } \varepsilon_{ik}\geq0 \mbox{ and } i\neq k \\
X_k^{-1} & \mbox{if } i=k.
\end{cases}
\]
\end{appendixtheorem}

\begin{proof}
By our formulas for $\mu_k'$ and $\mu_k^\sharp$, we have 
\begin{align*}
\mu_k^q(B_k') &= \mu_k^\sharp(\mu_k'(B_k')) = \mu_k^\sharp(\mathbb{B}_k^-/B_k) \\
&= \mathbb{B}_k^-\left(B_k(1+qX_k)(1+q\widehat{X}_k)^{-1}\right)^{-1} \\
&= \mathbb{B}_k^-(1+q\widehat{X}_k)(1+qX_k)^{-1}B_k^{-1}.
\end{align*}
By the definitions of $\mathbb{B}_k^-$ and $\widehat{X}_k$, this equals
\begin{align*}
\mu_k^q(B_k') &= \prod_{i|\varepsilon_{ki}<0}B_i^{-\varepsilon_{ki}}\big(1+qX_k\prod_iB_i^{\varepsilon_{ki}}\big)(1+qX_k)^{-1}B_k^{-1} \\
&= \big(\prod_{i|\varepsilon_{ki}<0}B_i^{-\varepsilon_{ki}}+qX_k\prod_{i|\varepsilon_{ki}>0}B_i^{\varepsilon_{ki}}\big)(1+qX_k)^{-1}B_k^{-1} \\
&=(qX_k\mathbb{B}_k^++\mathbb{B}_k^-)(1+qX_k)^{-1}B_k^{-1} \\
&=(qX_k\mathbb{B}_k^++\mathbb{B}_k^-)B_k^{-1}(1+q^{-1}X_k)^{-1}.
\end{align*}
From this calculation, we easily obtain the formula describing the action of the map $\mu_k^q$ on the generators $B_i'$.

On the other hand, if $\varepsilon_{ik}\leq0$ and $i\neq k$, then 
\begin{align*}
\mu_k^q(X_i') &= \mu_k^\sharp(\mu_k'(X_i')) = \mu_k^\sharp(X_i) \\
&= X_i(1+qX_k)(1+q^3X_k)\dots(1+q^{2|\varepsilon_{ik}|-1}X_k) \\
&= X_i\prod_{p=0}^{|\varepsilon_{ik}|-1}(1+q^{2p+1}X_k).
\end{align*}
If $\varepsilon_{ik}\geq0$ and $i\neq k$, then 
\begin{align*}
\mu_k^q(X_i') &= \mu_k^\sharp(\mu_k'(X_i')) = \mu_k^\sharp(q^{-\varepsilon_{ik}^2}X_iX_k^{\varepsilon_{ik}}) \\
&= q^{-\varepsilon_{ik}^2}X_i{\left((1+q^{-1}X_k)(1+q^{-3}X_k)\dots(1+q^{1-2|\varepsilon_{ik}|}X_k)\right)}^{-1}X_k^{\varepsilon_{ik}} \\
&= q^{-\varepsilon_{ik}^2}X_i{\left(q^{-1}(X_k+q)q^{-3}(X_k+q^3)\dots q^{1-2|\varepsilon_{ik}|}(X_k+q^{2|\varepsilon_{ik}|-1})\right)}^{-1}X_k^{\varepsilon_{ik}}.
\end{align*}
Applying the identity $1+3+5+\dots+(2N-1)=N^2$, this becomes 
\begin{align*}
\mu_k^q(X_i') &= X_i{\left((X_k+q)(X_k+q^3)\dots (X_k+q^{2|\varepsilon_{ik}|-1})\right)}^{-1}X_k^{\varepsilon_{ik}} \\
&= X_iX_k^{\varepsilon_{ik}} (X_k+q)^{-1}(X_k+q^3)^{-1}\dots (X_k+q^{2|\varepsilon_{ik}|-1})^{-1} \\
&= X_iX_k^{\varepsilon_{ik}}\prod_{p=0}^{\varepsilon_{ik}-1}(X_k+q^{2p+1})^{-1}.
\end{align*}
Finally, if $i=k$, then we have 
\[
\mu_k^q(X_i') = \mu_k^\sharp(\mu_k'(X_i')) = \mu_k^\sharp(X_k^{-1}) = X_k^{-1}.
\]
This proves the second formula.
\end{proof}

\chapter{Review of cluster algebras}
\label{ch:ReviewOfClusterAlgebras}

\section{General theory of cluster algebras}

In this appendix, we review the results we need from the theory of cluster algebras. All of the definitions and results of this section come from~\cite{FZIV}.

\begin{definition}
If $(\mathbb{P},\oplus,\cdot)$ is any semifield, then the group ring $\mathbb{ZP}$ is an integral domain, and hence we can form its fraction field $\mathbb{QP}$. We will write~$\mathcal{F}$ for a field isomorphic to the field of rational functions in~$n$ independent variables with coefficients in~$\mathbb{QP}$.
\end{definition}

\begin{definition}
A \emph{labeled seed} $(\mathbf{x},\mathbf{y},\mathbf{B})$ consists of a skew-symmetrizable $n\times n$ integer matrix~$\mathbf{B}=(b_{ij})$, an $n$-tuple $\mathbf{y}=(y_1,\dots,y_n)$ of elements of $\mathbb{P}$, and an $n$-tuple $\mathbf{x}=(x_1,\dots,x_n)$ of elements of $\mathcal{F}$ such that the $x_i$ are algebraically independent over~$\mathbb{QP}$ and $\mathcal{F}=\mathbb{QP}(x_1,\dots,x_n)$.
\end{definition}

\begin{definition}
Let $(\mathbf{x},\mathbf{y},\mathbf{B})$ be a labeled seed, and let $k\in\{1,\dots,n\}$. Then we define a new seed $(\mathbf{x}',\mathbf{y}',\mathbf{B}')$, called the seed obtained by \emph{mutation} in the direction $k$ as follows:
\begin{enumerate}
\item The entries of $\mathbf{B}'=(b_{ij}')$ are given by 
\[
b_{ij}'=
\begin{cases}
-b_{ij} & \mbox{if } k\in\{i,j\} \\
b_{ij}+\frac{|b_{ik}|b_{kj}+b_{ik}|b_{kj}|}{2} & \mbox{if } k\not\in\{i,j\}.
\end{cases}
\]
\item The elements of the $n$-tuple $\mathbf{y}'=(y_1',\dots,y_n')$ are given by 
\begin{align*}
y_j'=
\begin{cases}
y_k^{-1} & \mbox{if } j=k \\
y_jy_k^{[b_{kj}]_+}(y_k\oplus1)^{-b_{kj}} & \mbox{if } j\neq k
\end{cases}
\end{align*}
where we are using the notation $[b]_+=\max(b,0)$.
\item The elements of the $n$-tuple $\mathbf{x}'=(x_1',\dots,x_n')$ are given by 
\begin{align*}
x_j' =
\begin{cases}
\frac{y_k\prod_{i|b_{ik>0}}x_i^{b_{ik}} + \prod_{i|b_{ik<0}}x_i^{-b_{ik}}}{(y_k\oplus1)x_k} & \mbox{if } j=k \\
x_j & \mbox{if } j\neq k.
\end{cases}
\end{align*}
\end{enumerate}
\end{definition}

\begin{definition}
We denote by $\mathbb{T}_n$ an $n$-regular tree with edges labeled by the numbers~$1,\dots,n$ in such a way that the $n$ edges emanating from any vertex have distinct labels. A \emph{cluster pattern} is an assignment of a labeled seed $\Sigma_t=(\mathbf{x}_t,\mathbf{y}_t,\mathbf{B}_t)$ to every vertex $t\in\mathbb{T}_n$ so that if $t$ and $t'$ are vertices connected by an edge labeled~$k$, then $\Sigma_{t'}$ is obtained from $\Sigma_t$ by a mutation in the direction $k$. We will use the following notation for the data of $\Sigma_t$:
\[
\mathbf{x}_t=(x_{1;t},\dots,x_{n;t}), \quad \mathbf{y}_t=(y_{1;t},\dots,y_{n;t}), \quad \mathbf{B}_t=(b_{ij}^t).
\]
\end{definition}

\begin{definition}
Given a cluster pattern $t\mapsto(\mathbf{x}_t,\mathbf{y}_t,\mathbf{B}_t)$, we form the set of all cluster variables in all seeds of the cluster pattern:
\[
\mathcal{S}=\{x_{l;t}:t\in\mathbb{T}_n,1\leq l\leq n\}.
\]
Then the \emph{cluster algebra} with coefficients in $\mathbb{P}$ is the $\mathbb{ZP}$-subalgebra of $\mathcal{F}$ generated by elements in this set $\mathcal{S}$.
\end{definition}

\section{$F$-polynomials}

\begin{definition}
A \emph{cluster algebra with principal coefficients} at a vertex $t_0\in\mathbb{T}_n$ is a cluster algebra with $\mathbb{P}=\mathrm{Trop}(y_1,\dots,y_n)$ and $\mathbf{y}_{t_0}=(y_1,\dots,y_n)$.
\end{definition}

Let~$\mathcal{A}$ be a cluster algebra with principal coefficients, and denote the data of the initial seed $\Sigma_{t_0}=(\mathbf{x}_{t_0},\mathbf{y}_{t_0},\mathbf{B}_{t_0})$ by 
\[
\mathbf{x}_{t_0}=(x_1,\dots,x_n), \quad \mathbf{y}_{t_0}=(y_1,\dots,y_n), \quad \mathbf{B}_{t_0}=(b_{ij}^0).
\]
By iterating the exchange relations, we can express any cluster variable $x_{l;t}$ as a subtraction-free rational function of the variables $x_1,\dots,x_n,y_1,\dots,y_n$. We will denote this subtraction-free rational function by 
\[
X_{l;t}\in\mathbb{Q}_\mathrm{sf}(x_1,\dots,x_n,y_1,\dots,y_n).
\]
We will denote by $F_{l;t}\in\mathbb{Q}_\mathrm{sf}(y_1,\dots,y_n)$ the subtraction-free rational function obtained from~$X_{l;t}$ by specializing all the $x_i$ to 1. Thus 
\[
F_{l;t}(y_1,\dots,y_n)=X_{l;t}(1,\dots,1,y_1,\dots,y_n).
\]
By the Laurent phenomenon theorem of Fomin and Zelevinsky, $X_{l;t}$ is a Laurent polynomial in $x_1,\dots,x_n$ whose coefficients are integral polynomials in $y_1,\dots,y_n$, and $F_{l;t}$ is an integral polynomial in $y_1,\dots,y_n$.

\begin{definition}
The expressions $X_{l;t}$ and $F_{l;t}$ are called $X$- and \emph{$F$-polynomials}, respectively.
\end{definition}

There is a $\mathbb{Z}^n$-grading on the ring $\mathbb{Z}[x_1^{\pm1},\dots,x_n^{\pm1},y_1,\dots,y_n]$ given by the formulas 
\[
\deg(x_i) = \mathbf{e}_i, \quad \deg(y_i) = -\mathbf{b}_j^0
\]
where $\mathbf{e}_i$ is the $i$th standard basis vector in $\mathbb{Z}^n$ and $\mathbf{b}_j^0=\sum_i b_{ij}^0\mathbf{e}_i$ is the $j$th column of $\mathbf{B}_{t_0}$. By a result of \cite{FZIV}, each $X$-polynomial is homogeneous with respect to this $\mathbb{Z}^n$-grading.

\begin{definition}
The degree 
\[
\mathbf{g}_{l;t}=
\left( \begin{array}{ccc}
g_1 \\
\vdots \\
g_n \end{array} \right)
=\deg(X_{l;t})\in\mathbb{Z}^n
\]
is called the \emph{$\mathbf{g}$-vector} of the cluster variable $x_{l;t}$.
\end{definition}

The notions of $F$-polynomials and $\mathbf{g}$-vectors are important because they allow us to express an arbitrary cluster variable in terms of the variables $x_1,\dots,x_n,y_1,\dots,y_n$ of the initial seed $\Sigma_{t_0}$. To see this, we need one more piece of notation. It is a fact that any subtraction-free rational identity that holds in the semifield $\mathbb{Q}_\mathrm{sf}(u_1,\dots,u_n)$ will remain valid when we replace the $u_i$ by elements of an arbitrary semifield~$\mathbb{P}$. Thus if $f$ is a subtraction-free rational expression in $u_1,\dots,u_n$, there is a well defined element $f|_\mathbb{P}(y_1,\dots,y_n)$ of $\mathbb{P}$ obtained by evaluating $f$ at $y_1,\dots,y_n\in\mathbb{P}$.

\begin{proposition}[\cite{FZIV}, Corollary 6.3]
\label{prop:FZ63}
Let $\mathcal{A}$ be a cluster algebra over an arbitrary semifield $\mathbb{P}$ of coefficients. Then a cluster variable $x_{l;t}$ can be expressed in terms of the cluster variables at an initial seed as 
\[
x_{l;t}=\frac{F_{l;t}|_\mathcal{F}(\widehat{y}_1,\dots,\widehat{y}_n)}{F_{l;t}|_{\mathbb{P}}(y_1,\dots,y_n)}x_1^{g_1}\dots x_n^{g_n}
\]
where 
\[
\widehat{y}_j=y_j\prod_i x_i^{b_{ij}^0}.
\]
\end{proposition}

\section{Cluster algebras associated to surfaces}
\label{sec:ClusterAlgebrasAssociatedToSurfaces}

In~\cite{FST}, Fomin, Shapiro, and Thurston discuss the relationship between cluster algebras and the combinatorics of decorated surfaces. The idea of~\cite{FST} is to associate to a decorated surface $S$ a corresponding cluster algebra. This cluster algebra is defined in such a way that each seed corresponds to a ``tagged triangulation'' of the surface~$S$. An ordinary ideal triangulation is a special case of a tagged triangulation provided there are no self-folded triangles. Fomin, Shapiro, and Thurston assume that the surface $S$ is not a sphere with one, two, or three punctures, a monogon with zero or one puncture, or a bigon or triangle without punctures. According to Lemma~2.13 of~\cite{FST}, such a surface always admits an ideal triangulation $T$ with no self-folded triangles.

If $T$ is an ideal triangulation of $S$ with no self-folded triangles, then we get an exchange matrix $b_{ij}=\varepsilon_{ji}$ ($i,j\in J$), indexed by the internal edges of $T$. To each internal edge $i$ of $T$, we associate variables $x_i$ and~$y_i$. This defines a labeled seed, and hence a cluster algebra. This is the cluster algebra that Fomin, Shapiro, and Thurston associate to the surface~$S$.

If $c$ is any arc on $S$ which is an internal edge for some ideal triangulation and does not cut out a once punctured monogon, then there is a cluster variable $x_c$ in this cluster algebra corresponding to~$c$. In particular, this means that for any such arc $c$ on $S$ there is an associated $F$-polynomial 
\[
F_c(y_1,\dots,y_n)
\]
and a $\mathbf{g}$-vector $\mathbf{g}_c$.

As part of their work on the positivity conjecture for cluster algebras from surfaces, Musiker, Schiffler, and Williams~\cite{MSW1} gave a formula for computing the $\mathbf{g}$-vector associated to an arc. Their construction associates, to any arc $c$, a graph $\bar{G}_{T,c}$ in the plane with labeled edges. This graph is obtained by gluing together ``tiles'' of the form
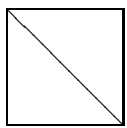
\[
\xy /l1.3pc/:
{\xypolygon4"A"{~:{(2,0):}}};
{"A2"\PATH~={**@{-}}'"A4"};
\endxy
\]

Indeed, suppose $c$ is an arc on a triangulated unpunctured surface. (We refer the reader to \cite{MSW1} for the case of a surface with punctures, which is similar.) Assume that this arc is not an edge of the triangulation. The illustration below shows an example of such a curve on a disk with ten marked points.
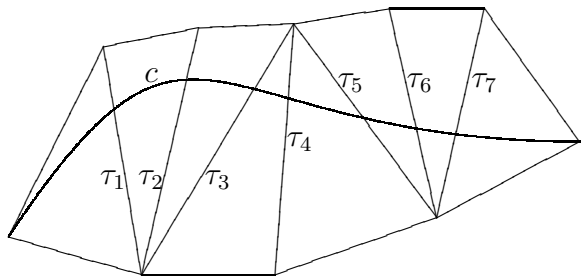
\[
\xy /l6pc/:
(3,0.5)*{}="1";
(2.5,-0.5)*{}="2";
(2,-0.6)*{}="3";
(1.5,-0.62)*{}="4";
(1,-0.7)*{}="5";
(0.5,-0.7)*{}="6";
(0,0)*{}="7";
(0.75,0.4)*{}="8";
(1.6,0.7)*{}="9";
(2.3,0.7)*{}="10";
{"1"\PATH~={**@{-}}'"2"},
{"2"\PATH~={**@{-}}'"3"},
{"3"\PATH~={**@{-}}'"4"},
{"4"\PATH~={**@{-}}'"5"},
{"5"\PATH~={**@{-}}'"6"},
{"6"\PATH~={**@{-}}'"7"},
{"7"\PATH~={**@{-}}'"8"},
{"8"\PATH~={**@{-}}'"9"},
{"9"\PATH~={**@{-}}'"10"},
{"10"\PATH~={**@{-}}'"1"},
{"10"\PATH~={**@{-}}'"2"},
{"10"\PATH~={**@{-}}'"3"},
{"10"\PATH~={**@{-}}'"4"},
{"9"\PATH~={**@{-}}'"4"},
{"8"\PATH~={**@{-}}'"4"},
{"8"\PATH~={**@{-}}'"5"},
{"8"\PATH~={**@{-}}'"6"},
"1";"7" **\crv{(2,-1) & (2,0)},
(2.25,-0.35)*{c};
(2.45,0.2)*{\tau_1};
(2.25,0.2)*{\tau_2};
(1.9,0.2)*{\tau_3};
(1.47,0)*{\tau_4};
(1.2,-0.31)*{\tau_5};
(0.84,-0.31)*{\tau_6};
(0.5,-0.31)*{\tau_7};
\endxy
\]

Choose an orientation for $c$, and label the arcs that $c$ crosses in order by $\tau_{i_1},\dots,\tau_{i_d}$. For any index~$j$, let $\Delta_{j-1}$ and $\Delta_j$ be the two triangles on either side of $\tau_{i_j}$. Then we can associate a tile $G_j$ as above to each $\tau_{i_j}$. It consists of two triangles with edges labeled as in~$\Delta_{j-1}$ and~$\Delta_j$ and glued together along the edge labeled $\tau_{i_j}$ so that the orientations of these triangles both agree or both disagree with those of~$\Delta_{j-1}$ and~$\Delta_j$. Note that there are two possible planar embeddings of the graph $G_j$.

The two arcs $\tau_{i_j}$ and $\tau_{i_{j+1}}$ are edges of the triangle $\Delta_j$. We will write $\tau_{[c_j]}$ for the third arc in this triangle. Then we can recursively glue together the tiles in order from~1 to~$d$ so that $G_{j+1}$ and $G_j$ are glued along the edges labeled $\tau_{[c_j]}$ and if the orientation of the triangles of~$G_j$ agrees with the orientation of $\Delta_{j-1}$ and $\Delta_j$ then the orientation of $G_{j+1}$ disagrees with the orientation of $\Delta_{j}$ and $\Delta_{j+1}$, and vice versa. We denote the resulting graph by $\bar{G}_{T,c}$.

For example, the graph $\bar{G}_{T,c}$ corresponding to the above example is
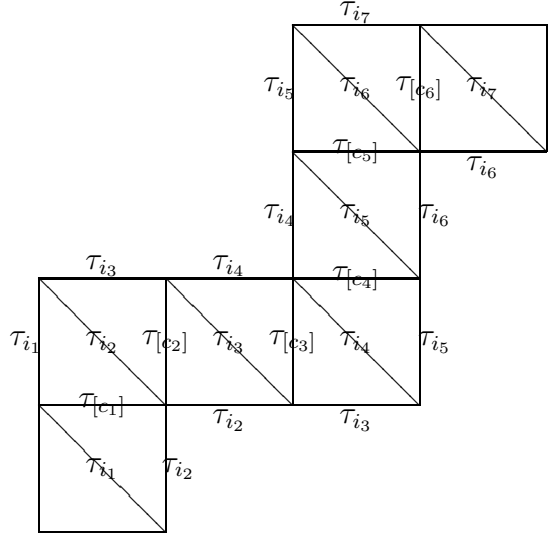
\[
\xy /l8pc/:
(1,-1)*{}="00";
(1.5,-1)*{}="10";
(2,-1)*{}="20";
(2.5,-1)*{}="30";
(3,-1)*{}="40";
(1,-0.5)*{}="01";
(1.5,-0.5)*{}="11";
(2,-0.5)*{}="21";
(2.5,-0.5)*{}="31";
(3,-0.5)*{}="41";
(1,0)*{}="02";
(1.5,0)*{}="12";
(2,0)*{}="22";
(2.5,0)*{}="32";
(3,0)*{}="42";
(1,0.5)*{}="03";
(1.5,0.5)*{}="13";
(2,0.5)*{}="23";
(2.5,0.5)*{}="33";
(3,0.5)*{}="43";
(2.5,1)*{}="34";
(3,1)*{}="44";
{"00"\PATH~={**@{-}}'"10"},
{"10"\PATH~={**@{-}}'"20"},
{"01"\PATH~={**@{-}}'"11"},
{"11"\PATH~={**@{-}}'"21"},
{"12"\PATH~={**@{-}}'"22"},
{"22"\PATH~={**@{-}}'"32"},
{"32"\PATH~={**@{-}}'"42"},
{"13"\PATH~={**@{-}}'"23"},
{"23"\PATH~={**@{-}}'"33"},
{"33"\PATH~={**@{-}}'"43"},
{"34"\PATH~={**@{-}}'"44"},
{"00"\PATH~={**@{-}}'"01"},
{"10"\PATH~={**@{-}}'"11"},
{"11"\PATH~={**@{-}}'"12"},
{"12"\PATH~={**@{-}}'"13"},
{"20"\PATH~={**@{-}}'"21"},
{"21"\PATH~={**@{-}}'"22"},
{"22"\PATH~={**@{-}}'"23"},
{"32"\PATH~={**@{-}}'"33"},
{"33"\PATH~={**@{-}}'"34"},
{"42"\PATH~={**@{-}}'"43"},
{"43"\PATH~={**@{-}}'"44"},
{"10"\PATH~={**@{-}}'"01"},
{"20"\PATH~={**@{-}}'"11"},
{"21"\PATH~={**@{-}}'"12"},
{"22"\PATH~={**@{-}}'"13"},
{"32"\PATH~={**@{-}}'"23"},
{"42"\PATH~={**@{-}}'"33"},
{"43"\PATH~={**@{-}}'"34"},
(3.05,0.25)*{\tau_{i_1}};
(2.75,-0.05)*{\tau_{i_3}};
(2.25,-0.05)*{\tau_{i_4}};
(2.05,-0.25)*{\tau_{i_4}};
(2.05,-0.75)*{\tau_{i_5}};
(1.75,-1.05)*{\tau_{i_7}};
(1.25,-0.44)*{\tau_{i_6}};
(1.44,-0.25)*{\tau_{i_6}};
(1.44,0.25)*{\tau_{i_5}};
(1.75,0.56)*{\tau_{i_3}};
(2.25,0.56)*{\tau_{i_2}};
(2.45,0.75)*{\tau_{i_2}};
(2.75,0.75)*{\tau_{i_1}};
(2.75,0.25)*{\tau_{i_2}};
(2.25,0.25)*{\tau_{i_3}};
(1.75,0.25)*{\tau_{i_4}};
(1.75,-0.25)*{\tau_{i_5}};
(1.75,-0.75)*{\tau_{i_6}};
(1.25,-0.75)*{\tau_{i_7}};
(2.75,0.5)*{\tau_{[c_1]}};
(2.5,0.25)*{\tau_{[c_2]}};
(2,0.25)*{\tau_{[c_3]}};
(1.75,0)*{\tau_{[c_4]}};
(1.75,-0.5)*{\tau_{[c_5]}};
(1.5,-0.75)*{\tau_{[c_6]}};
\endxy
\]

Write $G_{T,c}$ for the graph obtained from $\bar{G}_{T,c}$ by removing the diagonal in every tile. Recall that for any graph $G$, a \emph{perfect matching} of $G$ is a collection $P$ of edges such that every vertex of $G$ is incident to exactly one edge in $P$. It is easy to show that the graph~$G_{T,c}$ constructed in~\cite{MSW1} has exactly two perfect matchings consisting only of boundary edges. These perfect matchings are called the \emph{minimal matching} and \emph{maximal matching} and are denoted $P_{-}=P_{-}(G_{T,c})$ and $P_{+}=P_{+}(G_{T,c})$, respectively. In the above example, the maximal matching $P_{+}$ is the matching that contains the horizontal edge at the bottom of the graph~$\bar{G}_{T,c}$.

If the edges of a perfect matching $P$ are labeled $\tau_{j_1},\dots,\tau_{j_r}$, then we define the \emph{weight} $x(P)$ of $P$ as the product 
\[
x(P)=\prod_{s=1}^r x_{\tau_{j_s}}
\]
of the cluster variables associated to $\tau_{j_1},\dots,\tau_{j_r}$. Similarly, if $\tau_{i_1},\dots,\tau_{i_d}$ is the sequence of arcs in $T$ that $c$ crosses, then we define the \emph{crossing monomial} $\cross(T,c)$ of $c$ with respect to $T$ as the product 
\[
\cross(T,c)=\prod_{s=1}^d x_{\tau_{i_s}}.
\]
Note that the arcs $\tau_{i_1},\dots,\tau_{i_d}$ in this definition also appear as the labels on the diagonal edges in the graph $\bar{G}_{T,c}$.

\begin{proposition}[\cite{MSW1}]
Let $c$ be an arc on a decorated surface $S$. Then the $\mathbf{g}$-vector associated to $c$ is given by the formula 
\[
\mathbf{g}_{c}=\deg\left(\frac{x(P_{-})}{\cross(T,c)}\right).
\]
\end{proposition}

\chapter{Review of quantum cluster algebras}
\label{ch:ReviewOfQuantumClusterAlgebras}

\section{General theory of quantum cluster algebras}

Here we review the theory of quantum cluster algebras, following~\cite{BZq,Tran}. Throughout this section, $m$ and $n$ will be positive integers with $m\geq n$.

\begin{definition}
Let $k\in\{1,\dots,n\}$. We say that an $m\times n$ matrix $\mathbf{B}'=(b_{ij}')$ is obtained from an $m\times n$ matrix $\mathbf{B}=(b_{ij})$ by \emph{matrix mutation} in the direction $k$ if the entries of $\mathbf{B}'$ are given by 
\[
b_{ij}'=
\begin{cases}
-b_{ij} & \mbox{if } k\in\{i,j\} \\
b_{ij}+\frac{|b_{ik}|b_{kj}+b_{ik}|b_{kj}|}{2} & \mbox{if } k\not\in\{i,j\}.
\end{cases}
\]
In this case, we write $\mu_k(\mathbf{B})=\mathbf{B}'$.
\end{definition}

\begin{definition}
\label{def:compatibility}
Let $\mathbf{B}=(b_{ij})$ be an $m\times n$ integer matrix, and let  $\Lambda=(\lambda_{ij})$ be a skew-symmetric $m\times m$ integer matrix. We say that the pair $(\Lambda,\mathbf{B})$ is \emph{compatible} if for each $j\in\{1,\dots,n\}$ and $i\in\{1,\dots,m\}$, we have
\[
\sum_{k=1}^mb_{kj}\lambda_{ki}=\delta_{ij}d_j
\]
for some positive integers $d_j$~($j\in\{1,\dots,n\}$). Equivalently, the product $\mathbf{B}^t\Lambda$ equals the $n\times m$ matrix $(D|0)$ where $D$ is the $n\times n$ diagonal matrix with diagonal entries $d_1,\dots,d_n$.
\end{definition}

Let $k\in\{1,\dots,n\}$ and choose a sign $\epsilon\in\{\pm1\}$. Denote by $E_\epsilon$ the $m\times m$ matrix with entries given by 
\[
e_{ij}=
\begin{cases}
\delta_{ij} & \mbox{if } j\neq k \\
-1 & \mbox{if } i=j=k \\
\max(0,-\epsilon b_{ik}) & \mbox{if } i\neq j=k
\end{cases}
\]
and set 
\[
\Lambda'=E_\epsilon^t\Lambda E_\epsilon.
\]

\begin{proposition}[\cite{BZq}, Proposition~3.4]
The matrix $\Lambda'$ is skew-symmetric and independent of the sign $\epsilon$. Moreover, $(\Lambda',\mu_k(\mathbf{B}))$ is a compatible pair.
\end{proposition}

\begin{definition}[\cite{BZq}, Definition~3.5]
Let $(\Lambda,\mathbf{B})$ be a compatible pair and let $k\in\{1,\dots,n\}$. We say that the pair $(\Lambda',\mu_k(\mathbf{B}))$ is obtained from $(\Lambda,\mathbf{B})$ by \emph{mutation} in the direction~$k$ and write $\mu_k(\Lambda,\mathbf{B})=(\Lambda',\mu_k(\mathbf{B}))$.
\end{definition}

Let $L$ be a lattice of rank $m$ equipped with a skew-symmetric bilinear form $\Lambda:L\times L\rightarrow\mathbb{Z}$. Let $\omega$ be a formal variable. We can associate to these data a quantum torus algebra~$\mathcal{T}$. It is generated over $\mathbb{Q}[\omega,\omega^{-1}]$ by variables $A^v$~($v\in\Lambda$) subject to the commutation relations 
\[
A^{v_1}A^{v_2}=\omega^{-\Lambda(v_1,v_2)}A^{v_1+v_2}.
\]
In the literature on quantum cluster algebras, this quantum torus algebra is typically called a \emph{based quantum torus}, and the parameter is denoted $q^{-1/2}$, rather than $\omega$. (See~\cite{BZq,Muller,Tran} for example.) This quantum torus algebra has a noncommutative fraction field which we denote~$\mathcal{F}$.

\begin{definition}
A \emph{toric frame} in $\mathcal{F}$ is a mapping $M:\mathbb{Z}^m\rightarrow\mathcal{F}-\{0\}$ of the form 
\[
M(v)=\phi(A^{\eta(v)})
\]
where $\phi$ is an automorphism of $\mathcal{F}$ and $\eta:\mathbb{Z}^m\rightarrow L$ is an isomorphism of lattices.
\end{definition}

Note that the image $M(\mathbb{Z}^m)$ of a toric frame is a basis for an isomorphic copy of~$\mathcal{T}$ in~$\mathcal{F}$. We have the relations 
\begin{align*}
M(v_1)M(v_2) &= \omega^{-\Lambda_M(v_1,v_2)}M(v_1+v_2), \\
M(v_1)M(v_2) &= \omega^{-2\Lambda_M(v_1,v_2)}M(v_2)M(v_1), \\
M(v)^{-1} &= M(-v), \\
M(0) &=1
\end{align*}
where the form $\Lambda_M$ on $\mathbb{Z}^m$ is obtained from $\Lambda$ using the isomorphism $\eta$.

\begin{definition}
A \emph{quantum seed} is a pair $(M,\mathbf{B})$ where $M$ is a toric frame in~$\mathcal{F}$ and $\mathbf{B}$ is an $m\times n$ integer matrix such that $(\Lambda_M,\mathbf{B})$ is a compatible pair.
\end{definition}

\begin{definition}
\label{def:mutatedtoricframe}
Let $(M,\mathbf{B})$ be a quantum seed and write $\mathbf{B}=(b_{ij})$. For any index $k\in\{1,\dots,k\}$ and any sign $\epsilon\in\{\pm1\}$, we define a mapping $M':\mathbb{Z}^m\rightarrow\mathcal{F}-\{0\}$ by the formulas 
\begin{align*}
M'(v) &= \sum_{p=0}^{v_k}\binom{v_k}{p}_{\omega^{-d_k}}M(E_\epsilon v+\epsilon p b^k), \\
M'(-v) &= M'(v)^{-1}
\end{align*}
where $v=(v_1,\dots,v_m)\in\mathbb{Z}^m$ is such that $v_k\geq0$ and $b^k$ denotes the $k$th column of~$\mathbf{B}$. Here the $t$-binomial coefficient is given by 
\[
\binom{r}{p}_t=\frac{(t^r-t^{-r})\dots(t^{r-p+1}-t^{-r+p-1})}{(t^p-t^{-p})\dots(t-t^{-1})}.
\]
\end{definition}

\begin{proposition}[\cite{BZq}, Proposition~4.7]
The mapping $M'$ satisfies the following properties:
\begin{enumerate}
\item The mapping $M'$ is a toric frame which is independent of the sign~$\epsilon$.

\item The pair $(\Lambda_{M'},\mu_k(\mathbf{B}))$ is compatible and obtained from the pair $(\Lambda_M,\mathbf{B})$ by mutation in the direction~$k$.

\item The pair $(M',\mu_k(\mathbf{B}))$ is a quantum seed.
\end{enumerate}
\end{proposition}

\begin{definition}
Let $(M,\mathbf{B})$ be a quantum seed, and let $k\in\{1,\dots,n\}$. Let $M'$ be the mapping from Definition~\ref{def:mutatedtoricframe}, and let $\mathbf{B}'=\mu_k(\mathbf{B})$. Then we say that the quantum seed $(M',\mathbf{B}')$ is obtained from $(M,\mathbf{B})$ by \emph{mutation} in the direction~$k$.
\end{definition}

\begin{proposition}[\cite{BZq}, Proposition~4.9]
\label{prop:exchangetoricframe}
Let $(M,\mathbf{B})$ be a quantum seed, and suppose that $(M',\mathbf{B}')$ is obtained from $(M,\mathbf{B})$ by mutation in the direction~$k$. Then 
\[
M'(\mathbf{e}_k)=M\big(-\mathbf{e}_k+\sum_{i=1}^m[b_{ik}]_+\mathbf{e}_i\big) + M\big(-\mathbf{e}_k+\sum_{i=1}^m[-b_{ik}]_+\mathbf{e}_i\big)
\]
and $M'(\mathbf{e}_i)=M(\mathbf{e}_i)$ for $i\neq k$.
\end{proposition}

\begin{definition}
Denote by $\mathbb{T}_n$ an $n$-regular tree with edges labeled by the numbers $1,\dots,n$ in such a way that the $n$ edges emanating from any vertex have distinct labels. A \emph{quantum cluster pattern} is an assignment of a quantum seed $\Sigma_t=(M_t,\mathbf{B}_t)$ to each vertex $t\in\mathbb{T}_n$ so that if $t$ and $t'$ are vertices connected by an edge labeled~$k$, then $\Sigma_{t'}$ is obtained from $\Sigma_t$ by a mutation in the direction~$k$.
\end{definition}

Given a quantum cluster pattern, let us define $A_{i;t}=M_t(\mathbf{e}_i)$. For $i\in\{n+1,\dots,m\}$, we have $A_{i;t}=A_{i;t'}$ for all $t$,~$t'\in\mathbb{T}_n$, so we may omit one of the subscripts and write $A_i=A_{i;t}$ for all $t\in\mathbb{T}_n$. Let 
\[
\mathcal{S}=\{A_{i;t}:i\in\{1,\dots,n\},t\in\mathbb{T}_n\}.
\]

\begin{definition}[\cite{BZq}, Definition~4.12]
Given a quantum cluster pattern $t\mapsto(M_t,\mathbf{B}_t)$, the associated \emph{quantum cluster algebra} $\mathcal{A}$ is the $\mathbb{Z}[\omega^{\pm1},A_{n+1}^{\pm1},\dots,A_m^{\pm1}]$-subalgebra of the ambient skew-field~$\mathcal{F}$ generated by elements of $\mathcal{S}$.
\end{definition}

\section{Quantum $F$-polynomials}

One of the important tools that we apply in our construction of the map $\mathbb{I}_{\mathcal{D}}^q$ is the notion of a quantum $F$-polynomial from~\cite{Tran}. This extends Fomin and Zelevinsky's notion of $F$-polynomial~\cite{FZIV} to the noncommutative setting and allows us to express a generator $A_{j;t}$ of a quantum cluster algebra in terms of the generators associated with an initial seed.

\begin{theorem}[\cite{Tran}, Theorem~5.3]
\label{thm:quantumF}
Let $(M_0,\mathbf{B}_0)$ be an initial quantum seed in a quantum cluster algebra $\mathcal{A}$ and write $\mathbf{B}_0=(b_{ij})$. Then there exists an integer $\lambda_{j;t}\in\mathbb{Z}$ and a polynomial~$F_{j;t}$ in the variables 
\[
Y_j=M_0\big(\sum_i b_{ij}\mathbf{e}_i\big)
\]
with coefficients in $\mathbb{Z}[\omega,\omega^{-1}]$ such that the cluster variable $A_{j;t}\in\mathcal{A}$ is given by 
\[
A_{j;t}=\omega^{\lambda_{j;t}}F_{j;t}\cdot M_0(\mathbf{g}_{j;t})
\]
where $\mathbf{g}_{j;t}$ is an integer vector called the extended $\mathbf{g}$-vector of~$A_{j;t}$.
\end{theorem}

The polynomial $F_{j;t}$ appearing in the theorem is known as a \emph{quantum $F$-polynomial}. For the quantum cluster algebras considered in this paper, we have the following refinement of Theorem~\ref{thm:quantumF}.

\begin{corollary}
\label{cor:Fpositivity}
Let $\mathcal{A}$ be a quantum cluster algebra of type $\mathrm{A}_n$, and suppose the matrix~$D$ appearing in the compatibility condition of Definition~\ref{def:compatibility} is four times the identity. Then there exists a polynomial $F_{j;t}$ in the variables $Y_1,\dots,Y_n$ with coefficients in $\mathbb{Z}_{\geq0}[\omega^4,\omega^{-4}]$ such that the cluster variable $A_{j;t}\in\mathcal{A}$ is given by 
\[
A_{j;t}=F_{j;t}\cdot M_0(\mathbf{g}_{j;t})
\]
where $\mathbf{g}_{j;t}$ denotes the extended $\mathbf{g}$-vector of~$A_{j;t}$.
\end{corollary}

\begin{proof}
For algebras of type~$\mathrm{A}_n$, it is known that each classical $F$-polynomial has nonzero constant term (for example by~\cite{MSW1}). Hence, by~\cite{Tran}, Theorem~6.1, we have $\lambda_{j;t}=0$ in Theorem~\ref{thm:quantumF}. By~\cite{Tran}, Theorem~7.4, we know that the coefficients of $F_{j;t}$ are Laurent polynomials in $\omega^4$ with positive integral coefficients.
\end{proof}

\end{document}